\pdfoutput=1
\RequirePackage{ifpdf}
\ifpdf 
\documentclass[pdftex]{sigma}
\else
\documentclass{sigma}
\fi

\numberwithin{equation}{section}

\newtheorem{Theorem}{Theorem}[section]
\newtheorem{Corollary}[Theorem]{Corollary}
\newtheorem{Lemma}[Theorem]{Lemma}
\newtheorem{Proposition}[Theorem]{Proposition}
{ \theoremstyle{definition}

\newtheorem{Remark}[Theorem]{Remark} }

\usepackage[all]{xy}

\newcommand{\ds}{\displaystyle}
\newcommand{\bK}{\mathbf{K}}
\newcommand{\be}{\mathbf{e}}
\newcommand{\bk}{\mathbf{k}}
\newcommand{\bl}{\mathbf{l}}
\newcommand{\bm}{\mathbf{m}}
\newcommand{\bn}{\mathbf{n}}
\newcommand{\bs}{\mathbf{s}}

\newcommand{\bone}{\mathbf{1}}
\newcommand{\BN}{\mathbb{N}}
\newcommand{\BZ}{\mathbb{Z}}
\newcommand{\BR}{\mathbb{R}}
\newcommand{\BC}{\mathbb{C}}
\newcommand{\BH}{\mathbb{H}}
\newcommand{\BK}{\mathbb{K}}
\newcommand{\BO}{\mathbb{O}}
\newcommand{\fa}{\mathfrak{a}}

\newcommand{\fg}{\mathfrak{g}}
\newcommand{\fh}{\mathfrak{h}}
\newcommand{\fk}{\mathfrak{k}}
\newcommand{\fl}{\mathfrak{l}}

\newcommand{\fn}{\mathfrak{n}}
\newcommand{\fp}{\mathfrak{p}}

\newcommand{\fz}{\mathfrak{z}}
\newcommand{\cB}{\mathcal{B}}
\newcommand{\cF}{\mathcal{F}}
\newcommand{\cH}{\mathcal{H}}
\newcommand{\cK}{\mathcal{K}}
\newcommand{\cO}{\mathcal{O}}
\newcommand{\cP}{\mathcal{P}}

\newcommand{\rK}{\mathrm{K}}
\newcommand{\rT}{\mathrm{T}}

\newcommand{\ve}{\varepsilon}
\newcommand{\rank}{\operatorname{rank}}
\newcommand{\End}{\operatorname{End}}
\newcommand{\Hom}{\operatorname{Hom}}
\newcommand{\Ext}{\operatorname{Ext}}
\newcommand{\Tr}{\operatorname{Tr}}
\newcommand{\tr}{\operatorname{tr}}
\newcommand{\Det}{\operatorname{Det}}
\newcommand{\Pf}{\operatorname{Pf}}

\newcommand{\sgn}{\operatorname{sgn}}
\newcommand{\Proj}{\operatorname{Proj}}
\newcommand{\hsum}{\sideset{}{^\oplus}{\sum}}
\newcommand{\hboxtimes}{\mathbin{\hat{\boxtimes}}}
\newcommand{\dottimes}{\mathbin{\dot{\times}}}

\newcommand{\Sym}{\operatorname{Sym}}
\newcommand{\Skew}{\operatorname{Skew}}
\newcommand{\Supp}{\operatorname{Supp}}
\newcommand{\Herm}{\operatorname{Herm}}
\newcommand{\Rest}{\operatorname{Rest}}
\renewcommand{\Re}{\operatorname{Re}}

\begin{document}

\allowdisplaybreaks

\newcommand{\arXivNumber}{1804.07100}

\renewcommand{\PaperNumber}{036}

\FirstPageHeading

\ShortArticleName{Construction of Intertwining Operators}

\ArticleName{Construction of Intertwining Operators \\ between Holomorphic Discrete Series Representations}

\Author{Ryosuke NAKAHAMA}

\AuthorNameForHeading{R.~Nakahama}

\Address{Graduate School of Mathematical Sciences, The University of Tokyo, \\
3-8-1 Komaba Meguro-ku Tokyo 153-8914, Japan}
\Email{\href{mailto:nakahama@ms.u-tokyo.ac.jp}{nakahama@ms.u-tokyo.ac.jp}}

\ArticleDates{Received April 24, 2018, in final form April 02, 2019; Published online May 05, 2019}

\Abstract{In this paper we explicitly construct $G_1$-intertwining operators between holomorphic discrete series representations $\cH$ of a Lie group $G$ and those $\cH_1$ of a subgroup $G_1\subset G$ when $(G,G_1)$ is a symmetric pair of holomorphic type. More precisely, we construct $G_1$-intertwining projection operators from $\cH$ onto $\cH_1$ as differential operators, in the case $(G,G_1)=(G_0\times G_0,\Delta G_0)$ and both $\cH$, $\cH_1$ are of scalar type, and also construct $G_1$-intertwining embedding operators from $\cH_1$ into $\cH$ as infinite-order differential operators, in the case $G$ is simple, $\cH$ is of scalar type,
and $\cH_1$ is multiplicity-free under a maximal compact subgroup $K_1\subset K$. In the actual computation we make use of series expansions of integral kernels and the result of Faraut--Kor\'anyi (1990) or the author's previous result (2016) on norm computation. As an application, we observe the behavior of residues of the intertwining operators, which define the maps from some subquotient modules, when the parameters are at poles.}

\Keywords{branching laws; intertwining operators; symmetry breaking operators; symmetric pairs; holomorphic discrete series representations; highest weight modules}

\Classification{22E45; 43A85; 17C30}

\tableofcontents

\section{Introduction}
The purpose of this paper is to study intertwining operators between a holomorphic discrete series representation of some reductive Lie group~$G$ and that of some reductive subgroup $G_1\subset G$, and write down such an operator explicitly.

Let $G$ be a real reductive Lie group, $G_1$ be a reductive subgroup of $G$, and consider a representation $(\pi,\cH)$ of~$G$. Then it is a fundamental problem to understand how the representation $(\pi,\cH)$ of $G$ behaves
when it is restricted to the subgroup~$G_1$. Recently Kobayashi \cite{K1} proposed a program for such problems in the following three stages.
\begin{description}\itemsep=0pt
\item{(Stage A)} Abstract features of the restriction $\pi|_{G_1}$.
\item{(Stage B)} Branching laws.
\item{(Stage C)} Construction of symmetry breaking operators.
\end{description}
In general, the restriction $\pi|_{G_1}$ may behave wildly, for example, the multiplicity becomes infinite, or it contains continuous spectrum, even if $(G,G_1)$ is a symmetric pair, and $\pi$ is a unitary representation of~$G$. However Kobayashi and his collaborators found conditions for $(G,G_1,\pi)$ that the restriction $\pi|_{G_1}$ behaves nicely, that is, it is discretely decomposable \cite{Kdd1, Kdd2, Kdd3, Kbgg, KyO1, KyO2}, its multiplicity becomes finite or uniformly bounded \cite{Ksf,Kfm1, KM, KtO}, or decomposes multiplicity-freely \cite{Kmf0, Kmf1, Kmf2} (Stage~A). In particular, if $G$ is a reductive Lie group of \textit{Hermitian type}
(i.e., the Riemannian symmetric space $G/K$ has a natural complex structure), $(G,G_1)$ is a symmetric pair of \textit{holomorphic type} (i.e., a symmetric pair such that the embedding map $G_1/K_1\hookrightarrow G/K$ is holomorphic), and $\pi$ is in the nice class of representations, called the \textit{holomorphic discrete series representations} of~$G$, then the restriction $\pi|_{G_1}$ decomposes discretely~\cite{Kdd2, Ma}. Moreover, if the holomorphic discrete series representation $\pi$ is of \textit{scalar type}, then it decomposes multiplicity-freely. Also, under the assumption that $(G,G_1)$ is a symmetric pair of holomorphic type and $\pi$ is a holomorphic discrete series representation, the branching law
\begin{gather*} \pi|_{G_1}\simeq\hsum_{\pi_1\in\hat{G}_1}m(\pi,\pi_1)\pi_1 \end{gather*}
is determined in \cite{Kmf0, Kmf1} (see the survey \cite{Kmf0-1}) (Stage~B). Here $\hat{G}_1$ denotes the unitary dual of~$G_1$, i.e., the set of equivalence classes of unitary representations of~$G_1$, and $m(\pi,\pi_1)\in\BZ_{\ge 0}$. Thus our next interest is to understand the above decomposition explicitly, for example, to construct the $G_1$-intertwining operator between $\pi|_{G_1}$ and $\pi_1$ explicitly (Stage~C). Such problems have been considered by e.g.\ Clerc--Kobayashi--{\O}rsted--Pevzner \cite{CKOP}, Kobayashi--Kubo--Pevzner~\cite{KKP}, Kobayashi--{\O}rsted--Somberg--Sou\v{c}ek~\cite{KOSS}, Kobayashi--Speh \cite{KS1,KS2}, M\"ollers--{\O}rsted--Oshima \cite{MOO} and M\"ollers--Oshima \cite{MO} when $\pi$ are principal series or complementary series representations, and by e.g.\ Ibukiyama--Kuzumaki--Ochiai~\cite{IKO}, Kobayashi--Pevzner~\cite{KP1,KP2} and Peng--Zhang~\cite{PZ} when $\pi$ are holomorphic discrete series representations. The approach used in \cite{KKP,KOSS,KP1,KP2} is called the ``F-method'', in which the explicit intertwining operators are determined by solving certain differential equations. This idea first appeard in~\cite{Kfmeth}. In this paper, we also attack this problem when $\pi$ are holomorphic discrete series representations, but take an approach different from the F-method, namely, by computing some integrals using series expansion.

\looseness=-1 Now we review holomorphic discrete series representations. Let $G$ be a reductive Lie group of Hermitian type, and $K\subset G$ be a maximal compact subgroup with Cartan involution~$\vartheta$. Then there exists a complex subspace $\fp^+\subset\fg^\BC$ in the complexified Lie algebra of $G$ and a~bounded domain $D\subset\fp^+$ such that the Riemannian symmetric space $G/K$ is diffeomorphic to~$D$, and~$G/K$ admits a natural complex structure via this diffeomorphism. Next, let $(\tau,V)$ be a finite-dimensional representation of $\tilde{K}^\BC$, the universal covering group of $K^\BC$, and consider the space of holomorphic sections of the homogeneous vector bundle $\tilde{G}\times_{\tilde{K}}V$ on $G/K$. Then by the Borel embedding, it is isomorphic to the space of $V$-valued holomorphic functions on $D$
\begin{gather*} \Gamma_\cO\big(G/K,\tilde{G}\times_{\tilde{K}}V\big)\simeq \cO(D,V). \end{gather*}
Clearly this admits an action of $\tilde{G}$. If $(\tau,V)$ is sufficiently ``regular'', then $\cO(D,V)$ admits a~$\tilde{G}$-invariant inner product which is given by a converging integral on $D$. In this case the corresponding Hilbert subspace $\cH_\tau(D,V)\subset\cO(D,V)$ admits a~unitary representation. This family of representations is called the \textit{holomorphic discrete series representations}.

We take a subgroup $G_1\subset G$ which is stable under the Cartan involution $\vartheta$ of $G$. We assume that the embedding map $G_1/K_1\hookrightarrow G/K$ of Riemannian symmetric spaces is holomorphic. Let $\fp^+_1:=\fp^+\cap\fg^\BC_1$ be the intersection of $\fp^+$ and the complexfied Lie algebra of $G_1$, and $\fp^+_2:=(\fp^+_1)^\bot\subset\fp^+$ be the orthogonal complement under a suitable inner product on $\fp^+$. We take a finite-dimensional representation $(\rho,W)$ of $\tilde{K}_1^\BC$, and consider the corresponding holomorphic discrete series representation $\cH_{\rho}(D_1,W)$ of $\tilde{G}_1$. Then $\cH_{\rho}(D_1,W)$ appears in the direct summand of $\cH_{\tau}(D,V)|_{\tilde{G}_1}$ if and only if $(\rho,W)$ appears in the irreducible decomposition of $\cP(\fp^+_2)\otimes V$ under~$\tilde{K}_1^\BC$, where $\cP(\fp^+_2)$ is the space of holomorphic polynomials on~$\fp^+_2$~\cite{Kmf1}. Our aim is to write down the $\tilde{G}_1$-intertwining operator between $\cH_{\tau}(D,V)$ and each $\cH_{\rho}(D_1,W)$ explicitly.

We calculate the intertwining operator in the following way. First, we find the kernel function $\hat{\rK}(x;y_1)$ which is $\tilde{G}_1$-invariant in a suitable sense (see Proposition~\ref{kernelfunc}). Then the intertwining operators are given by (see Corollary~\ref{integral_expression})
\begin{alignat*}{3}
& \cH_{\tau}(D,V) \to \cH_{\rho}(D_1,W),\qquad && f \mapsto \big\langle f,\hat{\rK}(\cdot;y_1)\big\rangle_{\cH_\tau(D,V)},&\\
& \cH_{\rho}(D_1,W) \to \cH_{\tau}(D,V),\qquad && g \mapsto \big\langle g,\hat{\rK}(x;\cdot)^*\big\rangle_{\cH_\rho(D_1,W)}. &
\end{alignat*}
This gives an integral expression of the intertwining operator, and this step is similar to the method used in~\cite{KS0, KS1, KS2, MOO}. However, it seems to be difficult to get information on the bran\-ching from this expression. Also, in~\cite{KP1} it is proved that the intertwining operator from~$\cH_{\tau}(D,V)$ to~$\cH_{\rho}(D_1,W)$ is always given by a~differential operator (localness theorem), but we cannot derive this fact from our integral expression. Thus we try to rewrite the integral expression to a~differential expression (possibly of infinite order) by substituting $f(x)$ with~${\rm e}^{(x|z)}$, $g(y)$ with ${\rm e}^{(y|w)}$, where $(\cdot|\cdot)$ is a suitable inner product on $\fp^+$. Then we can show that there exists a polynomial $F^*(z_1,z_2)\in\cP(\overline{\fp^+_1\times\fp^+_2},\Hom(V,W))$ and a function $F(x_2;w_1)\in\cO(\fp^+_2\times\overline{D_1},\Hom(W,V))$ such that the intertwining operators are given by (see Theorem~\ref{main})
\begin{alignat*}{3}
&\cH_\tau(D,V)\to\cH_\rho(D_1,W),\qquad && f(x) \mapsto F^*\left.\left(\overline{\frac{\partial}{\partial x_1}},\overline{\frac{\partial}{\partial x_2}}\right)\right|_{x_2=0}f(x_1,x_2), & \\
& \cH_\rho(D_1,W) \to\cH_\tau(D,V),\qquad && g(x_1) \mapsto F\left(x_2;\overline{\frac{\partial}{\partial x_1}}\right)g(x_1).&
\end{alignat*}
The latter operator is of infinite order in general, but we can show that $F\big(x_2;\overline{\frac{\partial}{\partial x_1}}\big)g(x_1)$ converges uniformly on every compact set in some open subset of~$D$, extends holomorphically on whole~$D$, and defines a continuous map between spaces of all holomorphic functions (see Theo\-rems~\ref{conti_ext} and~\ref{extend}). The functions $F$ and $F^*$ are given by an explicit integral, and actual computation of~$F$ and~$F^*$ is performed in Section~\ref{sect_examples} case by case, by using the series expansion of integrands and the result of Faraut--Kor\'anyi~\cite{FK0} or the author's previous result \cite{N2} on norm computation. In this way, we get the explicit intertwining operators of both direction $\cH_\tau(D,V)\rightleftarrows\cH_\rho(D_1,W)$ in the case $(G,G_1)$ is one of
\begin{alignat}{3}
&(U(q,s),U(q,s')\times U(s'')),\qquad & &(\operatorname{SO}^*(2s),\operatorname{SO}^*(2(s-1))\times \operatorname{SO}(2)),& \notag\\
&(\operatorname{SO}^*(2s),U(s-1,1)),\qquad & &(\operatorname{SO}_0(2,2s),U(1,s)), & \label{list_normal_derivative}\\
&(E_{6(-14)},U(1)\times \operatorname{Spin}_0(2,8)), \qquad &&& \notag
\end{alignat}
which are given by normal derivatives (Corollaries~\ref{normal_der_1} and~\ref{normal_der_2}). We also get the projection operators $\cH_\tau(D,V)\rightarrow\cH_\rho(D_1,W)$ in the case $(G,G_1)$ is of the form
\begin{gather*} (G,G_1)=(G_0\times G_0,\Delta G_0), \end{gather*}
where $G_0$ is a simple Lie group of Hermitian type, when both $(\tau,V)$ and $(\rho,W)$ are scalar (and some few other cases) (Theorem~\ref{thm_tensor}), which gives essentially the same result as in~\cite{PZ} (see also, e.g.,~\cite{BCK,OR,P}).
In addition we get the embedding operators $\cH_\rho(D_1,W)\rightarrow\cH_\tau(D,V)$ in the case $(G,G_1)$ is one of
\begin{alignat}{3}
&(\operatorname{Sp}(s,\BR), \operatorname{Sp}(s',\BR)\times \operatorname{Sp}(s'',\BR)),\qquad & &(U(q,s), U(q',s')\times U(q'',s'')),&\notag\\
&(\operatorname{SO}^*(2s), \operatorname{SO}^*(2s')\times \operatorname{SO}^*(2s'')),\qquad & &(\operatorname{Sp}(s,\BR), U(s',s'')),&\notag\\
&(\operatorname{SO}^*(2s), U(s',s'')),\qquad & &(\operatorname{SU}(s,s), \operatorname{Sp}(s,\BR)),&\notag\\
&(\operatorname{SU}(s,s), \operatorname{SO}^*(2s)),\qquad & &(\operatorname{SO}_0(2,n), \operatorname{SO}_0(2,n')\times \operatorname{SO}(n'')), &\notag\\
&(E_{6(-14)},{\rm SL}(2,\BR)\times \operatorname{SU}(1,5)),\qquad & &(E_{6(-14)},U(1)\times \operatorname{SO}^*(10)),&\notag\\
&(E_{6(-14)},\operatorname{SU}(2,4)\times \operatorname{SU}(2)),\qquad & &(E_{7(-25)},{\rm SL}(2,\BR)\times \operatorname{Spin}_0(2,10)),&\notag\\
&(E_{7(-25)},U(1)\times E_{6(-14)}),\qquad & &(E_{7(-25)},\operatorname{SU}(2)\times \operatorname{SO}^*(12)),&\notag\\
&(E_{7(-25)},\operatorname{SU}(2,6)),\qquad &&&\label{list_infinite_order}
\end{alignat}
when $(\tau,V)$ is scalar and $(\rho,W)$ is multiplicity-free under the maximal compact subgroup $K_1\subset G_1$ (or more generally when $(\rho,W)$ satisfies the assumption (\ref{assumption_orth}) given later) (Theo\-rems~\ref{main1}, \ref{main2}, \ref{main3}, \ref{main4}, \ref{main5}, \ref{main6}), but $(E_{7(-25)},U(1)\times E_{6(-14)})$ case is based on some unproved assumption (Theorems~\ref{main2}(4)). We note that this assumption for $(\rho,W)$, which is the same assumption used in the author's previous paper~\cite{N2}, is needed since the explicit computation of intertwining operators is deeply based on the explicit norm computation of $\cH_\rho(D_1,W)$ given in~\cite{N2}. The symmetric pairs $(G,G_1)$ in the lists (\ref{list_normal_derivative}), (\ref{list_infinite_order}) exhaust all symmetric spaces of holomorphic type such that $G$ is simple (if we replace $(U(q,s), U(q',s')\times U(q'',s''))$ with $(\operatorname{SU}(q,s), S(U(q',s')\times U(q'',s'')))$). It remains as a future task to construct embedding operators for tensor product case, and to construct projection operators in the list~(\ref{list_infinite_order}) (for some special cases it is already done; see~\cite{IKO,Ju,KP2}).

\looseness=-1 The embedding intertwining operators $\cH_\rho(D_1,W)\rightarrow\cH_\tau(D,V)$ we compute in this paper are written uniformly
in the following form, although they are computed case by case. Let $(G,G_1)$ be a symmetric pair in the list (\ref{list_infinite_order}),
and $\chi$, $\chi_1$ be (suitably normalized) characters of maximal compact subgroups $K$, $K_1$ of $G$, $G_1$ respectively.
We assume $(\tau,V)=\big(\chi^{-\lambda},\BC\big)$, $(\rho,W)=\big(\chi_1^{-\varepsilon(\lambda+\delta k)}\otimes\rho_0,W\big)$,
where $(\rho_0,W)$ is a representation of $K_1$ which appears in the decomposition of $\cP(\fp^+_2)$, $\varepsilon$ and $\delta$ are
1 or 2 according to $(G,G_1)$, and $k\in\BZ_{\ge 0}$ if $\fp^+_2$ is of tube type, $k=0$ if $\fp^+_2$ is not of tube type.
We write $\cH_\tau(D,V)=\cH_\lambda(D)$, $\cH_\rho(D_1,W)=\cH_{\varepsilon(\lambda+\delta k)}(D_1,W)$.
We assume $\cH_{\varepsilon(\lambda+\delta k)}(D_1,W)$ to be multiplicity-free under $K_1$.
Then the intertwining operator is of the form
\begin{gather*}
\cF_{\lambda,k,W}\colon \ \cH_{\varepsilon(\lambda+\delta k)}(D_1,W)\to\cH_\lambda(D), \\
\cF_{\lambda,k,W}f(x_1,x_2)=\Delta(x_2)^k\sum_{\substack{W'\in\Supp_{K_1}(\cP(\fp^+_1)\otimes W) \\ \hspace{15pt}\cap\Supp_{K_1}(\cP(\fp^+_2))}}
\frac{1}{b_{W,W'}(\lambda+\delta k)}
\cK_{W,W'}\left(x_2;\frac{1}{\varepsilon}\overline{\frac{\partial}{\partial x_1}}\right)f(x_1),
\end{gather*}
where $x_1\in\fp^+_1$, $x_2\in\fp^+_2$, $\Delta(x_2)$ is a polynomial on $\fp^+_2$, $\Supp_{K_1}(\cP(\fp^+_1)\otimes W)$ and $\Supp_{K_1}(\cP(\fp^+_2))$ denote all $K_1$-types which appear in the decomposition of $\cP(\fp^+_1)\otimes W$ and $\cP(\fp^+_2)$ respectively, and for each~$W'$, $b_{W,W'}(\lambda)\in\BC[\lambda]$ is a monic polynomial given by a product of Pochhammer symbols, and
\begin{gather*} \cK_{W,W'}(x_2;y_1)\in \big(W'\otimes \overline{W'}\big)^{K_1}\subset \cP(\fp^+_2)\otimes\overline{\cP(\fp^+_1)\otimes W} \end{gather*}
is a $\overline{W}\simeq\Hom(W,\BC)$-valued $K_1$-invariant polynomial, normalized such that
\begin{gather*} \sum_{W'}\cK_{W,W'}(x_2;y_1)={\rm e}^{\frac{1}{2}(x_2|Q(y_1)x_2)_{\fp^+}}\rK_{W}(x_2)
={\rm e}^{\frac{1}{2}(Q(x_2)y_1|y_1)_{\fp^+}}\rK_W(x_2), \end{gather*}
where $\rK_W(x_2)\in\cP(\fp^+_2,\Hom(W,\BC))^{K_1}$ is a fixed polynomial, and $Q\colon \fp^+\to\Hom(\overline{\fp^+},\fp^+)$ is a~quadratic map determined from the Jordan triple system structure of $\fp^+$.
On the other hand, when $(G,G_1)$ is in the list~(\ref{list_normal_derivative}), we have $Q(y_1)x_2=0$ and
${\rm e}^{\frac{1}{2}(x_2|Q(y_1)x_2)_{\fp^+}}\rK_{W}(x_2)=\rK_W(x_2)$.
In this case the embedding intertwining operator is given by the multiplication operator
\begin{gather*}
\cF_{\lambda,k,W}\colon \cH_{\varepsilon(\lambda+\delta k)}(D_1,W)\to\cH_\lambda(D), \\
\cF_{\lambda,k,W}f(x_1,x_2)=\Delta(x_2)^k\rK_{W}(x_2)f(x_1), \qquad x_1\in\fp^+_1,\quad x_2\in\fp^+_2.
\end{gather*}

By the explicit computation of the intertwining operators, we can study how the operator depends on the parameter of the holomorphic discrete series representation. More precisely, since each $b_{W,W'}(\lambda)$ in the above formula is a holomorphic polynomial, $\cF_{\lambda,k,W}$ extends meromorphically for all $\lambda\in\BC$, and defines an intertwining operator $\cO_{\varepsilon(\lambda+\delta k)}(D_1,W)_{\tilde{K}_1}\to\cO_\lambda(D)_{\tilde{K}}$. Now suppose $\nu=\lambda$ is a pole of $\cF_{\nu,k,W}$ of order $i_0$. In this case $\cO_{\varepsilon(\lambda+\delta k)}(D_1,W)_{\tilde{K}_1}$ and $\cO_\lambda(D)_{\tilde{K}}$ are no longer unitarizable. Then for $i=0,1,\ldots,i_0$, there exists a submodule $M_i\subset \cO_{\varepsilon(\lambda+\delta k)}(D_1,W)_{\tilde{K}_1}$ such that
\begin{gather*} \tilde{\cF}_{\lambda,k,W}^i:=\lim_{\nu\to\lambda}(\nu-\lambda)^i\cF_{\nu,k,W}\colon M_i\to\cO_\lambda(D)_{\tilde{K}} \end{gather*}
is well-defined. This $M_i$ contains all $K_1$-types $W'$ such that it appears as the summand of $\cF_{\nu,k,W}$ and the corresponding $b_{W,W'}(\nu+\delta k)$ has a zero of order at most $i$ at $\nu=\lambda$. Moreover, $\tilde{\cF}_{\lambda,k,W}^i$ is trivial on $M_{i-1}$, and defines a map from $M_i/M_{i-1}$. However, this is not intertwining unless $i=i_0$. But fortunately, if there exists a submodule $M'$ such that $M_{i-1}\subsetneq M'\subsetneq M_i$, then the restriction
\begin{gather*} \tilde{\cF}_{\lambda,k,W}^i\colon \ M'/M_{i-1}\to\cO_\lambda(D)_{\tilde{K}} \end{gather*}
is intertwining. Whether such submodule $M'$ exists or not depends on the pair $(G,G_1)$. In this paper we observe this phenomenon only when $G$ is classical and the minimal $K_1$-type $W$ is 1-dimensional, but this also occur when $G$ is exceptional or $W$ is not 1-dimensional.

This paper is organized as follows. In Section~\ref{section2} we prepare some notations and review some facts on Lie algebras of Hermitian type, Jordan triple systems, and holomorphic discrete series representations. In Section~\ref{section3} we construct a general theory on the intertwining operators between holomorphic discrete series representations. In Section~\ref{section4}, as a preparation for case by case analysis, we fix the explicit realization of Lie groups and their root systems. In Sections~\ref{sect_examples} and \ref{sect_remaining} we compute the explicit intertwining operators by using the results of Sections~\ref{section3} and~\ref{section4}. In Section~\ref{section6}, we study what occurs when the parameter is at a pole in the cases $G$ is classical and both $\cH$ and $\cH_1$ are of scalar type.

\section{Preliminaries for general theory}\label{section2}
\subsection{Root systems}\label{root_sys}
Let $\fg$ be a reductive Lie algebra with Cartan involution $\vartheta$. We decompose $\fg$ into a sum of simple non-compact subalgebras, a semi-simple compact subalgebra and an abelian subalgebra as
\begin{gather*} \fg=\fg_{(1)}\oplus\cdots\oplus\fg_{(m)}\oplus\fg_{\text{cpt}}\oplus\fz(\fg). \end{gather*}
We assume that each simple non-compact subalgebra $\fg_{(i)}$ is of Hermitian type, that is,
its maximal compact subalgebra $\fk_{(i)}:=\fg^\vartheta_{(i)}$ has a 1-dimensional center $\fz(\fk_{(i)})$, and also that the abelian part $\fz(\fg)$ is fixed by $\vartheta$. For each $i$, we fix an element $z_{(i)}\in\fz(\fk_{(i)})$ such that $\operatorname{ad}(z_{(i)})$ has eigenvalues $+\sqrt{-1}$, $0$, $-\sqrt{-1}$,
and decompose the complexified Lie algebra $\fg_{(i)}^\BC$ into eigenspaces under $\operatorname{ad}(z_{(i)})$ as
\begin{gather*} \fg_{(i)}^\BC=\fp^+_{(i)}\oplus \fk^\BC_{(i)}\oplus\fp^-_{(i)}. \end{gather*}
We denote
\begin{alignat*}{3}
& \fp^+ :=\fp^+_{(1)}\oplus\cdots\oplus\fp^+_{(m)},\qquad && \fk^\BC:=\fk_{(1)}^\BC\oplus\cdots\oplus\fk_{(m)}^\BC\oplus\fg_{\text{cpt}}^\BC\oplus\fz(\fg)^\BC,&\\
&\fp^-:=\fp^-_{(1)}\oplus\cdots\oplus\fp^-_{(m)},\qquad && \fk :=\fk_{(1)}\oplus\cdots\oplus\fk_{(m)}\oplus\fg_{\text{cpt}}\oplus\fz(\fg)=\fg^\vartheta,&
\end{alignat*}
so that
\begin{gather*} \fg^\BC=\fp^+\oplus\fk^\BC\oplus\fp^-. \end{gather*}
We denote the anti-holomorphic extension of the Cartan involution $\vartheta$ on $\fg^\BC$ by the same sym\-bol~$\vartheta$.
Also, let $\hat{\vartheta}:=\vartheta\circ \operatorname{Ad}({\rm e}^{\pi z})$ ($z:=\sum_i{z_{(i)}}$) be the anti-holomorphic involution on~$\fg^\BC$ fixing~$\fg$.

Next, we fix a Cartan subalgebra $\fh\subset\fk$. Then $\fh^\BC$ automatically becomes a Cartan subalgebra of $\fg^\BC$.
We set $\fh_{(i)}:=\fh\cap\fg_{(i)}$.
Let $\Delta_{\fg^\BC_{(i)}}=\Delta(\fg^\BC_{(i)},\fh^{\BC}_{(i)})$ be the root system of $\fg^\BC_{(i)}$,
and let $\Delta_{\fp^\pm_{(i)}}$, $\Delta_{\fk^\BC_{(i)}}$ be the set of roots
such that the corresponding root space is contained in $\fp^\pm_{(i)}$, $\fk^\BC_{(i)}$ respectively.
We fix a positive system $\Delta_{\fg^\BC_{(i)},+}\subset \Delta_{\fg^\BC_{(i)}}$
such that $\Delta_{\fp^+_{(i)}}\subset\Delta_{\fg^\BC_{(i)},+}$,
and denote $\Delta_{\fk^\BC_{(i)},+}:=\Delta_{\fk^\BC_{(i)}}\cap \Delta_{\fg^\BC_{(i)},+}$.
Then we can take a system of strongly orthogonal roots
$\{\gamma_{1,(i)},\ldots,\gamma_{r_{(i)},(i)}\}\subset \Delta_{\fp^+_{(i)}}$,
where $r_{(i)}=\rank_\BR\fg_{(i)}$, such that
\begin{enumerate}\itemsep=0pt
\item $\gamma_{1,(i)}$ is the highest root in $\Delta_{\fp^+_{(i)}}$,
\item $\gamma_{k,(i)}$ is the root in $\Delta_{\fp^+_{(i)}}$
which is highest among the roots strongly orthogonal to each $\gamma_{j,(i)}$ with $1\le j\le k-1$.
\end{enumerate}
For each $j$, let $\fp^+_{jj,(i)}$ be the root space corresponding to $\gamma_{j,(i)}$.
We take an element $e_{j,(i)}\in\fp^+_{jj,(i)}$ such that
\begin{gather*} -[[e_{j,(i)},\vartheta e_{j,(i)}],e_{j,(i)}]=2e_{j,(i)}, \end{gather*}
and set
\begin{alignat*}{4}
& h_{j,(i)}:=-[e_{j,(i)},\vartheta e_{j,(i)}]\in\sqrt{-1}\fh_{(i)},\qquad && e_{(i)}:=\sum_{j=1}^{r_{(i)}}e_{j,(i)}\in\fp^+_{(i)},\qquad && e:=\sum_{i=1}^m e_{(i)}\in\fp^+,& \\
&\fa_{\fl,(i)}:=\bigoplus_{j=1}^{r_{(i)}}\BR h_{j,(i)}\subset\sqrt{-1}\fh_{(i)},\qquad && \fa_{(i)}^+:=\bigoplus_{j=1}^{r_{(i)}}\BR e_{j,(i)}\subset\fp^+_{(i)}.&&&
\end{alignat*}
Then the restricted root system $\Sigma=\Sigma\big(\fg^\BC_{(i)},\fa_{\fl,(i)}^\BC\big)$ is one of
\begin{gather*} \Sigma=\big\{ \tfrac{1}{2}(\gamma_{j,(i)}-\gamma_{k,(i)})\big|_{\fa_{\fl,(i)}}\!\!\colon
 1\le j, k\le r_{(i)}, j\ne k \big\}\\
 \hphantom{\Sigma=}{}
\cup\big\{{\pm}\tfrac{1}{2}(\gamma_{j,(i)}+\gamma_{k,(i)})\big|_{\fa_{\fl,(i)}}\!\!\colon 1\le j\le k\le r_{(i)}\big\} \end{gather*}
(type $C_{r_{(i)}}$), or
\begin{gather*} \Sigma=(\text{as above})\cup\big\{ {\pm}\tfrac{1}{2}\gamma_{j,(i)}\big|_{\fa_{\fl,(i)}}\colon 1\le j\le r_{(i)}\big\} \end{gather*}
(type $BC_{r_{(i)}}$). For $1\le j\le k\le r_{(i)}$ we set
\begin{gather*}
\fp^+_{jk,(i)} :=\big\{x\in\fp^+_{(i)}\colon \operatorname{ad}(l)x=\tfrac{1}{2}(\gamma_{j,(i)}+\gamma_{k,(i)})(l)x\text{ for all }l\in\fa_{\fl,(i)}\big\},\\
\fp^+_{0j,(i)} :=\big\{x\in\fp^+_{(i)}\colon \operatorname{ad}(l)x=\tfrac{1}{2}\gamma_{j,(i)}(l)x\text{ for all } l\in\fa_{\fl,(i)}\big\}.
\end{gather*}
Then we have
\begin{gather*} \fp^+_{(i)}=\bigoplus_{\substack{0\le j\le k\le r_{(i)}\\ (j,k)\ne (0,0)}}\fp^+_{jk,(i)}. \end{gather*}
We set
\begin{alignat*}{4}
&\fp^+_{\rT,(i)}:=\bigoplus_{1\le j\le k\le r_{(i)}}\fp^+_{jk,(i)},\qquad &&
\fp^-_{\rT,(i)}:=\vartheta\fp^+_{\rT,(i)},\qquad &&
\fp^+_{\rT}:=\bigoplus_{i=1}^m \fp^+_{\rT,(i)}, &\\
&\fk^\BC_{\rT,(i)}:=[\fp^+_{\rT,(i)},\fp^-_{\rT,(i)}], \qquad && \fk_{\rT,(i)} :=\fk^\BC_{\rT,(i)}\cap \fk_{(i)}, &&&\\
& \fg^\BC_{\rT,(i)}:=\fp^+_{\rT,(i)}\oplus \fk^\BC_{\rT,(i)}\oplus \fp^-_{\rT,(i)},\qquad && \fg_{\rT,(i)}:=\fg^\BC_{\rT,(i)}\cap \fg_{(i)},&&&
\end{alignat*}
and we define the integers
\begin{gather*}
d_{(i)} :=\dim\fp^+_{12,(i)},\qquad b_{(i)}:=\dim\fp^+_{01,(i)},\\
n_{(i)}:=\dim\fp^+_{(i)}=r_{(i)}+\tfrac{1}{2}r_{(i)}(r_{(i)}-1)d_{(i)}+b_{(i)}r_{(i)},\qquad
n:=\dim\fp^+=\sum_{i=1}^m n_{(i)},\\
n_{\rT,(i)}:=\dim\fp^+_{\rT,(i)}=r_{(i)}+\tfrac{1}{2}r_{(i)}(r_{(i)}-1)d_{(i)},\\
p_{(i)}:=2+(r_{(i)}-1)d_{(i)}+b_{(i)}.
\end{gather*}

Throughout the paper, let $G^\BC$ be a connected complex Lie group with Lie algebra $\fg^\BC$, and let
$G$, $K^\BC$, $K$, $G^\BC_{(i)}$, $G_{(i)}$, $K^\BC_{(i)}$, $K_{(i)}$, $G^\BC_{\rT,(i)}$, $G_{\rT,(i)}$, $K^\BC_{\rT,(i)}$, $K_{\rT,(i)}$
be the connected Lie subgroup with Lie algebras
$\fg$, $\fk^\BC$, $\fk$, $\fg^\BC_{(i)}$, $\fg_{(i)}$, $\fk^\BC_{(i)}$, $\fk_{(i)}$,
$\fg^\BC_{\rT,(i)}$, $\fg_{\rT,(i)}$, $\fk^\BC_{\rT,(i)}$, $\fk_{\rT,(i)}$ respectively.
Also, let
\begin{gather*} K_{L,(i)}:=\big\{ k\in K_{\rT,(i)}\colon \operatorname{Ad}(k)e_{(i)}=e_{(i)}\big\}, \end{gather*}
which is possibly non-connected, and we denote its Lie algebra by $\fk_{\fl,(i)}$.

For $k\in K^\BC$, we write $k^*:=(\vartheta k)^{-1}$.
Then for each $i$, there exists a unique Hermitian inner product $(\cdot|\cdot)_{\fp^+_{(i)}}$,
holomorphic in the first variable and anti-holomorphic in the second variable, such that
\begin{gather*}
(\operatorname{Ad}(k)x|y)_{\fp^+_{(i)}} =(x|\operatorname{Ad}(k^*)y)_{\fp^+_{(i)}}, \qquad x,y\in\fp^+_{(i)},\quad k\in K^\BC_{(i)},\qquad
(e_{1,(i)}|e_{1,(i)})_{\fp^+_{(i)}} =1.
\end{gather*}
This is proportional to the restriction of the Killing form of $\fg_{(i)}^\BC$ on $\fp^+_{(i)}\times \fp^-_{(i)}$,
if we identify $\fp^+_{(i)}$ and $\fp^-_{(i)}$ through $\vartheta$. By summing these inner products, we define
\begin{gather}\label{innerprod}
(x|y)=(x|y)_{\fp^+}:=\sum_{i=1}^m (x_i|y_i)_{\fp^+_{(i)}} \qquad
 x=\sum_{i=1}^m x_i, \qquad y=\sum_{i=1}^m y_i\in\fp^+=\bigoplus_{i=1}^m\fp^+_{(i)} .
\end{gather}
From now on we omit $\operatorname{Ad}$ or $\operatorname{ad}$ if there is no confusion, so that $(kx|y)_{\fp^+}=(x|k^*y)_{\fp^+}$.

\subsection{Operations on Jordan triple systems}
$\fp^+$ has a Hermitian positive Jordan triple system structure with the product
\begin{gather*} (x,y,z)\mapsto -[[x,\vartheta y],z]. \end{gather*}
For $x,y\in\fp^+$, let $D(x,y)$ be the linear map, $Q(x,y)$ be the anti-linear map on $\fp^+$ defined by
\begin{gather*} D(x,y):=-\operatorname{ad}([x,\vartheta y])\bigr|_{\fp^+},\qquad Q(x,y):=\operatorname{ad}(x)\operatorname{ad}(y)\vartheta\bigr|_{\fp^+}, \end{gather*}
and let $Q(x):=\frac{1}{2}Q(x,x)$.
We recall that, for $x,y\in\fp^+$, the \textit{Bergman operator} $B(x,y)\in\End(\fp^+)$ is defined as
\begin{gather*} B(x,y):=I-D(x,y)+Q(x)Q(y)\in\End(\fp^+). \end{gather*}
We say $(x,y)\in\fp^+\times\fp^+$ is \textit{quasi-invertible} if $B(x,y)$ (or equivalently $B(y,x)$) is invertible,
and in this case the \textit{quasi-inverse} $x^y$ is defined as
\begin{gather*} x^y:=B(x,y)^{-1} \left(x-Q(x)y\right)\in\fp^+. \end{gather*}
Then if $B(x,y)$ is invertible, then there exists an element $k\in K^\BC$ such that $B(x,y)z=\operatorname{Ad}(k)z$ holds for any $z\in\fp^+$.
Also, $B(x,y)$ and $x^y$ satisfy the following properties. For $x,y,z\in\fp^+$ and $k\in K^\BC$, if $(x,y)$ is quasi-invertible, then
\begin{alignat}{3}
& B(kx,k^{*-1}y) =kB(x,y)k^{-1}, &&\label{Bergman_k}\\
& B(x,y)B(x^y,z) =B(x,y+z) \qquad & & \text{\cite[Part V, Proposition III.3.1, (J6.4)]{FKKLR}},&\label{Bergman_left}\\
& B(z,x^y)B(y,x) =B(y+z,x) \qquad & & \text{\cite[Part V, Proposition III.3.1, (J6.4$'$)]{FKKLR}},&\label{Bergman_right}\\
& (kx)^{k^{*-1}y} =k(x^y), &&& \label{quasiinv_k}\\
& x^{y+z} =(x^y)^z \qquad & & \text{\cite[Part V, Theorem III.5.1(i)]{FKKLR}},& \label{quasiinv_add}\\
&(x+z)^y=x^y+B(x,y)^{-1}z^{(y^x)} \qquad & & \text{\cite[Part V, Theorem III.5.1(ii)]{FKKLR}}& \label{quasiinv_twice}
\end{alignat}
hold. Here, the equality (\ref{quasiinv_add}) holds when one of $(x,y+z)$ or $(x^y,z)$ is quasi-invertible,
and the other also becomes quasi-invertible. Similarly, the equality~(\ref{quasiinv_twice}) holds when one of $(x+z,y)$ or $(z,y^x)$ is
quasi-invertible, and then the other also is. Also, for the Bergman operator, we can show directly from the definition that,
if $\fp^+_1,\fp^+_2\subset \fp^+$ are Jordan triple subsystems such that $D(\fp^+_1,\fp^+_2)=\{0\}$ (we do not assume they are ideals),
then we have
\begin{gather}\label{Bergman_decomp}
B(x_1+x_2,y_1+y_2)=B(x_1,y_1)B(x_2,y_2), \qquad x_1,y_1\in\fp^+_1, \quad x_2,y_2\in\fp^+_2.
\end{gather}

Next, we recall the spectral decomposition and the spectral norm. For any $x_i\in\fp^+_{(i)}$, there exist complex numbers
$a_{1,i},\ldots, a_{r_{(i)},i}$ and an element $k_i\in K_{(i)}$ such that
\begin{gather*} x_i=k_i\sum_{j=1}^{r_{(i)}} a_{j,i}e_{j,(i)}. \end{gather*}
The set $\{a_{1,i},\ldots,a_{r_{(i)},i}\}$ is unique under the condition that $a_{j,i}\in\BR_{\ge 0}$ and
$a_{1,i}\ge\cdots\ge a_{r_{(i)},i}\ge 0$. This is called the \textit{spectral decomposition}.
For $x=\sum\limits_{i=1}^m x_i=\sum\limits_{i=1}^m k_i\sum\limits_{j=1}^{r_{(i)}}a_{j,i}e_{j,(i)}\in\fp^+=\bigoplus\limits_{i=1}^m \fp^+_{(i)}$,
the \textit{spectral norm} is defined as
\begin{gather}\label{spectral_norm}
|x|_\infty=|x|_{\fp^+,\infty}:=\max_{1\le i\le m}\max_{1\le j\le r_{(i)}}|a_{j,i}|.
\end{gather}
In fact this becomes a norm on the vector space $\fp^+$ (see \cite[Part V, Proposition VI.4.1]{FKKLR}).

Next, for each $i$, let $h_{(i)}(x,y)\in\cP(\fp^+\times\overline{\fp^+})$ be the \textit{generic norm} on $\fp^+_{(i)}$.
This is the polynomial, holomorphic in $x$ and anti-holomorphic in $y$, satisfying
\begin{gather*} \Det_{\fp^+_{(i)}}(B(x_i,y_i))=h_{(i)}(x_i,y_i)^{p_{(i)}}, \qquad x_i,y_i\in\fp^+_{(i)} . \end{gather*}
If $x_i=\sum\limits_{j=1}^{r_{(i)}}a_je_{j,(i)}$, $y_i=\sum\limits_{j=1}^{r_{(i)}}b_je_{j,(i)}\in\fa^+_{(i)}\subset\fp^+_{(i)}$,
then $h_{(i)}(x_i,y_i)$ is given by
\begin{gather*} h_{(i)}(x_i,y_i)=\prod_{j=1}^{r_{(i)}}(1-a_j\overline{b_j}). \end{gather*}
For later use we abbreviate
\begin{gather*} \Det_{\fp^+}(B(x,y))^{-1}=\prod_{i=1}^m h_{(i)}(x_i,y_i)^{-p_{(i)}}=:h(x,y)^{-p}. \end{gather*}
Also, we abbreviate $B(x,x)=:B(x)$, $h_{(i)}(x_i,x_i)=h_{(i)}(x_i)$. Let
\begin{align}
D\colon \hspace{-3pt}&=(\text{connected component of }\{x\in \fp^+\colon B(x)\text{ is positive definite.}\}\text{ which contains }0) \notag \\
&=\big\{x\in \fp^+\colon |x|_\infty <1\}=\{x\in \fp^+\colon |D(x,x)|_{\fp^+,\mathrm{op}}<2\big\}\label{BSD}
\end{align}
be the \textit{bounded symmetric domain}, which is diffeomorphic to $G/K$ via the Borel embedding which we will review later
(for these equalities see \cite[Part~V, Proposition~VI.4.2]{FKKLR}. Here $|\cdot|_{\fp^+,\mathrm{op}}$ denotes the
operator norm on $\End(\fp^+)$ with respect to $|\cdot|_{\fp^+}$).
Then if $x,y\in D$, $B(x,y)$ is invertible, and thus it is in the image of $K^\BC$.
Moreover, since $D$ is simply connected, there exists a holomorphic map $\tilde{B}\colon D\times\overline{D}\to K^\BC$
(or $\tilde{B}\colon D\times\overline{D}\to \tilde{K}^\BC$, where $\tilde{K}^\BC$ is the universal covering group of $K^\BC$)
such that
\begin{gather*} \operatorname{Ad}(\tilde{B}(x,y))=B(x,y)\in\End(\fp^+), \qquad \tilde{B}(0,0)=\bone_{K^\BC}\in K^\BC \quad \big(\text{resp.} \ {} \in \tilde{K}^\BC\big) \end{gather*}
holds. From now on we omit the tilde, and use the same symbol $B$ instead of $\tilde{B}$.

Next we consider $\fp^+_\rT$. This has a complex Jordan algebra structure with the product
\begin{gather*} (x,y)\mapsto x\cdot y:=-\tfrac{1}{2}[[x,\vartheta e],y]. \end{gather*}
We recall the quadratic map $P\colon \fp^+_\rT\to\End(\fp^+_\rT)$ defined by
\begin{gather}\label{quad_repn}
P(x)y:=2x\cdot(y\cdot x)-y\cdot(x\cdot x)=Q(x)Q(e)y, \qquad x,y\in\fp^+_\rT .
\end{gather}
If $y$ is in the real form $\left\{y\in\fp^+_\rT\colon Q(e)y=y\right\}$ of $\fp^+_\rT$,
then $P(x)y=-\frac{1}{2}[[x,\vartheta y],x]=Q(x)y$ holds.
Next we review the \textit{determinant polynomials} on Jordan algebras.
On each simple component~$\fp^+_{\rT,(i)}$ there exists a determinant polynomial $\Delta_{(i)}$,
which is the homogeneous polynomial of degree~$r_{(i)}$ satisfying
\begin{gather*}
\Delta_{(i)}(kx)=\Delta_{(i)}(ke_{(i)})\Delta_{(i)}(x)
\qquad \text{for all} \ k\in K^\BC_{\rT,(i)},\ x\in \fp^+_{\rT,(i)},\qquad
\Delta_{(i)}(e_{(i)})=1.
\end{gather*}
The quadratic map $P$ and the determinant polynomials are related as
\begin{gather*} \Det_{\fp^+_{\rT,(i)}}(P(x_i))=\Delta(x_i)^{2n_{\rT,(i)}/r_{(i)}}, \qquad x_i\in\fp^+_{\rT,(i)} . \end{gather*}
We \looseness=-1 extend $\Delta_{(i)}$ on $\fp^+_{(i)}$ such that it does not depend on $\big(\fp^+_{\rT,(i)}\big)^\bot=\bigoplus\limits_{j=1}^{r_{(i)}} \fp^+_{0j,(i)}$, and denote by the same symbol $\Delta_{(i)}$.
Then the determinant polynomial $\Delta_{(i)}$ and the generic norm $h_{(i)}$ are related as
\begin{gather}\label{det_norm}
\Delta_{(i)}(e_{(i)}-x)=h_{(i)}(x,e_{(i)}), \qquad x\in\fp^+_{(i)}.
\end{gather}
For the theory of Jordan algebras and Jordan triple systems, see, e.g., \cite[Part V]{FKKLR}, \cite{FK, L,Sat}.

\subsection{Polynomials on Jordan triple systems}
Let $\cP(\fp^+)$ be the space of all holomorphic polynomials on $\fp^+$. Then $K^\BC$ acts on $\cP(\fp^+)$ by
\begin{gather*} (\operatorname{Ad}|_{\fp^+})^\vee(k)f(x):=f\big(k^{-1}x\big), \qquad k\in K^\BC,\quad f\in\cP(\fp^+). \end{gather*}
Then clearly we have $\cP(\fp^+)\simeq \cP(\fp^+_{(1)})\otimes\cdots\otimes\cP(\fp^+_{(m)})$, according to the
simple decomposition of the Jordan triple system $\fp^+=\fp^+_{(1)}\oplus\cdots\oplus\fp^+_{(m)}$.
In the rest of this subsection, we assume $\fg$ is simple, and we drop the subscript $(i)$. We set
\begin{gather*} \BZ_{++}^r:=\big\{\bm=(m_1,\ldots,m_r)\in\BZ^r\colon m_1\ge \cdots \ge m_r\ge 0\big\}. \end{gather*}
Then $\cP(\fp^+)$ is decomposed as follows.
\begin{Theorem}[Hua--Kostant--Schmid, {\cite[Part~III, Theorem~V.2.1]{FKKLR}}]\label{HKS}
Under the $K^\BC$-action, $\cP(\fp^+)$ is decomposed as
\begin{gather*} \cP(\fp^+)=\bigoplus_{\bm\in\BZ^{r}_{++}}\cP_\bm(\fp^+), \end{gather*}
where $\cP_\bm(\fp^+)$ is the irreducible representation of $K^\BC$ with lowest weight
$-m_1\gamma_1-\cdots-m_r\gamma_r$. Moreover, each $\cP_\bm(\fp^+)$ has a nonzero $K_L$-invariant polynomial,
which is unique up to scalar multiple.
\end{Theorem}
Let $d_\bm^{(d,r,b)}:=\dim\cP_\bm(\fp^+)$, $d_\bm^{(d,r,0)}:=\dim\cP_\bm(\fp^+_\rT)$,
where $\cP_\bm(\fp^+_\rT)$ is the irreducible representation of $K^\BC_\rT$ with lowest weight $-m_1\gamma_1-\cdots-m_r\gamma_r$,
and let $\Phi_\bm^{(d,r)}$ be the $K_L$-invariant polynomial in $\cP_\bm(\fp^+)$
such that $\Phi_\bm^{(d,r)}(e)=1$. Especially, when $\bm=(m,\ldots,m)$, then $\Phi_{(m,\ldots,m)}^{(d,r)}(x)=\Delta(x)^m$ holds.

Next we recall the \textit{Fischer inner product}.
For two holomorphic polynomials $f,g\in\cP(\fp^+)$, it is defined as
\begin{gather}\label{Fischer}
\langle f,g\rangle_F:=\frac{1}{\pi^n}\int_{\fp^+}f(x)\overline{g(x)} {\rm e}^{-|x|_{\fp^+}^2}{\rm d}x.
\end{gather}
This integral converges for polynomials $f$ and $g$, and the reproducing kernel is given by ${\rm e}^{(x|y)_{\fp^+}}$.
Let $\bK_\bm^{(d)}(x,y)\in\cP(\fp^+\times\overline{\fp^+})$ be the reproducing kernel of $\cP_\bm(\fp^+)$
with respect to $\langle\cdot,\cdot\rangle_F$, so that $\sum\limits_{\bm\in\BZ_{++}^r}\bK_\bm^{(d)}(x,y)={\rm e}^{(x|y)_{\fp^+}}$. Then the following holds.
\begin{Proposition}[{\cite[Part III, Lemma V.3.1(a), Theorem V.3.4]{FKKLR}}]\label{Kmxe}
\begin{gather*} \bK_\bm^{(d)}(x,e)=\frac{d_\bm^{(d,r,b)}}{\left(\frac{n}{r}\right)_{\bm,d}}\Phi_\bm^{(d,r)}(x)
=\frac{d_\bm^{(d,r,0)}}{\left(\frac{n_\rT}{r}\right)_{\bm,d}}\Phi_\bm^{(d,r)}(x). \end{gather*}
\end{Proposition}
Here, for $\lambda\in\BC$, $\bs\in\BC^r$, $\bm\in(\BZ_{\ge 0})^r$ and $d\in\BZ_{\ge 0}$, $(\lambda+\bs)_{\bm,d}$ is defined as
\begin{gather}\label{Pochhammer}
(\lambda+\bs)_{\bm,d}:=\prod_{j=1}^r\left(\lambda+s_j-\frac{d}{2}(j-1)\right)_{m_j},\qquad
(\lambda)_m:=\lambda(\lambda+1)\cdots(\lambda+m-1),
\end{gather}
and we write $(\lambda+(0,\ldots,0))_{\bm,d}=(\lambda)_{\bm,d}$.
We renormalize $\Phi_\bm^{(d,r)}$ as
\begin{gather}\label{Phitilde1}
\tilde{\Phi}_\bm^{(d)}(x):=\frac{d_\bm^{(d,r,b)}}{\left(\frac{n}{r}\right)_{\bm,d}}\Phi_\bm^{(d,r)}(x)
=\frac{d_\bm^{(d,r,0)}}{\left(\frac{n_\rT}{r}\right)_{\bm,d}}\Phi_\bm^{(d,r)}(x),
\end{gather}
so that
\begin{gather*} {\rm e}^{(x|e)_{\fp^+}}=\sum_{\bm\in\BZ_{++}^r}\bK_\bm^{(d)}(x,e)=\sum_{\bm\in\BZ_{++}^r}\tilde{\Phi}_\bm^{(d)}(x). \end{gather*}
For example, when $\fp^+=M(r,\BC)$ (i.e., $G=\operatorname{SU}(r,r)$), if the eigenvalues of $x\in M(r,\BC)$ are $t_1,\ldots,t_r$, we have
\begin{gather}
\tilde{\Phi}_\bm^{(2)}(x) = \!\left(\!\frac{\prod\limits_{i<j}(m_i-m_j-i+j)}{\prod\limits_{i=1}^r (r-i)!}\!\right)^2\!\!\! \frac{1}{\prod\limits_{i=1}^r (r-i+1)_{m_i}}
 \frac{\prod\limits_{i=1}^r (r-i)!}{\prod\limits_{i<j}(m_i-m_j-i+j)}
\frac{\det\big(\big(t_i^{m_j+r-j}\big)_{i,j}\big)}{\det\big(\big(t_i^{r-j}\big)_{i,j}\big)}\notag \\
\hphantom{\tilde{\Phi}_\bm^{(2)}(x)}{} =\frac{\prod\limits_{i<j}(m_i-m_j-i+j)}{\prod\limits_{i=1}^r (m_i+r-i)!}
\frac{\det\left((t_i^{m_j+r-j})_{i,j}\right)}{\det\left((t_i^{r-j})_{i,j}\right)}.\label{Schur}
\end{gather}
Then $\tilde{\Phi}_\bm^{(d)}(x)$ does not depend on $r$ in the following sense.
Since $\tilde{\Phi}_\bm^{(d)}$ is $K_L$-invariant, it is determined by the value on $\fa^+\subset\fp^+$.
Thus for $x=a_1e_1+\cdots+a_re_r\in\fa^+$, we write
\begin{gather*} \tilde{\Phi}_\bm^{(d)}(x)=:\tilde{\Phi}_\bm^{(d)}(a_1,\ldots,a_r). \end{gather*}
Then this does not depend on $r$, that is,
\begin{gather*} \tilde{\Phi}_\bm^{(d,r)}(a_1,\ldots,a_{r-1},0)=\tilde{\Phi}_\bm^{(d,r-1)}(a_1,\ldots,a_{r-1}) \end{gather*}
holds. Also, when $\fp^+=\Sym(r,\BC)$, $M(q,s;\BC)$ or $\Skew(s,\BC)$ (i.e., $G=\operatorname{Sp}(r,\BR)$, $\operatorname{SU}(q,s)$ or $\operatorname{SO}^*(2s)$ respectively),
for $x,y\in\fp^+$, $\bK_\bm^{(d)}(x,y)$ depends only on the eigenvalues of $xy^*$, so following the notation in \cite{M} we write
\begin{gather}\label{Phitilde2}
\bK_\bm^{(d)}(x,y)=:\tilde{\Phi}_\bm^{(d)}(xy^*)=\tilde{\Phi}_\bm^{(d)}(y^*x).
\end{gather}

\subsection{Holomorphic discrete series representations}\label{HDS}
In this subsection we recall the explicit realization of the holomorphic discrete series representation
of the universal covering group $\tilde{G}$. First we recall the Borel embedding,
\begin{gather*} \xymatrix{ G/K \ar[r] \ar@{-->}[d]^{\mbox{\rotatebox{90}{$\sim$}}} & G^\BC/K^\BC P^- \\
 D \ar@{^{(}->}[r] & \fp^+, \ar[u]_{\exp} } \end{gather*}
where $P^\pm:=\exp(\fp^\pm)$. When $g\in G^\BC$ and $x\in \fp^+$ satisfy $g\exp(x)\in P^+K^\BC P^-$, we write
\begin{gather}\label{cocycle}
g\exp(x)=\exp(\pi^+(g,x))\kappa(g,x)\exp(\pi^-(g,x)),
\end{gather}
where $\pi^+(g,x)\in\fp^+$, $\kappa(g,x)\in K^\BC$, and $\pi^-(g,x)\in\fp^-$.
If $g=k\in K^\BC$, $g=\exp(y)\in P^+$ or $g=\exp(\vartheta y)\in P^-$ with $y\in\fp^+$, we have
\begin{alignat*}{3}
&\pi^+(k,x)=kx,\qquad && \kappa(k,x)=k,&\\
& \pi^+(\exp(y),x)=x+y,\qquad && \kappa(\exp(y),x) =\bone_{K^\BC},& \\
& \pi^+(\exp(\vartheta y),x) =x^y,\qquad && \operatorname{Ad}(\kappa(\exp(\vartheta y),x))|_{\fp^+}=B(x,y)^{-1}.&
\end{alignat*}
$\pi^+$ gives the birational action of $G^\BC$ on $\fp^+$, and from now on we abbreviate $\pi^+(g,x)=:gx$.
Especially, if $x\in D$ and $g\in G$, then automatically $gx\in D$ and $\kappa(g,x)$ is well-defined,
and the action of $G$ on $D$ is transitive. Since $D$ is simply connected, the map $\kappa\colon G\times D\to K^\BC$
lifts to the universal covering space, that is, $\kappa\colon \tilde{G}\times D\to \tilde{K}^\BC$ is well-defined.
We denote this extended map by the same symbol~$\kappa$. Then for $x,y\in\fp^+$ and $g\in G^\BC$,
\begin{gather}\label{Bergman_property}
B\big(gx,\big(\hat{\vartheta}g\big)y\big)=\kappa(g,x)B(x,y)\kappa\big(\hat{\vartheta}g,y\big)^*
\end{gather}
holds in $\End(\fp^+)$, where $\hat{\vartheta}$ is the anti-holomorphic involution of $G^\BC$ fixing $G$, and $\operatorname{Ad}$ is omitted.
If $g\in G$ (i.e., $g=\hat{\vartheta}g$) and $x,y\in D$, this also holds in $K^\BC$, regarding $B(x,y)$ as the element of~$K^\BC$.
This formula is also verified in $\tilde{K}^\BC$ if $g\in \tilde{G}$.

Now let $(\tau,V)$ be an irreducible holomorphic representation of $\tilde{K}^\BC$ with $\tilde{K}$-invariant inner product $(\cdot,\cdot)_\tau$.
We consider the space of holomorphic sections of the homogeneous vector bundle on $G/K$ with fiber~$V$. Then since $D\simeq G/K$ is contractible, it is isomorphic to the space of $V$-valued holomorphic functions on~$D$
\begin{gather*} \Gamma_\cO\big(G/K, \tilde{G}\times_{\tilde{K}}V\big)\simeq \cO(D,V). \end{gather*}
Via this identification, $\tilde{G}$ acts on $\cO(D,V)=\cO_\tau(D,V)$ by
\begin{gather*} \hat{\tau}(g)f(x)=\tau\big(\kappa\big(g^{-1},x\big)\big)^{-1}f\big(g^{-1}x\big),\qquad g\in\tilde{G}, \quad x\in D, \quad f\in\cO(D,V) , \end{gather*}
and the function $\tau(B(x,y))\in \cO(D\times\overline{D},\End(V))$ is invariant under the diagonal action of~$\tilde{G}$.
If the function $\tau(B(x,y))$ is positive-definite, that is,
$\sum\limits_{j,k=1}^N(\tau(B(x_j,x_k))v_j,v_k)_\tau\ge 0$ holds for any $\{x_j\}_{j=1}^N\subset D$ and $\{v_j\}_{j=1}^N\subset V$,
then there exists a unique Hilbert subspace $\cH_\tau(D,V)\subset\cO_\tau(D,V)$ with the reproducing kernel $\tau(B(x,y))$,
on which $\tilde{G}$ acts unitarily via $\hat{\tau}$.
This representation $(\hat{\tau},\cH_\tau(D,V))$ is called a unitary highest weight representation.
Especially, if its inner product is given by the converging integral
\begin{gather*} \langle f,g\rangle_{\hat{\tau}}:=C_\tau\int_D \big(\tau\big(B(x)^{-1}\big)f(x),g(x)\big)_{\tau}h(x)^{-p}{\rm d}x, \end{gather*}
where $h(x)^{-p}{\rm d}x:=\prod\limits_{i=1}^m h_{(i)}(x_i)^{-p_{(i)}}{\rm d}x=\Det(B(x))^{-1}{\rm d}x$ is the $G$-invariant measure on~$D$,
and~$C_\tau$ is the constant such that $\Vert v\Vert_{\hat{\tau}}=|v|_\tau$ holds for any constant functions
(or elements in the minimal $K$-type)~$v$,
then $(\hat{\tau},\cH_\tau(D,V))$ is called a holomorphic discrete series representation.
In this case, all bounded holomorphic functions on $D$ belong to $\cH_\tau(D,V)$,
and especially the space of $\tilde{K}$-finite vectors is equal to the space of all polynomials,
\begin{gather*} \cH_\tau(D,V)_{\tilde{K}}=\cO_\tau(D,V)_{\tilde{K}}=\cP(\fp^+,V). \end{gather*}
For general $\cH_\tau(D,V)$ such that the above integral does not converge for any non-zero function,
it may happen that the $\tilde{K}$-finite part $\cH_\tau(D,V)_{\tilde{K}}$ is strictly smaller than $\cP(\fp^+,V)$.

Now we assume $G$ is simple. Let $\chi$ be the character of $\tilde{K}^\BC$ such that
\begin{gather}\label{char}
\chi(k)^p=\Det(\operatorname{Ad}(k)|_{\fp^+}),\qquad \text{or} \qquad \chi(B(x,y))=h(x,y).
\end{gather}
Then for $x,y\in\fp^+$ we have ${\rm d}\chi([x,-\vartheta y])=(x|y)_{\fp^+}$, where $(\cdot|\cdot)_{\fp^+}$ is as in (\ref{innerprod}).
Let $(\tau_0,V)$ be a fixed irreducleble representation of $K^\BC$.
Then for $\lambda\in\BR$, $(\tau,V)=(\tau_0\otimes\chi^{-\lambda},V)$ is again a~representation of $\tilde{K}^\BC$.
In this case we denote $\cH_\tau(D,V)=:\cH_\lambda(D,V)$.
If $\lambda$ is sufficiently large, this becomes a holomorphic discrete series representation.
The parameter $\lambda$ such that the unitary subrepresentation $\cH_\lambda(D,V)\subset\cO_\lambda(D,V)$ exists is classified by
Enright--Howe--Wallach~\cite{EHW} and Jakobsen~\cite{J}.

When $\cH_\lambda(D,V)$ is holomorphic discrete, we consider the irreducible decomposition of \linebreak
$\cH_\lambda(D,V)_{\tilde{K}}\simeq \cP(\fp^+,V)\otimes \chi^{-\lambda}$ as $\tilde{K}^\BC$-modules,
\begin{gather*} \cP(\fp^+,V)\otimes \chi^{-\lambda}\simeq \bigoplus_m W_m\otimes\chi^{-\lambda}, \end{gather*}
such that its components are orthogonal to one another with respect to the Fischer inner product
\begin{gather*} \langle f,g\rangle_{F,\tau_0}:=\frac{1}{\pi^n}\int_{\fp^+} (f(x),g(x) )_{\tau_0}{\rm e}^{-|x|_{\fp^+}^2}{\rm d}x. \end{gather*}
Then since both $\langle \cdot,\cdot\rangle_{F,\tau_0}$ and
$\langle\cdot,\cdot\rangle_{\hat{\tau}}=\langle \cdot,\cdot\rangle_{\lambda,\tau_0}$ are $\tilde{K}$-invariant,
there exists a constant $p_m(\lambda)>0$ such that $\Vert f_m\Vert^2_{\lambda,\tau_0}/\Vert f_m\Vert^2_{F,\tau_0}=p_m(\lambda)$
for any $f_m\in W_m$.
Now we additionally assume that
\begin{gather}\label{assumption_orth}
W_m\perp W_n \text{ in } \langle\cdot,\cdot\rangle_{F,\tau_0}\qquad \text{implies}\qquad
W_m\perp W_n \text{ in } \langle\cdot,\cdot\rangle_{\lambda,\tau_0}
\end{gather}
for sufficiently large $\lambda$.
This holds if $\cP(\fp^+,V)$ is multiplicity-free, or $G=U(q,s)$ and one of~$U(q)$ and~$U(s)$ acts trivially on $V$.
Then $\Vert \cdot\Vert_{\lambda,\tau_0}$ is computed as
\begin{gather}\label{norm_compare}
\Vert f\Vert_{\lambda,\tau_0}^2=\sum_{m,n}\langle f_m,f_n\rangle_{\lambda,\tau_0}=\sum_m p_m(\lambda)\Vert f_m\Vert_{F,\tau_0}^2,
\end{gather}
where $f_m$ is the orthogonal projection of $f$ onto $W_m$, that is, the cross terms vanish.
Especially, if the reproducing kernel with respect to $\langle\cdot,\cdot\rangle_{F,\tau_0}$ is expanded as
\begin{gather*} {\rm e}^{(x|y)_{\fp^+}}I_V=\sum_m \bK_m(x,y)\in\cO\big(\fp^+\times\overline{\fp^+},\End(V)\big), \end{gather*}
where \looseness=-1 $\bK_m(x,y)\in W_m\otimes\overline{W_m}$, then the reproducing kernel with respect to $\langle\cdot,\cdot\rangle_{\lambda,\tau_0}$ is expanded as
\begin{gather}\label{kernel_expansion}
h(x,y)^{-\lambda}\tau_0(B(x,y))=\sum_m p_m(\lambda)^{-1}\bK_m(x,y)\in\cO\big(D\times\overline{D},\End(V)\big).
\end{gather}
Each $p_m(\lambda)$ is meromorphically continued for all $\lambda\in\BC$, and defines a positive definite kernel function if
$p_m(\lambda)^{-1}\ge 0$ for all $m$, and for such $\lambda$ there exists a unitary subrepresentation $\cH_\lambda(D,V)\subset\cO(D,V)$.
Also, if $f_m\in W_m$ we have
\begin{gather}\label{exp_onD}
\langle f_m,{\rm e}^{(\cdot|y)_{\fp^+}}\rangle_{\lambda,\tau_0}=p_m(\lambda)\langle f_m,\bK_m(\cdot,y)\rangle_{F,\tau}=p_m(\lambda)f_m(y).
\end{gather}
Under this assumption (\ref{assumption_orth}), $p_m(\lambda)$ are explicitly computed in \cite{N2} for classical $G$.
We note that if we drop the assumption (\ref{assumption_orth}), then the formulas (\ref{norm_compare}) and (\ref{kernel_expansion}) become
more complicated, and their explicit formulas are not known so far.
Especially if $(\tau_0,V)$ is trivial,
then by \cite[Part~III, Corollary~V.3.9]{FKKLR}, \cite{FK0} for $f_\bm\in\cP_\bm(\fp^+)$ we have
\begin{gather}\label{scalar_norm}
p_\bm(\lambda)=\frac{\Vert f_\bm\Vert_\lambda^2}{\Vert f_\bm\Vert_F^2}=\frac{1}{(\lambda)_{\bm,d}},
\end{gather}
where $(\lambda)_{\bm,d}$ is as (\ref{Pochhammer}), and therefore we have
\begin{gather}\label{ext_of_h}
h(x,y)^{-\lambda}=\sum_{\bm\in\BZ_{++}^r}(\lambda)_{\bm,d}\bK_\bm^{(d)}(x,y),
\end{gather}
where $\bK_\bm^{(d)}(x,y)\in\cP_\bm(\fp^+)\otimes \overline{\cP_\bm(\fp^+)}$ is as in the previous subsection.

\section[Intertwining operators between holomorphic discrete series representations]{Intertwining operators between holomorphic discrete series\\ representations}\label{section3}
\subsection{Setting}
Let $G$ be a connected real reductive Lie group such that each simple non-compact component is of Hermitian type, as in Section~\ref{root_sys},
and let $z\in\fz(\fk)$ be the element such that $\operatorname{ad}(z)|_{\fp^+}=\sqrt{-1}I_{\fp^+}$.
Let $G_1\subset G$ be a connected reductive subgroup which is stable under the Cartan involution $\vartheta$ of $G$.
We denote the Lie algebra of $G_1$ and its Cartan decomposition under $\vartheta$ by $\fg_1=\fk_1\oplus \fp_1$. We assume
\begin{gather}\label{assumption}
z\in \fg_1.
\end{gather}
We set $\fp_1^+:=\fg_1^\BC\cap \fp^+$, $\fp_1^-:=\fg_1^\BC\cap \fp^-$, so that
$\fg_1^\BC=\fp^+_1\oplus \fk^\BC_1\oplus \fp^-_1$.
Also, let $\fp_2^+\subset\fp^+$ be the orthogonal complement of $\fp^+_1$ with respect to the inner product $(\cdot|\cdot)_{\fp^+}$
defined in (\ref{innerprod}).
We define another inner product $(\cdot|\cdot)_{\fp^+_1}$ on $\fp^+_1$ as in (\ref{innerprod}), changing $\fg$ to $\fg_1$,
and let $D_1\subset\fp^+_1$ be the bounded symmetric domain, defined as in (\ref{BSD}). Then we have
\begin{Proposition}\label{BSD_intersection}
$D_1=D\cap \fp^+_1$.
\end{Proposition}
\begin{proof}$D_1=G_1.0\subset D\cap\fp^+_1$ is clear. For $x,y\in \fp^+_1$, let $B_1(x,y)$ be the Bergman operator on~$\fp^+_1$.
Then $B_1(x,y)=B(x,y)|_{\fp^+_1}$ holds. Therefore if $x\in D\cap \fp^+_1$, then $B(x)$ is positive definite,
and hence $B_1(x)=B(x)|_{\fp^+_1}$ is also positive definite. Since $D\cap \fp^+_1$ is connected, $D\cap\fp^+_1\subset D_1$ holds.
\end{proof}

This implies that the spectral norms (\ref{spectral_norm}) on $\fp^+$ and $\fp^+_1$ coincide.
\begin{Corollary}
For $x\in\fp^+_1\subset\fp^+$, $|x|_{\fp^+,\infty}=|x|_{\fp^+_1,\infty}$ holds.
\end{Corollary}
\begin{proof}This follows from $|x|_{\fp^+,\infty}=\inf\left\{t>0\colon \frac{1}{t}x\in D\right\}$,
$|x|_{\fp^+_1,\infty}=\inf\left\{t>0\colon \frac{1}{t}x\in D_1\right\}$ \linebreak (see~(\ref{BSD})) and Proposition~\ref{BSD_intersection}.
\end{proof}

Let $(\tau,V)$ be a representation of $\tilde{K}^\BC$, and consider the representation $(\hat{\tau},\cH_\tau(D,V))$ of $\tilde{G}$, as in Section~\ref{HDS}.
We want to discuss the restriction $\cH_\tau(D,V)|_{\tilde{G}_1}$.
Then since it is discretely decomposable, the space of $\tilde{K}_1$-finite vectors coincides with the space of $\tilde{K}$-finite vectors
(see \cite[Theorem~4.5]{K1}),
which is a subspace of the space of $V$-valued polynomials on $\fp^+$.
\begin{gather*} \cH_\tau(D,V)_{\tilde{K}_1}=\cH_\tau(D,V)_{\tilde{K}}\subset\cP(\fp^+,V). \end{gather*}
Since $\fp^+$ acts on $\cH_\tau(D,V)_{\tilde{K}}\subset\cP(\fp^+,V)$ by 1st order differential operators with constant coefficients,
every $(\fg_1,\tilde{K}_1)$-submodule in $\cH_\tau(D,V)_{\tilde{K}_1}$ has $\fp^+_1$-invariant vectors,
and the space of $\fp^+_1$-invariant vectors is equal to
\begin{gather*} \cH_\tau(D,V)_{\tilde{K}_1}^{\fp^+_1}=\cH_\tau(D,V)\cap\cP(\fp^+_2,V). \end{gather*}
Thus if we write the decomposition of the above space under $\tilde{K}_1^\BC$ as
\begin{gather*} \cH_\tau(D,V)\cap\cP(\fp^+_2,V)\simeq \bigoplus_i m(\rho_i)(\rho_i,W_i), \end{gather*}
then $\cH_\tau(D,V)$ is decomposed under $\tilde{G}_1$ abstractly as
\begin{gather*}
\cH_\tau(D,V)_{\tilde{K}}|_{(\fg_1,\tilde{K}_1)} \simeq \bigoplus_i m(\rho_i)\cH_{\rho_i}(D_1,W_i)_{\tilde{K}_1},\\
\cH_\tau(D,V)|_{\tilde{G}_1} \simeq \hsum_i m(\rho_i)\cH_{\rho_i}(D_1,W_i)
\end{gather*}
(see \cite{JV}, \cite[Section 8]{Kmf1}, \cite{Ma}). Especially, if $\cH_\tau(D,V)_{\tilde{K}}=\cP(\fp^+, V)$, the decomposition of
$\cH_\tau(D,V)$ under $\tilde{G}_1$ corresponds to the decomposition of $\cP(\fp^+_2,V)\simeq\cP(\fp^+_2)\otimes V$ under $\tilde{K}_1^\BC$.
Now let $(\rho,W)$ be an irreducible representation of $\tilde{K}_1^\BC$ which appears in $\cP(\fp^+_2)\otimes V|_{\tilde{K}_1^\BC}$,
with $\tilde{K}_1$-invariant inner product $(u,v)_\rho$.
Our aim is to construct $\tilde{G}_1$-intertwining operators between $\cH_\tau(D,V)|_{\tilde{G}_1}$ and $\cH_\rho(D_1,W)$
explicitly. To do this, we assume $\cH_\tau(D,V)_{\tilde{K}}=\cP(\fp^+,V)$.

\subsection{Integral expression}

First we find the intertwining operators in integral expression.
Let $(\rho,W)$ be an irreducible representation of $\tilde{K}_1^\BC$ which appears in $\cP(\fp^+_2)\otimes V|_{\tilde{K}_1^\BC}$,
and let $\cF_{\tau\rho}^*\colon \cH_\tau(D,V)\to \cH_{\rho}(D_1,W)$ be a $\tilde{G}_1$-intertwining operator.
Then for any $y_1\in D_1$, the linear map $\cH_\tau(D,V)\to W$, $f\mapsto (\cF_{\tau\rho}^* f)(y_1)$ is continuous,
and by the Riesz representation theorem, there exists $\hat{\rK}_{y_1}\in\cH_\tau(D,V)\otimes \overline{W}$ such that
\begin{gather*} \langle f,\hat{\rK}_{y_1}\rangle_{\hat{\tau}}=(\cF_{\tau\rho}^* f)(y_1), \qquad f\in\cH_\tau(D,V),\quad y_1\in D_1. \end{gather*}
We write $\hat{\rK}(x;y_1):=\hat{\rK}_{y_1}(x)$ for $x\in D$, $y_1\in D_1$.
We identify $V\otimes \overline{W}$ and $\Hom(W,V)$ via the inner product of $W$.
Then $\cF_{\tau\rho}^*\colon \cH_\tau(D,V)\to \cH_{\rho}(D_1,W)$ and its adjoint operator $\cF_{\tau\rho}\colon \cH_\rho(D_1,W)\to \cH_\tau(D,V)$ are given by
\begin{alignat*}{3}
&(\cF_{\tau\rho}^* f)(y_1)=\langle f,\hat{\rK}(\cdot,y_1)\rangle_{\hat{\tau}},\qquad && f\in\cH_\tau(D,V),\quad y_1\in D_1,\\
& (\cF_{\tau\rho} f)(x)=\langle f,\hat{\rK}(x,\cdot)^*\rangle_{\hat{\rho}}, \qquad && f\in\cH_\rho(D,W),\quad x\in D.
\end{alignat*}
By the intertwining property, $\hat{\rK}(x;y_1)$ must satisfy
\begin{gather}\label{G1-invariance}
\hat{\rK}(gx;gy_1)=\tau(\kappa(g,x))\hat{\rK}(x;y_1)\rho(\kappa(g,y_1))^*
\end{gather}
for any $g\in\tilde{G}_1$, where $\kappa(g,x)$ is as (\ref{cocycle}). Thus we find a kernel function satisfying (\ref{G1-invariance}).

Let $\rK(x_2)\in\cP(\fp^+_2)\otimes V\otimes\overline{W}\simeq\cP(\fp^+_2,\Hom(W,V))$ be an operator-valued polynomial satisfying
\begin{gather}\label{K-invariance}
\rK(kx_2)=\tau(k)\rK(x_2)\rho(k)^{-1}, \qquad x_2\in \fp^+_2,\quad k\in \tilde{K}^\BC_1.
\end{gather}
Let $\Proj_2\colon \fp^+\to\fp^+_2$ be the orthogonal projection, and we define an operator-valued function $\hat{\rK}\in\cO(D\times \overline{D_1}, \Hom(W,V))$ by
\begin{gather}\label{Khat}
\hat{\rK}(x;y_1)=\hat{\rK}(x_1,x_2;y_1):=\tau(B(x,y_1))\rK(\Proj_2(x^{y_1})),
\end{gather}
where $x=(x_1,x_2)\in D$, $y_1\in D_1$. Then the following holds.
\begin{Proposition}\label{kernelfunc}
For any $x\in D$, $y_1\in D_1$, and $g\in\tilde{G}_1$, $\hat{\rK}(x;y_1)$ satisfies the identity \eqref{G1-invariance}.
\end{Proposition}
\begin{proof}
By (\ref{Bergman_property}), we have
\begin{gather*} \tau(B(gx,gy_1))=\tau(\kappa(g,x))\tau(B(x,y_1))\tau(\kappa(g,y_1))^*. \end{gather*}
Thus it suffices to show
\begin{gather*} \rK(\Proj_2((gx)^{gy_1}))=\tau(\kappa(g,y_1))^{*-1}\rK(\Proj_2(x^{y_1}))\rho(\kappa(g,y_1))^*. \end{gather*}
By $\tilde{K}^\BC_1$-invariance of $\rK(\cdot)$, this is equivalent to
\begin{gather*} \Proj_2((gx)^{gy_1})=\kappa(g,y_1)^{*-1}\Proj_2(x^{y_1}), \qquad x\in D,\quad y_1\in D_1,\quad g\in G_1. \end{gather*}
First we show
\begin{gather*} \Proj_2\big((gx)^{(\hat{\vartheta}g)y_1}\big)=\kappa(\hat{\vartheta}g,y_1)^{*-1}\Proj_2(x^{y_1}), \qquad x\in \fp^+,\quad y_1\in \fp^+_1 \end{gather*}
for $g=k\in K_1^\BC$ or $g=\exp(-z_1)$, $g=\exp(\vartheta w_1)\in G_1^\BC$ with $z_1,w_1\in\fp^+_1$,
when one side is well-defined, that is, we show
\begin{gather*}
\Proj_2\big((kx)^{k^{*-1}y_1}\big)=k\Proj_2(x^{y_1}), \\
\Proj_2\big((x-z_1)^{(y_1^{z_1})}\big)=B(z_1,y_1)\Proj_2(x^{y_1}), \\
\Proj_2\big((x^{w_1})^{y_1-w_1}\big)=\Proj_2(x^{y_1}).
\end{gather*}
In fact, the 1st and 3rd formulas follow from (\ref{quasiinv_k}), (\ref{quasiinv_add}),
and the 2nd formula follows from $(x-z_1)^{(y_1^{z_1})}=B(z_1,y_1)x^{y_1}-B(z_1,y_1)z_1^{y_1}$,
which is a consequence of~(\ref{quasiinv_twice}), and that $B(z_1,y_1)z_1^{y_1}\in\fp^+_1$ is annihilated by $\Proj_2$.
Since any $g\in G_1$ is written as the form $g=\exp(\vartheta w_1)k\exp(-z_1)$ with $z_1,w_1\in D_1$ and $k\in K^\BC_1$
(which is proved by using the $KAK$-decomposition and \cite[Part~III, Lemma~III.2.4]{FKKLR}),
the proposition follows from the cocycle condition of $\kappa$.
\end{proof}

We write $\cP(\fp^+_2,\Hom(W,V))^{\tilde{K}^\BC_1}:=\{\rK(x_2)\in \cP(\fp^+_2,\Hom(W,V))\colon \text{satisfying }(\ref{K-invariance})\}$,
$\cO(D\times \overline{D_1}$, $\Hom(W,V))^{\tilde{G}_1}
:=\{\hat{\rK}(x;y_1)\in\cO(D\times \overline{D_1}, \Hom(W,V))\colon \text{satisfying }(\ref{G1-invariance})\}$.
Then we have the following.
\begin{Lemma}
The linear map $\cO(D\times \overline{D_1}, \Hom(W,V))^{\tilde{G}_1}\to\cP(\fp^+_2,\Hom(W,V))^{\tilde{K}^\BC_1}$,
$\hat{\rK}(x;y_1)\mapsto \hat{\rK}(0,x_2;0)$ is bijective, and its inverse is given by $\rK(x_2)\mapsto \hat{\rK}(x;y_1)$ in \eqref{Khat}.
\end{Lemma}
\begin{proof}We write the restriction map by $\Rest$, that is, $(\Rest \hat{\rK})(x_2):=\hat{\rK}(0,x_2;0)$,
and write the linear map in~(\ref{Khat}) by $\Ext$.
For any $\hat{\rK}\in\cO(D\times \overline{D_1}, \Hom(W,V))^{\tilde{G}_1}$, clearly $(\Rest \hat{\rK})(x_2)$ satisfies
(\ref{K-invariance}), and therefore
\begin{gather*}
\big(\Rest \hat{\rK}\big)(x_2)\in \cO(D\cap \fp^+_2,\Hom(W,V))^{\tilde{K}^\BC_1}\\
\qquad {} \simeq \Hom_{\tilde{K}^\BC_1}(W,\cO(D\cap\fp^+_2,V)) \simeq\Hom_{\tilde{K}^\BC_1}(W,\cP(\fp^+_2,V))\simeq\cP(\fp^+_2,\Hom(W,V))^{\tilde{K}^\BC_1}
\end{gather*}
holds.
Also, for any $\rK(x_2)\in\cP(\fp^+_2,\allowbreak\Hom(W,V))^{\tilde{K}^\BC_1}$ we have
\begin{gather*} (\Rest\circ\Ext \rK)(x_2)=\tau(B(x_2,0))\rK\big(\Proj_2\big((x_2)^0\big)\big)=\rK(x_2). \end{gather*}
Thus $\Rest$ is surjective. Therefore it suffices to show that $\Rest$ is injective.
Suppose $\hat{\rK}_1$ and $\hat{\rK}_2\in\cO(D\times \overline{D_1}, \Hom(W,V))^{\tilde{G}_1}$ satisfy $\Rest\hat{\rK}_1=\Rest\hat{\rK}_2$.
Then by~(\ref{G1-invariance}),
\begin{gather*} \hat{\rK}_1(g.(0,x_2);g.0)=\hat{\rK}_2(g.(0,x_2);g.0) \end{gather*}
holds for any $g\in\tilde{G}_1$.
We fix $x_2\in D\cap \fp^+_2$, and set
\begin{gather*}
S_{x_2} :=\left\{ (g.(0,x_2);g.0)\in D\times\overline{D_1}\colon g\in G_1\right\}\subset D\times\overline{D_1}, \\
D_{x_2} :=\left\{ (g.(0,x_2);h.0)\in D\times\overline{D_1}\colon g,h\in G_1\right\}\subset D\times\overline{D_1}.
\end{gather*}
We show that $S_{x_2}$ contains a totally real submanifold of full dimension of $D_{x_2}$.
Let $\mathrm{pr}_1\colon D\times\overline{D_1}\to D\to D\cap\fp^+_1=D_1$, $\mathrm{pr}_2\colon D\times\overline{D_1}\to \overline{D_1}$
be the projections. Then since for every $x_2\in D\cap\fp^+_2$, $\{\exp(z).(0,x_2)\colon z\in\fp_1\}\subset D$ intersects transversally
with $\fp^+_2$ at $x_2$, the differential of $\mathrm{pr}_1|_{S_{x_2}}$ at $(0,x_2;0)$ is surjective.
Similarly, since $G_1$ acts transitively on $D_1$, the differential of $\mathrm{pr}_2|_{S_{x_2}}$ at $(0,x_2;0)$ is also surjective.
Therefore, $\mathrm{pr}_1|_{S_{x_2}}$ and $\mathrm{pr}_2|_{S_{x_2}}$ are both submersive near $(0,x_2;0)\in S_{x_2}$,
and $T_{(x;y_1)}S_{x_2}+JT_{(x;y_1)}S_{x_2}=T_{(x;y_1)}D_{x_2}$ holds on this neighborhood,
where $J$ is the complex structure of $D_{x_2}$. Hence $S_{x_2}$ contains a totally real submanifold of full dimension of $D_{x_2}$,
and $\hat{\rK}_1=\hat{\rK}_2$ holds on $D_{x_2}$ for each $x_2\in D\cap\fp^+_2$.
Finally, since $\bigcup_{x_2\in D\cap\fp^+_2} D_{x_2}$ contains an open subset of $D\times\overline{D_1}$,
$\hat{\rK}_1=\hat{\rK}_2$ holds on whole $D\times\overline{D_1}$. Therefore $\Rest$ is injective.
\end{proof}

Therefore if we assume $\cH_\tau(D,V)_{\tilde{K}}=\cP(\fp^+,V)$, then for any $(\rho,W)\subset \cP(\fp^+_2,V)$,
\begin{gather*}
 \overline{\Hom_{\tilde{G}_1}(\cH_\tau(D,V),\cH_\rho(D_1,W))}\simeq \Hom_{\tilde{G}_1}(\cH_\rho(D_1,W),\cH_\tau(D,V)) \\
\hphantom{\overline{\Hom_{\tilde{G}_1}(\cH_\tau(D,V),\cH_\rho(D_1,W))}}{} \simeq \cP(\fp^+_2,\Hom(W,V))^{\tilde{K}^\BC_1}\simeq \cO(D\times \overline{D_1}, \Hom(W,V))^{\tilde{G}_1}
\end{gather*}
holds, and by the above argument,
for any $\rK\in\cP(\fp^+_2,\Hom(W,V))^{\tilde{K}^\BC_1}$, $\hat{\rK}(x;y_1)$ in (\ref{Khat}) becomes
the kernel function of the intertwining operator.
Especially, $\hat{\rK}(\cdot;y_1)\in\cH_\tau(D,V)\otimes \overline{W}$ holds for any $y_1\in D_1$,
and $\hat{\rK}(x;\cdot)^*\in\cH_{\rho}(D_1,W)\otimes \overline{V}$ holds for any $x\in D$.
Also, even if $\cH_\tau(D,V)_{\tilde{K}}\ne\cP(\fp^+)\otimes V$, since $\hat{\rK}(x;\cdot)$ is a bounded function on $D_1$ for any $x\in D$,
$\hat{\rK}(x;\cdot)^*\in\cH_{\rho}(D_1,W)\otimes \overline{V}$ holds
if we assume $\cH_\rho(D_1,W)$ is a holomorphic discrete series representation.
\begin{Corollary}\label{integral_expression}
Let $(\tau,V)$ be a $\tilde{K}^\BC$-module, and $(\rho,W)$ be a $\tilde{K}_1^\BC$-module
which appears in $\cP(\fp^+_2)\otimes V|_{\tilde{K}_1^\BC}$.
Let $\rK(x_2)\in \cP(\fp^+_2)\otimes V\otimes\overline{W}\simeq\cP(\fp^+_2,\Hom(W,V))$ be an operator-valued polynomial
satisfying \eqref{K-invariance}, and define $\hat{\rK}(x;y_1)\in\cO(D\times\overline{D_1},\Hom(W,V))$ by \eqref{Khat}.
\begin{enumerate}\itemsep=0pt
\item[$(1)$] Assume $\cH_\tau(D,V)_{\tilde{K}}=\cP(\fp^+,V)$. Then the linear maps
\begin{gather*}
\cF_{\tau\rho}^*\colon \ \cH_\tau(D,V)\to\cH_\rho(D_1,W), \\
(\cF_{\tau\rho}^*f)(y_1):=\big\langle f,\hat{\rK}(\cdot;y_1)\big\rangle_{\hat{\tau}}
=C_\tau\int_D\hat{\rK}(x;y_1)^*\tau\big(B(x)^{-1}\big)f(x)h(x)^{-p}{\rm d}x, \\
\cF_{\tau\rho}\colon \ \cH_\rho(D_1,W)\to \cH_\tau(D,V), \\
(\cF_{\tau\rho}f)(x):=\big\langle f,\hat{\rK}(x;\cdot)^*\big\rangle_{\hat{\rho}}
=C_\rho\int_{D_1}\hat{\rK}(x;y_1)\rho\big(B(y_1)^{-1}\big)f(y_1)h_1(y_1)^{-p_1}{\rm d}y_1
\end{gather*}
intertwine the $\tilde{G}_1$-action. Here the second equalities hold only if $\cH_\tau(D{,}V)$ resp.~$\cH_\rho(D_1{,}W)$ are holomorphic discrete series representations.
\item[$(2)$] Assume $\cH_\rho(D_1,W)$ is a holomorphic discrete series representation.
Then the linear map $\cF_{\tau\rho}\colon \cH_\rho(D_1,W)\to \cO_\tau(D,V)$,
\begin{gather*} (\cF_{\tau\rho}f)(x):=\big\langle f,\hat{\rK}(x;\cdot)^*\big\rangle_{\hat{\rho}}
=C_\rho\int_{D_1}\hat{\rK}(x;y_1)\rho\big(B(y_1)^{-1}\big)f(y_1)h_1(y_1)^{-p_1}{\rm d}y_1 \end{gather*}
intertwines the $\tilde{G}_1$-action.
\end{enumerate}
\end{Corollary}
These operators are defined as maps from $\cH_\tau(D,V)$ resp.~$\cH_\rho(D_1,W)$, but in fact these extend continuously to maps between
$\cO_\tau(D,V)$ and $\cO_\rho(D_1,W)$ if $\cH_\tau(D,V)$ or $\cH_\rho(D_1,W)$ is a~holomorphic discrete series representation.
\begin{Theorem}\label{conti_ext}\quad
\begin{enumerate}\itemsep=0pt
\item[$(1)$] Assume $\cH_\tau(D,V)$ is a holomorphic discrete series representation. Then
the linear map $\cF_{\tau\rho}^*\colon\! \cH_\tau(D,V)\!\to\!\cH_\rho(D_1{,}W)$ extends continuously to the map $\cF_{\tau\rho}^*\colon\! \cO_\tau(D,V)\!\to\!\cO_\rho(D_1{,}W)$.
\item[$(2)$] Assume $\cH_\rho(D_1,W)$ is a holomorphic discrete series representation. Then
the linear map $\cF_{\tau\rho}\colon\! \cH_\rho(D_1,W)\!\to\! \cO_\tau(D{,}V)$ extends continuously to the map $\cF_{\tau\rho}\colon\! \cO_\rho(D_1,W)\!\to \! \cO_\tau(D{,}V)$.
\end{enumerate}
\end{Theorem}
\begin{proof}
We only prove (2). By the $\tilde{G}_1$-invariance of $\hat{\rK}$, for $k\in \tilde{K}_1$ we have
\begin{gather*} (\cF_{\tau\rho}f)(x)=\big\langle f(y_1),\hat{\rK}(x;y_1)^*\big\rangle_{\hat{\rho},y_1}
=\big\langle \rho(k)^{-1}f(ky_1),\hat{\rK}\big(k^{-1}x;y_1\big)^*\tau(k)^*\big\rangle_{\hat{\rho},y_1}. \end{gather*}
Let $\hbar=-\sqrt{-1}z\in\sqrt{-1}\fz(\fk)$ be the element such that $\operatorname{ad}(\hbar)|_{\fp^+}=I_{\fp^+}$.
Then for $t\in\sqrt{-1}\BR$ and $v\in V$, by setting $k={\rm e}^{-t\hbar}$ we have
\begin{gather*} ((\cF_{\tau\rho}f)(x),v)_\tau=\chi_\rho\big({\rm e}^{t\hbar}\big)\chi_\tau\big({\rm e}^{-t\hbar}\big)\big\langle f\big({\rm e}^{-t}y_1\big),
\hat{\rK}\big({\rm e}^tx;y_1\big)^*v\big\rangle_{\hat{\rho},y_1}, \end{gather*}
where we write $\rho({\rm e}^{t\hbar})=\chi_\rho({\rm e}^{t\hbar})I_W$, $\tau({\rm e}^{t\hbar})=\chi_\tau({\rm e}^{t\hbar})I_V$.
Here, though $\chi_\rho$ and $\chi_\tau$ are defined as a~function on $Z(\tilde{K})\subset \tilde{G}$,
their ratio $\chi_\rho\chi_\tau^{-1}$ is well-defined as a function on $Z(K)\subset G$.
By analytic continuation this holds for $t\in\BC$ such that $|{\rm e}^t|\ge 1$ and $|{\rm e}^t x|_\infty <1$.
Then for $|x|_\infty<{\rm e}^{-t}\le 1$ we have
\begin{align*}
|((\cF_{\tau\rho}f)(x),v)_\tau|&=\big|\chi_\rho({\rm e}^{t\hbar})\chi_\tau({\rm e}^{-t\hbar})\big\langle f\big({\rm e}^{-t}y_1\big),
\hat{\rK}\big({\rm e}^tx;y_1\big)^*v\big\rangle_{\hat{\rho},y_1}\big| \\
&\le \chi_\rho({\rm e}^{t\hbar})\chi_\tau\big({\rm e}^{-t\hbar}\big)\big\Vert f\big({\rm e}^{-t}y_1\big)\big\Vert_{\hat{\rho},y_1}
\big\Vert\hat{\rK}\big({\rm e}^tx;y_1\big)^*v\big\Vert_{\hat{\rho},y_1}.
\end{align*}
Now since $\cH_\rho(D_1,W)$ is a holomorphic discrete series representation, we have
\begin{align*}
\big\Vert f\big({\rm e}^{-t}y_1\big)\big\Vert_{\hat{\rho},y_1}^2
&=C_\rho\int_{D_1}\big(\rho(B(y_1)^{-1})f\big({\rm e}^{-t}y_1\big),f\big({\rm e}^{-t}y_1\big)\big)_\rho h_1(y_1)^{-p_1}{\rm d}y_1 \\
&\le C_\rho\int_{D_1}\big|\rho\big(B(y_1)^{-1}\big)\big|_{\rho,\mathrm{op}}h_1(y_1)^{-p_1}{\rm d}y_1 \sup_{y_1\in D_1}\big|f\big({\rm e}^{-t}y_1\big)\big|_\rho^2 \\
&=C_\rho^{\prime 2}\sup_{|y_1|_\infty\le {\rm e}^{-t}} |f(y_1)|_\rho^2.
\end{align*}
Therefore for $|x|_\infty<{\rm e}^{-t}$ we have
\begin{gather*}
 |((\cF_{\tau\rho}f)(x),v)_\tau|
\le C_\rho'\chi_\rho\big({\rm e}^{t\hbar}\big)\chi_\tau\big({\rm e}^{-t\hbar}\big)\sup_{|y_1|_\infty\le {\rm e}^{-t}}|f(y_1)|_\rho
\big(\big\langle\hat{\rK}\big({\rm e}^tx;\cdot\big)^*,\hat{\rK}\big({\rm e}^tx;\cdot\big)^*\big\rangle_{\hat{\rho}}v,v\big)_\tau^{1/2},
\end{gather*}
and since $\big\langle\hat{\rK}(x;\cdot)^*,\hat{\rK}(x;\cdot)^*\big\rangle_{\hat{\rho}}\in C^\omega(D,\End(V))$,
$\cF_{\tau\rho}$ extends continuously to the map from $\cO_\rho(D_1,W)$ to $\cO_\tau(D,V)$.
\end{proof}

\subsection{Differential expression}
Next we will establish differential expressions for the intertwining operators.
To do this for $\cF_{\tau\rho}\colon \cH_\rho(D_1,W)\to \cH_\tau(D,V)$, we assume that $(G,G_1)$ is a symmetric pair, that is,
there exists an involution $\sigma$ of $G$ (without loss of generality we assume $\sigma\vartheta=\vartheta\sigma$) such that $G_1=(G^\sigma)_0$.
A~symmetric pair which satisfies the assumption (\ref{assumption}) is called a symmetric pair of \textit{holomorphic type}.
We extend $\sigma$ on $\fg^\BC$ holomorphically. Then this defines the involution on $\fp^+$ and also on~$D$.

Since the reproducing kernel of $\cP(\fp^+,V)$ with respect to the Fischer norm
is given by ${\rm e}^{(x|z)_{\fp^+}}I_V$, that is,
\begin{gather*} f(x)=\frac{1}{\pi^n}\int_{\fp^+}{\rm e}^{(x|z)_{\fp^+}}f(z){\rm e}^{-|z|_{\fp^+}^2}{\rm d}z \end{gather*}
holds for $f\in \cH_\tau(D,V)_{\tilde{K}}=\cP(\fp^+,V)$, we have
\begin{align*}
(\cF_{\tau\rho}^*f)(y_1)&=\left\langle \frac{1}{\pi^n}\int_{\fp^+}{\rm e}^{(\cdot|z)_{\fp^+}}f(z){\rm e}^{-|z|_{\fp^+}^2}{\rm d}z,
\hat{\rK}(\cdot;y_1)\right\rangle_{\hat{\tau}} \\
&=\frac{1}{\pi^n}\int_{\fp^+}\big\langle {\rm e}^{(\cdot|z)_{\fp^+}}I_V,\hat{\rK}(\cdot;y_1)\big\rangle_{\hat{\tau}}f(z){\rm e}^{-|z|_{\fp^+}^2}{\rm d}z \\
&=\frac{1}{\pi^n}\int_{\fp^+}\cF_{\tau\rho}^*\big({\rm e}^{(\cdot|z)_{\fp^+}}I_V\big)(y_1)f(z){\rm e}^{-|z|_{\fp^+}^2}{\rm d}z,
\end{align*}
and similarly, for $f\in \cH_\rho(D_1,W)_{\tilde{K}_1}=\cP(\fp^+_1,W)$ we have
\begin{align*}
(\cF_{\tau\rho}f)(x)&=\left\langle \frac{1}{\pi^{n_1}}\int_{\fp^+_1}{\rm e}^{(\cdot|w_1)_{\fp^+_1}}f(w_1){\rm e}^{-|w_1|_{\fp^+_1}^2}{\rm d}w_1,
\hat{\rK}(x;\cdot)^*\right\rangle_{\hat{\rho}} \\
&=\frac{1}{\pi^{n_1}}\int_{\fp^+_1}\big\langle {\rm e}^{(\cdot|w_1)_{\fp^+_1}}I_W,\hat{\rK}(x;\cdot)^*\big\rangle_{\hat{\rho}}
f(w_1){\rm e}^{-|w_1|_{\fp^+_1}^2}{\rm d}w_1 \\
&=\frac{1}{\pi^{n_1}}\int_{\fp^+_1}\cF_{\tau\rho}\big({\rm e}^{(\cdot|w_1)_{\fp^+_1}}I_W\big)(x)f(w_1){\rm e}^{-|w_1|_{\fp^+_1}^2}{\rm d}w_1.
\end{align*}
Now we have
\begin{Lemma}\quad
\begin{enumerate}\itemsep=0pt
\item[$(1)$] $\ds \cF_{\tau\rho}^*\big({\rm e}^{(\cdot|z)_{\fp^+}}I_V\big)(y_1)=\cF_{\tau\rho}^*\big({\rm e}^{(\cdot|z)_{\fp^+}}I_V\big)(0){\rm e}^{(y_1|z)_{\fp^+}}$.
\item[$(2)$] If $(G,G_1)$ is symmetric, then
\begin{gather*} \cF_{\tau\rho}\big({\rm e}^{(\cdot|w_1)_{\fp^+_1}}I_W\big)(x)=\cF_{\tau\rho}\big({\rm e}^{(\cdot|w_1)_{\fp^+_1}}I_W\big)(0,x_2){\rm e}^{(x_1|w_1)_{\fp^+_1}}. \end{gather*}
\end{enumerate}
\end{Lemma}
\begin{proof}
(1) Since $\cF_{\tau\rho}^*$ intertwines the $\tilde{G}_1$-action, it also intertwines the $\fg^\BC_1$-action.
Especially, since $\fp^+_1\subset \fg^\BC_1$ acts as a 1st-order differential operator with constant coefficients, we have
\begin{align*}
\frac{{\rm d}}{{\rm d}t}\cF_{\tau\rho}^*\big({\rm e}^{(\cdot|z)_{\fp^+}}I_V\big)(ty_1)
&=\left.\frac{{\rm d}}{{\rm d}s}\right|_{s=0}\cF_{\tau\rho}^*\big({\rm e}^{(\cdot|z)_{\fp^+}}I_V\big)(ty_1+sy_1) \\
&=\left.\frac{{\rm d}}{{\rm d}s}\right|_{s=0}\cF_{\tau\rho}^*\big({\rm e}^{(\cdot+sy_1|z)_{\fp^+}}I_V\big)(ty_1)\\
& =\left.\frac{{\rm d}}{{\rm d}s}\right|_{s=0}\cF_{\tau\rho}^*\big({\rm e}^{(\cdot|z)_{\fp^+}}I_V\big)(ty_1){\rm e}^{s(y_1|z)_{\fp^+}} \\
&=\cF_{\tau\rho}^*\big({\rm e}^{(\cdot|z)_{\fp^+}}I_V\big)(ty_1)(y_1|z)_{\fp^+}.
\end{align*}
Therefore, as functions of $t$, both
\begin{gather*} \cF_{\tau\rho}^*\big({\rm e}^{(\cdot|z)_{\fp^+}}I_V\big)(ty_1)\qquad \text{and} \qquad
\cF_{\tau\rho}^*\big({\rm e}^{(\cdot|z)_{\fp^+}}I_V\big)(0){\rm e}^{t(y_1|z)_{\fp^+}} \end{gather*}
satisfy the same differential equation with the same initial condition, and thus they coincide.
By substituting $t=1$, we get the desired formula.

(2) First we note that under the assumption that $(G,G_1)$ is symmetric,
if $x=(x_1,x_2)\in D$, where $x_1\in\fp^+_1$, $x_2\in\fp^+_2$, then $(tx_1,x_2)\in D$ holds for any $t\in[-1,1]$,
because $-\sigma(x_1,x_2)=(-x_1,x_2)\in D$ and $D$ is convex.
Therefore $\cF_{\tau\rho}\big({\rm e}^{(\cdot|w_1)_{\fp^+_1}}I_W\big)(tx_1,x_2)$ is well-defined for any $t\in[-1,1]$ if $(G,G_1)$ is symmetric.
Then as for (1), we can show that $\cF_{\tau\rho}\big({\rm e}^{(\cdot|w_1)_{\fp^+_1}}I_W\big)(tx_1,x_2)$ and
$\cF_{\tau\rho}\big({\rm e}^{(\cdot|w_1)_{\fp^+_1}}I_W\big)(0,x_2){\rm e}^{t(x_1|w_1)_{\fp^+_1}}$ satisfy the same differential equation
with the same initial condition, and thus they coincide.
\end{proof}

Thus we define $F_{\tau\rho}^*(z)\in\cP(\overline{\fp^+},\Hom(V,W))$ and
$F_{\tau\rho}(x_2;w_1)\in\cO((D\cap\fp^+_2)\times\overline{\fp^+_1},\allowbreak\Hom(W,V))$ by
\begin{gather*}
F_{\tau\rho}^*(z)=F_{\tau\rho}^*(z_1,z_2):=\cF_{\tau\rho}^*\big({\rm e}^{(\cdot|z)_{\fp^+}}I_V\big)(0)
=\big\langle {\rm e}^{(\cdot|z)_{\fp^+}}I_V,\hat{\rK}(\cdot;0)\big\rangle_{\hat{\tau}} \\
\hphantom{F_{\tau\rho}^*(z)}{} =\big\langle {\rm e}^{(\cdot|z)_{\fp^+}}I_V,\rK(\Proj_2(\cdot))\big\rangle_{\hat{\tau}}
=C_\tau\int_D \rK(x_2)^*\tau\big(B(x)^{-1}\big){\rm e}^{(x|z)_{\fp^+}}h(x)^{-p}{\rm d}x,\\
F_{\tau\rho}(x_2;w_1)=\cF_{\tau\rho}\big({\rm e}^{(\cdot|w_1)_{\fp^+_1}}I_W\big)(0,x_2)
=\big\langle {\rm e}^{(y_1|w_1)_{\fp^+_1}}I_W,\hat{\rK}(0,x_2;y_1)^*\big\rangle_{\hat{\rho},y_1}\\
\hphantom{F_{\tau\rho}(x_2;w_1)}{}=\big\langle {\rm e}^{(y_1|w_1)_{\fp^+_1}}I_W,\bigl(\tau(B(x_2,y_1))\rK(\Proj_2((x_2)^{y_1}))\bigr)^*\big\rangle_{\hat{\rho},y_1}\\
\hphantom{F_{\tau\rho}(x_2;w_1)}{}=C_\rho\int_{D_1} \tau(B(x_2,y_1))\rK(\Proj_2((x_2)^{y_1}))\rho\big(B(y_1)^{-1}\big){\rm e}^{(y_1|w_1)_{\fp^+_1}}h_1(y_1)^{-p}{\rm d}y_1
\end{gather*}
(where the last equalities hold only if $\cH_\tau(D,V)$ resp.~$\cH_\rho(D_1,W)$ is a holomorphic discrete series representation).
Then we have
\begin{align*}
(\cF_{\tau\rho}^*f)(y_1)&=\frac{1}{\pi^n}\int_{\fp^+}F_{\tau\rho}^*(z_1,z_2){\rm e}^{(y_1|z)_{\fp^+}}f(z){\rm e}^{-|z|_{\fp^+}^2}{\rm d}z\\
&=\frac{1}{\pi^n}\left.\int_{\fp^+}F_{\tau\rho}^*(z_1,z_2){\rm e}^{(x|z)_{\fp^+}}f(z){\rm e}^{-|z|_{\fp^+}^2}{\rm d}z\right|_{x_1=y_1,x_2=0}\\
&=F_{\tau\rho}^*\left(\left.\overline{\frac{\partial}{\partial x_1}}\right|_{x_1=y_1},
\left.\overline{\frac{\partial}{\partial x_2}}\right|_{x_2=0}\right)
\frac{1}{\pi^n}\int_{\fp^+}{\rm e}^{(x|z)_{\fp^+}}f(z){\rm e}^{-|z|_{\fp^+}^2}{\rm d}z\\
&=F_{\tau\rho}^*\left.\left(\overline{\frac{\partial}{\partial x_1}},\overline{\frac{\partial}{\partial x_2}}\right)\right|_{x_1=y_1,x_2=0}f(x),
\end{align*}
and similarly we have
\begin{gather*}
(\cF_{\tau\rho}f)(x) =\frac{1}{\pi^{n_1}}\int_{\fp^+_1}F_{\tau\rho}(x_2;w_1){\rm e}^{(x_1|w_1)_{\fp^+_1}}f(w_1){\rm e}^{-|w_1|_{\fp^+_1}^2}{\rm d}w_1
=F_{\tau\rho}\left.\left(x_2;\overline{\frac{\partial}{\partial y_1}}\right)\right|_{y_1=x_1}f(y_1).
\end{gather*}
Here, for anti-holomorphic polynomial $f\in\cP(\overline{\fp^+})$, let $f\big(\overline{\frac{\partial}{\partial x}}\big)$ be the
holomorphic differential operator characterized by
\begin{gather*} f\left(\overline{\frac{\partial}{\partial x}}\right){\rm e}^{(x|y)_{\fp^+}}=f(y), \end{gather*}
and similarly for anti-holomorphic function $f\in\cO(\overline{\fp^+_1})$, let $f\big(\overline{\frac{\partial}{\partial x_1}}\big)$ be the
operator characterized by
\begin{gather*} f\left(\overline{\frac{\partial}{\partial x_1}}\right){\rm e}^{(x_1|y_1)_{\fp^+_1}}=f(y_1). \end{gather*}

Next we describe $\Proj_2((x_2)^{y_1})$, and $\tau(B(x_2,y_1))$ when $\tau=\chi^{-\lambda}$ is one-dimensional, where $\chi$ is as (\ref{char}),
namely $\tau(B(x_2,y_1))=\chi^{-\lambda}(B(x_2,y_1))=h(x_2,y_1)^{-\lambda}$.
\begin{Proposition}\label{projprop}
Assume $(G,G_1)$ is a symmetric pair.
\begin{enumerate}\itemsep=0pt
\item[$(1)$] $\Proj_2((x_2)^{y_1})=(x_2)^{Q(y_1)x_2}$.
\item[$(2)$] Assume $G$ is simple. Then $h(x_2,y_1)^2=h(Q(x_2)y_1,y_1)=h(x_2,Q(y_1)x_2)$.
\end{enumerate}
\end{Proposition}
\begin{Lemma}\label{projlemma}
Let $x,y\in\fp^+$.
\begin{enumerate}\itemsep=0pt
\item[$(1)$] $B(-x,y)B(x,y)=B(Q(x)y,y)=B(x,Q(y)x)$.
\item[$(2)$] $x^y=(Q(x)y)^y+x^{Q(y)x}$.
\end{enumerate}
\end{Lemma}
\begin{proof}
(1) Use \cite[Part V, Proposition I.5.1, (J4.2), (J4.2$'$)]{FKKLR} and $B(x,-y)=B(-x,y)$, $Q(-x)=Q(x)$.

(2) Both sides are computed as
\begin{align*}
({\rm l.h.s.})&=B(x,y)^{-1}(x-Q(x)y),\\
({\rm r.h.s.})&=B(Q(x)y,y)^{-1}(Q(x)y-Q(Q(x)y)y)+B(x,Q(y)x)^{-1}(x-Q(x)Q(y)x) \\
&=B(x,y)^{-1}B(-x,y)^{-1}(x+Q(x)y-Q(x)Q(y)x-Q(x)Q(y)Q(x)y),
\end{align*}
where we used (1) and \cite[Part V, Proposition I.4.1]{FKKLR}. Thus it suffices to show
\begin{gather*} B(-x,y)(x-Q(x)y)=x+Q(x)y-Q(x)Q(y)x-Q(x)Q(y)Q(x)y. \end{gather*}
In fact, we have
\begin{gather*}
B(-x,y)(x-Q(x)y) =(I+D(x,y)+Q(x)Q(y))(x-Q(x)y) \\
\qquad{}=x+D(x,y)x+Q(x)Q(y)x-Q(x)y-D(x,y)Q(x)y-Q(x)Q(y)Q(x)y \\
\qquad{}=x+D(x,y)x+Q(x)Q(y)x-Q(x)y-Q(x)D(y,x)y-Q(x)Q(y)Q(x)y \\
\qquad{}=x+2Q(x)y+Q(x)Q(y)x-Q(x)y-2Q(x)Q(y)x-Q(x)Q(y)Q(x)y \\
\qquad{}=x+Q(x)y-Q(x)Q(y)x-Q(x)Q(y)Q(x)y,
\end{gather*}
where we used \cite[Part V, Proposition I.2.1(J1)]{FKKLR} at the 3rd equality, and $D(x,y)x=2Q(x)y$ at the 4th equality.
Thus the lemma follows.
\end{proof}

\begin{proof}[Proof of Proposition \ref{projprop}]
(1) When $x_2\in\fp^+_2$, $y_1\in\fp^+_1$, we have $(Q(x_2)y_1)^{y_1}\in\fp^+_1$, $(x_2)^{Q(y_1)x_2}\allowbreak\in\fp^+_2$. Therefore
\begin{gather*} \Proj_2((x_2)^{y_1})=\Proj_2\big((Q(x_2)y_1)^{y_1}+(x_2)^{Q(y_1)x_2}\big)=(x_2)^{Q(y_1)x_2}. \end{gather*}

(2) We extend $\sigma$ on $\fg^\BC$ holomorphically. Then since $\sigma$ acts by $+1$ on $\fp^+_1$ and $-1$ on $\fp^+_2$,
$B(-x_2,y_1)=\sigma B(x_2,y_1)\sigma$ holds. Therefore by Lemma \ref{projlemma},
\begin{gather*}
\sigma B(x_2,y_1)\sigma B(x_2,y_1)=B(Q(x_2)y_1,y_1)=B(x_2,Q(y_1)x_2), \\
\therefore \Det(B(x_2,y_1))^2=\Det(B(Q(x_2)y_1,y_1))=\Det(B(x_2,Q(y_1)x_2)).
\end{gather*}
Since $\Det(B(x_2,y_1))=h(x_2,y_1)^p$, the proposition follows.
\end{proof}

Therefore when $(G,G_1)$ is symmetric, we have
\begin{gather*}
 F_{\tau\rho}(x_2;w_1)
:=\big\langle {\rm e}^{(y_1|w_1)_{\fp^+_1}}I_W,\bigl(\tau(B(x_2,y_1))\rK\big((x_2)^{Q(y_1)x_2}\big) \bigr)^*\big\rangle_{\hat{\rho},y_1},
\end{gather*}
and moreover if $G$ is simple and $\tau=\chi^{-\lambda}$ is one-dimensional, we have
\begin{align*}
F_{\tau\rho}(x_2;w_1)&=\big\langle {\rm e}^{(y_1|w_1)_{\fp^+_1}}I_W,
\big(h(x_2,Q(y_1)x_2)^{-\lambda/2}\rK\big((x_2)^{Q(y_1)x_2}\big)\big)^*\big\rangle_{\hat{\rho},y_1}\\
&=\big\langle {\rm e}^{(y_1|w_1)_{\fp^+_1}}I_W,
\big(h(Q(x_2)y_1,y_1)^{-\lambda/2}\rK\big((x_2)^{Q(y_1)x_2}\big)\big)^*\big\rangle_{\hat{\rho},y_1}.
\end{align*}
We summarize the above results.
\begin{Theorem}\label{main}
Let $(\tau,V)$ be a $\tilde{K}^\BC$-module, and $(\rho,W)$ be a $\tilde{K}_1^\BC$-module
which appears in $\cP(\fp^+_2)\otimes V|_{\tilde{K}_1^\BC}$.
Let $\rK(x_2)\in \cP(\fp^+_2)\otimes V\otimes\overline{W}\simeq\cP(\fp^+_2,\Hom(W,V))$ be an operator-valued polynomial
satisfying \eqref{K-invariance}.
\begin{enumerate}\itemsep=0pt
\item[$(1)$] Assume $\cH_\tau(D,V)_{\tilde{K}}=\cP(\fp^+,V)$.
We define $F_{\tau\rho}^*(z)\in\cP(\overline{\fp^+},\Hom(V,W))$ by
\begin{gather*} F_{\tau\rho}^*(z)=F_{\tau\rho}^*(z_1,z_2):=\big\langle {\rm e}^{(\cdot|z)_{\fp^+}}I_V,\rK(\Proj_2(\cdot))\big\rangle_{\hat{\tau}}. \end{gather*}
Then the linear map
\begin{gather*}
\cF_{\tau\rho}^*\colon \cH_\tau(D,V)_{\tilde{K}}\to\cH_\rho(D_1,W)_{\tilde{K}_1}, \qquad
(\cF_{\tau\rho}^*f)(x_1)
=F_{\tau\rho}^*\left.\left(\overline{\frac{\partial}{\partial x_1}},\overline{\frac{\partial}{\partial x_2}}\right)\right|_{x_2=0}f(x)
\end{gather*}
intertwines the $(\fg_1,\tilde{K}_1)$-action.
\item[$(2)$] Assume $(G,G_1)$ is symmetric, and also assume ``$\cH_\tau(D,V)_{\tilde{K}}=\cP(\fp^+,V)$'' or ``$\cH_\rho(D_1,W)$ is
a holomorphic discrete series representation''.
We define $F_{\tau\rho}(x_2;w_1)\in\cO((D\cap\fp^+_2)\times\overline{\fp^+_1},$ $\Hom(W,V))$ by
\begin{gather*}
 F_{\tau\rho}(x_2;w_1)
:=\big\langle {\rm e}^{(y_1|w_1)_{\fp^+_1}}I_W,\bigl(\tau(B(x_2,y_1))\rK\big((x_2)^{Q(y_1)x_2}\big)\bigr)^*\big\rangle_{\hat{\rho},y_1}.
\end{gather*}
Then the linear map
\begin{gather*}
\cF_{\tau\rho}\colon \cH_\rho(D_1,W)_{\tilde{K}_1}\to \cO_\tau(D,V)_{\tilde{K}}, \qquad
(\cF_{\tau\rho}f)(x) =F_{\tau\rho}\left(x_2;\overline{\frac{\partial}{\partial x_1}}\right)f(x_1)
\end{gather*}
intertwines the $(\fg_1,\tilde{K}_1)$-action.
\end{enumerate}
\end{Theorem}

\begin{Remark}
For $w\in\fp^+$, we define the differential operator $\cB_\tau(w)$ on $\cP(\overline{\fp^+},V)$ by
\begin{gather*} \cB_\tau(w)f(z):=\sum_{\alpha\beta}\frac{1}{2}(Q(e_\alpha,e_\beta)w|z)_{\fp^+}
\frac{\partial^2f}{\partial\bar{z}_\alpha\partial\bar{z}_\beta}(z)
+\sum_{\alpha}{\rm d}\tau([e_\alpha,\vartheta w])\frac{\partial f}{\partial \bar{z}_\alpha}(z), \end{gather*}
where $\{e_\alpha\}$ is a basis of $\fp^+$, with the dual basis $\{e_\alpha^\vee\}$,
and $\frac{\partial}{\partial\bar{z}_\alpha}$ is the anti-holomorphic directional derivative along $e_\alpha^\vee$.
Then this is a generalization of the Bessel operator $\cB_\nu$ in~\cite{D} or \cite[Section~XV.2]{FK}.
Then for $w_1\in\fp^+_1$, $\cB_\tau(w_1)$ annihilates $F_{\tau\rho}^*(z)$, because
\begin{gather*}
 (\cB_\tau(w_1))_zF_{\tau\rho}^*(z)
=(\cB_\tau(w_1))_z\big\langle {\rm e}^{(x|z)_{\fp^+}}I_V,\rK(\Proj_2(x))\big\rangle_{\hat{\tau},x} \\
 =\big\langle ((Q(x)w_1|z)_{\fp^+}+{\rm d}\tau([x,\vartheta w_1])){\rm e}^{(x|z)_{\fp^+}}I_V,\rK(\Proj_2(x))\big\rangle_{\hat{\tau},x} \\
 =\big\langle {\rm d}\hat{\tau}(-\vartheta w_1)_x{\rm e}^{(x|z)_{\fp^+}}I_V,\rK(\Proj_2(x))\big\rangle_{\hat{\tau},x}
=\big\langle {\rm e}^{(x|z)_{\fp^+}}I_V,{\rm d}\hat{\tau}(w_1)_x\rK(\Proj_2(x))\big\rangle_{\hat{\tau},x} \\
 =\left\langle {\rm e}^{(x|z)_{\fp^+}}I_V,\left.\frac{{\rm d}}{{\rm d}t}\right|_{t=0}\rK(\Proj_2(x-tw_1))\right\rangle_{\hat{\tau},x}
=\left\langle {\rm e}^{(x|z)_{\fp^+}}I_V,\left.\frac{{\rm d}}{{\rm d}t}\right|_{t=0}\rK(\Proj_2(x))\right\rangle_{\hat{\tau},x}=0.
\end{gather*}
This differential equation coincides with $\widehat{{\rm d}\pi_\mu}$ on $\fn_+$ appeared in Proposition 3.10 or Section 4.4, Step 1
of \cite{KP1}, and thus the operator $\cF_{\tau\rho}^*$ coincides with the one given by the F-method.
\end{Remark}
These operators are defined as maps from the space of polynomials $\cH_\tau(D,V)_{\tilde{K}}$ resp.\ \linebreak $\cH_\rho(D_1,W)_{\tilde{K}_1}$,
but in fact these are well-defined as maps between
$\cO_\tau(D,V)$ and $\cO_\rho(D_1,W)$ in the following sense.
\begin{Theorem}\label{extend}\quad
\begin{enumerate}\itemsep=0pt
\item[$(1)$] $\cF_{\tau\rho}^*$ is well-defined as the map $\cF_{\tau\rho}^*\colon \cO_\tau(D,V)\to\cO_\rho(D_1,W)$.
\item[$(2)$] Assume $\cH_\rho(D_1\!{,}W)$ is a holomorphic discrete series representation. Then for  $f{\in} \cO_\rho(D_1\!{,}W)$,
$\cF_{\tau\rho}f(x)$ converges uniformly on every compact subset in $\{x=x_1+x_2\in D\colon |x_1|_\infty+|x_2|_\infty\allowbreak <1\}$,
and it continues holomorphically on whole $D$. Especially $\cF_{\tau\rho}$ is well-defined as the map
$\cF_{\tau\rho}\colon \cO_\rho(D_1,W)\to \cO_\tau(D,V)$.
\end{enumerate}
\end{Theorem}
\begin{proof}
(1) Clear since $\cF_{\tau\rho}^*$ is a finite-order differential operator.

(2) First we decompose $F_{\tau\rho}(x_2;w_1)$ as the sum of homogeneous polynomials in $w_1$.
\begin{align*}
F_{\tau\rho}(x_2;w_1) &=\big\langle {\rm e}^{(y_1|w_1)_{\fp^+_1}}I_W,\bigl(\tau(B(x_2,y_1))\rK(\Proj_2((x_2)^{y_1}))\bigr)^*\big\rangle_{\hat{\rho},y_1} \\
&=\sum_{n=0}^\infty \frac{1}{n!}\big\langle (y_1|w_1)_{\fp^+_1}^nI_W, \hat{\rK}(x_2;y_1)^*\big\rangle_{\hat{\rho},y_1} =:\sum_{n=0}^\infty F_n(x_2;w_1).
\end{align*}
Then $\cF_{\tau\rho}$ is written as
\begin{gather*} (\cF_{\tau\rho}f)(x)=\sum_{n=0}^\infty F_n\left(x_2;\overline{\frac{\partial}{\partial x_1}}\right)f(x_1). \end{gather*}
Now by the $\tilde{G}_1$-invariance of $\hat{\rK}$, for $k\in \tilde{K}_1$ we have
\begin{gather*} F_n(x_2;w_1)=\tau(k)F_n\big(k^{-1}x_2;k^*w_1\big)\rho(k)^{-1}. \end{gather*}
Therefore for $k={\rm e}^{-t\hbar}$ we have
\begin{gather}\label{invariance_Fn}
F_n(x_2;w_1)=\chi_\rho\big({\rm e}^{t\hbar}\big)\chi_\tau\big({\rm e}^{-t\hbar}\big) F_n\big({\rm e}^tx_2;\overline{{\rm e}^{-t}}w_1\big),
\end{gather}
where the notations are the same as in the previous subsection.
Now we fix $t>0$, and let $|x_2|_\infty<{\rm e}^{-t}<1$. Then for $v\in V$ we have
\begin{gather*}
 \left|\left(F_n\left(x_2;\overline{\frac{\partial}{\partial x_1}}\right)f(x_1),v\right)_\tau\right| =\left|\left(\chi_\rho\big({\rm e}^{t\hbar}\big)\chi_\tau\big({\rm e}^{-t\hbar}\big)
F_n\left({\rm e}^tx_2;\overline{{\rm e}^{-t}\frac{\partial}{\partial x_1}}\right)f(x_1),v\right)_\tau\right| \\
\qquad{} =\frac{\chi_\rho({\rm e}^{t\hbar})\chi_\tau({\rm e}^{-t\hbar})}{n!}\left|\left\langle {\rm e}^{-nt}
\left(y_1\,\middle|\overline{\frac{\partial}{\partial x_1}}\right)_{\fp^+_1}^nf(x_1),
\hat{\rK}\big({\rm e}^tx_2;y_1\big)^*v\right\rangle_{\hat{\rho},y_1}\right| \\
\qquad{} \le \chi_\rho\big({\rm e}^{t\hbar}\big)\chi_\tau\big({\rm e}^{-t\hbar}\big)\frac{{\rm e}^{-nt}}{n!}
\left\Vert\left(y_1\,\middle|\overline{\frac{\partial}{\partial x_1}}\right)_{\fp^+_1}^nf(x_1)\right\Vert_{\hat{\rho},y_1}
\big\Vert\hat{\rK}\big({\rm e}^tx_2;y_1\big)^*v\big\Vert_{\hat{\rho},y_1}.
\end{gather*}
Next we estimate
$\ds \left\Vert\left(y_1\,\middle|\overline{\frac{\partial}{\partial x_1}}\right)_{\fp^+_1}^nf(x_1)\right\Vert_{\hat{\rho},y_1}$.
For $x_1,y_1\in D_1$, we take $R>0$ such that $x_1+R{\rm e}^{\sqrt{-1}\theta}y_1\in D_1$ for any $\theta\in\BR$. Then
\begin{align*}
\left|\left(y_1\,\middle|\,\overline{\frac{\partial}{\partial x_1}}\right)_{\fp^+_1}^nf(x_1)\right|_\rho
&=\left|\left.\frac{{\rm d}^n}{{\rm d}t^n}\right|_{t=0}f(x_1+ty_1)\right|_\rho
=\left|\frac{n!}{2\pi\sqrt{-1}}\oint_{|z|=R}\frac{f(x_1+zy_1)}{z^{n+1}}{\rm d}z\right|_\rho \\
&=\left|\frac{n!}{2\pi\sqrt{-1}}\int_0^{2\pi}\frac{f\big(x_1+R{\rm e}^{\sqrt{-1}\theta}y_1\big)}{R^{n+1}{\rm e}^{\sqrt{-1}(n+1)\theta}}
\sqrt{-1}R{\rm e}^{\sqrt{-1}\theta}{\rm d}\theta\right|_\rho \\
&=\left|\frac{n!}{2\pi}\int_0^{2\pi}\frac{f\big(x_1+R{\rm e}^{\sqrt{-1}\theta}y_1\big)}{R^n{\rm e}^{\sqrt{-1}n\theta}}{\rm d}\theta\right|_\rho\\
& \le \frac{n!}{R^n}\max_{\theta\in[0,2\pi]}\big|f\big(x_1+R{\rm e}^{\sqrt{-1}\theta}y_1\big)\big|_\rho.
\end{align*}
Now fix $0<s<1$, let $|x_1|_\infty<s$, and we set $\ds R:=\frac{s-|x_1|_\infty}{|y_1|_\infty}$. Then we have
\begin{gather*} \big|x_1+R{\rm e}^{\sqrt{-1}\theta}y_1\big|_\infty\le |x_1|_\infty+R|y_1|_\infty=s, \end{gather*}
and therefore
\begin{align*}
\left|\left(y_1\,\middle|\,\overline{\frac{\partial}{\partial x_1}}\right)_{\fp^+_1}^nf(x_1)\right|_\rho
&\le n!\left(\frac{|y_1|_\infty}{s-|x_1|_\infty}\right)^n\sup_{|\xi_1|_\infty\le s}|f(\xi_1)|_\rho \\
&\le n!\left(\frac{1}{s-|x_1|_\infty}\right)^n\sup_{|\xi_1|_\infty\le s}|f(\xi_1)|_\rho
\end{align*}
holds. Hence we have
\begin{gather}
 \left\Vert\left(y_1\,\middle|\overline{\frac{\partial}{\partial x_1}}\right)_{\fp^+_1}^nf(x_1)\right\Vert_{\hat{\rho},y_1}^2 \notag \\
\qquad{} =C_\rho\int_{D_1} \left(\rho\big(B(x_1)^{-1}\big)\left(y_1\,\middle|\,\overline{\frac{\partial}{\partial x_1}}\right)_{\fp^+_1}^nf(x_1)
,\left(y_1\,\middle|\,\overline{\frac{\partial}{\partial x_1}}\right)_{\fp^+_1}^nf(x_1)\right)_\rho h_1(y_1)^{-p_1}{\rm d}y_1 \notag \\
\qquad{} \le C_\rho\int_{D_1} \big|\rho\big(B(x_1)^{-1}\big)\big|_{\rho,\mathrm{op}}h_1(y_1)^{-p_1}{\rm d}y_1
\sup_{y_1\in D_1}\left|\left(y_1\,\middle|\,\overline{\frac{\partial}{\partial x_1}}\right)_{\fp^+_1}^nf(x_1)\right|_\rho^2 \notag \\
\qquad{} \le C_\rho^{\prime 2}(n!)^2\left(\frac{1}{s-|x_1|_\infty}\right)^{2n}\sup_{|\xi_1|_\infty\le s}|f(\xi_1)|_\rho^2. \label{higherdiff}
\end{gather}
Therefore we have
\begin{gather*}
 \left|\left(F_n\left(x_2;\overline{\frac{\partial}{\partial x_1}}\right)f(x_1),v\right)_\tau\right| \\
 \le C_\rho'\chi_\rho\big({\rm e}^{t\hbar}\big)\chi_\tau\big({\rm e}^{-t\hbar}\big)
\left(\frac{{\rm e}^{-t}}{s-|x_1|_\infty}\right)^n\sup_{|\xi_1|_\infty\le s}|f(\xi_1)|_\rho
\big(\big\langle\hat{\rK}\big({\rm e}^tx_2;\cdot\big)^*,
\hat{\rK}\big({\rm e}^tx_2;\cdot\big)^*\big\rangle_{\hat{\rho}}v,v\big)_\tau^{1/2}.
\end{gather*}
As a consequence, if we assume ${\rm e}^{-t}<s-|x_1|_\infty$, we get
\begin{gather*}
 \left|\left((\cF_{\tau\rho}f)(x),v\right)_\tau\right|\le\sum_{n=0}^\infty
\left|\left(F_n\left(x_2;\overline{\frac{\partial}{\partial x_1}}\right)f(x_1),v\right)_\tau\right| \\
 \le C_\rho'\chi_\rho\big({\rm e}^{t\hbar}\big)\chi_\tau\big({\rm e}^{-t\hbar}\big)
\sum_{n=0}^\infty \left(\frac{{\rm e}^{-t}}{s-|x_1|_\infty}\right)^n\!\sup_{|\xi_1|_\infty\le s}\!|f(\xi_1)|_\rho
\big(\big\langle\hat{\rK}\big({\rm e}^tx_2;\cdot\big)^* ,
\hat{\rK}\big({\rm e}^tx_2;\cdot\big)^*\big\rangle_{\hat{\rho}}v,v\big)_\tau^{1/2} \\
 =C_\rho'\chi_\rho\big({\rm e}^{t\hbar}\big)\chi_\tau\big({\rm e}^{-t\hbar}\big)\frac{s-|x_1|_\infty}{s-|x_1|_\infty-{\rm e}^{-t}}\sup_{|\xi_1|_\infty\le s}|f(\xi_1)|_\rho
\big(\big\langle\hat{\rK}\big({\rm e}^tx_2;\cdot\big)^*,
\hat{\rK}\big({\rm e}^tx_2;\cdot\big)^*\big\rangle_{\hat{\rho}}v,v\big)_\tau^{1/2}.
\end{gather*}
This estimate holds if $0<{\rm e}^{-t}<s<1$, $|x_1|_\infty<s-{\rm e}^{-t}$ and $|x_2|_\infty<{\rm e}^{-t}$, and thus
\begin{gather*}
 ((\cF_{\tau\rho}f)(x),v )_\tau=\sum_{n=0}^\infty\left(F_n\left(x_2;\overline{\frac{\partial}{\partial x_1}}\right)f(x_1),v\right)_\tau
\end{gather*}
converges absolutely and uniformly on every compact subset in $\{|x_1|_\infty<s-{\rm e}^{-t},|x_2|_\infty<{\rm e}^{-t}\}$ for any $0<{\rm e}^{-t}<s<1$,
and hence also does on every compact subset in $\{|x_1|_\infty+|x_2|_\infty<1\}$.
By Theorem \ref{conti_ext}, $\cF_{\tau\rho}$ is a continuous operator from $\cO_\rho(D_1,W)$ to $\cO_\tau(D,V)$,
and therefore $\cF_{\tau\rho} f(x)$ must extend to whole $D$.
\end{proof}

\subsection{Analytic continuation of intertwining operators}
\looseness=-1 In this subsection we assume $G$ to be simple and that $(\tau,V)$ is of the form $(\tau,V)=\big(\tau_0\otimes\chi^{-\lambda},V\big)$.
Then we may assume $(\rho,W)$ is also of the form $(\rho,W)=\big(\rho_0\otimes\chi|_{\tilde{K}_1}^{-\lambda},W\big)$.
In this section we denote the representation of $\tilde{G}_1$
with minimal $\tilde{K}_1$-type $\big(\rho_0\otimes\chi|_{\tilde{K}_1}^{-\lambda},W\big)$ by $\cH_\lambda(D_1,W)$ and $\cO_\lambda(D_1,W)$
(this notation has not the same meaning as the one in Section \ref{HDS}, since we do not assume $G_1$ to be simple,
and even if $G_1$ is simple, $\chi|_{\tilde{K}_1}$ is not normalized as (\ref{char}) for $G_1$ in general).
Then the intertwining operators $\cF_{\tau\rho}^*\colon \cH_\lambda(D,V)\to\cH_\lambda(D_1,W)$ and $\cF_{\tau\rho}\colon \cH_\lambda(D_1,W)\to\cH_\lambda(D,V)$ depend
holomorphically on the parameter $\lambda$, and as the operators between the space of polynomials, these continue meromorphically for
all $\lambda\in\BC$, and define the intertwining operators $\cF_{\tau\rho}^*\colon \cO_\lambda(D,V)_{\tilde{K}}\to\cO_\lambda(D_1,W)_{\tilde{K}_1}$ and
$\cF_{\tau\rho}\colon \cO_\lambda(D_1,W)_{\tilde{K}_1}\to\cO_\lambda(D,V)_{\tilde{K}}$ if $\lambda$ is not a~pole.
Since $\cF_{\tau\rho}^*$ is a finite-order differential operator, this is clearly well-defined as the map
$\cF_{\tau\rho}^*\colon \cO_\lambda(D,V)\to\cO_\lambda(D_1,W)$ for all $\lambda\in\BC$ except for the poles. On the other hand,
for the infinite-order differential operator $\cF_{\tau\rho}$, it is not clear. The goal of this subsection is to prove that
$\cF_{\tau\rho}\colon \cO_\lambda(D_1,W)\to\cO_\lambda(D,V)$ is well-defined if $\lambda$ is not a pole.
In this subsection we write $\cF_{\tau\rho}=\cF_{\lambda,\rho}$.

To do this, we consider the $\tilde{K}_1$-type decomposition of $\cO_\lambda(D_1,W)_{\tilde{K}_1}$ as
\begin{align*}
\cO_\lambda(D_1,W)_{\tilde{K}_1}&\simeq \cP(\fp^+_1,W)\otimes(\chi|_{\tilde{K}_1})^{-\lambda}
\simeq \bigoplus_{m=0}^\infty \bigoplus_{j=1}^{N_m} W_{m,j}\otimes(\chi|_{\tilde{K}_1})^{-\lambda},
\end{align*}
where $W_{m,j}\subset\cP_m(\fp^+_1,W)=\{W\text{-valued homogeneous polynomials on }\fp^+_1\text{ of degree }m\}$, and $N_m\in\BN$.
We assume that the norm of $\cH_\lambda(D_1,W)$ is expanded as
\begin{gather}\label{assumption_decomposition}
\Vert f\Vert_{\lambda,\rho_0}^2=\sum_{m=0}^\infty \sum_{j=1}^{N_m}p_{m,j}(\lambda)\Vert f_{m,j}\Vert_{F,\rho_0}^2
\end{gather}
as (\ref{norm_compare}), where $f_{m,j}$ is the orthogonal projection of $f$ onto $W_{m,j}$, $\Vert f_{m,j}\Vert_{F,\rho_0}$ is the Fischer norm,
and $p_{m,j}(\lambda)$ depends meromorphically on $\lambda$.

Next we additionally assume that $\hat{\rK}(x_2;y_1)^*=\hat{\rK}_\lambda(x_2;y_1)^*$ is expanded as
\begin{gather}\label{assumption_expansion}
\hat{\rK}_\lambda(x_2;y_1)^*=\sum_{m=0}^\infty \sum_{j=1}^{N_m}q_{m,j}(\lambda)\hat{\rK}_{m,j}(x_2;y_1)^*,
\end{gather}
where $\hat{\rK}_{m,j}(x_2;y_1)^*\in \big(\overline{\cP(\fp_2^+,V)_{x_2}}\otimes (W_{m,j})_{y_1}\big)^{\tilde{K}_1}$
does not depend on $\lambda$, and $q_{m,j}(\lambda)$ depends meromorphically on $\lambda$.
We note that $\hat{\rK}_{m,j}(x_2;y_1)^*$ is non-zero only if $W_{m,j}$ appears commonly in the decomposition of both
$\cP(\fp_2^+,V|_{\tilde{K}_1})$ and $\cP(\fp^+_1,W)$ since $\hat{\rK}_{m,j}(x_2;y_1)^*$ is $\tilde{K}_1^\BC$-invariant.
We also note that if both $\cP(\fp_2^+,V|_{\tilde{K}_1})$ and $\cP(\fp^+_1,W)$ are multiplicity-free under $\tilde{K}_1^\BC$,
then we can always expand $\hat{\rK}_\lambda(x_2;y_1)^*$ as~(\ref{assumption_expansion}) since
$\big(\overline{\cP(\fp_2^+,V)_{x_2}}\otimes (W_{m,j})_{y_1}\big)^{\tilde{K}_1}$ is 1-dimensional.
Then by~(\ref{exp_onD}) $F_{\tau\rho}(x_2;w_1)=F_{\lambda,\rho}(x_2;w_1)$ is given by
\begin{align}
F_{\lambda,\rho}(x_2;w_1)&=\big\langle {\rm e}^{(\cdot|w_1)_{\fp^+_1}}I_W,\hat{\rK}_\lambda(x_2;\cdot)^*\big\rangle_{\lambda,\rho_0}
=\sum_{m=0}^\infty \sum_{j=1}^{N_m}q_{m,j}(\lambda)\big\langle {\rm e}^{(\cdot|w_1)_{\fp^+_1}}I_W,\hat{\rK}_{m,j}(x_2;\cdot)^*\big\rangle_{\lambda,\rho_0} \notag \\
&=\sum_{m=0}^\infty \sum_{j=1}^{N_m}p_{m,j}(\lambda)q_{m,j}(\lambda)\hat{\rK}_{m,j}(x_2;w_1), \label{ext_of_FW}
\end{align}
and $\cF_{\lambda,\rho}$ is given by substituting $w_1$ with $\overline{\frac{\partial}{\partial x_1}}$.
This continues meromorphically for all $\lambda$, and defines an intertwining operator from $\cO_\lambda(D_1,W)_{\tilde{K}_1}$
to $\cO_\lambda(D,V)_{\tilde{K}}$ if $\lambda$ is not a pole of $p_{m,j}(\lambda)q_{m,j}(\lambda)$.
In fact, this is well-defined as a map from $\cO_\lambda(D_1,W)$ to $\cO_\lambda(D,V)$ under some assumption.
\begin{Theorem}\label{continuation}
Assume \eqref{assumption_decomposition}, \eqref{assumption_expansion} holds, and also assume that
for any $\lambda\in\BC$ which is not a pole of $p_{m,j}(\lambda)$, $p_{m,j}(\lambda)q_{m,j}(\lambda)$,
there exist $\mu>p_1-1$, $C>0$, $k\ge 0$ such that $\left|\frac{p_{m,j}(\lambda)}{p_{m,j}(\mu)}\right|<C(1+m^k)$,
$\left|\frac{p_{m,j}(\lambda)q_{m,j}(\lambda)}{p_{m,j}(\mu)q_{m,j}(\mu)}\right|<C(1+m^k)$ holds.
Then if $\lambda$ is not a pole of $p_{m,j}(\lambda)q_{m,j}(\lambda)$, then for $f\in\cO_\lambda(D_1,W)$,
$\cF_{\lambda,\rho}f(x)$ converges uniformly on every compact subset in $\{x=x_1+x_2\in D\colon |x_1|_\infty+|x_2|_\infty<1\}$,
and it continues holomorphically on whole $D$. Especially $\cF_{\lambda,\rho}$ defines a~continuous map
$\cF_{\lambda,\rho}\colon \cO_\lambda(D_1,W)\to \cO_\lambda(D,V)$.
\end{Theorem}
In Section \ref{sect_examples}, we compute $\cF_{\lambda,\rho}$ explicitly when $(\tau,V)$ is 1-dimensional.
In these cases, $p_{m,j}(\lambda)$ and $p_{m,j}(\lambda)q_{m,j}(\lambda)$ are given by inverse of products of Pochhammer symbols,
and the polynomial-growth condition is satisfied by Stirling's formula.
Therefore, explicit intertwining operators in Section \ref{sect_examples} are well-defined as operators
between spaces of all holomorphic functions for any $\lambda\in\BC$ except for poles.
\begin{proof}
First we note that this theorem says that $\cF_{\lambda,\rho}f(x)$ converges if $\lambda$ is not a pole of $p_{m,j}(\lambda)q_{m,j}(\lambda)$,
but by continuity, it is enough to prove it when $\lambda$ is not a pole of either $p_{m,j}(\lambda)$ or $p_{m,j}(\lambda)q_{m,j}(\lambda)$.
First we prove the continuity as a map from $\cO_\lambda(D_1,W)$ to $\cO_\lambda(D,V)$ in the integral expression,
and second we prove the uniform convergence in the differential expression.

First we work with the integral expression. As in the holomorphic discrete series case,
for $f\in\cO_\lambda(D_1,W)_{\tilde{K}_1}=\cP(\fp^+_1,W)$, if $|x|_\infty<{\rm e}^{-t}$ then
\begin{align*}
((\cF_{\lambda,\rho}f)(x),v)_\tau&=\big\langle f(y_1),\hat{\rK}_\lambda(x;y_1)^*v\big\rangle_{\lambda,\rho_0,y_1} \\
&=\chi_\rho\big({\rm e}^{t\hbar}\big)\chi_\tau\big({\rm e}^{-t\hbar}\big)\big\langle f\big({\rm e}^{-t}y_1\big),
\hat{\rK}_\lambda\big({\rm e}^tx;y_1\big)^*v\big\rangle_{\lambda,\rho_0,y_1}
\end{align*}
holds. Then by assumption (\ref{assumption_decomposition}), we have
\begin{gather*}
 ((\cF_{\lambda,\rho}f)(x),v)_\tau
=\chi_\rho\big({\rm e}^{t\hbar}\big)\chi_\tau\big({\rm e}^{-t\hbar}\big)\sum_{m=0}^\infty \sum_{j=1}^{N_m}
p_{m,j}(\lambda)\big\langle f_{m,j}\big({\rm e}^{-t}y_1\big),\hat{\rK}_\lambda\big({\rm e}^tx;y_1\big)^*_{m,j}v\big\rangle_{F,\rho_0,y_1},
\end{gather*}
where $f_{m,j}$ is the orthogonal projection of $f$ onto $W_{m,j}$. Hence for $1<s<{\rm e}^t$,
\begin{gather*}
 \chi_\rho\big({\rm e}^{-t\hbar}\big)\chi_\tau\big({\rm e}^{t\hbar}\big)|((\cF_{\lambda,\rho}f)(x),v)_\tau| \\
\qquad{}\le \sum_{m=0}^\infty \sum_{j=1}^{N_m}
|p_{m,j}(\lambda)|\big\Vert f_{m,j}\big({\rm e}^{-t}y_1\big)\big\Vert_{F,\rho_0,y_1}\big\Vert\hat{\rK}_\lambda\big({\rm e}^tx;y_1\big)^*_{m,j}v\big\Vert_{F,\rho_0,y_1} \\
\qquad{}\le C\sum_{m=0}^\infty \sum_{j=1}^{N_m}\big(1+m^k\big)
p_{m,j}(\mu)\big\Vert f_{m,j}\big({\rm e}^{-t}y_1\big)\big\Vert_{F,\rho_0,y_1}\big\Vert\hat{\rK}_\lambda\big({\rm e}^tx;y_1\big)^*_{m,j}v\big\Vert_{F,\rho_0,y_1} \\
\qquad{}=C\sum_{m=0}^\infty \sum_{j=1}^{N_m}\big(1+m^k\big)
\big\Vert f_{m,j}\big({\rm e}^{-t}y_1\big)\big\Vert_{\mu,\rho_0,y_1}\big\Vert\hat{\rK}_\lambda\big({\rm e}^tx;y_1\big)^*_{m,j}v\big\Vert_{\mu,\rho_0,y_1} \\
\qquad{}=C\sum_{m=0}^\infty \sum_{j=1}^{N_m}\big(1+m^k\big)
s^{-m}\big\Vert f_{m,j}\big(s{\rm e}^{-t}y_1\big)\big\Vert_{\mu,\rho_0,y_1}\big\Vert\hat{\rK}_\lambda\big({\rm e}^tx;y_1\big)^*_{m,j}v\big\Vert_{\mu,\rho_0,y_1} \\
\qquad{}\le C\sum_{m=0}^\infty \sum_{j=1}^{N_m}\big(1+m^k\big)
s^{-m}\big\Vert f\big(s{\rm e}^{-t}y_1\big)\big\Vert_{\mu,\rho_0,y_1}\big\Vert\hat{\rK}_\lambda\big({\rm e}^tx;y_1\big)^*v\big\Vert_{\mu,\rho_0,y_1} \\
\qquad{}\le C\sum_{m=0}^\infty N_m\big(1+m^k\big)
s^{-m}\big\Vert f\big(s{\rm e}^{-t}y_1\big)\big\Vert_{\mu,\rho_0,y_1}\big\Vert\hat{\rK}_\lambda\big({\rm e}^tx;y_1\big)^*v\big\Vert_{\mu,\rho_0,y_1} \\
\qquad{}\le C'\sum_{m=0}^\infty N_m\big(1+m^k\big)
s^{-m}\sup_{|y_1|_\infty\le s{\rm e}^{-t}}|f(y_1)|_{\rho_0}\sup_{y_1\in D_1}\big|\hat{\rK}_\lambda({\rm e}^tx;y_1)^*v\big|_{\rho_0},
\end{gather*}
and thus we get{\samepage
\begin{gather*}
 |((\cF_{\lambda,\rho}f)(x),v)_\tau|\\
 \le C'\chi_\rho\big({\rm e}^{t\hbar}\big)\chi_\tau\big({\rm e}^{-t\hbar}\big)\sum_{m=0}^\infty N_m(1+m^k)
s^{-m}\sup_{|y_1|_\infty\le s{\rm e}^{-t}}|f(y_1)|_{\rho_0}\sup_{y_1\in D_1}\big|\hat{\rK}_\lambda\big({\rm e}^tx;y_1\big)^*v\big|_{\rho_0}.
\end{gather*}
Therefore $\cF_{\lambda,\rho}$ extends as a continuous operator from $\cO_\lambda(D_1,W)$ to $\cO_\lambda(D,V)$.}

Next we work with the differential expression. By (\ref{ext_of_FW}), as in (\ref{invariance_Fn}),
\begin{align*}
F_{\lambda,\rho}(x_2;w_1)
&=\sum_{m=0}^\infty \sum_{j=1}^{N_m}p_{m,j}(\lambda)q_{m,j}(\lambda)\hat{\rK}_{m,j}(x_2;w_1) \\
&=\chi_\rho\big({\rm e}^{t\hbar}\big)\chi_\tau\big({\rm e}^{-t\hbar}\big)
\sum_{m=0}^\infty \sum_{j=1}^{N_m}p_{m,j}(\lambda)q_{m,j}(\lambda)\hat{\rK}_{m,j}\big({\rm e}^tx_2;{\rm e}^{-t}w_1\big),
\end{align*}
and for $f\in\cO(D_1,W)$, $v\in V$,
\begin{gather*}
 |(\cF_{\lambda,\rho}f(x),v)_\tau|=\left|\left(F_{\lambda,\rho}\left(x_2;\overline{\frac{\partial}{\partial x_1}}\right)f(x_1),v\right)_\tau\right| \\
 =\left|\chi_\rho\big({\rm e}^{t\hbar}\big)\chi_\tau\big({\rm e}^{-t\hbar}\big)\sum_{m=0}^\infty \sum_{j=1}^{N_m}p_{m,j}(\lambda)q_{m,j}(\lambda)
\left(\hat{\rK}_{m,j}\left({\rm e}^tx_2;{\rm e}^{-t}\overline{\frac{\partial}{\partial x_1}}\right)f(x_1),v\right)_\tau\right| \\
 \le C\chi_\rho\big({\rm e}^{t\hbar}\big)\chi_\tau\big({\rm e}^{-t\hbar}\big)\sum_{m=0}^\infty \sum_{j=1}^{N_m}\big(1+m^k\big)p_{m,j}(\mu)q_{m,j}(\mu)
\left|\left(\hat{\rK}_{m,j}\left({\rm e}^tx_2;{\rm e}^{-t}\overline{\frac{\partial}{\partial x_1}}\right)f(x_1),v\right)_\tau\right|.
\end{gather*}
As in the holomorphic discrete series case, if ${\rm e}^{-t}<s-|x_2|_\infty$ then
\begin{gather*}
 p_{m,j}(\mu)q_{m,j}(\mu)\left|\left(\hat{\rK}_{m,j}\left({\rm e}^tx_2;{\rm e}^{-t}\overline{\frac{\partial}{\partial x_1}}\right)f(x_1),
v\right)_\tau\right| \\
\qquad{}=\frac{p_{m,j}(\mu)q_{m,j}(\mu)}{m!}\left|\left\langle {\rm e}^{-mt}
\left(y_1\,\middle|\overline{\frac{\partial}{\partial x_1}}\right)_{\fp^+_1}^mf(x_1),
\hat{\rK}_{m,j}\big({\rm e}^tx_2;y_1\big)^*v\right\rangle_{F,\rho_0,y_1}\right| \\
\qquad{}=\frac{1}{m!}\left|\left\langle {\rm e}^{-mt}
\left(y_1\,\middle|\overline{\frac{\partial}{\partial x_1}}\right)_{\fp^+_1}^mf(x_1),
q_{m,j}(\mu)\hat{\rK}_{m,j}\big({\rm e}^tx_2;y_1\big)^*v\right\rangle_{\mu,\rho_0,y_1}\right| \\
\qquad{}\le \frac{{\rm e}^{-mt}}{m!}
\left\Vert\left(y_1\,\middle|\overline{\frac{\partial}{\partial x_1}}\right)_{\fp^+_1}^mf(x_1)\right\Vert_{\mu,\rho_0,y_1}
\big\Vert q_{m,j}(\mu)\hat{\rK}_{m,j}\big({\rm e}^tx_2;y_1\big)^*v\big\Vert_{\mu,\rho_0,y_1} \\
\qquad{}\le \frac{{\rm e}^{-mt}}{m!}
\left\Vert\left(y_1\,\middle|\overline{\frac{\partial}{\partial x_1}}\right)_{\fp^+_1}^mf(x_1)\right\Vert_{\mu,\rho_0,y_1}
\big\Vert \hat{\rK}_\mu\big({\rm e}^tx_2;y_1\big)^*v\big\Vert_{\mu,\rho_0,y_1} \\
\qquad{}\le C_\rho'\left(\frac{{\rm e}^{-t}}{s-|x_1|_\infty}\right)^m\sup_{|\xi_1|_\infty\le s}|f(\xi_1)|_{\rho_0}
\big(\big\langle\hat{\rK}_\mu({\rm e}^tx_2;\cdot)^*,\hat{\rK}_\mu\big({\rm e}^tx_2;\cdot\big)^*\big\rangle_{\mu,\rho_0}v,v\big)_\tau^{1/2}.
\end{gather*}
where we have used (\ref{higherdiff}). Therefore we have
\begin{gather*}
 |(\cF_{\lambda,\rho}f(x),v)_\tau|\\
 \le C\chi_\rho\big({\rm e}^{t\hbar}\big)\chi_\tau\big({\rm e}^{-t\hbar}\big)
\sum_{m=0}^\infty \sum_{j=1}^{N_m}\big(1+m^k\big)p_{m,j}(\mu)q_{m,j}(\mu)
\left|\left(\hat{\rK}_{m,j}\left({\rm e}^tx_2;{\rm e}^{-t}\overline{\frac{\partial}{\partial x_1}}\right)f(x_1),v\right)_\tau\right| \\
 \le CC_\rho'\chi_\rho\big({\rm e}^{t\hbar}\big)\chi_\tau\big({\rm e}^{-t\hbar}\big)
\sum_{m=0}^\infty N_m\big(1+m^k\big)\left(\frac{{\rm e}^{-t}}{s-|x_1|_\infty}\right)^m \\
\quad{}\times\sup_{|\xi_1|_\infty\le s}|f(\xi_1)|_{\rho_0}\big(\big\langle\hat{\rK}_\mu\big({\rm e}^tx_2;\cdot\big)^*,
\hat{\rK}_\mu\big({\rm e}^tx_2;\cdot\big)^*\big\rangle_{\mu,\rho_0}v,v\big)_\tau^{1/2}.
\end{gather*}
This converges if ${\rm e}^{-t}<s-|x_1|_\infty$, and as in the holomorphic discrete series case, this estimate holds if $|x_1|_\infty+|x_2|_\infty<1$.
Since $\cF_{\lambda,\rho}$ is continuous as an operator from $\cO_\lambda(D_1,W)$ to $\cO_\lambda(D,V)$,
$\cF_{\lambda,\rho}f(x)$ must extend holomorphically to whole $D$.
\end{proof}

\section{Preliminaries for explicit calculation}\label{section4}
\subsection[Parametrization of representations of classical $K^\BC$]{Parametrization of representations of classical $\boldsymbol{K^\BC}$}
In this subsection we fix the realization of root systems and parametrization of irreducible finite-dimensional representations of $K^\BC$
when it is classical.
First we set $K^\BC:={\rm GL}(r,\BC)$ or $\operatorname{SO}(n,\BC)$. We take a Cartan subalgebra $\fh^\BC\subset\fk^\BC$,
and take a basis $\{t_1,\ldots,t_r\}\subset\fh^\BC$, with the dual basis $\{\varepsilon_1,\ldots,\varepsilon_r\}\subset\big(\fh^\BC\big)^\vee$,
where $r=\big\lfloor\frac{n}{2}\big\rfloor$ when $K^\BC=\operatorname{SO}(n,\BC)$, such that the positive root system $\Delta_+\big(\fk^\BC,\fh^\BC\big)$ is given by
\begin{gather*} \Delta_+(\fk^\BC,\fh^\BC)=\begin{cases}\{\varepsilon_j-\varepsilon_k\colon 1\le j<k\le r\},& K^\BC={\rm GL}(r,\BC) ,\\
\{\varepsilon_j\pm \varepsilon_k\colon 1\le j<k\le r\},& K^\BC=\operatorname{SO}(2r,\BC) ,\\
\{\varepsilon_j\pm \varepsilon_k\colon 1\le j<k\le r\}\cup\{\varepsilon_j\colon 1\le j\le r\},& K^\BC=\operatorname{SO}(2r+1,\BC) .\end{cases} \end{gather*}
For $\bm\in\BZ^r$ with $m_1\ge\cdots\ge m_r$, we denote
the irreducible representation of ${\rm GL}(r,\BC)$ with highest weight $m_1\varepsilon_1+\cdots+m_r\varepsilon_r$ by
$\big(\tau_\bm^{(r)},V_\bm^{(r)}\big)$,
the irreducible representation of ${\rm GL}(r,\BC)$ with lowest weight $-m_1\varepsilon_1-\cdots-m_r\varepsilon_r$ by
$\big(\tau_\bm^{(r)\vee},V_\bm^{(r)\vee}\big)$,
and for $\bm\in\BZ^r$ with $m_1\ge\cdots\ge m_{r-1}\ge|m_r|$ (when $n=2r$) or with $m_1\ge\cdots\ge m_r\ge 0$ (when $n=2r+1$), we denote
the irreducible representation of $\operatorname{SO}(n,\BC)$ with lowest weight $-m_1\varepsilon_1-\cdots-m_r\varepsilon_r$ by
$\big(\tau_\bm^{[n]\vee},V_\bm^{[n]\vee}\big)$.
We omit the superscript $(r)$ and $[n]$ if there is no confusion.

Next we set $G:=\operatorname{Sp}(r,\BR)$, $U(q,s)$, $\operatorname{SO}^*(2s)$, or $\operatorname{SO}_0(2,n)$, and let $K^\BC$ be the complexification of their maximal compact subgroups,
that is, $K^\BC={\rm GL}(r,\BC)$, ${\rm GL}(q,\BC)\times {\rm GL}(s,\BC)$, ${\rm GL}(s,\BC)$ or $\operatorname{SO}(2,\BC)\times \operatorname{SO}(n,\BC)$ respectively.
Then the irreducible finite-dimensional representations of~$K^\BC$ are of the form
$V_\bm^{(r)}$, $V_\bm^{(q)}\boxtimes V_\bn^{(s)\vee}$, $V_\bm^{(s)}$, or $\BC_{m_0}\boxtimes V_\bm^{[n]}$ respectively,
where we normalize the representation $(\chi^{m_0},\BC_{m_0})$ of $\operatorname{SO}(2,\BC)$ later as in (\ref{chi_so(2n)}).
Also, under the suitable ordering of $\Delta\big(\fg^\BC,\fh^\BC\big)$, $\cP_\bm(\fp^+)$ in Theorem \ref{HKS} is given by
\begin{gather*} \cP_\bm(\fp^+)\simeq\begin{cases}
V_{(2m_1,2m_2,\ldots,2m_r)}^{(r)\vee}=:V_{2\bm}^{(r)\vee}, & G=\operatorname{Sp}(r,\BC),\ \bm\in\BZ_{++}^r,\\
V_\bm^{(q)\vee}\boxtimes V_\bm^{(s)}, & G=U(q,s),\ \bm\in\BZ_{++}^{\min\{q,s\}},\\
V_{(m_1,m_1,m_2,m_2,\ldots,m_{\lfloor s/2\rfloor},m_{\lfloor s/2\rfloor}(,0))}^{(s)\vee}=:V_{\bm^2}^{(s)\vee},
& G=\operatorname{SO}^*(2s),\ \bm\in\BZ_{++}^{\lfloor s/2\rfloor},\\
\BC_{-m_1-m_2}\boxtimes V_{(m_1-m_2,0,0,\ldots,0)}^{[n]\vee}, & G=\operatorname{SO}_0(2,n),\ \bm\in\BZ_{++}^2, \end{cases} \end{gather*}
where, when $s<q$ and $\bm\in\BZ_{++}^s$, we denote $V_{(m_1,\ldots,m_s,0,\ldots,0)}^{(q)}=:V_\bm^{(q)}$ etc.,
and the character of~$K^\BC$ normalized as (\ref{char}) is given by
\begin{gather*} \chi\simeq\begin{cases}
V_{(1,1,\ldots,1)}^{(r)}, & G=\operatorname{Sp}(r,\BC),\\
V_{\frac{s}{q+s}(1,1,\ldots,1)}^{(q)}\boxtimes V_{\frac{q}{q+s}(1,1,\ldots,1)}^{(s)\vee}, & G=\operatorname{SU}(q,s),\\
V_{\left(\frac{1}{2},\frac{1}{2},\ldots,\frac{1}{2}\right)}^{(s)}, & G=\operatorname{SO}^*(2s),\\
\BC_{1}\boxtimes V_{(0,0,\ldots,0)}^{[n]}, &G=\operatorname{SO}_0(2,n). \end{cases} \end{gather*}

We have the local isomorphism $\operatorname{SO}^*(6)\simeq \operatorname{SU}(1,3)$. Accordingly, we identify the representation
\begin{gather*} V_{(m_1,m_2,m_3)}^{(3)\vee}=V_{\frac{1}{3}(2m_1-m_2-m_3,-m_1+2m_2-m_3,-m_1-m_2+2m_3)}^{(3)\vee}\otimes\chi_{\operatorname{SO}^*(6)}^{-\frac{2}{3}|\bm|} \end{gather*}
of $U(3)\subset \operatorname{SO}^*(6)$ and the representation
\begin{gather*}
 \big(V_0^{(1)\vee}\boxtimes V_{\frac{1}{3}(2m_1-m_2-m_3,-m_1+2m_2-m_3,-m_1-m_2+2m_3)}^{(3)\vee}\big)\otimes\chi_{\operatorname{SU}(1,3)}^{-\frac{2}{3}|\bm|}\\
 \qquad {} =V_{\frac{1}{2}|\bm|}^{(1)\vee}\boxtimes V_{\frac{1}{2}(m_1-m_2-m_3,-m_1+m_2-m_3,-m_1-m_2+m_3)}^{(3)\vee}
\end{gather*}
of $S(U(1)\times U(3))\subset \operatorname{SU}(1,3)$. Also we write $V_{n_0}^{(1)\vee}\boxtimes V_{(n_1,n_2,n_3)}^{(3)\vee}=:V_{(n_0;n_1,n_2,n_3)}^{(1,3)\vee}
\simeq V_{\substack{(n_0+a;n_1+a,\ \\ \ n_2+a,n_3+a)}}^{(1,3)\vee}$ for any $a\in\BR$, so that
\begin{gather*} V_{(m_1,m_2,m_3)}^{(3)\vee}\simeq V_{\frac{1}{2}(|\bm|;m_1-m_2-m_3,-m_1+m_2-m_3,-m_1-m_2+m_3)}^{(1,3)\vee}
\simeq V_{(0;-m_2-m_3,-m_1-m_3,-m_1-m_2)}^{(1,3)\vee}. \end{gather*}

\subsection{Explicit realization of classical groups and bounded symmetric domains}\label{realize}
In this subsection, we review and fix the explicit realization of groups
\begin{gather*} G=\operatorname{Sp}(r,\BR),\; U(q,s),\; \operatorname{SO}^*(2s),\; \operatorname{SO}_0(2,n). \end{gather*}
First we deal with $G=\operatorname{Sp}(r,\BR)$, $U(q,s)$, and $\operatorname{SO}^*(2s)$. For these groups we have
\begin{gather*} (r,n,d,p)=\begin{cases}
\big(r,\frac{1}{2}r(r+1),1,r+1\big),& G=\operatorname{Sp}(r,\BR) ,\\
(\min\{q,s\},qs,2,q+s),& G=U(q,s) ,\\
\big(\big\lfloor\frac{s}{2}\big\rfloor,\frac{1}{2}s(s-1),4,2(s-1)\big), & G=\operatorname{SO}^*(2s). \end{cases} \end{gather*}
We realize these groups as
\begin{gather*}
\operatorname{Sp}(r,\BR) :=\left\{ g\in {\rm GL}(2r,\BC)\colon g\begin{pmatrix}0&I_r\\-I_r&0\end{pmatrix}{}^t\hspace{-1pt}g=\begin{pmatrix}0&I_r\\-I_r&0\end{pmatrix},\;
g\begin{pmatrix}0&I_r\\I_r&0\end{pmatrix}=\begin{pmatrix}0&I_r\\I_r&0\end{pmatrix}\bar{g}\right\},\\
U(q,s) :=\left\{ g\in {\rm GL}(q+s,\BC)\colon g\begin{pmatrix}I_q&0\\0&-I_s\end{pmatrix}g^*=\begin{pmatrix}I_q&0\\0&-I_s\end{pmatrix}\right\},\\
\operatorname{SO}^*(2s) :=\left\{ g\in {\rm GL}(2s,\BC)\colon g\begin{pmatrix}0&I_s\\I_s&0\end{pmatrix}{}^t\hspace{-1pt}g=\begin{pmatrix}0&I_s\\I_s&0\end{pmatrix},\;
g\begin{pmatrix}0&I_s\\-I_s&0\end{pmatrix}=\begin{pmatrix}0&I_s\\-I_s&0\end{pmatrix}\bar{g}\right\}.
\end{gather*}
Then $K$ is isomorphic to $U(r)$, $U(q)\times U(s)$, and $U(s)$ respectively. We embed $K$ into $G$ as
\begin{alignat*}{3}
& k \mapsto \begin{pmatrix}k&0\\0&{}^t\hspace{-1pt}k^{-1}\end{pmatrix}, \qquad & & G=\operatorname{Sp}(r,\BR),\; \operatorname{SO}^*(2s),& \\
& (k_1,k_2) \mapsto \begin{pmatrix}k_1&0\\0&k_2\end{pmatrix}, \qquad & &G=U(q,s).&
\end{alignat*}
Clearly these extend to the embeddings of complexified Lie groups $K^\BC\to G^\BC$.
When $G=\operatorname{Sp}(r,\BR)$ or $\operatorname{SO}^*(2s)$, we sometimes write the elements of $K$ or $K^\BC$ as $\big(k,{}^t\hspace{-1pt}k^{-1}\big)$,
and deal with these inclusions in a unified way. Similarly, $\fp^+$ is isomorphic to $\Sym(r,\BC)$, $M(q,s;\BC)$ and $\Skew(s,\BC)$ respectively.
We embed $\fp^+$ into $\fg^\BC$ as $x\mapsto \left(\begin{smallmatrix}0&x\\0&0\end{smallmatrix}\right)$.
Then the rational action of $G$ on~$\fp^+$ is given by
\begin{gather*} \begin{pmatrix}a&b\\c&d\end{pmatrix}x=(ax+b)(cx+d)^{-1},\qquad
 \begin{pmatrix}a&b\\c&d\end{pmatrix}\in G,\quad x\in\fp^+ . \end{gather*}
The Jordan triple system structure on $\fp^+$ is given by
\begin{gather*} Q(x)y=xy^*x, \qquad x,y\in\fp^+, \end{gather*}
the inner product (\ref{innerprod}) is given by
\begin{gather*} (x|y)_{\fp^+}=\frac{1}{\varepsilon}\Tr(xy^*), \qquad x,y\in\fp^+, \qquad
\varepsilon=\begin{cases} 1, & G=\operatorname{Sp}(r,\BR),\; U(q,s),\\ 2, &G=\operatorname{SO}^*(2s),\end{cases} \end{gather*}
the Bergman operator $B\colon D\times \overline{D}\to K^\BC$ is given by
\begin{gather*} B(x,y)=\big(I-xy^*,(I-y^*x)^{-1}\big), \qquad x,y\in\fp^+, \end{gather*}
the quasi-inverse is given by
\begin{gather*} x^y=x(I-y^*x)^{-1}=(I-xy^*)^{-1}x, \qquad x,y\in\fp^+, \end{gather*}
and the bounded symmetric domain $D$ is given by
\begin{gather*} D=\{x\in\fp^+\colon I-xx^*\text{ is positive definite}\}. \end{gather*}

Let $(\tau,V)$ be an irreducible representation of $\tilde{K}^\BC$ with $\tilde{K}$-invariant inner product $(\cdot,\cdot)_\tau$.
Then~$\tilde{G}$ acts on $\cO(D,V)$ as
\begin{gather*} \hat{\tau}\left(\begin{pmatrix}a&b\\c&d\end{pmatrix}^{-1}\right)f(w)
=\tau\big(a^*+xb^*,(cx+d)^{-1}\big)f\big((ax+b)(cx+d)^{-1}\big), \end{gather*}
where we regard $\big(a^*+xb^*,(cx+d)^{-1}\big)$ as the lift on $\tilde{K}^\BC$, and this action preserves the inner product
\begin{gather*} \langle f,g\rangle_{\hat{\tau}}=C_\tau\int_D
\big(\tau\big((I-xx^*)^{-1},I-x^*x\big)f(x),g(x)\big)_\tau \det(I-xx^*)^{-p/\varepsilon}{\rm d}x. \end{gather*}

Especially, for $G=\operatorname{Sp}(r,\BR)$ or $\operatorname{SO}^*(2s)$, let $(\tau,V)=\big(\chi^{-\lambda},\BC\big)$ be a 1-dimensional representation of $\tilde{K}^\BC$,
normalized as in (\ref{char}), that is,
\begin{gather*} \chi(k):=\det(k)^{1/\varepsilon}. \end{gather*}
Then the $\tilde{G}$-invariant inner product on $\cH_\tau(D,\BC)=\cH_\lambda(D)$ is given by
\begin{gather}\label{inner_prod_scalar_type}
\langle f,g\rangle_\lambda=C_\lambda\int_D f(x)\overline{g(x)} \det(I-xx^*)^{(\lambda-p)/\varepsilon}{\rm d}x,
\end{gather}
which converges for any polynomial $f$, $g$ if $\lambda>p-1$. When $G=U(q,s)$, we define $\big(\chi^{-\lambda_1-\lambda_2},\BC\big)$ as
\begin{gather}\label{charuqs}
\chi^{-\lambda_1-\lambda_2}(k_1,k_2):=\det(k_1)^{-\lambda_1}\det(k_2)^{\lambda_2},
\end{gather}
and write the corresponding representation of $\tilde{G}$ as $\cH_{\lambda_1+\lambda_2}(D)$. Then again the $\tilde{G}$-invariant inner product is given by~(\ref{inner_prod_scalar_type}) with $\lambda=\lambda_1+\lambda_2$.

Next we deal with $G=\operatorname{SO}_0(2,n)$ case with $n\ge 3$. In this case, we have
\begin{gather*} (r,n,d,p)=(2,n,n-2,n). \end{gather*}
We realize this group as
\begin{gather*} \operatorname{SO}_0(2,n):=\left\{ g\in {\rm SL}(2+n,\BR)\colon g\begin{pmatrix}I_2&0\\0&-I_n\end{pmatrix}{}^t\hspace{-1pt}g=\begin{pmatrix}I_2&0\\0&-I_n\end{pmatrix}
\right\}_0 \end{gather*}
as usual, where the subscript 0 means the identity component.
We have $K\simeq \operatorname{SO}(2)\times \operatorname{SO}(n)$, embedded into $G$ as $(k_1,k_2)\mapsto \left(\begin{smallmatrix}k_1&0\\0&k_2\end{smallmatrix}\right)$,
and $\fp^+\simeq\BC^n$, embedded into $\fg^\BC$ as
\begin{gather*} x\mapsto \begin{pmatrix}0&0&{}^t\hspace{-1pt}x\\0&0&\sqrt{-1}\hspace{1pt}{}^t\hspace{-1pt}x\\x&\sqrt{-1}x&0\end{pmatrix}, \end{gather*}
where we regard $x$ as a column vector. For $x={}^t\hspace{-1pt}(x_1,\ldots,x_n),y={}^t\hspace{-1pt}(y_1,\ldots,y_n)\in\fp^+$, we write
\begin{gather}\label{quadra}
q(x):=x_1^2+\cdots+x_n^2,\qquad q(x,y):=x_1y_1+\cdots+x_ny_n.
\end{gather}
Then the Jordan triple system structure on $\fp^+$ is given by
\begin{gather*} Q(x)y=2q(x,\overline{y})x-q(x)\overline{y}, \qquad x,y\in\fp^+, \end{gather*}
the inner product (\ref{innerprod}) is given by
\begin{gather*} (x|y)_{\fp^+}=2q(x,\overline{y}), \qquad x,y\in\fp^+, \end{gather*}
the generic norm is given by
\begin{gather*} h(x,y)=1-2q(x,\overline{y})+q(x)\overline{q(y)}, \qquad x,y\in\fp^+, \end{gather*}
the Bergman operator $B\colon D\times \overline{D}\to \End(\fp^+)\simeq M(n,\BC)$ is given by
\begin{gather*} B(x,y)=h(x,y)I-2(1-q(x,\overline{y}))\big(xy^*-\overline{y}\hspace{1pt}{}^t\hspace{-1pt}x\big)
+2\big(xy^*-\overline{y}\hspace{1pt}{}^t\hspace{-1pt}x\big)^2, \qquad x,y\in\fp^+, \end{gather*}
the quasi-inverse is given by
\begin{gather*} x^y=h(x,y)^{-1}(x-q(x)\overline{y}), \qquad x,y\in\fp^+, \end{gather*}
and the bounded symmetric domain $D$ is the connected component of $\{h(x,x)>0\}$ which contains the origin.

Let $(\tau,V)=\big(\chi^{-\lambda},\BC\big)$ be a 1-dimensional representation of $\tilde{K}^\BC$,
where $\chi$ is normalized as in~(\ref{char}), that is,
\begin{gather}\label{chi_so(2n)}
\chi\left(\exp\left(a\begin{pmatrix}0&-\sqrt{-1}\\\sqrt{-1}&0\end{pmatrix}\right),k_2\right)={\rm e}^a,
\qquad a\in\BC,\quad k_2\in \operatorname{SO}(n,\BC).
\end{gather}
Then the $\tilde{G}$-action on $\cO(D)$ preserves the inner product
\begin{gather}\label{inn_prod_so(2n)}
\langle f,g\rangle_\lambda=C_\lambda\int_D f(x)\overline{g(x)}\big(1-2q(x,\overline{x})+|q(x)|^2\big)^{\lambda-n}{\rm d}x.
\end{gather}
By Theorem \ref{HKS}, the space of $\tilde{K}$-finite vectors $\cO(D)_{\tilde{K}}=\cP(\fp^+)$ is decomposed as
\begin{gather}\label{poly_rank2}
\cP(\fp^+)=\bigoplus_{\bm\in\BZ_{++}^2}\cP_\bm(\fp^+)
\simeq\bigoplus_{\bm\in\BZ_{++}^2}\chi^{-m_1-m_2}\boxtimes V_{(m_1-m_2,0,\ldots,0)}^{[n]\vee},
\end{gather}
and by (\ref{scalar_norm}), for $f=f_\bm\in\cP_\bm(\fp^+)$, the ratio of this norm and the Fischer norm (\ref{Fischer}) is given by
\begin{gather}\label{norm_rank2}
\frac{\Vert f_\bm\Vert_\lambda^2}{\Vert f_\bm\Vert_F^2}=\frac{1}{(\lambda)_{\bm,n-2}}
=\frac{1}{(\lambda)_{m_1}\big(\lambda-\frac{n-2}{2}\big)_{m_2}}.
\end{gather}
When $n=1,2$, we have $\mathfrak{so}(2,1)\simeq\mathfrak{sl}(2,\BR)$, which is of real rank 1,
or $\mathfrak{so}(2,2)\simeq\mathfrak{sl}(2,\BR)\oplus\mathfrak{sl}(2,\BR)$,
which is not simple, and thus their properties are a bit different from those of $n\ge 3$ cases.
However, for convenience, we use the same inner product as (\ref{inn_prod_so(2n)}), so that
\begin{gather*} \cH_\lambda(D_{\operatorname{SO}_0(2,1)})\simeq \cH_{2\lambda}(D_{{\rm SL}(2,\BR)}),\qquad
\cH_\lambda(D_{\operatorname{SO}_0(2,2)})\simeq \cH_\lambda(D_{{\rm SL}(2,\BR)}) \hboxtimes \cH_\lambda(D_{{\rm SL}(2,\BR)}). \end{gather*}
When $n=1$, the space of $\tilde{K}$-finite vectors $\cO(D)_{\tilde{K}}=\cP(\fp^+)$ is decomposed as
\begin{gather*} \cP(\fp^+)=\BC[x]=\bigoplus_{l=0}^\infty\BC x^l
=:\bigoplus_{l=0}^\infty\cP_{\left(\left\lceil\frac{l}{2}\right\rceil,\left\lfloor\frac{l}{2}\right\rfloor\right)}(\fp^+) \end{gather*}
(so that we only allow $m_1=m_2$ and $m_1=m_2+1$ cases).
Then for $f=f_l\in\cP_{\left(\left\lceil\frac{l}{2}\right\rceil,\left\lfloor\frac{l}{2}\right\rfloor\right)}(\fp^+)$,
the ratio of two norms is given by
\begin{gather*} \frac{\Vert f_l\Vert_\lambda^2}{\Vert f_l\Vert_F^2}
=\frac{1}{(\lambda)_{\left\lceil\frac{l}{2}\right\rceil}\big(\lambda+\frac{1}{2}\big)_{\left\lfloor\frac{l}{2}\right\rfloor}}
=\frac{2^l}{(2\lambda)_l}. \end{gather*}
When $n=2$, the space of $\tilde{K}$-finite vectors $\cO(D)_{\tilde{K}}=\cP(\fp^+)$ is decomposed as
\begin{gather*} \cP(\fp^+)=\BC[x_1,x_2]=\bigoplus_{\bm\in(\BZ_{\ge 0})^2}\BC\big(x_1+\sqrt{-1}x_2\big)^{m_1}\big(x_1-\sqrt{-1}x_2\big)^{m_2}
=:\bigoplus_{\bm\in(\BZ_{\ge 0})^2}\cP_\bm(\fp^+) \end{gather*}
(we do not assume $m_1\ge m_2$), and for $f=f_\bm\in\cP_\bm(\fp^+)$, the ratio of two norms is given by
\begin{gather*} \frac{\Vert f_\bm\Vert_\lambda^2}{\Vert f_\bm\Vert_F^2}=\frac{1}{(\lambda)_{m_1}(\lambda)_{m_2}}. \end{gather*}
That is, when $n=1,2$ the range of summation in (\ref{poly_rank2}) changes, but the formula of the norm~(\ref{norm_rank2}) does not change.

\subsection{Root systems of exceptional Lie algebras}
First we consider the Lie algebra $\mathfrak{e}_{7(-25)}$. We take a Cartan subalgebra
$\fh\subset\mathfrak{so}(2)\oplus\mathfrak{e}_6\subset\mathfrak{e}_{7(-25)}$, and we take three kinds of basis
$\{\gamma_1,\gamma_2,\gamma_3,\ve_1,\ve_2,\ve_3,\ve_4\}\subset\big(\fh^\BC\big)^\vee$,
$\big\{\delta_1^{(i)},\ldots,\delta_8^{(i)}\big\}\subset\big(\fh^\BC\big)^\vee\oplus\BC$ $(i=1,2)$
such that the simple system of positive roots is given by
\begin{gather*}
\alpha_1 =\tfrac12 (\gamma_1-\gamma_2)+\tfrac{1}{2}(-\ve_1-\ve_2-\ve_3-\ve_4) \\
\hphantom{\alpha_1}{} =\tfrac{1}{2}\big({-}\delta_1^{(i)}+\delta_2^{(i)}+\delta_3^{(i)}+\delta_4^{(i)}-\delta_5^{(i)}-\delta_6^{(i)}-\delta_7^{(i)}+\delta_8^{(i)}\big),\\
\alpha_2 =\ve_3-\ve_4 =\delta_5^{(i)}-\delta_6^{(i)},\qquad
\alpha_3 =\ve_3+\ve_4 =\delta_7^{(i)}-\delta_8^{(i)},\\
\alpha_4 =\ve_2-\ve_3
 =\tfrac{1}{2}\big(\delta_1^{(i)}-\delta_2^{(i)}-\delta_3^{(i)}+\delta_4^{(i)}-\delta_5^{(i)}+\delta_6^{(i)}-\delta_7^{(i)}+\delta_8^{(i)}\big),\\
\alpha_5 =\ve_1-\ve_2 =\delta_3^{(i)}-\delta_4^{(i)},\\
\alpha_6 =\tfrac12 (\gamma_2-\gamma_3)-\ve_1
 =\tfrac{1}{2}\big({-}\delta_1^{(i)}+\delta_2^{(i)}-\delta_3^{(i)}+\delta_4^{(i)}+\delta_5^{(i)}+\delta_6^{(i)}-\delta_7^{(i)}-\delta_8^{(i)}\big),\\
\beta_{\mathfrak{e}_{7(-25)}} =\gamma_3
 =\tfrac{1}{2}\big(\delta_1^{(1)}+\delta_2^{(1)}-\delta_3^{(1)}-\delta_4^{(1)}-\delta_5^{(1)}-\delta_6^{(1)}+\delta_7^{(1)}+\delta_8^{(1)}\big)
 =\delta_7^{(2)}+\delta_8^{(2)},
\end{gather*}
where $\beta_{\mathfrak{e}_{7(-25)}}$ is the unique non-compact simple root. Here, the expression in the basis $\{\gamma_1,\gamma_2,\gamma_3,\allowbreak\ve_1,\ve_2,\ve_3,\ve_4\}\subset\big(\fh^\BC\big)^\vee$
is a modification of the one used in \cite{Y2}. Next let
\begin{gather*}
\alpha_{134} =\alpha_1+\alpha_3+\alpha_4 =\delta_4^{(i)}-\delta_5^{(i)},\qquad
\alpha_{456} =\alpha_4+\alpha_5+\alpha_6 =\delta_6^{(i)}-\delta_7^{(i)},\\
\alpha_{23445} =\alpha_2+\alpha_3+2\alpha_4+\alpha_5 =\delta_1^{(i)}-\delta_2^{(i)},
\end{gather*}
and
\begin{gather*}
\beta_{\mathfrak{sl}(2,\BR)} =\gamma_1,\qquad\!\!
\beta_{\mathfrak{so}(2,10)} =\gamma_3=\beta_{\mathfrak{e}_{7(-25)}},\qquad\!\!
\beta_{\mathfrak{e}_{6(-14)}} =\tfrac{1}{2}(\gamma_2+\gamma_3)+\tfrac{1}{2}(-\ve_1-\ve_2-\ve_3+\ve_4),\\
\beta_{\mathfrak{su}(2,6)} =\delta_2^{(1)}-\delta_3^{(1)}
=\tfrac{1}{2}\big({-}\delta_1^{(2)}+\delta_2^{(2)}-\delta_3^{(2)}+\delta_4^{(2)}+\delta_5^{(2)}+\delta_6^{(2)}+\delta_7^{(2)}+\delta_8^{(2)}\big),\\
\beta_{\mathfrak{so}^*(12)}
 =\tfrac{1}{2}\big(\delta_1^{(1)}+\delta_2^{(1)}-\delta_3^{(1)}-\delta_4^{(1)}-\delta_5^{(1)}-\delta_6^{(1)}+\delta_7^{(1)}+\delta_8^{(1)}\big)
=\delta_7^{(2)}+\delta_8^{(2)}=\beta_{\mathfrak{e}_{7(-25)}}.
\end{gather*}
We realize $\mathfrak{sl}(2,\BR)\oplus\mathfrak{so}(2,10),\mathfrak{u}(1)\oplus\mathfrak{e}_{6(-14)},
\mathfrak{su}(2,6),\mathfrak{su}(2)\oplus\mathfrak{so}^*(12)\subset\mathfrak{e}_{7(-25)}$
such that the simple systems of positive roots are given by
\begin{gather*}
\Pi(\mathfrak{sl}(2,\BR)\oplus\mathfrak{so}(2,10)) =\{\beta_{\mathfrak{sl}(2,\BR)}\}
\cup\{\alpha_2,\alpha_3,\alpha_4,\alpha_5,\alpha_6,\beta_{\mathfrak{so}(2,10)}\},\\
\Pi(\mathfrak{u}(1)\oplus\mathfrak{e}_{6(-14)}) =\{\alpha_2,\alpha_3,\alpha_4,\alpha_5,\alpha_6,\beta_{\mathfrak{e}_{6(-14)}}\},\\
\Pi(\mathfrak{su}(2,6)) =\{\alpha_{23445},\beta_{\mathfrak{su}(2,6)},\alpha_5,\alpha_{134},\alpha_2,\alpha_{456},\alpha_3\},\\
\Pi(\mathfrak{su}(2)\oplus\mathfrak{so}^*(12)) =\{\alpha_{23445}\}
\cup\{\alpha_5,\alpha_{134},\alpha_2,\alpha_{456},\alpha_3,\beta_{\mathfrak{so}^*(12)}\},
\end{gather*}
where for $\fg$ simple of Hermitian type and for a choice of a system of positive roots $\beta_{\fg}$ is the unique non-compact simple root.
Next let
\begin{gather*}
\beta_{\mathfrak{so}(2,8)} =\tfrac12 (\gamma_1+\gamma_2)+\tfrac{1}{2}(-\ve_1-\ve_2-\ve_3-\ve_4),\\
\beta_{\mathfrak{so}(2,8)'} =\tfrac12 (\gamma_1+\gamma_3)+\tfrac{1}{2}(-\ve_1-\ve_2-\ve_3+\ve_4),\\
\beta_{\mathfrak{su}(2,4)} =\delta_2^{(1)}-\delta_5^{(1)}=\beta_{\mathfrak{e}_{6(-14)}},\\
\beta_{\mathfrak{su}(4,2)}
 =\tfrac{1}{2}\big(\delta_1^{(1)}+\delta_2^{(1)}-\delta_5^{(1)}-\delta_6^{(1)}-\delta_7^{(1)}+\delta_8^{(1)}-\delta_3^{(1)}+\delta_4^{(1)}\big).
\end{gather*}
We realize $\mathfrak{u}(1)\oplus\mathfrak{so}(2,8),\mathfrak{u}(1)\oplus\mathfrak{so}(2,8)',
\mathfrak{su}(2,4)\oplus\mathfrak{su}(2),\mathfrak{su}(2)\oplus\mathfrak{su}(4,2)\subset\mathfrak{e}_{6(-14)}$ such that
the simple systems of positive roots are given by
\begin{gather*}
\Pi(\mathfrak{u}(1)\oplus\mathfrak{so}(2,8)) =\{\alpha_2,\alpha_3,\alpha_4,\alpha_5,\beta_{\mathfrak{so}(2,8)}\},\\
\Pi(\mathfrak{u}(1)\oplus\mathfrak{so}(2,8)') =\{\alpha_2,\alpha_3,\alpha_4,\alpha_5,\beta_{\mathfrak{so}(2,8)'}\},\\
\Pi(\mathfrak{su}(2,4)\oplus\mathfrak{su}(2)) =\{\alpha_{23445},\beta_{\mathfrak{su}(2,4)},\alpha_2,\alpha_{456},\alpha_3\}\cup\{\alpha_5\},\\
\Pi(\mathfrak{su}(2)\oplus\mathfrak{su}(4,2)) =\{\alpha_{23445}\}\cup\{\alpha_2,\alpha_{456},\alpha_3,\beta_{\mathfrak{su}(4,2)},\alpha_5\}.
\end{gather*}
Then $\mathfrak{su}(2,4)\oplus\mathfrak{su}(2)=\mathfrak{e}_{6(-14)}\cap\mathfrak{su}(2,6)$,
$\mathfrak{su}(2)\oplus\mathfrak{su}(4,2)=\mathfrak{e}_{6(-14)}\cap(\mathfrak{su}(2)\oplus\mathfrak{so}^*(12))$ holds.
Next we take another simple system of positive roots of $\mathfrak{e}_{7(-25)}$ as
\begin{gather*}
\alpha_1' =\delta_1^{(i)}-\delta_2^{(i)}=\alpha_{23445},\qquad
\alpha_2' =\delta_3^{(i)}-\delta_4^{(i)}=\alpha_5,\\
\alpha_3' =\tfrac{1}{2}\big({-}\delta_1^{(i)}+\delta_2^{(i)}-\delta_3^{(i)}-\delta_4^{(i)}+\delta_5^{(i)}+\delta_6^{(i)}+\delta_7^{(i)}
-\delta_8^{(i)}\big),\\
\alpha_4' =\delta_4^{(i)}-\delta_5^{(i)}=\alpha_{134},\qquad
\alpha_5' =\delta_5^{(i)}-\delta_6^{(i)}=\alpha_2,\qquad
\alpha_6' =\delta_6^{(i)}-\delta_7^{(i)}=\alpha_{456},\\
\beta_{\mathfrak{e}_{7(-25)}}
 =\tfrac{1}{2}\big(\delta_1^{(1)}+\delta_2^{(1)}-\delta_3^{(1)}-\delta_4^{(1)}-\delta_5^{(1)}-\delta_6^{(1)}+\delta_7^{(1)}+\delta_8^{(1)}\big)
=\delta_7^{(2)}+\delta_8^{(2)},
\end{gather*}
and let
\begin{gather*}
\alpha_{\substack{1233\\445}}' =\alpha_1'+\alpha_2'+2\alpha_3'+2\alpha_4'+\alpha_5'=\delta_7^{(i)}-\delta_8^{(i)}=\alpha_3,\\
\beta_{\mathfrak{e}_{6(-14)}'} =\delta_1^{(1)}-\delta_3^{(1)}
=\tfrac{1}{2}\big(\delta_1^{(2)}-\delta_2^{(2)}-\delta_3^{(2)}+\delta_4^{(2)}+\delta_5^{(2)}+\delta_6^{(2)}+\delta_7^{(2)}+\delta_8^{(2)}\big),
\end{gather*}
To this choice corresponds another realization of $\mathfrak{e}_{6(-14)}'\subset\mathfrak{e}_{7(-25)}$ such that
the simple system of positive roots is given by
\begin{gather*} \Pi(\mathfrak{e}_{6(-14)}')=\big\{\alpha_2',\alpha_3',\alpha_4',\alpha_5',\alpha_6',\beta_{\mathfrak{e}_{6(-14)}'}\big\}. \end{gather*}
In addition let
\begin{gather*}
\beta_{\mathfrak{sl}(2,\BR)'} =\delta_2^{(1)}-\delta_8^{(1)}
=\tfrac{1}{2}\big({-}\delta_1^{(2)}+\delta_2^{(2)}+\delta_3^{(2)}+\delta_4^{(2)}+\delta_5^{(2)}+\delta_6^{(2)}+\delta_7^{(2)}-\delta_8^{(2)}\big),\\
\beta_{\mathfrak{su}(1,5)} =\delta_1^{(1)}-\delta_3^{(1)}
=\tfrac{1}{2}\big(\delta_1^{(2)}-\delta_2^{(2)}-\delta_3^{(2)}+\delta_4^{(2)}+\delta_5^{(2)}+\delta_6^{(2)}+\delta_7^{(2)}+\delta_8^{(2)}\big)
=\beta_{\mathfrak{e}_{6(-14)}'},\\
\beta_{\mathfrak{so}^*(10)}
 =\tfrac{1}{2}\big(\delta_1^{(1)}+\delta_2^{(1)}-\delta_3^{(1)}-\delta_4^{(1)}-\delta_5^{(1)}+\delta_6^{(1)}+\delta_7^{(1)}-\delta_8^{(1)}\big)
=\delta_6^{(2)}+\delta_7^{(2)},
\end{gather*}
and we realize $\mathfrak{sl}(2,\BR)\oplus\mathfrak{su}(1,5),\mathfrak{u}(1)\oplus\mathfrak{so}^*(10)\subset\mathfrak{e}_{6(-14)}'$
such that the systems of positive roots are given by
\begin{gather*}\begin{split}&
\Pi(\mathfrak{sl}(2,\BR)\oplus\mathfrak{su}(1,5)) =\{\beta_{\mathfrak{sl}(2,\BR)'}\}
\cup\{\beta_{\mathfrak{su}(1,5)},\alpha_2',\alpha_4',\alpha_5',\alpha_6'\},\\
& \Pi(\mathfrak{u}(1)\oplus\mathfrak{so}^*(10))
 =\{\alpha_2',\alpha_4',\alpha_5',\alpha_6',\beta_{\mathfrak{so}^*(10)}\}.\end{split}
\end{gather*}
Then $\mathfrak{sl}(2,\BR)\oplus\mathfrak{su}(1,5)=\mathfrak{e}_{6(-14)}'\cap\mathfrak{su}(2,6)$, $\mathfrak{u}(1)\oplus\mathfrak{so}^*(10)=\mathfrak{e}_{6(-14)}'\cap(\mathfrak{su}(2)\oplus\mathfrak{so}^*(12))$ holds.
The Vogan diagrams for each Lie algebra are as in Fig.~\ref{Vogan_diag}, and the roots in $\Delta_{\fp^+}(\mathfrak{e}_{7(-25)})$ are described in Fig.~\ref{Deltafp} (quoted from~\cite[Appendix]{J}), where each arrow with label $j$ means that adding the simple root $\alpha_j$ to the root at the source of the arrow we get the root at the target of the arrow. The pattern of the vertices represents the roots in $\Delta_{\fp^+}$ of each subalgebra.

\renewcommand{\textfraction}{0.01}
\renewcommand{\topfraction}{0.99}

\begin{figure}[t!]
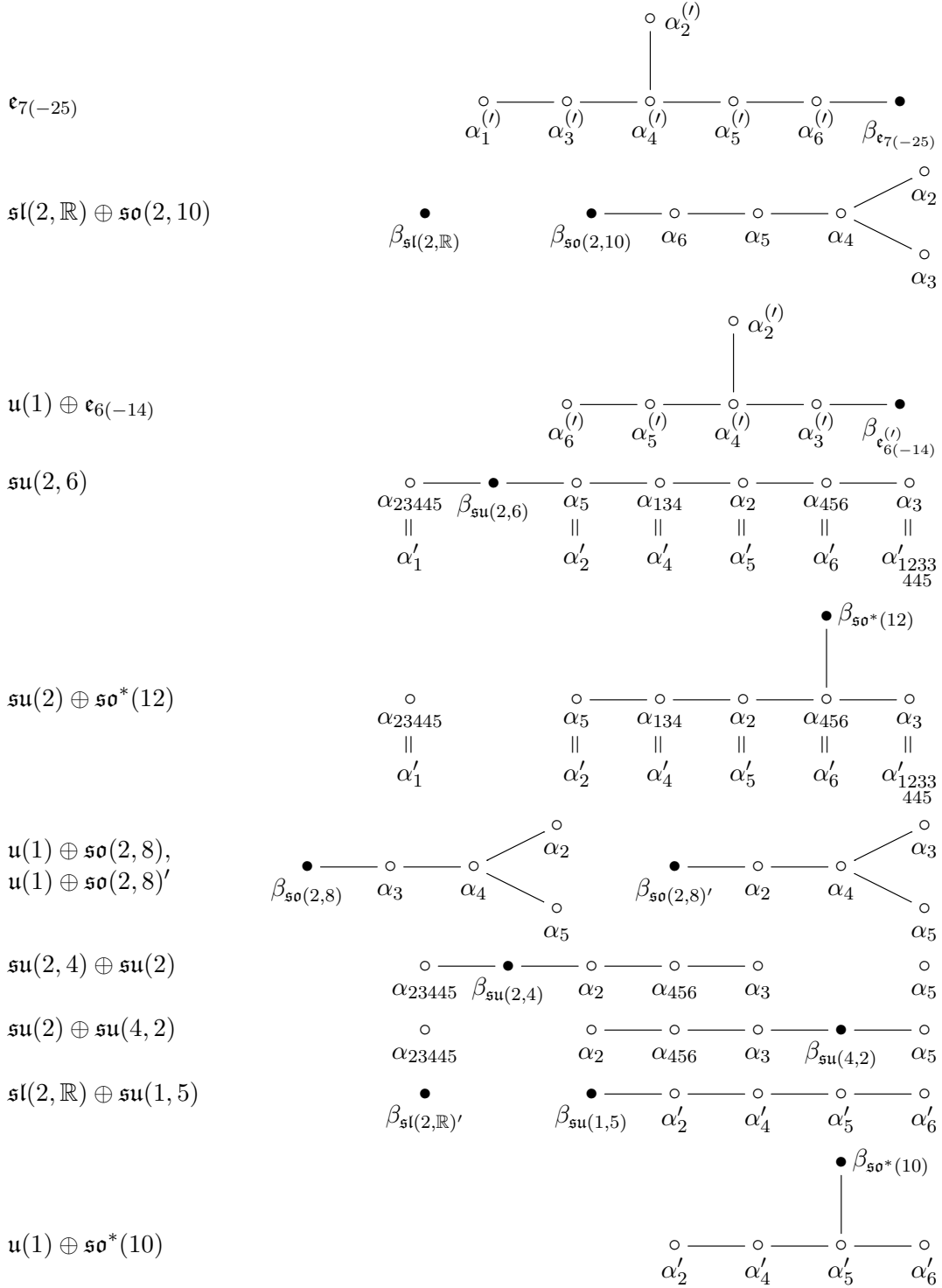

\begin{align*}
&\mathfrak{e}_{7(-25)} &\xygraph{ \circ ([]!{+(0,-.3)} {\alpha_1^{(\prime)}}) - [r]
\circ ([]!{+(0,-.3)} {\alpha_3^{(\prime)}}) - [r]
\circ ([]!{+(0,-.3)} {\alpha_4^{(\prime)}}) (
- [u] \circ ([]!{+(.4,0)} {\alpha_2^{(\prime)}}),
- [r] \circ ([]!{+(0,-.3)} {\alpha_5^{(\prime)}})
- [r] \circ ([]!{+(0,-.3)} {\alpha_6^{(\prime)}})
- [r] \bullet ([]!{+(0,-.4)} {\beta_{\mathfrak{e}_{7(-25)}}}))}\\[-1mm]
&\mathfrak{sl}(2,\BR)\oplus\mathfrak{so}(2,10) &\xygraph{ \bullet ([]!{+(0,-.3)} {\beta_{\mathfrak{sl}(2,\BR)}}) [r] [r]
\bullet ([]!{+(0,-.3)} {\beta_{\mathfrak{so}(2,10)}}) - [r]
\circ ([]!{+(0,-.3)} {\alpha_6}) - [r]
\circ ([]!{+(0,-.3)} {\alpha_5}) - [r]
\circ ([]!{+(0,-.3)} {\alpha_4}) (
- []!{+(1,.5)} \circ ([]!{+(0,-.3)} {\alpha_2}),
- []!{+(1,-.5)} \circ ([]!{+(0,-.3)} {\alpha_3}))}\\[-1mm]
&\mathfrak{u}(1)\oplus\mathfrak{e}_{6(-14)} &\xygraph{ \circ ([]!{+(0,-.3)} {\alpha_6^{(\prime)}}) - [r]
\circ ([]!{+(0,-.3)} {\alpha_5^{(\prime)}}) - [r]
\circ ([]!{+(0,-.3)} {\alpha_4^{(\prime)}}) (
- [u] \circ ([]!{+(.4,0)} {\alpha_2^{(\prime)}}),
- [r] \circ ([]!{+(0,-.3)} {\alpha_3^{(\prime)}})
- [r] \bullet ([]!{+(0,-.4)} {\beta_{\mathfrak{e}_{6(-14)}^{(\prime)}}}))}\\[-1mm]
&\mathfrak{su}(2,6) &\xygraph{ \circ ([]!{+(0,-.6)} {\shortstack{$\alpha_{23445}$\\$\rotatebox{90}{=}$\\$\alpha_1'$}}) - [r]
\bullet ([]!{+(0,-.3)} {\beta_{\mathfrak{su}(2,6)}}) - [r]
\circ ([]!{+(0,-.6)} {\shortstack{$\alpha_5$\\$\rotatebox{90}{=}$\\$\alpha_2'$}}) - [r]
\circ ([]!{+(0,-.6)} {\shortstack{$\alpha_{134}$\\$\rotatebox{90}{=}$\\$\alpha_4'$}}) - [r]
\circ ([]!{+(0,-.6)} {\shortstack{$\alpha_2$\\$\rotatebox{90}{=}$\\$\alpha_5'$}}) - [r]
\circ ([]!{+(0,-.6)} {\shortstack{$\alpha_{456}$\\$\rotatebox{90}{=}$\\$\alpha_6'$}}) - [r]
\circ ([]!{+(0,-.7)} {\shortstack{$\alpha_3$\\$\rotatebox{90}{=}$\\$\alpha_{\substack{1233\\445}}'$}})} \\[-1mm]
&\mathfrak{su}(2)\oplus\mathfrak{so}^*(12) &
\xygraph{ \circ ([]!{+(0,-.6)} {\shortstack{$\alpha_{23445}$\\$\rotatebox{90}{=}$\\$\alpha_1'$}}) [r] [r]
\circ ([]!{+(0,-.6)} {\shortstack{$\alpha_5$\\$\rotatebox{90}{=}$\\$\alpha_2'$}}) - [r]
\circ ([]!{+(0,-.6)} {\shortstack{$\alpha_{134}$\\$\rotatebox{90}{=}$\\$\alpha_4'$}}) - [r]
\circ ([]!{+(0,-.6)} {\shortstack{$\alpha_2$\\$\rotatebox{90}{=}$\\$\alpha_5'$}}) - [r]
\circ ([]!{+(0,-.6)} {\shortstack{$\alpha_{456}$\\$\rotatebox{90}{=}$\\$\alpha_6'$}}) (
- [u] \bullet ([]!{+(.6,0)} {\beta_{\mathfrak{so}^*(12)}}),
- [r] \circ ([]!{+(0,-.7)} {\shortstack{$\alpha_3$\\$\rotatebox{90}{=}$\\$\alpha_{\substack{1233\\445}}'$}}))} \\[-1mm]
&\begin{matrix}\mathfrak{u}(1)\oplus\mathfrak{so}(2,8),\\ \mathfrak{u}(1)\oplus\mathfrak{so}(2,8)' \end{matrix}
&\xygraph{ \bullet ([]!{+(0,-.3)} {\beta_{\mathfrak{so}(2,8)}}) - [r]
\circ ([]!{+(0,-.3)} {\alpha_3}) - [r]
\circ ([]!{+(0,-.3)} {\alpha_4}) (
- []!{+(1,.5)} \circ ([]!{+(0,-.3)} {\alpha_2}),
- []!{+(1,-.5)} \circ ([]!{+(0,-.3)} {\alpha_5}))} \qquad
\xygraph{ \bullet ([]!{+(0,-.3)} {\beta_{\mathfrak{so}(2,8)'}}) - [r]
\circ ([]!{+(0,-.3)} {\alpha_2}) - [r]
\circ ([]!{+(0,-.3)} {\alpha_4}) (
- []!{+(1,.5)} \circ ([]!{+(0,-.3)} {\alpha_3}),
- []!{+(1,-.5)} \circ ([]!{+(0,-.3)} {\alpha_5}))} \\[-1mm]
&\mathfrak{su}(2,4)\oplus\mathfrak{su}(2) &\xygraph{ \circ ([]!{+(0,-.3)} {\alpha_{23445}}) - [r]
\bullet ([]!{+(0,-.3)} {\beta_{\mathfrak{su}(2,4)}}) - [r]
\circ ([]!{+(0,-.3)} {\alpha_2}) - [r]
\circ ([]!{+(0,-.3)} {\alpha_{456}}) - [r]
\circ ([]!{+(0,-.3)} {\alpha_3}) [r] [r]
\circ ([]!{+(0,-.3)} {\alpha_5}))} \\[-1mm]
&\mathfrak{su}(2)\oplus\mathfrak{su}(4,2) &\xygraph{ \circ ([]!{+(0,-.3)} {\alpha_{23445}}) [r] [r]
\circ ([]!{+(0,-.3)} {\alpha_2}) - [r]
\circ ([]!{+(0,-.3)} {\alpha_{456}}) - [r]
\circ ([]!{+(0,-.3)} {\alpha_3}) - [r]
\bullet ([]!{+(0,-.3)} {\beta_{\mathfrak{su}(4,2)}}) - [r]
\circ ([]!{+(0,-.3)} {\alpha_5}))} \\[-1mm]
&\mathfrak{sl}(2,\BR)\oplus\mathfrak{su}(1,5) &\xygraph{ \bullet ([]!{+(0,-.3)} {\beta_{\mathfrak{sl}(2,\BR)'}}) [r] [r]
\bullet ([]!{+(0,-.3)} {\beta_{\mathfrak{su}(1,5)}}) - [r]
\circ ([]!{+(0,-.3)} {\alpha_2'}) - [r]
\circ ([]!{+(0,-.3)} {\alpha_4'}) - [r]
\circ ([]!{+(0,-.3)} {\alpha_5'}) - [r]
\circ ([]!{+(0,-.3)} {\alpha_6'})} \\[-1mm]
&\mathfrak{u}(1)\oplus\mathfrak{so}^*(10) &\xygraph{ \circ ([]!{+(0,-.3)} {\alpha_2'}) - [r]
\circ ([]!{+(0,-.3)} {\alpha_4'}) - [r]
\circ ([]!{+(0,-.3)} {\alpha_5'}) (
- [u] \bullet ([]!{+(.6,0)} {\beta_{\mathfrak{so}^*(10)}}),
- [r] \circ ([]!{+(0,-.3)} {\alpha_6'}))}
\end{align*}
\caption{Vogan diagrams.}\label{Vogan_diag}\vspace{-3mm}
\end{figure}

\begin{figure}[pt!]
\begin{gather*} \xymatrix{\boxdot&\otimes\ar[l]_1&\otimes\ar[l]_3&\boxtimes\ar[l]_4&\boxtimes\ar[l]_5&\oplus\ar[l]_6&&& \\
&&&\boxtimes\ar[u]^2&\boxtimes\ar[u]^2\ar[l]_5&\oplus\ar[u]^2\ar[l]_6&&& \\
&&&&\otimes\ar[u]^4&\boxplus\ar[u]^4\ar[l]_6&\boxplus\ar[l]_5&& \\
&&&&\otimes\ar[u]^3&\boxplus\ar[u]^3\ar[l]_6&\boxplus\ar[u]^3\ar[l]_5&\oplus\ar[l]_4&\oplus\ar[l]_2 \\
&&&&\boxdot\ar[u]^1&\odot\ar[u]^1\ar[l]_6&\odot\ar[u]^1\ar[l]_5&\boxdot\ar[u]^1\ar[l]_4&\boxdot\ar[u]^1\ar[l]_2 \\
&&&&&&&\boxdot\ar[u]^3&\boxdot\ar[u]^3\ar[l]_2 \\
&&&&&&&&\odot\ar[u]^4 \\ &&&&&&&&\odot\ar[u]^5 \\ &&&&&&&&\boxdot\ar[u]^6} \end{gather*}
\vspace{-160pt}
\begin{align*}
\Delta_{\fp^+}(\mathfrak{sl}(2,\BR)\oplus\mathfrak{so}(2,10))&=\{\boxdot,\odot\}\hspace{140pt}\\
\Delta_{\fp^+}(\mathfrak{u}(1)\oplus\mathfrak{e}_{6(-14)})&=\{\boxtimes,\boxplus,\otimes,\oplus\}\\
\Delta_{\fp^+}(\mathfrak{su}(2,6))&=\{\odot,\otimes,\oplus\}\\
\Delta_{\fp^+}(\mathfrak{su}(2)\oplus\mathfrak{so}^*(12))&=\{\boxdot,\boxtimes,\boxplus\}\\
\Delta_{\fp^+}(\mathfrak{u}(1)\oplus\mathfrak{so}(2,8))&=\{\boxtimes,\otimes\}\\
\Delta_{\fp^+}(\mathfrak{u}(1)\oplus\mathfrak{so}(2,8)')&=\{\boxplus,\oplus\}\\
\Delta_{\fp^+}(\mathfrak{su}(2,4)\oplus\mathfrak{su}(2))&=\{\otimes,\oplus\}\\
\Delta_{\fp^+}(\mathfrak{su}(2)\oplus\mathfrak{su}(4,2))&=\{\boxtimes,\boxplus\}
\end{align*}
\begin{gather*} \xymatrix{\odot&\otimes\ar[l]_{1'}&\boxtimes\ar[l]_{3'}&\boxtimes\ar[l]_{4'}&\boxtimes\ar[l]_{5'}&\boxplus\ar[l]_{6'}&&& \\
&&&\boxtimes\ar[u]^{2'}&\boxtimes\ar[u]^{2'}\ar[l]_{5'}&\boxplus\ar[u]^{2'}\ar[l]_{6'}&&& \\
&&&&\boxtimes\ar[u]^{4'}&\boxplus\ar[u]^{4'}\ar[l]_{6'}&\boxplus\ar[l]_{5'}&& \\
&&&&\otimes\ar[u]^{3'}&\oplus\ar[u]^{3'}\ar[l]_{6'}&\oplus\ar[u]^{3'}\ar[l]_{5'}&\oplus\ar[l]_{4'}&\oplus\ar[l]_{2'} \\
&&&&\odot\ar[u]^{1'}&\odot\ar[u]^{1'}\ar[l]_{6'}&\odot\ar[u]^{1'}\ar[l]_{5'}&\odot\ar[u]^{1'}\ar[l]_{4'}&\odot\ar[u]^{1'}\ar[l]_{2'} \\
&&&&&&&\boxdot\ar[u]^{3'}&\boxdot\ar[u]^{3'}\ar[l]_{2'} \\
&&&&&&&&\boxdot\ar[u]^{4'} \\ &&&&&&&&\boxdot\ar[u]^{5'} \\ &&&&&&&&\boxdot\ar[u]^{6'}} \end{gather*}
\vspace{-130pt}
\begin{align*}
\Delta_{\fp^+}(\mathfrak{u}(1)\oplus\mathfrak{e}_{6(-14)}')&=\{\boxtimes,\boxplus,\otimes,\oplus\}\hspace{130pt}\\
\Delta_{\fp^+}(\mathfrak{su}(2,6))&=\{\odot,\otimes,\oplus\}\\
\Delta_{\fp^+}(\mathfrak{su}(2)\oplus\mathfrak{so}^*(12))&=\{\boxdot,\boxtimes,\boxplus\}\\
\Delta_{\fp^+}(\mathfrak{sl}(2,\BR)'\oplus\mathfrak{su}(1,5))&=\{\otimes,\oplus\}\\
\Delta_{\fp^+}(\mathfrak{u}(1)\oplus\mathfrak{so}^*(10))&=\{\boxtimes,\boxplus\}
\end{align*}
\caption{Description of $\Delta_\fp^+(\mathfrak{e}_{7(-25)})$.}\label{Deltafp}
\end{figure}

Next we give the set of strongly orthogonal roots $\{\gamma_1,\ldots,\gamma_r\}$ and the central character ${\rm d}\chi$ of each Lie algebra.
First, for $\mathfrak{e}_{7(-25)}$ we have
\begin{gather*}
\gamma_1(\mathfrak{e}_{7(-25)}) =\gamma_1
=\tfrac{1}{2}\big(\delta_1^{(1)}+\delta_2^{(1)}+\delta_3^{(1)}+\delta_4^{(1)}-\delta_5^{(1)}-\delta_6^{(1)}-\delta_7^{(1)}-\delta_8^{(1)}\big)
=\delta_3^{(2)}+\delta_4^{(2)},\\
\gamma_2(\mathfrak{e}_{7(-25)}) =\gamma_2
=\tfrac{1}{2}\big(\delta_1^{(1)}+\delta_2^{(1)}-\delta_3^{(1)}-\delta_4^{(1)}+\delta_5^{(1)}+\delta_6^{(1)}-\delta_7^{(1)}-\delta_8^{(1)}\big)
=\delta_5^{(2)}+\delta_6^{(2)},\\
\gamma_3(\mathfrak{e}_{7(-25)}) =\gamma_3
=\tfrac{1}{2}\big(\delta_1^{(1)}+\delta_2^{(1)}-\delta_3^{(1)}-\delta_4^{(1)}-\delta_5^{(1)}-\delta_6^{(1)}+\delta_7^{(1)}+\delta_8^{(1)}\big)
=\delta_7^{(2)}+\delta_8^{(2)},\\
{\rm d}\chi_{\mathfrak{e}_{7(-25)}} =\tfrac{1}{2}(\gamma_1+\gamma_2+\gamma_3)
=\tfrac{1}{4}\big(3\delta_1^{(1)}+3\delta_2^{(1)}-\delta_3^{(1)}-\delta_4^{(1)}-\delta_5^{(1)}-\delta_6^{(1)}-\delta_7^{(1)}-\delta_8^{(1)}\big)\\
\hphantom{{\rm d}\chi_{\mathfrak{e}_{7(-25)}}}{}
=\tfrac{1}{2}\big(\delta_3^{(2)}+\delta_4^{(2)}+\delta_5^{(2)}+\delta_6^{(2)}+\delta_7^{(2)}+\delta_8^{(2)}\big).
\end{gather*}
Next, for $\mathfrak{sl}(2,\BR)$, $\mathfrak{so}(2,10)$ and $\mathfrak{e}_{6(-14)}$ we have
\begin{gather*}
\gamma_1(\mathfrak{sl}(2,\BR))=\gamma_1,\qquad
\gamma_1(\mathfrak{so}(2,10))=\gamma_2,\qquad
\gamma_2(\mathfrak{so}(2,10))=\gamma_3,\\
\gamma_1(\mathfrak{e}_{6(-14)})=\tfrac 12 (\gamma_1+\gamma_2)+\tfrac{1}{2}(\ve_1+\ve_2+\ve_3+\ve_4),\\
\gamma_2(\mathfrak{e}_{6(-14)})=\tfrac12 (\gamma_1+\gamma_2)-\tfrac{1}{2}(\ve_1+\ve_2+\ve_3+\ve_4),\\
{\rm d}\chi_{\mathfrak{sl}(2,\BR)}=\tfrac{1}{2}\gamma_1,\qquad
{\rm d}\chi_{\mathfrak{so}(2,10)}=\tfrac{1}{2}(\gamma_2+\gamma_3),\qquad
{\rm d}\chi_{\mathfrak{e}_{6(-14)}}=\tfrac{1}{3}(2\gamma_1+\gamma_2+\gamma_3),
\end{gather*}
and the character of $\fz_{\mathfrak{e}_{7(-25)}}(\mathfrak{e}_{6(-14)})\simeq\mathfrak{u}(1)$ is given by
\begin{gather*} {\rm d}\chi_{\mathfrak{u}(1)}=\tfrac{1}{6}(-\gamma_1+\gamma_2+\gamma_3). \end{gather*}
We write $\frac{1}{2}(\gamma_2-\gamma_3)=:\ve_0$. Then we have
\begin{gather*}
\gamma_1(\mathfrak{sl}(2,\BR)) =2{\rm d}\chi_{\mathfrak{sl}(2,\BR)}={\rm d}\chi_{\mathfrak{e}_{6(-14)}}-2{\rm d}\chi_{\mathfrak{u}(1)},\\
\gamma_1(\mathfrak{so}(2,10)) ={\rm d}\chi_{\mathfrak{so}(2,10)}+\ve_0=\tfrac{1}{2}({\rm d}\chi_{\mathfrak{e}_{6(-14)}}+4{\rm d}\chi_{\mathfrak{u}(1)})+\ve_0,\\
\gamma_2(\mathfrak{so}(2,10)) ={\rm d}\chi_{\mathfrak{so}(2,10)}-\ve_0=\tfrac{1}{2}({\rm d}\chi_{\mathfrak{e}_{6(-14)}}+4{\rm d}\chi_{\mathfrak{u}(1)})-\ve_0,\\
\gamma_1(\mathfrak{e}_{6(-14)}) ={\rm d}\chi_{\mathfrak{sl}(2,\BR)}+\tfrac12 {{\rm d}\chi_{\mathfrak{so}(2,10)}} +\tfrac{1}{2}(\ve_0+\ve_1+\ve_2+\ve_3+\ve_4)\\
\hphantom{\gamma_1(\mathfrak{e}_{6(-14)})}{} =\tfrac{3}{4}{\rm d}\chi_{\mathfrak{e}_{6(-14)}}+\tfrac{1}{2}(\ve_0+\ve_1+\ve_2+\ve_3+\ve_4),\\
\gamma_2(\mathfrak{e}_{6(-14)}) ={\rm d}\chi_{\mathfrak{sl}(2,\BR)}+\tfrac12 {{\rm d}\chi_{\mathfrak{so}(2,10)}}+\tfrac{1}{2}(\ve_0-\ve_1-\ve_2-\ve_3-\ve_4)\\
 \hphantom{\gamma_2(\mathfrak{e}_{6(-14)})}{} =\tfrac{3}{4}{\rm d}\chi_{\mathfrak{e}_{6(-14)}}+\tfrac{1}{2}(\ve_0-\ve_1-\ve_2-\ve_3-\ve_4).
\end{gather*}

Let $V_{(a_0,a_1,a_2,a_3,a_4)}^{[10]\vee}$ be the irreducible representation of $\mathfrak{so}(10)$ with lowest weight
$-a_0\ve_0-a_1\ve_1-a_2\ve_2-a_3\ve_3-a_4\ve_4$. Then as $\mathfrak{u}(1)\oplus\mathfrak{u}(1)\oplus\mathfrak{so}(10)$-modules we have
\begin{gather*}
\cP_m(\fp^+(\mathfrak{sl}(2,\BR))) \simeq -2m{\rm d}\chi_{\mathfrak{sl}(2,\BR)}\simeq -m({\rm d}\chi_{\mathfrak{e}_{6(-14)}}-2{\rm d}\chi_{\mathfrak{u}(1)}),\\
\cP_{(m_1,m_2)}(\fp^+(\mathfrak{so}(2,10))) \simeq -(m_1+m_2){\rm d}\chi_{\mathfrak{so}(2,10)}\boxtimes V_{(m_1-m_2,0,0,0,0)}^{[10]\vee}\\
\hphantom{\cP_{(m_1,m_2)}(\fp^+(\mathfrak{so}(2,10)))}{} \simeq -\tfrac{1}{2} (m_1+m_2)({\rm d}\chi_{\mathfrak{e}_{6(-14)}}+4{\rm d}\chi_{\mathfrak{u}(1)})\boxtimes V_{(m_1-m_2,0,0,0,0)}^{[10]\vee},\\
\cP_{(m_1,m_2)}(\fp^+(\mathfrak{e}_{6(-14)})) \simeq -(m_1+m_2)\big({\rm d}\chi_{\mathfrak{sl}(2,\BR)}+\tfrac12 {{\rm d}\chi_{\mathfrak{so}(2,10)}}\big) \\
 \hphantom{\cP_{(m_1,m_2)}(\fp^+(\mathfrak{e}_{6(-14)})) \simeq}{} \boxtimes V_{\left(\frac{m_1+m_2}{2},\frac{m_1-m_2}{2},\frac{m_1-m_2}{2},\frac{m_1-m_2}{2},\frac{m_1-m_2}{2}\right)}^{[10]\vee}\\
 \hphantom{\cP_{(m_1,m_2)}(\fp^+(\mathfrak{e}_{6(-14)}))}{} \simeq -\tfrac{3}{4}(m_1+m_2){\rm d}\chi_{\mathfrak{e}_{6(-14)}}
\boxtimes V_{\left(\frac{m_1+m_2}{2},\frac{m_1-m_2}{2},\frac{m_1-m_2}{2},\frac{m_1-m_2}{2},\frac{m_1-m_2}{2}\right)}^{[10]\vee}.
\end{gather*}
Next we consider $\mathfrak{su}(2,6)$, $\mathfrak{su}(2)\oplus\mathfrak{so}^*(12)$. We have
\begin{gather*}
\gamma_1(\mathfrak{su}(2,6))=\delta_1^{(1)}-\delta_8^{(1)}
=\tfrac{1}{2}\big(\delta_1^{(2)}-\delta_2^{(2)}+\delta_3^{(2)}+\delta_4^{(2)}+\delta_5^{(2)}+\delta_6^{(2)}+\delta_7^{(2)}-\delta_8^{(2)}\big),\\
\gamma_2(\mathfrak{su}(2,6))=\delta_2^{(1)}-\delta_7^{(1)}
=\tfrac{1}{2}\big({-}\delta_1^{(2)}+\delta_2^{(2)}+\delta_3^{(2)}+\delta_4^{(2)}+\delta_5^{(2)}+\delta_6^{(2)}-\delta_7^{(2)}+\delta_8^{(2)}\big),\\
\gamma_1(\mathfrak{so}^*(12))
=\tfrac{1}{2}\big(\delta_1^{(1)}+\delta_2^{(1)}+\delta_3^{(1)}+\delta_4^{(1)}-\delta_5^{(1)}-\delta_6^{(1)}-\delta_7^{(1)}-\delta_8^{(1)}\big)
=\delta_3^{(2)}+\delta_4^{(2)},\\
\gamma_2(\mathfrak{so}^*(12))
=\tfrac{1}{2}\big(\delta_1^{(1)}+\delta_2^{(1)}-\delta_3^{(1)}-\delta_4^{(1)}+\delta_5^{(1)}+\delta_6^{(1)}-\delta_7^{(1)}-\delta_8^{(1)}\big)
=\delta_5^{(2)}+\delta_6^{(2)},\\
\gamma_3(\mathfrak{so}^*(12))
=\tfrac{1}{2}\big(\delta_1^{(1)}+\delta_2^{(1)}-\delta_3^{(1)}-\delta_4^{(1)}-\delta_5^{(1)}-\delta_6^{(1)}+\delta_7^{(1)}+\delta_8^{(1)}\big)
=\delta_7^{(2)}+\delta_8^{(2)},\\
{\rm d}\chi_{\mathfrak{su}(2,6)}={\rm d}\chi_{\mathfrak{so}^*(12)}={\rm d}\chi_{\mathfrak{e}_7(-25)}
=\tfrac{1}{4}\big(3\delta_1^{(1)}+3\delta_2^{(1)}-\delta_3^{(1)}-\delta_4^{(1)}-\delta_5^{(1)}-\delta_6^{(1)}-\delta_7^{(1)}-\delta_8^{(1)}\big)\\
\hphantom{{\rm d}\chi_{\mathfrak{su}(2,6)}={\rm d}\chi_{\mathfrak{so}^*(12)}={\rm d}\chi_{\mathfrak{e}_7(-25)}}{} =\tfrac{1}{2}\big(\delta_3^{(2)}+\delta_4^{(2)}+\delta_5^{(2)}+\delta_6^{(2)}+\delta_7^{(2)}+\delta_8^{(2)}\big).
\end{gather*}
Let $V_{(a_1,a_2;a_3,\ldots,a_8)}^{(2,6)\vee}$, $V_{(b_1,b_2)}^{(2)\vee}\boxtimes V_{(b_3,\ldots,b_8)}^{(6)\vee}$ be the irreducible
$\mathfrak{s}(\mathfrak{u}(2)\oplus\mathfrak{u}(6))\simeq\mathfrak{su}(2)\oplus\mathfrak{u}(6)$-module with lowest weight
$-\big(a_1\delta_1^{(1)}+\cdots+a_8\delta_8^{(1)}\big)$, $-\big(b_1\delta_1^{(2)}+\cdots+b_8\delta_8^{(2)}\big)$ respectively.
We also write $V_{(a_1,a_2;a_3,\ldots,a_8)}^{(2,6)\vee}\simeq V_{(a_1+c,a_2+c;a_3+c\ldots,a_8+c)}^{(2,6)\vee}$,
$V_{(b_1,b_2)}^{(2)\vee}\boxtimes V_{(b_3,\ldots,b_8)}^{(6)\vee}\simeq V_{(b_1+d,b_2+d)}^{(2)\vee}\boxtimes V_{(b_3,\ldots,b_8)}^{(6)\vee}$
for any $c,d\in\BR$. Then as $\mathfrak{s}(\mathfrak{u}(2)\oplus\mathfrak{u}(6))\simeq\mathfrak{su}(2)\oplus\mathfrak{u}(6)$-modules we have
\begin{gather*}
\cP_{(m_1,m_2)}(\fp^+(\mathfrak{su}(2,6))) \simeq V_{(m_1,m_2;0,0,0,0,-m_2,-m_1)}^{(2,6)\vee}\\
\hphantom{\cP_{(m_1,m_2)}(\fp^+(\mathfrak{su}(2,6)))}{} \simeq V_{\left(\frac{m_1-m_2}{2},\frac{-m_1+m_2}{2}\right)}^{(2)\vee}\boxtimes
V_{\substack{\bigl(\frac{m_1+m_2}{2},\frac{m_1+m_2}{2},\frac{m_1+m_2}{2},\frac{m_1+m_2}{2},\hspace{40pt}\\
\hspace{95pt}\frac{m_1-m_2}{2},\frac{-m_1+m_2}{2}\bigr)}}^{(6)\vee}\\
\hphantom{\cP_{(m_1,m_2)}(\fp^+(\mathfrak{su}(2,6)))}{} \simeq V_{(m_1-m_2,0)}^{(2)\vee}\boxtimes
V_{\left(\frac{m_1+m_2}{2},\frac{m_1+m_2}{2},\frac{m_1+m_2}{2},\frac{m_1+m_2}{2},\frac{m_1-m_2}{2},\frac{-m_1+m_2}{2}\right)}^{(6)\vee},\\
\cP_{(m_1,m_2,m_3)}(\fp^+(\mathfrak{so}^*(12))) \simeq
V_{\substack{\left(\frac{m_1+m_2+m_3}{2},\frac{m_1+m_2+m_3}{2};\frac{m_1-m_2-m_3}{2},\frac{m_1-m_2-m_3}{2},\right.\qquad\qquad \\
\left.\frac{-m_1+m_2-m_3}{2},\frac{-m_1+m_2-m_3}{2},\frac{-m_1-m_2+m_3}{2},\frac{-m_1-m_2+m_3}{2}\right)}}^{(2,6)\vee}\\
\hphantom{\cP_{(m_1,m_2,m_3)}(\fp^+(\mathfrak{so}^*(12)))}{} \simeq V_{(0,0;-m_2-m_3,-m_2-m_3,-m_1-m_3,-m_1-m_3,-m_1-m_2,-m_1-m_2)}^{(2,6)\vee}\\
\hphantom{\cP_{(m_1,m_2,m_3)}(\fp^+(\mathfrak{so}^*(12)))}{}\simeq V_{(0,0)}^{(2)\vee}\boxtimes V_{(m_1,m_1,m_2,m_2,m_3,m_3)}^{(6)\vee}.
\end{gather*}
Next we consider $\mathfrak{u}(1)\oplus\mathfrak{so}(2,8)$. We have
\begin{gather*}
\gamma_1(\mathfrak{so}(2,8))=\tfrac12 (\gamma_1+\gamma_2)+\tfrac{1}{2}(\ve_1+\ve_2+\ve_3+\ve_4),\\
\gamma_2(\mathfrak{so}(2,8))=\tfrac12 (\gamma_1+\gamma_2) -\tfrac{1}{2}(\ve_1+\ve_2+\ve_3+\ve_4),\\
\gamma_1(\mathfrak{so}(2,8)')=\tfrac12 (\gamma_1+\gamma_3)+\tfrac{1}{2}(\ve_1+\ve_2+\ve_3-\ve_4),\\
\gamma_2(\mathfrak{so}(2,8)')=\tfrac12 (\gamma_1+\gamma_3)-\tfrac{1}{2}(\ve_1+\ve_2+\ve_3-\ve_4),\\
{\rm d}\chi_{\mathfrak{so}(2,8)}=\tfrac12 (\gamma_1+\gamma_2),\qquad {\rm d}\chi_{\mathfrak{so}(2,8)'}=\tfrac12 (\gamma_1+\gamma_3),
\end{gather*}
and the character of $\fz_{\mathfrak{e}_{6(-14)}}(\mathfrak{so}(2,8))\simeq\mathfrak{u}(1)$ is given by
\begin{gather*} {\rm d}\chi_{\mathfrak{u}(1)}=\tfrac{1}{6}(\gamma_1-\gamma_2+2\gamma_3). \end{gather*}
Especially we have
\begin{gather*}
\gamma_1(\mathfrak{so}(2,8))={\rm d}\chi_{\mathfrak{so}(2,8)}+\tfrac{1}{2}(\ve_1+\ve_2+\ve_3+\ve_4),\\
\gamma_2(\mathfrak{so}(2,8))={\rm d}\chi_{\mathfrak{so}(2,8)}-\tfrac{1}{2}(\ve_1+\ve_2+\ve_3+\ve_4),\\
\gamma_1(\mathfrak{so}(2,8)')=\tfrac{1}{2}({\rm d}\chi_{\mathfrak{so}(2,8)}+3{\rm d}\chi_{\mathfrak{u}(1)})+\tfrac{1}{2}(\ve_1+\ve_2+\ve_3-\ve_4),\\
\gamma_2(\mathfrak{so}(2,8)')=\tfrac{1}{2}({\rm d}\chi_{\mathfrak{so}(2,8)}+3{\rm d}\chi_{\mathfrak{u}(1)})-\tfrac{1}{2}(\ve_1+\ve_2+\ve_3-\ve_4).
\end{gather*}
Let $V_{(a_1,a_2,a_3,a_4)}^{[8]\vee}$ be the irreducible representation of $\mathfrak{so}(8)$ with lowest weight
$-a_1\ve_1-a_2\ve_2-a_3\ve_3-a_4\ve_4$. Then as $\mathfrak{u}(1)\oplus\mathfrak{u}(1)\oplus\mathfrak{so}(8)$-modules we have
\begin{gather*}
\cP_{(m_1,m_2)}(\fp^+(\mathfrak{so}(2,8))) \simeq -(m_1+m_2){\rm d}\chi_{\mathfrak{so}(2,8)}
\boxtimes V_{\left(\frac{m_1-m_2}{2},\frac{m_1-m_2}{2},\frac{m_1-m_2}{2},\frac{m_1-m_2}{2}\right)}^{[10]\vee},\\
\cP_{(m_1,m_2)}(\fp^+(\mathfrak{so}(2,8)')) \!\simeq\! -\tfrac12 (m_1 \!+\! m_2)({\rm d}\chi_{\mathfrak{so}(2,8)}\! + \!3{\rm d}\chi_{\mathfrak{u}(1)})
\! \boxtimes\! V_{\left(\!\frac{m_1{-}m_2}{2}{,}\frac{m_1{-}m_2}{2}{,}\frac{m_1{-}m_2}{2}{,}\frac{-m_1{+}m_2}{2}\!\right)}^{[10]\vee}\!.
\end{gather*}
Next we consider $\mathfrak{su}(2,4)\oplus\mathfrak{su}(2)$. We have
\begin{gather*}
\gamma_1(\mathfrak{su}(2,4))=\delta_1^{(1)}-\delta_8^{(1)},\qquad
\gamma_2(\mathfrak{su}(2,4))=\delta_2^{(1)}-\delta_7^{(1)},\\
\gamma_1(\mathfrak{su}(4,2))=\tfrac{1}{2}\big(\delta_1^{(1)}+\delta_2^{(1)}+\delta_5^{(1)}-\delta_6^{(1)}-\delta_7^{(1)}-\delta_8^{(1)}+\delta_3^{(1)}-\delta_4^{(1)}\big),\\
\gamma_2(\mathfrak{su}(4,2))
=\tfrac{1}{2}\big(\delta_1^{(1)}+\delta_2^{(1)}-\delta_5^{(1)}+\delta_6^{(1)}-\delta_7^{(1)}-\delta_8^{(1)}-\delta_3^{(1)}+\delta_4^{(1)}\big),\\
{\rm d}\chi_{\mathfrak{su}(2,4)}={\rm d}\chi_{\mathfrak{su}(4,2)}={\rm d}\chi_{\mathfrak{e}_{6(-14)}}
=\tfrac{1}{3}\big(2\delta_1^{(1)}+2\delta_2^{(1)}-\delta_5^{(1)}-\delta_6^{(1)}-\delta_7^{(1)}-\delta_8^{(1)}\big).
\end{gather*}
Let $V_{(a_1,a_2;a_5,a_6,a_7,a_8)}^{(2,4)\vee}\boxtimes V_{(a_3,a_4)}^{(2)\vee}$ be the irreducible
$\mathfrak{s}(\mathfrak{u}(2)\oplus\mathfrak{u}(4))\oplus\mathfrak{su}(2)$-module with lowest weight
$-\big(a_1\delta_1^{(1)}+\cdots+a_8\delta_8^{(1)}\big)$.
We also write $V_{(a_1,a_2;a_5,a_6,a_7,a_8)}^{(2,4)\vee}\boxtimes V_{(a_3,a_4)}^{(2)\vee}
\simeq V_{\substack{(a_1+b,a_2+b;a_5+b,\ \\ \ a_6+b,a_7+b,a_8+b)}}^{(2,4)\vee}\boxtimes V_{(a_3+c,a_4+c)}^{(2)\vee}$ for any $b,c\in\BR$.
Then as $\mathfrak{s}(\mathfrak{u}(2)\oplus\mathfrak{u}(4))\oplus\mathfrak{su}(2)$-modules we have
\begin{gather*}
\cP_{(m_1,m_2)}(\fp^+(\mathfrak{su}(2,4))) \simeq V_{(m_1,m_2;0,0,-m_2,-m_1)}^{(2,4)\vee}\boxtimes V_{(0,0)}^{(2)\vee},\\
\cP_{(m_1,m_2)}(\fp^+(\mathfrak{su}(4,2))) \simeq
V_{\substack{\bigl(\frac{m_1+m_2}{2},\frac{m_1+m_2}{2};\hspace{80pt}\\
\; \frac{m_1-m_2}{2},\frac{-m_1+m_2}{2},\frac{-m_1-m_2}{2},\frac{-m_1-m_2}{2}\bigr)}}^{(2,4)\vee}
\boxtimes V_{\left(\frac{m_1-m_2}{2},\frac{-m_1+m_2}{2}\right)}^{(2)\vee}\\
\hphantom{\cP_{(m_1,m_2)}(\fp^+(\mathfrak{su}(4,2)))}{} \simeq V_{(0,0;-m_2,-m_1,-m_1-m_2,-m_1-m_2)}^{(2,4)\vee}\boxtimes V_{(m_1-m_2,0)}^{(2)\vee}.
\end{gather*}
Finally we consider $\mathfrak{sl}(2,\BR)\oplus\mathfrak{su}(1,5)$, $\mathfrak{u}(1)\oplus\mathfrak{so}^*(10)$. We have
\begin{gather*}
\gamma_1(\mathfrak{sl}(2,\BR)) =\delta_2^{(1)}-\delta_8^{(1)}
=\tfrac{1}{2}\big({-}\delta_1^{(2)}+\delta_2^{(2)}+\delta_3^{(2)}+\delta_4^{(2)}+\delta_5^{(2)}+\delta_6^{(2)}+\delta_7^{(2)}-\delta_8^{(2)}\big),\\
\gamma_1(\mathfrak{su}(1,5)) =\delta_1^{(1)}-\delta_7^{(1)}
=\tfrac{1}{2}\big(\delta_1^{(2)}-\delta_2^{(2)}+\delta_3^{(2)}+\delta_4^{(2)}+\delta_5^{(2)}+\delta_6^{(2)}-\delta_7^{(2)}+\delta_8^{(2)}\big),\\
\gamma_1(\mathfrak{so}^*(10))
 =\tfrac{1}{2}\big(\delta_1^{(1)}+\delta_2^{(1)}+\delta_3^{(1)}+\delta_4^{(1)}-\delta_5^{(1)}-\delta_6^{(1)}-\delta_7^{(1)}-\delta_8^{(1)}\big)
=\delta_3^{(2)}+\delta_4^{(2)},\\
\gamma_2(\mathfrak{so}^*(10))
 =\tfrac{1}{2}\big(\delta_1^{(1)}+\delta_2^{(1)}-\delta_3^{(1)}-\delta_4^{(1)}+\delta_5^{(1)}+\delta_6^{(1)}-\delta_7^{(1)}-\delta_8^{(1)}\big)
=\delta_5^{(2)}+\delta_6^{(2)},\\
{\rm d}\chi_{\mathfrak{sl}(2,\BR)} =\tfrac{1}{2}(\delta_2^{(1)}-\delta_8^{(1)})
=\tfrac{1}{4}\big({-}\delta_1^{(2)}+\delta_2^{(2)}+\delta_3^{(2)}+\delta_4^{(2)}+\delta_5^{(2)}+\delta_6^{(2)}+\delta_7^{(2)}-\delta_8^{(2)}\big),\\
{\rm d}\chi_{\mathfrak{su}(1,5)} =\tfrac{1}{6}\big(5\delta_1^{(1)}-\delta_3^{(1)}-\delta_4^{(1)}-\delta_5^{(1)}-\delta_6^{(1)}-\delta_7^{(1)}\big)\\
\hphantom{{\rm d}\chi_{\mathfrak{su}(1,5)}}{} =\tfrac{1}{12}\big(5\delta_1^{(2)}-5\delta_2^{(2)}+3\delta_3^{(2)}+3\delta_4^{(2)}+3\delta_5^{(2)}+3\delta_6^{(2)}+3\delta_7^{(2)}+5\delta_8^{(2)}\big),\\
{\rm d}\chi_{\mathfrak{so}^*(10)}
 =\tfrac{1}{8}\big(5\delta_1^{(1)}+5\delta_2^{(1)}-\delta_3^{(1)}-\delta_4^{(1)}-\delta_5^{(1)}-\delta_6^{(1)}-\delta_7^{(1)}-5\delta_8^{(1)}\big)\\
\hphantom{{\rm d}\chi_{\mathfrak{so}^*(10)}}{} =\tfrac{1}{2}\big(\delta_3^{(2)}+\delta_4^{(2)}+\delta_5^{(2)}+\delta_6^{(2)}+\delta_7^{(2)}\big),\\
{\rm d}\chi_{\mathfrak{u}(1)}
 =\tfrac{1}{24}\big(5\delta_1^{(1)}-3\delta_2^{(1)}-\delta_3^{(1)}-\delta_4^{(1)}-\delta_5^{(1)}-\delta_6^{(1)}-\delta_7^{(1)}+3\delta_8^{(1)}\big)
 =\tfrac{1}{6}\big(\delta_1^{(2)}-\delta_2^{(2)}+\delta_8^{(2)}\big),
\end{gather*}
where ${\rm d}\chi_{\mathfrak{u}(1)}$ is the character of $\fz_{\mathfrak{e}_{6(-14)}'}(\mathfrak{so}^*(10))$. Especially we have
\begin{gather*}
\gamma_1(\mathfrak{sl}(2,\BR)) =2{\rm d}\chi_{\mathfrak{sl}(2,\BR)}
=-3{\rm d}\chi_{\mathfrak{u}(1)}+\tfrac{1}{2}\big(\delta_3^{(2)}+\delta_4^{(2)}+\delta_5^{(2)}+\delta_6^{(2)}+\delta_7^{(2)}\big),\\
\gamma_1(\mathfrak{su}(1,5)) =\delta_1^{(1)}-\delta_7^{(1)}
=3{\rm d}\chi_{\mathfrak{u}(1)}+\tfrac{1}{2}\big(\delta_3^{(2)}+\delta_4^{(2)}+\delta_5^{(2)}+\delta_6^{(2)}-\delta_7^{(2)}\big),\\
\gamma_1(\mathfrak{so}^*(10))
 ={\rm d}\chi_{\mathfrak{sl}(2,\BR)}+\tfrac{1}{2}\big(\delta_1^{(1)}+\delta_3^{(1)}+\delta_4^{(1)}-\delta_5^{(1)}-\delta_6^{(1)}-\delta_7^{(1)}\big)
=\delta_3^{(2)}+\delta_4^{(2)},\\
\gamma_2(\mathfrak{so}^*(10))
 ={\rm d}\chi_{\mathfrak{sl}(2,\BR)}+\tfrac{1}{2}\big(\delta_1^{(1)}-\delta_3^{(1)}-\delta_4^{(1)}+\delta_5^{(1)}+\delta_6^{(1)}-\delta_7^{(1)}\big)
=\delta_5^{(2)}+\delta_6^{(2)}.
\end{gather*}
Let $V_{(a_1;a_3,\ldots,a_7)}^{(1,5)\vee}$, $V_{(b_3,\ldots,b_7)}^{(5)\vee}$ be the irreducible
$\mathfrak{s}(\mathfrak{u}(1)\oplus\mathfrak{u}(5))\simeq\mathfrak{u}(5)$-module with lowest weight
$-\big(a_1\delta_1^{(1)}+a_3\delta_3^{(1)}+\cdots+a_7\delta_7^{(1)}\big)$, $-\big(b_3\delta_3^{(2)}+\cdots+b_7\delta_7^{(2)}\big)$ respectively.
We also write $V_{(a_1;a_3,\ldots,a_7)}^{(1,5)\vee}\simeq V_{(a_1+c;a_3+c\ldots,a_7+c)}^{(1,5)\vee}$ for any $c\in\BR$.
Then as $\mathfrak{u}(1)\oplus\mathfrak{s}(\mathfrak{u}(1)\oplus\mathfrak{u}(5))\simeq\mathfrak{u}(1)\oplus\mathfrak{u}(5)$-modules we have
\begin{gather*}
\cP_m(\fp^+(\mathfrak{sl}(2,\BR))) \simeq -2m{\rm d}\chi_{\mathfrak{sl}(2,\BR)}
\simeq 3m{\rm d}\chi_{\mathfrak{u}(1)}\boxtimes V_{\left(\frac{m}{2},\frac{m}{2},\frac{m}{2},\frac{m}{2},\frac{m}{2}\right)}^{(5)\vee},\\
\cP_m(\fp^+(\mathfrak{su}(1,5))) \simeq V_{(m;0,0,0,0,-m)}^{(1,5)\vee}
\simeq -3m{\rm d}\chi_{\mathfrak{u}(1)}\boxtimes V_{\left(\frac{m}{2},\frac{m}{2},\frac{m}{2},\frac{m}{2},-\frac{m}{2}\right)}^{(5)\vee},\\
\cP_{(m_1,m_2)}(\fp^+(\mathfrak{so}^*(10))) \simeq -(m_1+m_2){\rm d}\chi_{\mathfrak{sl}(2,\BR)}\boxtimes
V_{\substack{\bigl(\frac{m_1+m_2}{2};\hspace{150pt}\\
\; \frac{m_1-m_2}{2},\frac{m_1-m_2}{2},\frac{-m_1+m_2}{2},\frac{-m_1+m_2}{2},\frac{-m_1-m_2}{2}\bigr)}}^{(1,5)\vee}\\
\hphantom{\cP_{(m_1,m_2)}(\fp^+(\mathfrak{so}^*(10)))}{} \simeq -(m_1+m_2){\rm d}\chi_{\mathfrak{sl}(2,\BR)}\boxtimes V_{(0;-m_2,-m_2,-m_1,-m_1,-m_1-m_2)}^{(1,5)\vee}\\
\hphantom{\cP_{(m_1,m_2)}(\fp^+(\mathfrak{so}^*(10)))}{}\simeq V_{(m_1,m_1,m_2,m_2,0)}^{(5)\vee}.
\end{gather*}

\subsection{Exceptional Jordan triple systems}\label{EJTS}

When $\fg=\mathfrak{e}_{7(-25)}$, we have $\fp^+=\Herm(3,\BO)^\BC$. In this subsection we consider the Jordan triple system structure
of $\Herm(3,\BO)^\BC$ and its subsystems. For $x\in\Herm(3,\BO)^\BC$, the adjoint element~$x^\sharp$ is defined by
\begin{gather}\label{adjoint3}
\begin{pmatrix}\xi_1&x_3&\hat{x}_2\\\hat{x}_3&\xi_2&x_1\\x_2&\hat{x}_1&\xi_3\end{pmatrix}^\sharp\!
:=\!\begin{pmatrix}\xi_2\xi_3-x_1\hat{x}_1&\hat{x}_2\hat{x}_1-\xi_3x_3&x_3x_1-\xi_2\hat{x}_2\\
x_1x_2-\xi_3\hat{x}_3&\xi_3\xi_1-x_2\hat{x}_2&\hat{x}_3\hat{x}_2-\xi_1x_1\\
\hat{x}_1\hat{x}_3-\xi_2x_2&x_2x_3-\xi_1\hat{x}_1&\xi_1\xi_2-x_3\hat{x}_3\end{pmatrix}\!, \quad \xi_i\in\BC,\ \ x_i\in\BO^\BC,
\end{gather}
where $x\mapsto\hat{x}$ is the ($\BC$-linear) conjugate in the octonion $\BO$, so that $\frac{1}{2}\big(xx^\sharp+x^\sharp x\big)=(\det x)I$ holds,
where $\det x=\frac{1}{3}\Tr\big(xx^\sharp\big)$ is the determinant polynomial in the sense of Jordan algebras.
Then for $x,y\in\Herm(3,\BO)^\BC$ the \textit{Freudenthal product} $x\times y$ is defined by
\begin{gather}\label{Freudenthal_def}
x\times y:=(x+y)^\sharp-x^\sharp-y^\sharp
\end{gather}
so that $x\times x=2x^\sharp$. Also let $(x|y):=\Tr(x\overline{y})$, where $y\mapsto\overline{y}$ is the complex conjugate with respect to the
real form $\Herm(3,\BO)\subset\Herm(3,\BO)^\BC$. Then the Jordan triple system structure of $\Herm(3,\BO)^\BC$ is given by
\begin{gather*} Q(x)y:=(x|y)x-x^\sharp\times\overline{y}, \end{gather*}
and the generic norm $h(x,y)$ is given by
\begin{gather*} h(x,y):=1-(x|y)+(x^\sharp|y^\sharp)-(\det x)\overline{(\det y)}. \end{gather*}
We have the linear isomorphism
\begin{gather}\label{Herm3O}
\BC\oplus M(1,2;\BO)^\BC\oplus\Herm(2,\BO)^\BC\xrightarrow{\sim}\Herm(3,\BO)^\BC,\qquad\!\!
(x_{11},x_{12},x_{22})\mapsto \begin{pmatrix}x_{11}&x_{12}\\{}^t\hspace{-1pt}\hat{x}_{12}&x_{22}\end{pmatrix},\!\!\!
\end{gather}
and $\Herm(2,\BO)^\BC\simeq\BC^{10}$ as Jordan algebras, on which $\operatorname{SO}(10)$ acts such that it preserves the quadratic form
$\det\left(\begin{smallmatrix}\xi_1&x\\\hat{x}&\xi_2\end{smallmatrix}\right)=\xi_1\xi_2-x\hat{x}$, and the sesqui-linear inner product
induced from $\Herm(3,\BO)^\BC$. For $x_{12}\in M(1,2;\BO)^\BC$ and $y_{11}\in\BC$, $y_{22}\in\Herm(2,\BO)^\BC$, we have
\begin{gather*} Q\left(\begin{pmatrix}0&x_{12}\\{}^t\hspace{-1pt}\hat{x}_{12}&0\end{pmatrix}\right)\begin{pmatrix}y_{11}&0\\0&y_{22}\end{pmatrix}
=\begin{pmatrix}\Re_\BO(x_{12}(\overline{y_{22}}{}^t\hspace{-1pt}\hat{x}_{12}))&0\\
0&\overline{y_{11}}{}^t\hspace{-1pt}\hat{x}_{12}x_{12}\end{pmatrix}, \end{gather*}
where $\Re_\BO x$ is the real part of $x$ in the sense of the octonion $\BO$ ($x\mapsto \Re_\BO x$ is $\BC$-linear).

Next, for a while we consider simple Euclidean Jordan algebras of rank 3, $\Herm(3,\BK)$, where $\BK=\BR, \BC, \BH, \BO$.
Then we have the linear isomorphism
\begin{gather*}
M(1,3;\BK')\oplus\Herm(3,\BK')\simeq \Skew(3,\BK')\oplus\Herm(3,\BK')\xrightarrow{\sim}\Herm(3,\BK), \\
\left((a_1,a_2,a_3),\begin{pmatrix}\xi_1&x_3&\hat{x}_2\\\hat{x}_3&\xi_2&x_1\\x_2&\hat{x}_1&\xi_3\end{pmatrix}\right)
\mapsto \begin{pmatrix}\xi_1&x_3+a_3j&\hat{x}_2-a_2j\\\hat{x}_3-a_3j&\xi_2&x_1+a_1j\\x_2+a_2j&\hat{x}_1-a_1j&\xi_3\end{pmatrix},
\end{gather*}
where $(\BK,\BK')=(\BC,\BR)$, $(\BH,\BC)$, $(\BO,\BH)$, $j\in \BK$ is an imaginary unit of the Cayley-Dickson extension $\BK=\BK'\oplus\BK' j$,
and $x\mapsto \hat{x}$ is the conjugate in $\BK'$.
Let $x\times y$ and $(x|y)$ be the Freudenthal product and the inner product on $\Herm(3,\BK)$ defined by the same formula
as in the case of $\Herm(3,\BO)$.
Then by \cite{Y1,Y2} we have
\begin{gather*}
(a,x)\times (b,y)=\big({-}(ay+bx),x\times y-\big({}^t\hspace{-1pt}\hat{a}b+{}^t\hspace{-1pt}\hat{b}a\big)\big),\\
((a,x)|(b,y))=2\Re_{\BK'}\big(a{}^t\hspace{-1pt}\hat{b}\big)+\Re_{\BK'}\Tr(xy),
\end{gather*}
where $\Re_{\BK'} x$ is the real part of $x$ in $\BK'$.
We note that in \cite{Y1,Y2} the Freudenthal product is normalized such that $x\times x=x^\sharp$,
but in this paper we use the different normalization.
Then since
\begin{gather}\label{Freudenthal}
Q(x)y=(x|y)x-x^\sharp\times y=xyx, \qquad x,y\in\Herm(3,\BK'),\quad \BK'=\BR,\BC,\BH
\end{gather}
holds, we have
\begin{gather*}
Q((a,x))(b,y)=\bigl(a{}^t\hspace{-1pt}ba-axy+bx^\sharp+\Re_{\BK'}\Tr(xy)a,\\
\hphantom{Q((a,x))(b,y)=\bigl(}{} xyx+\big({}^t\hspace{-1pt}aa\big)\times y+2\Re_{\BK'}\Tr\big(a{}^t\hspace{-1pt}\hat{b}\big)x-x{}^t\hspace{-1pt}\hat{a}b-{}^t\hspace{-1pt}\hat{b}ax\bigr),
\end{gather*}
and especially we have
\begin{gather*} Q((a,0))(0,y)=\big(0,\big({}^t\hspace{-1pt}aa\big)\times y\big),\qquad Q((0,x))(b,0)=\big(bx^\sharp,0\big). \end{gather*}
Now we return to the case $(\BK,\BK')=(\BO,\BH)$, and extend the above formula holomorphically in $(a,x)$, anti-holomorphically in $(b,y)$.
Since we have the isomorphism $\BH\simeq\{a\in M(2,\BC)\colon aJ_2=J_2\overline{a}\}$
where $J_2:=\left(\begin{smallmatrix}0&1\\-1&0\end{smallmatrix}\right)$,
$M(1,3;\BH)$ and $\Herm(3,\BH)$ are naturally identified with
\begin{gather*}\begin{split}&
M(1,3;\BH) \simeq \{ a\in M(2,6;\BC)\colon aJ_6=J_2\overline{a}\}\subset M(2,6;\BC), \\
& \Herm(3,\BH) \simeq \{ x\in \Herm(6,\BC)\colon xJ_6=J_6\overline{x}\}=:\Herm(3,\BH)',\end{split}
\end{gather*}
where $J_6:=\left(\begin{smallmatrix}J_2&0&0\\0&J_2&0\\0&0&J_2\end{smallmatrix}\right)$,
and we again identify $\Herm(3,\BH)'$ with
\begin{gather*} \Herm(3,\BH)'\simeq \{ x\in \Skew(6,\BC)\colon xJ_6=J_6\overline{x}\}\subset\Skew(6,\BC) \end{gather*}
via $x\mapsto xJ_6$. Now we define a quadratic map $\Skew(6,\BC)\to\Skew(6,\BC)$, $x\mapsto x^\#$ by
\begin{gather}\label{adjoint6}
(x^\#)_{kl}:=(-1)^{k+l}\Pf\big((x_{ij})_{i,j\in\{1,\ldots,6\}\setminus\{k,l\}}\big), \qquad 1\le k<l\le 6 ,
\end{gather}
so that $xx^\#=x^\#x=\Pf(x)I$ holds. Then $\sharp$ in $\Herm(3,\BH)'$ and $\#$ in $\Skew(6,\BC)$ are related as
\begin{gather*} (x^\sharp)J_6^{-1}=(J_6x)^\#,\qquad J_6^{-1}(x^\sharp)=(xJ_6)^\#, \qquad x\in\Herm(3,\BH)'. \end{gather*}
For $x,y\in\Skew(6,\BC)$ we define
\begin{gather*} x\dottimes y:=(x+y)^\#-x^\#-y^\#. \end{gather*}
Then we have the equality
\begin{gather*} \tfrac{1}{2}\Tr(xy)x+x^\#\dottimes y=xyx. \end{gather*}
Now we identify $\Herm(3,\BO)^\BC\simeq M(2,6;\BC)\oplus\Skew(6,\BC)$ via the above isomorphism. Then the Jordan triple system structure
is induced on $M(2,6;\BC)\oplus\Skew(6,\BC)$ as
\begin{gather*}
Q((a,x))(b,y)=\big(ab^*a-axy^*+J_2\overline{b}x^\#+\tfrac{1}{2}\Tr(xy^*)a,\\
\hphantom{Q((a,x))(b,y)=\big(}{} xy^*x+({}^t\hspace{-1pt}aJ_2a)\dottimes y^*+\Tr(ab^*)x-x{}^t\hspace{-1pt}a\overline{b}-b^*ax\big).
\end{gather*}
Especially we have
\begin{gather*} Q((a,0))(0,y)=(0,({}^t\hspace{-1pt}aJ_2a)\dottimes y^*),\qquad Q((0,x))(b,0)=(J_2\overline{b}x^\#,0). \end{gather*}
The group $S(U(2)\times U(6))\simeq \operatorname{SU}(2)\times U(6)$ acts on $M(2,6;\BC)\oplus\Skew(6,\BC)$ by
\begin{alignat*}{3}
& (k_1,k_2)(a,x)=\big(k_1ak_2^{-1},\det(k_2)^{-1}k_2x{}^t\hspace{-1pt}k_2\big), \qquad && (k_1,k_2)\in S(U(2)\times U(6)),&\\
&(l_1,l_2)(a,x)=\big(\det(l_2)^{1/2}l_1al_2^{-1},l_2x{}^t\hspace{-1pt}l_2\big), \qquad && (l_1,l_2)\in \operatorname{SU}(2)\times U(6).&
\end{alignat*}

Next we consider $M(1,2;\BO)^\BC$, which is the $\fp^+$-part of $\fg=\mathfrak{e}_{6(-14)}$.
The Jordan triple system structure and the generic norm are the restriction of those of $\Herm(3,\BO)^\BC$ via the identification~(\ref{Herm3O}),
so that for $x,y,z\in M(1,2;\BO)^\BC$,
\begin{gather*} Q(x)y=x\big({}^t\hspace{-1pt}\overline{\hat{y}}x\big),\qquad
B(x,y)z=z-x\big({}^t\hspace{-1pt}\overline{\hat{y}}z\big)-z\big({}^t\hspace{-1pt}\overline{\hat{y}}x\big)
+x\big(\big(\big({}^t\hspace{-1pt}\overline{\hat{y}}z\big){}^t\hspace{-1pt}\overline{\hat{y}}\big)x\big) \end{gather*}
holds. Then we have the isomorphism $M(1,2;\BO)^\BC\simeq\BO^\BC\oplus\BO^\BC\simeq\BC^8\oplus\BC^8$.
Similarly, we have the isomorphism
\begin{gather*} M(1,2;\BO)^\BC\simeq M(2,4;\BC)\oplus M(4,2;\BC)\subset M(2,6;\BC)\oplus\Skew(6,\BC), \end{gather*}
where the inclusion is given by $(a,x)\mapsto\left((0,a),\left(\begin{smallmatrix}0&-{}^t\hspace{-1pt}x\\x&0\end{smallmatrix}\right)\right)$.
Then the Jordan triple system structure of $M(2,4;\BC)\oplus M(4,2;\BC)$ is given by
\begin{gather*}
Q((a,x))(b,y)=\big(ab^*a-J_2\overline{b}\big(xJ_2{}^t\hspace{-1pt}x\big)^\#+\Tr(xy^*)a-axy^*,\\
\hphantom{Q((a,x))(b,y)=\big(}{} xy^*x-\big({}^t\hspace{-1pt}aJ_2a\big)^\#\overline{y}J_2+\Tr(ab^*)x-b^*ax\big),
\end{gather*}
where we define the linear map $\Skew(4,\BC)\to\Skew(4,\BC)$, $x\mapsto x^\#$ by
\begin{gather}\label{adjoint4}
\begin{pmatrix}0&a&b&c\\-a&0&d&e\\-b&-d&0&f\\-c&-e&-f&0\end{pmatrix}^\#
:=\begin{pmatrix}0&-f&e&-d\\f&0&-c&b\\-e&c&0&-a\\d&-b&a&0\end{pmatrix}.
\end{gather}
Especially we have
\begin{gather*} Q((a,0))(0,y)=\big(0,-\big({}^t\hspace{-1pt}aJ_2a\big)^\#\overline{y}J_2\big),\qquad Q((0,x))(b,0)=\big({-}J_2\overline{b}\big(xJ_2{}^t\hspace{-1pt}x\big)^\#,0\big). \end{gather*}
The group $S(U(2)\times U(4))\times \operatorname{SU}(2)$ acts on $M(2,4;\BC)\oplus M(4,2;\BC)$ by
\begin{gather*} (k_1,k_2,k_3)(a,x)=\big(k_1ak_2^{-1},\det(k_2)^{-1}k_2xk_3^{-1}\big). \end{gather*}
Finally, let $M(1,2;\BO)^{\BC\prime}\subset\Herm(3,\BO)^\BC$ be the $\fp^+$-part of $\mathfrak{e}_{6(-14)}'\subset\mathfrak{e}_{7(-25)}$.
Then we have the isomorphism
\begin{gather*} M(1,2;\BO)^{\BC\prime}\simeq\BC\oplus\Skew(5,\BC)\oplus M(1,5;\BC)\subset M(2,6;\BC)\oplus\Skew(6,\BC)\simeq\Herm(3,\BO)^\BC, \end{gather*}
where the inclusion is given by
$(\alpha,x,a)\mapsto \left(\left(\begin{smallmatrix}x&0\\0&0\end{smallmatrix}\right),\left(\begin{smallmatrix}a&0\\0&\alpha\end{smallmatrix}\right)\right)$.
Then the Jordan triple system structure is given by
\begin{gather*}
Q((\alpha,x,a))(\beta,y,b)=\bigg(\alpha\overline{\beta}\alpha+\tfrac{1}{2}\Tr(xy^*)\alpha+\overline{b}\,{}^t\mathbf{Pf}(x),\\
\hphantom{Q((\alpha,x,a))(\beta,y,b)=\bigg(}{} xy^*x+\alpha\Proj\left(\begin{pmatrix}y^*&-{}^t\hspace{-1pt}a\\a&0\end{pmatrix}^\#\right)
+(\alpha\overline{\beta}+ab^*)x-x{}^t\hspace{-1pt}a\overline{b}-b^*ax,\\
\hphantom{Q((\alpha,x,a))(\beta,y,b)=\bigg(}{} ab^*a+\tfrac{1}{2}\Tr(xy^*)a-axy^*+\overline{\beta}\mathbf{Pf}(x)\bigg),
\end{gather*}
where $\Proj\colon \Skew(6,\BC)\to\Skew(5,\BC)$ is defined by $\left(\begin{smallmatrix}x&-{}^t\hspace{-1pt}a\\a&0\end{smallmatrix}\right)\mapsto x$,
and for $x\in\Skew(5,\BC)$, $\mathbf{Pf}(x)\in M(1,5;\BC)$ is defined by
\begin{gather}\label{Pfaff_vector}
(\mathbf{Pf}(x))_k:=(-1)^k\Pf\big((x_{ij})_{i,j\in\{1,\ldots,5\}\setminus\{k\}}\big).
\end{gather}
Especially we have
\begin{gather*} Q((0,x,0))(\beta,0,b)=\big(\overline{b}\,{}^t\mathbf{Pf}(x),0,\overline{\beta}\mathbf{Pf}(x)\big). \end{gather*}
The group $U(1)\times S(U(1)\times U(5))\simeq U(1)\times U(5)$ acts on $\BC\oplus\Skew(5,\BC)\oplus M(1,5;\BC)$ by
\begin{gather*}
\big(k_1,\det(k_2)^{-1},k_2\big)(\alpha,x,a) =\big(k_1^2\alpha,k_1\det(k_2)^{-1}k_2x{}^t\hspace{-1pt}k_2,\det(k_2)^{-1}ak_2^{-1}\big),\\
\hphantom{\big(k_1,\det(k_2)^{-1},k_2\big)(\alpha,x,a) =}{} \ (k_1,\det(k_2)^{-1},k_2)\in U(1)\times S(U(1)\times U(5)),\\
(l_1,l_2)(\alpha,x,a) =\big(l_1^{-3}\det(l_2)^{1/2}\alpha,l_2x{}^t\hspace{-1pt}l_2,l_1^3\det(l_2)^{1/2}al_2^{-1}\big),\qquad
 (l_1,l_2)\in U(1)\times U(5).
\end{gather*}

\section{Explicit calculation of intertwining operators}\label{sect_examples}
\subsection{Normal derivative case}\label{sect_normal_derivative}
In this subsection, we find a sufficient condition for $\cF_{\tau\rho}^*$ to become a normal derivative,
that is, a differential operator in the direction of~$\fp^+_2$,
and a sufficient condition for $\cF_{\tau\rho}$ to become a~multiplication operator.
Let $G\supset G_1$ be two real reductive groups of Hermitian type satisfying the assumption~(\ref{assumption}),
$(\tau,V)$ be an irreducible finite-dimensional representation of $\tilde{K}^\BC$ such that~$\cH_\tau(D,V)$ is holomorphic discrete,
and let $(\rho,W)\subset\cP(\fp^+_2,V)$ be an irreducible submodule of~$\tilde{K}^\BC_1$.
Let $\rK(x_2)\in W\otimes\overline{W}\subset\cP(\fp^+_2,V)\otimes\overline{W}\simeq\cP(\fp^+_2,\Hom(W,V))$
be the $\tilde{K}_1^\BC$-invariant polynomial in the sense of~(\ref{K-invariance}).
Here we regard $W$ both as a submodule of $\cP(\fp^+_2,V)$ and as an abstract $\tilde{K}^\BC_1$-module. Then the following holds.
\begin{Theorem}\label{normal_derivative}\quad
\begin{enumerate}\itemsep=0pt
\item[$(1)$] Assume that there exists an irreducible subrepresentation $V'\subset\cP(\fp^+,V)$ of $\tilde{K}$ such that
$W\subset V'\cap \cP(\fp^+_2,V)\subset\cP(\fp^+,V)$. Then the linear map
\begin{gather*}
\cF_{\tau\rho}^*\colon \ \cO_\tau(D,V)\to\cO_\rho(D_1,W), \qquad
(\cF_{\tau\rho}^*f)(x_1)=\rK\left.\left(\overline{\frac{\partial}{\partial x_2}}\right)^*\right|_{x_2=0}f(x_1,x_2)
\end{gather*}
intertwines the $\tilde{G}_1$-action.
\item[$(2)$] Suppose $(G,G_1)$ is a symmetric pair. We take a subrepresentation $V_1\subset V$ of $\tilde{K}_1$ such that
$W\subset\cP(\fp^+_2,V_1)\subset\cP(\fp^+_2,V)$.
Assume that $(x_2)^{Q(y_1)x_2}=x_2$, and $\tau(B(x_2,y_1))|_{V_1}=I_{V_1}$ for any $x_2\in\fp^+_2$, $y_1\in\fp^+_1$.
Then the linear map
\begin{gather*}
\cF_{\tau\rho}\colon \ \cO_\rho(D_1,W)\to \cO_\tau(D,V), \qquad
(\cF_{\tau\rho}f)(x_1,x_2)=\rK(x_2)f(x_1)
\end{gather*}
intertwines the $\tilde{G}_1$-action.
\end{enumerate}
\end{Theorem}

\begin{proof}(1) Since ${\rm e}^{(x|z)_{\fp^+}}I_V$ is the reproducing kernel of $\cP(\fp^+,V)$ with respect to the inner product $\langle\cdot,\cdot\rangle_F$,
the projection of ${\rm e}^{(x|z)_{\fp^+}}I_V$ onto any subrepresentation of $\cP(\fp^+,V)$ is non-zero.
Let $\bK_{V'}(x,z)\in\cP(\fp^+\times\overline{\fp^+},\End(V))$ be the orthogonal projection of ${\rm e}^{(x|z)_{\fp^+}}I_V$ onto $V'$
with respect to the inner product $\langle\cdot,\cdot\rangle_{\hat{\tau}}$. Then we have
\begin{gather*}
F_{\tau\rho}^*(z_1,z_2)^* =\big\langle \rK(\Proj_2(\cdot)),{\rm e}^{(\cdot|z)_{\fp^+}}I_V\big\rangle_{\hat{\tau}}
=\big\langle \rK(\Proj_2(\cdot)),\bK_{V'}(\cdot,z)\big\rangle_{\hat{\tau}}.
\end{gather*}
Then since the map $f\mapsto \big\langle f,\bK_{V'}(\cdot,z)\big\rangle_{\hat{\tau}}$ in $\End(V')$ intertwines the $\tilde{K}$-action,
by Schur's lemma, there exists a constant $C$ such that
\begin{gather*} F_{\tau\rho}^*(z_1,z_2)^*=C\rK(\Proj_2(z_1,z_2)),\qquad \therefore F_{\tau\rho}^*(z_1,z_2)=\bar{C}\rK(z_2)^*. \end{gather*}
Since the intertwining property does not change by scalar multiplication, we may omit $\bar{C}$.
Then the corresponding $\cF_{\tau\rho}^*$ intertwines the $(\fg_1,\tilde{K}_1)$-action.
Since this is a finite-order differential operator, this extends to the operator between the spaces of all holomorphic functions,
and the claim follows.

(2) By the assumption, we have
\begin{align*}
F_{\tau\rho}(x_2;w_1)
&=\big\langle {\rm e}^{(y_1|w_1)_{\fp^+}}I_W,\bigl(\tau(B(x_2,y_1))\rK\big((x_2)^{Q(y_1)x_2}\big)\bigr)^*\big\rangle_{\hat{\rho},y_1}\\
&=\big\langle {\rm e}^{(y_1|w_1)_{\fp^+}}I_W,\rK(x_2)^*\big\rangle_{\hat{\rho},y_1}.
\end{align*}
Then since $\rK(x_2)^*\in\overline{\cP(\fp^+_2,\Hom(W,V))}\simeq \Hom(\cP(\fp^+_2,V),W)$ and the orthogonal projection of
${\rm e}^{(y_1|w_1)_{\fp^+}}I_W$ onto $W\subset\cP(\fp^+_2,W)$ is $I_W$, we have
\begin{gather*} F_{\tau\rho}(x_2;w_1)=\big\langle I_W,\rK(x_2)^*\big\rangle_{\hat{\rho},y_1}=\rK(x_2). \end{gather*}
Then the corresponding $\cF_{\tau\rho}$ intertwines the $(\fg_1,\tilde{K}_1)$-action.
Since this is a multiplication operator, this extends to the operator between the spaces of all holomorphic functions,
and the claim follows.
\end{proof}

Especially, if $G$ is simple and $(\tau,V)$, $(\rho,W)$ are of the form $(\tau,V)=\big(\tau_0\otimes\chi^{-\lambda},V\big)$,
$(\rho,W)=\big(\rho_0\otimes\chi|_{\tilde{K}_1}^{-\lambda}{,}W\big)$ respectively, then since
$\cP\big(\fp^+_2{,}\Hom\big(W\otimes\chi|_{\tilde{K}_1}^{-\lambda}{,}V\otimes\chi^{-\lambda}\big)\big)^{\tilde{K}_1}
\!\simeq\!\cP(\fp^+_2{,}\Hom(W,V))^{\tilde{K}_1}$ holds for any $\lambda$, $\rK(x_2)$ does not depend on $\lambda$.
Therefore $\cF_{\tau\rho}^*$ and $\cF_{\tau\rho}$ intertwine $\tilde{G}_1$-action for any $\lambda$.

The condition in Theorem \ref{normal_derivative}(1) is the same as \cite[Lemma 5.5(3)]{KP2} when $(G,G_1)$ is of split rank 1
(i.e., $(G,G_1)=(U(q,s),U(q,s-1)\times U(1))$, $(\operatorname{SO}^*(2s),\operatorname{SO}^*(2(s-1))\times \operatorname{SO}(2))$, or $(\operatorname{SO}_0(2,2s),U(1,s))$),
and $(\tau,V)$ is 1-dimensional.
That is also satisfied when $(G,G_1)=(U(q,s),U(q,s')\times U(s''))$ with $s'+s''=s$, $(\operatorname{SO}^*(2s),U(s-1,1))$,
or $(E_{6(-14)},U(1)\times \operatorname{Spin}_0(2,8))$ (up to covering), and $(\tau,V)$ is 1-dimensional. That is,
\begin{Corollary}\label{normal_der_1}
Let $(G,G_1)=(U(q,s),U(q,s')\times U(s''))$, $(\operatorname{SO}^*(2s),\operatorname{SO}^*(2(s-1))\times \operatorname{SO}(2))$, $(\operatorname{SO}^*(2s),U(s-1,1))$, $(\operatorname{SO}_0(2,2s),U(1,s))$
or $(E_{6(-14)},U(1)\times \operatorname{Spin}_0(2,8))$ $($up to covering$)$, and $(\tau,V)=\big(\chi^{-\lambda},\BC\big)$ be $1$-dimensional.
Then for any subrepresentation $W\subset\cP(\fp^+_2)$ of $\tilde{K}^\BC_1$, the differential operator
\begin{gather*}
\cF_{\tau\rho}^*\colon \ \cO_\lambda(D)\to\cO_\rho(D_1,W), \qquad
(\cF_{\tau\rho}^*f)(x_1)=\rK\left.\left(\overline{\frac{\partial}{\partial x_2}}\right)^*\right|_{x_2=0}f(x_1,x_2)
\end{gather*}
intertwines the $\tilde{G}_1$-action.
\end{Corollary}
\begin{proof}
Since it is already proved for $(G,G_1)=(U(q,s),U(1)\times U(q-1,s))$, $(\operatorname{SO}^*(2s),\allowbreak \operatorname{SO}^*(2(s-1))\times \operatorname{SO}(2))$,
or $(\operatorname{SO}_0(2,2s),U(1,s))$ in \cite{KP2},
we only deal with $(G,G_1)=(U(q,s),U(q,s')\times U(s''))$, $(\operatorname{SO}^*(2s),U(s-1,1))$ and $(E_{6(-14)},U(1)\times \operatorname{Spin}_0(2,8))$.
In the first case we have $\fp^+=M(q,s;\BC)$, $\fp^+_1=M(q,s';\BC)$, $\fp^+_2=M(q,s'';\BC)$, and
\begin{gather*}
\cP(\fp^+)=\bigoplus_{\bm\in\BZ_{++}^{\min\{q,s\}}}\cP_\bm(\fp^+)
=\bigoplus_{\bm\in\BZ_{++}^{\min\{q,s\}}}V_\bm^{(q)\vee}\boxtimes V_\bm^{(s)},\\
\cP(\fp^+_2)=\bigoplus_{\bm\in\BZ_{++}^{\min\{q,s''\}}}\cP_\bm(\fp^+_2)
=\bigoplus_{\bm\in\BZ_{++}^{\min\{q,s''\}}}V_\bm^{(q)\vee}\boxtimes V_\bm^{(s'')}.
\end{gather*}
Then by comparing the weights for ${\rm GL}(q,\BC)$, we get $\cP_\bm(\fp^+_2)\subset \cP_\bm(\fp^+)$.
Therefore for any $W=\cP_\bm(\fp^+_2)\otimes\chi^{-\lambda}\subset\cP(\fp^+_2,\chi^{-\lambda})$,
if we set $V'=\cP_\bm(\fp^+)\otimes\chi^{-\lambda}\subset\cP(\fp^+,\chi^{-\lambda})$, then $W\subset V'$ holds, and the condition in
Theorem~\ref{normal_derivative}(1) is satisfied. In the second case we have $\fp^+=\Skew(s,\BC)$, $\fp^+_1=\BC^{s-1}$,
$\fp^+_2=\Skew(s-1,\BC)$, and
\begin{gather*}
\cP(\fp^+) =\bigoplus_{\bm\in\BZ_{++}^{\lfloor s/2\rfloor}}\cP_\bm(\fp^+)
=\bigoplus_{\bm\in\BZ_{++}^{\lfloor s/2\rfloor}}V_{(m_1,m_1,m_2,m_2,\ldots,m_{\lfloor s/2\rfloor},m_{\lfloor s/2\rfloor}(,0))}^{(s)\vee},\\
\cP(\fp^+_2) =\bigoplus_{\bn\in\BZ_{++}^{\lfloor (s-1)/2\rfloor}}\cP_\bn(\fp^+_2)
=\bigoplus_{\bn\in\BZ_{++}^{\lfloor (s-1)/2\rfloor}}
V_{(n_1,n_1,n_2,n_2,\ldots,n_{\lfloor (s-1)/2\rfloor},n_{\lfloor (s-1)/2\rfloor}(,0))}^{(s-1)\vee}.
\end{gather*}
Then by the branching law of $U(s)\downarrow U(s-1)$, we can show that abstractly $\cP_\bn(\fp^+_2)\subset\cP_\bm(\fp^+)$ implies
\begin{gather*}
 m_1\ge n_1\ge m_1\ge n_1\ge m_2\ge\cdots\ge m_r\ge n_r\ge m_r\ge n_r\ge 0, \\
 \therefore (m_1,\ldots,m_r)=(n_1,\ldots,n_r), \qquad s=2r+1,\\
m_1\ge n_1\ge m_1\ge n_1\ge m_2\ge\cdots\ge m_{r-1}\ge n_{r-1}\ge m_r\ge 0\ge m_r, \\
 \therefore (m_1,\ldots,m_r)=(n_1,\ldots,n_{r-1},0), \qquad s=2r.
\end{gather*}
Therefore $(W=)\cP_\bm(\fp^+_2)\otimes\chi^{-\lambda}\subset\cP_\bm(\fp^+)\otimes\chi^{-\lambda}(=V')$ holds as a concrete submodule,
and the condition in Theorem~\ref{normal_derivative}(1) is satisfied.
In the third case $\cP_{(m_1,m_2)}(\fp^+)$ is isomorphic to
\begin{gather*} \cP_{(m_1,m_2)}(\fp^+)\simeq \chi_{\mathfrak{e}_{6(-14)}}^{-\frac{3}{4}(m_1+m_2)}\boxtimes
V_{\left(\frac{m_1+m_2}{2},\frac{m_1-m_2}{2},\frac{m_1-m_2}{2},\frac{m_1-m_2}{2},\frac{m_1-m_2}{2}\right)}^{[10]\vee}, \end{gather*}
and by \cite[Theorem 1.1]{T} we can show
\begin{gather*}
\left.V_{\left(\frac{m_1+m_2}{2},\frac{m_1-m_2}{2},\frac{m_1-m_2}{2},\frac{m_1-m_2}{2},\frac{m_1-m_2}{2}\right)}^{[10]\vee}
\right|_{\mathfrak{so}(2)\oplus\mathfrak{so}(8)}\\
\qquad{} \simeq\bigoplus_{k_1=0}^{m_2}\bigoplus_{\substack{|k_2|\le\frac{m_1-m_2}{2}\\k_2-\frac{m_1-m_2}{2}\in\BZ}}
\bigoplus_{\substack{|l-k_2|\le m_2-k_1\\ l-k_2-m_2+k_1\in 2\BZ}}
V_l^{[2]\vee}\boxtimes V_{\left(\frac{m_1-m_2}{2}+k_1,\frac{m_1-m_2}{2},\frac{m_1-m_2}{2},k_2\right)}^{[8]\vee}.
\end{gather*}
Therefore a $\mathfrak{u}(1)\oplus\mathfrak{u}(1)\oplus\mathfrak{so}(8)$-submodule in $\cP_{(m_1,m_2)}(\fp^+)$ has a lowest weight of the form
\begin{gather*}
 -\tfrac{3}{4}(m_1+m_2){\rm d}\chi_{\mathfrak{e}_{6(-14)}}-l\ve_0
-\big(\tfrac12 (m_1-m_2)+k_1\big)\ve_1\\
\qquad\quad{} -\tfrac12 (m_1-m_2) \ve_2-\tfrac12 (m_1-m_2)\ve_3-k_2\ve_4\\
\qquad {} =-\tfrac{1}{4}(m_1+m_2)(2\gamma_1+\gamma_2+\gamma_3)-\tfrac{1}{2}l(\gamma_2-\gamma_3)\\
\qquad\quad{} -\big(\tfrac12 (m_1-m_2)+k_1\big)\ve_1-\tfrac12 (m_1-m_2) \ve_2-\tfrac12 (m_1-m_2)\ve_3-k_2\ve_4\\
\qquad{} =-\tfrac{1}{2}(m_1+m_2)\gamma_1-\tfrac{1}{4}(m_1+m_2+2l)\gamma_2-\tfrac{1}{4}(m_1+m_2-2l)\gamma_3\\
\qquad\quad{} -\big(\tfrac12 (m_1-m_2)+k_1\big)\ve_1-\tfrac12 (m_1-m_2)\ve_2-\tfrac12 (m_1-m_2)\ve_3-k_2\ve_4.
\end{gather*}
On the other hand, $\cP_{(n_1,n_2)}(\fp^+_2)$ has the lowest weight
\begin{gather*} -\tfrac{1}{2}(n_1+n_2)(\gamma_1+\gamma_3)-\tfrac{1}{2}(n_1-n_2)(\ve_1+\ve_2+\ve_3-\ve_4). \end{gather*}
Therefore if $\cP_{(n_1,n_2)}(\fp^+_2)\subset\cP_{(m_1,m_2)}(\fp^+)$ abstractly, then we have
\begin{gather*} n_1+n_2=m_1+m_2,\qquad l=-\tfrac12 (m_1+m_2),\qquad k_1=0,\\ n_1-n_2=m_1-m_2,\qquad k_2=-\tfrac12 (m_1-m_2),
\end{gather*}
and especially $(n_1,n_2)=(m_1,m_2)$ holds. Therefore
$(W=)\cP_\bm(\fp^+_2)\otimes\chi^{-\lambda}\subset\cP_\bm(\fp^+)\otimes\chi^{-\lambda}(=V')$ holds as a concrete submodule,
and the condition in Theorem~\ref{normal_derivative}(1) is also satisfied.
\end{proof}

Next we consider $\cF_{\tau\rho}$. We again consider
\begin{gather*} (G,G_1)=\begin{cases} (U(q,s),U(q,s')\times U(s''))&(\text{Case }1),\\
(\operatorname{SO}^*(2s),\operatorname{SO}^*(2(s-1))\times \operatorname{SO}(2))&(\text{Case }2),\\
(\operatorname{SO}^*(2s),U(s-1,1))&(\text{Case }3),\\
(\operatorname{SO}_0(2,2s),U(1,s))&(\text{Case }4),\\
(E_{6(-14)},U(1)\times \operatorname{Spin}_0(2,8))&(\text{Case }5) \end{cases} \end{gather*}
(up to covering). Then $\fp^+=M(q,s;\BC)$, $\Skew(s,\BC)$, $\Skew(s,\BC)$, $\BC^{2s}$ and $M(1,2;\BO)^\BC$ respectively.
We realize $G_1\subset G$ such that
\begin{align*}
\fp^+_1=\fg_1\cap\fp^+&=\begin{cases}
\left\{y_1=\begin{pmatrix}y&0\end{pmatrix}\colon y\in M(q,s';\BC)\right\}& (\text{Case }1),\vspace{1mm}\\
\left\{y_1=\begin{pmatrix}y&0\\0&0\end{pmatrix}\colon y\in \Skew(s-1,\BC)\right\}& (\text{Case }2),\vspace{1mm}\\
\left\{y_1=\begin{pmatrix}0&y\\-{}^t\hspace{-1pt}y&0\end{pmatrix}\colon y\in M(s-1,1;\BC)\right\}& (\text{Case }3),\vspace{1mm}\\
\left\{y_1={\vphantom{\bigl(}}^t\!\!\left(\frac{1}{2}{}^t\hspace{-1pt}y,\frac{\sqrt{-1}}{2}{}^t\hspace{-1pt}y\right)\colon
y\in \BC^s\right\}& (\text{Case }4),\vspace{1mm}\\
\left\{y_1=(y,0)\colon y\in\BO^\BC\right\}& (\text{Case }5), \end{cases}\\
\fp^+_2=(\fp^+_1)^\bot&=\begin{cases}
\left\{x_2=\begin{pmatrix}0&x\end{pmatrix}\colon x\in M(q,s'';\BC)\right\}& (\text{Case }1),\vspace{1mm}\\
\left\{x_2=\begin{pmatrix}0&x\\-{}^t\hspace{-1pt}x&0\end{pmatrix}\colon x\in M(s-1,1;\BC)\right\}& (\text{Case }2),\vspace{1mm}\\
\left\{x_2=\begin{pmatrix}x&0\\0&0\end{pmatrix}\colon x\in \Skew(s-1,\BC)\right\}& (\text{Case }3),\vspace{1mm}\\
\left\{x_2={\vphantom{\bigl(}}^t\!\!\left(\frac{1}{2}{}^t\hspace{-1pt}x,-\frac{\sqrt{-1}}{2}{}^t\hspace{-1pt}x\right)\colon
x\in \BC^s\right\}& (\text{Case }4),\\
\left\{x_2=(0,x)\colon x\in\BO^\BC\right\}& (\text{Case }5). \end{cases}
\end{align*}
Then for $(y_1,x_2)\in\fp^+_1\times\fp^+_2$, $(x_2)^{Q(y_1)x_2}=x_2$ holds since
\begin{align*}
Q(y_1)x_2=\begin{cases}
\begin{pmatrix}y&0\end{pmatrix}\begin{pmatrix}0\\x^*\end{pmatrix}\begin{pmatrix}y&0\end{pmatrix}=0& (\text{Case }1),\\
\begin{pmatrix}y&0\\0&0\end{pmatrix}\begin{pmatrix}0&-\bar{x}\\x^*&0\end{pmatrix}
\begin{pmatrix}y&0\\0&0\end{pmatrix}=0& (\text{Case }2),\\
\begin{pmatrix}0&y\\-{}^t\hspace{-1pt}y&0\end{pmatrix}
\begin{pmatrix}x^*&0\\0&0\end{pmatrix}\begin{pmatrix}0&y\\-{}^t\hspace{-1pt}y&0\end{pmatrix}=0& (\text{Case }3),\\
2q(y_1,\overline{x_2})y_1-q(y_1)\overline{x_2}=0& (\text{Case }4),\\
\begin{pmatrix}y&0\end{pmatrix}\left(\begin{pmatrix}0\\\overline{\hat{x}}\end{pmatrix}\begin{pmatrix}y&0\end{pmatrix}\right)=0& (\text{Case }5),
\end{cases}
\end{align*}
and the Bergman operators are computed as
\begin{gather*}
 B(x_2,y_1)=\left(I_q-\begin{pmatrix}0&x\end{pmatrix}\begin{pmatrix}y^*\\0\end{pmatrix},
\left(I_s-\begin{pmatrix}y^*\\0\end{pmatrix}\begin{pmatrix}0&x\end{pmatrix}\right)^{-1}\right) \\
 \hphantom{B(x_2,y_1)=}{} =\left(I_q,\begin{pmatrix}I_{s'}&-y^*x\\0&I_{s''}\end{pmatrix}^{-1}\right)\in K_1^\BC \hspace{62mm} (\text{Case }1), \\
 B(x_2,y_1)=I_s-\begin{pmatrix}0&x\\-{}^t\hspace{-1pt}x&0\end{pmatrix}\begin{pmatrix}y^*&0\\0&0\end{pmatrix}
=\begin{pmatrix}I_{s-1}&0\\{}^t\hspace{-1pt}xy^*&1\end{pmatrix}\in K_1^\BC\hspace{39mm} (\text{Case }2), \\
 B(x_2,y_1)=I_s-\begin{pmatrix}x&0\\0&0\end{pmatrix}\begin{pmatrix}0&-\bar{y}\\y^*&0\end{pmatrix}
=\begin{pmatrix}I_{s-1}&x\bar{y}\\0&1\end{pmatrix}\in K_1^\BC\hspace{38mm} (\text{Case }3), \\
B(x_2,y_1) =h(x_2,y_1)I_{2s}-2(1-q(x_2,\overline{y}_1))(x_2y^*_1-\overline{y}_1{}^t\hspace{-1pt}x_2)
+2(x_2y^*_1-\overline{y}_1{}^t\hspace{-1pt}x_2)^2 \\
\hphantom{B(x_2,y_1)}{} =I_{2s}-2(x_2y^*_1-\overline{y}_1{}^t\hspace{-1pt}x_2)
=I_{2s}-\frac{1}{2}\begin{pmatrix}xy^*-\overline{y}\hspace{1pt}{}^t\hspace{-1pt}x &
-\sqrt{-1}(xy^*-\overline{y}\hspace{1pt}{}^t\hspace{-1pt}x) \\
-\sqrt{-1}(xy^*-\overline{y}\hspace{1pt}{}^t\hspace{-1pt}x) &
-(xy^*-\overline{y}\hspace{1pt}{}^t\hspace{-1pt}x) \end{pmatrix} \\
\hphantom{B(x_2,y_1)}{}
 =\frac{1}{2}\begin{pmatrix}1&1\\\sqrt{-1}&-\sqrt{-1}\end{pmatrix}\!
\begin{pmatrix}I_s& 0\\ -(xy^*-\overline{y}\hspace{1pt}{}^t\hspace{-1pt}x)& I_s\end{pmatrix}\!
\begin{pmatrix}1&-\sqrt{-1}\\1&\sqrt{-1}\end{pmatrix}\in\End(\fp^+)\qquad\!\! (\text{Case }4),
\end{gather*}
for Cases 1--4, and
\begin{gather*}
B(x_2,y_1)z =\begin{pmatrix}z_1& z_2\end{pmatrix}
-\begin{pmatrix}0& x\end{pmatrix}\left(\begin{pmatrix}{}^t\hspace{-1pt}\overline{\hat{y}}\\0\end{pmatrix}
\begin{pmatrix}z_1& z_2\end{pmatrix}\right)
-\begin{pmatrix}z_1& z_2\end{pmatrix}\left(\begin{pmatrix}{}^t\hspace{-1pt}\overline{\hat{y}}\\0\end{pmatrix}
\begin{pmatrix}0& x\end{pmatrix}\right)\\
\hphantom{B(x_2,y_1)z =}{} +\begin{pmatrix}0& x\end{pmatrix}\left(\left(\left(\begin{pmatrix}{}^t\hspace{-1pt}\overline{\hat{y}}\\0\end{pmatrix}
\begin{pmatrix}z_1& z_2\end{pmatrix}\right)\begin{pmatrix}{}^t\hspace{-1pt}\overline{\hat{y}}\\0\end{pmatrix}\right)
\begin{pmatrix}0& x\end{pmatrix}\right)\\
\hphantom{B(x_2,y_1)z }{} =\begin{pmatrix}z_1& z_2-z_1(\overline{\hat{y}}x)\end{pmatrix}
=\begin{pmatrix}z_1& z_2\end{pmatrix}\begin{pmatrix}1&-\overline{\hat{y}}x\\0&1\end{pmatrix},\qquad
 z=\begin{pmatrix}z_1& z_2\end{pmatrix}\in \fp^+
\end{gather*}
for Case 5. That is, each $B(x_2,y_1)$ sits in the nilpotent radical of the parabolic subgroup of $K^\BC$ whose Levi subgroup is $K^\BC_1$.
Therefore, for the representation
\begin{gather*} V=\begin{cases} \chi^{-\lambda_1-\lambda_2}_{U(q,s)}\otimes \big(V_\bk^{(q)\vee}\boxtimes V_\bm^{(s)}\big) & (\text{Case }1),\\
\chi^{-\lambda}_{\operatorname{SO}^*(2s)}\otimes V_\bm^{(s)\vee} & (\text{Cases }2,3),\\
\chi^{-\lambda}_{\operatorname{SO}_0(2,2s)}\otimes V_\bm^{[2s]\vee} & (\text{Case 4}),\\
\chi^{-\lambda}_{E_{6(-14)}}\otimes V_{(m_0,m_1,\ldots,m_4)}^{[10]\vee} & (\text{Case 5})\end{cases} \end{gather*}
of $\tilde{K}^\BC$, if we take the subrepresentation
\begin{align*}
V_1&=\begin{cases} \chi^{-\lambda_1-\lambda_2}_{U(q,s)}\otimes \big(V_\bk^{(q)\vee}\boxtimes V_{(m_1,\ldots,m_s')}^{(s')}
\boxtimes V_{(m_{s'+1},\ldots,m_{s})}^{(s'')}\big) & (\text{Case }1),\\
\chi^{-\lambda}_{\operatorname{SO}^*(2s)}\otimes \big(V_{(m_1,\ldots,m_{s-1})}^{(s-1)\vee}\boxtimes \BC_{-m_s}\big) & (\text{Case }2),\\
\chi^{-\lambda}_{\operatorname{SO}^*(2s)}\otimes \big(V_{(m_2,\ldots,m_s)}^{(s-1)\vee}\boxtimes \BC_{-m_1}\big) & (\text{Case }3),\\
\chi^{-\lambda}_{\operatorname{SO}_0(2,2s)}\otimes V_\bm^{(s)\vee} & (\text{Case }4),\\
\chi^{-\lambda}_{E_{6(-14)}}\otimes V_{(m_0;m_1,\ldots,m_4)}^{[2,8]\vee} & (\text{Case }5),\end{cases} \\
&=\begin{cases} \big(\chi^{-\lambda_1-\lambda_2}_{U(q,s')}\otimes \big(V_\bk^{(q)\vee}\boxtimes V_{(m_1,\ldots,m_s')}^{(s')}\big)\big)
\boxtimes V_{(\lambda_2+m_{s'+1},\ldots,\lambda_2+m_{s})}^{(s'')} & (\text{Case }1),\\
\big(\chi^{-\lambda}_{\operatorname{SO}^*(2(s-1))}\otimes V_{(m_1,\ldots,m_{s-1})}^{(s-1)\vee}\big)\boxtimes \BC_{-\frac{\lambda}{2}-m_s} & (\text{Case }2),\\
\chi^{-\frac{\lambda}{2}-\frac{\lambda}{2}}_{U(s-1,1)}\otimes \big(V_{(m_2,\ldots,m_s)}^{(s-1)\vee}\boxtimes \BC_{-m_1}\big) & (\text{Case }3),\\
\chi^{-\lambda-0}_{U(1,s)}\otimes (\BC\boxtimes V_\bm^{(s)\vee}) & (\text{Case }4),\\
\big(\chi^{-\lambda-\frac{m_0}{2}}_{\operatorname{Spin}_0(2,8)}\otimes V_{(m_1,\ldots,m_4)}^{[8]\vee}\big)
\boxtimes \chi^{-\lambda+\frac{3}{2}m_0}_{U(1)} & (\text{Case }5)\end{cases}\end{align*}
of $\tilde{K}^\BC_1$, then $\tau(B(x_2,y_1))|_{V_1}=I_{V_1}$ holds.
Thus we have proved the following.
\begin{Corollary}\label{normal_der_2}\quad
\begin{enumerate}\itemsep=0pt
\item[$(1)$] Let $(G,G_1)=(U(q,s),U(q,s')\times U(s''))$,
and $(\tau,V)=\big(\chi^{-\lambda_1-\lambda_2}\otimes\big(\tau_\bk^{(q)\vee}\boxtimes \tau_\bm^{(s)}\big),V_\bk^{(q)\vee}\otimes V_\bm^{(s)}\big)$.
Then for any subrepresentation $W\subset\cP\big(\fp^+_2,V_\bk^{(q)\vee}\boxtimes V_{(m_1,\ldots,m_s')}^{(s')}
\boxtimes V_{(m_{s'+1},\ldots,m_{s})}^{(s'')}\big)$ of~$\tilde{K}^\BC_1$, the multiplication operator
$\cF_{\tau\rho}\colon \cO_{\lambda_1+\lambda_2}(D_1,W)\to\cO_{\lambda_1+\lambda_2}(D,V)$, $(\cF_{\tau\rho}f)(x_1,x_2)$ $=\rK(x_2)f(x_1)$ intertwines the $\tilde{G}_1$-action.
\item[$(2)$] Let $(G,G_1)=(\operatorname{SO}^*(2s),\operatorname{SO}^*(2(s-1))\times \operatorname{SO}(2))$,
and $(\tau,V)=\big(\chi^{-\lambda}\otimes\tau_\bm^{(s)\vee},V_\bm^{(s)\vee}\big)$.
Then for any subrepresentation $W\subset\cP\big(\fp^+_2,V_{(m_1,\ldots,m_{s-1})}^{(s-1)\vee}\boxtimes \BC_{-m_s}\big)$ of~$\tilde{K}^\BC_1$,
the multiplication operator
$\cF_{\tau\rho}\colon \cO_\lambda(D_1,W)\to\cO_\lambda(D,V)$, $(\cF_{\tau\rho}f)(x_1,x_2)=\rK(x_2)f(x_1)$ intertwines the $\tilde{G}_1$-action.
\item[$(3)$] Let $(G,G_1)=(\operatorname{SO}^*(2s),U(s-1,1))$,
and $(\tau,V)=\big(\chi^{-\lambda}\otimes\tau_\bm^{(s)\vee},V_\bm^{(s)\vee}\big)$.
Then for any subrepresentation $W\subset\cP\big(\fp^+_2,V_{(m_2,\ldots,m_s)}^{(s-1)\vee}\boxtimes \BC_{-m_1}\big)$ of $\tilde{K}^\BC_1$,
the multiplication operator
$\cF_{\tau\rho}\colon \cO_\lambda(D_1,W)\to\cO_{\frac{\lambda}{2}+\frac{\lambda}{2}}(D,V)$,
$(\cF_{\tau\rho}f)(x_1,x_2)=\rK(x_2)f(x_1)$ intertwines the $\tilde{G}_1$-action.
\item[$(4)$] Let $(G,G_1)=(\operatorname{SO}_0(2,2s),U(1,s))$, and $(\tau,V)=\big(\chi^{-\lambda}\otimes \tau_\bm^{[2s]\vee},V_\bm^{[2s]\vee}\big)$,
Then for any subrepresentation $W\subset\cP\big(\fp^+_2,V_\bm^{(s)\vee}\big)$ of~$\tilde{K}^\BC_1$,
the multiplication operator $\cF_{\tau\rho}\colon \cO_\lambda(D_1,W)\to\cO_{\lambda+0}(D,V)$, $(\cF_{\tau\rho}f)(x_1,x_2)=\rK(x_2)f(x_1)$ intertwines the $\tilde{G}_1$-action.
\item[$(5)$] Let $(G,G_1)=(E_{6(-14)},U(1)\times \operatorname{Spin}_0(2,8))$ $($up to covering$)$,
and $(\tau,V)=\big(\chi^{-\lambda}\otimes \tau_{\bm}^{[10]\vee},\allowbreak V_{\bm}^{[10]\vee}\big)$.
Then for any subrepresentation $W\subset\cP\big(\fp^+_2,V_{(m_0;m_1,\ldots,m_4)}^{[2,8]\vee}\big)$ of~$\tilde{K}^\BC_1$,
the multiplication operator $\cF_{\tau\rho}\colon \cO_\lambda(D_1,W)\to\cO_\lambda(D,V)$, $(\cF_{\tau\rho}f)(x_1,x_2)=\rK(x_2)f(x_1)$ intertwines the $\tilde{G}_1$-action.
\end{enumerate}
\end{Corollary}

\subsection[$\cF_{\tau\rho}^*$ for $(G,G_1)=(G_0\times G_0, \Delta G_0)$]{$\boldsymbol{\cF_{\tau\rho}^*}$ for $\boldsymbol{(G,G_1)=(G_0\times G_0, \Delta G_0)}$}
In this subsection we find the operator $\cF_{\tau\rho}^*$ for $(G,G_1)=(G_0\times G_0, \Delta G_0)$,
where $G_0$ is a simple Lie group of Hermitian type, although it is already done by Peng--Zhang \cite{PZ} (see also, e.g., \cite{BCK,OR,P}).
We denote the complexified Lie algebra of $G_0$ by $\fg_0^\BC=\fp^+_0\oplus\fk^\BC_0\oplus\fp^-_0$.
Similarly, we denote the objects such as $D\subset\fp^+$, $h(x,y)\in\cP(\fp^+\times\overline{\fp^+})$, $p\in\BZ$ for $G_0$
by writing the subscript~$0$.
Then we have
\begin{gather*} \fp^+_1=\{(x_0,x_0)\colon x_0\in\fp^+_0\},\qquad \fp^+_2=\{(x_0,-x_0)\colon x_0\in\fp^+_0\}\subset\fp^+=\fp^+_0\oplus\fp^+_0. \end{gather*}
We identify $\fp^+_0$ and $\fp^+_1$, $\fp^+_2$ via $x_0\mapsto (x_0,x_0)$ and $x_0\mapsto (x_0,-x_0)$ respectively.
Then for $x=(x_L,x_R)\in\fp^+$, the projection onto $\fp^+_2$ is given by
\begin{gather*} x_2=\Proj_2((x_L,x_R))=\tfrac{1}{2}(x_L-x_R). \end{gather*}

Let $(\tau,V)=(\tau_L\boxtimes\tau_R,V_L\otimes V_R)$ be a finite-dimensional irreducible representation
of $\tilde{K}=\tilde{K}_0\times \tilde{K}_0$.
We take an irreducible $\tilde{K}_1^\BC\simeq\tilde{K}_0^\BC$-submodule $W\subset\cP(\fp^+_2,V)$.
Let $\rK(x_2)\in\cP(\fp^+_2,\allowbreak\Hom(W,V))$ be a~$\tilde{K}^\BC_0$-invariant polynomial in the sense of~(\ref{K-invariance}).
Then the function $F_{\tau\rho}^*(z_L,z_R)\in\cP(\overline{\fp^+},\Hom(V,W))$ in Theorem~\ref{main}(1) is given by
\begin{gather*} F_{\tau\rho}^*(z_L,z_R)=\big\langle {\rm e}^{(x_L|z_L)_{\fp^+_0}+(x_R|z_R)_{\fp^+_0}}I_V,\rK\big(\tfrac{1}{2}(x_L-x_R)\big)
\big\rangle_{\hat{\tau}_L\boxtimes\hat{\tau}_R,x_L,x_R}. \end{gather*}
We rewrite $\rK\left(\frac{x_2}{2}\right)$ as $\rK(x_2)$, so that
\begin{gather*} F_{\tau\rho}^*(z_L,z_R)=\big\langle {\rm e}^{(x_L|z_L)_{\fp^+_0}+(x_R|z_R)_{\fp^+_0}}I_V,\rK(x_L-x_R)
\big\rangle_{\hat{\tau}_L\boxtimes\hat{\tau}_R,x_L,x_R}. \end{gather*}
Now we assume $(\tau,V)=\big(\chi^{-\lambda}_0\boxtimes \chi^{-\mu}_0,\BC\big)$ is 1-dimensional, with $\lambda,\mu>p_0-1$ so that
$\cH_\lambda(D_0)$ and $\cH_\mu(D_0)$ are holomorphic discrete,
and let $W:=\cP_\bk(\fp^+_0)\otimes\chi^{-\lambda-\mu}_0$ with $\bk\in\BZ_{++}^{r_0}$.
We realize~$W$ in $\cP(\fp^+_0)\otimes\chi^{-\lambda-\mu}$ with the variable $y_2$,
and write $\rK(x_2)=\rK(x_2,y_2)\in\cP(\fp^+_0\times\overline{\fp^+_0})$.
Then if $\rK(x_L-x_R,y_2)\in\cP(\fp^+_0)\otimes\cP(\fp^+_0)\otimes\overline{\cP_\bk(\fp^+_0)}$ is expanded as
\begin{align*}
\rK(x_L-x_R,y_2)={} &\sum_{\bm\in\BZ_{++}^{r_0}}\sum_{\bn\in\BZ_{++}^{r_0}}\cK_{\bm,\bn}(x_L,x_R;y_2) \\
&{}\in\bigoplus_{\bm\in\BZ_{++}^{r_0}}\bigoplus_{\bn\in\BZ_{++}^{r_0}}\cP_\bm(\fp^+_0)\otimes\cP_\bn(\fp^+_0)\otimes\overline{\cP_\bk(\fp^+_0)},
\end{align*}
then by (\ref{exp_onD}) and (\ref{scalar_norm}) we have
\begin{align*}
F_{\tau\rho}^*(z_L,z_R;y_2)
&=\sum_{\bm\in\BZ_{++}^{r_0}}\sum_{\bn\in\BZ_{++}^{r_0}}\big\langle {\rm e}^{(x_L|z_L)_{\fp^+_0}+(x_R|z_R)_{\fp^+_0}},
\cK_{\bm,\bn}(x_L,x_R;y_2)\big\rangle_{\hat{\tau}_L\boxtimes\hat{\tau}_R,x_L,x_R}\\
&=\sum_{\bm\in\BZ_{++}^{r_0}}\sum_{\bn\in\BZ_{++}^{r_0}}\frac{1}{(\lambda)_{\bm,d_0}(\mu)_{\bn,d_0}}\overline{\cK_{\bm,\bn}(z_L,z_R;y_2)}.
\end{align*}
Now we write $\overline{\cK_{\bm,\bn}(z_L,z_R;y_2)}=:\cK_{\bm,\bn}(y_2;z_L,z_R)$. Then by Theorem~\ref{main}, the linear map
\begin{gather*}
\cF_{\tau\rho}^*\colon \cH_\lambda(D_0)_{\tilde{K}_0}\boxtimes\cH_\mu(D_0)_{\tilde{K}_0}\to \cH_{\lambda+\mu}(D_0,\cP_\bk(\fp^+_0))_{\tilde{K}_0}, \\
\cF_{\lambda,\mu,k}^*f(y_1,y_2)
:=\sum_{\bm\in\BZ_{++}^{r_0}}\sum_{\bn\in\BZ_{++}^{r_0}}\frac{1}{(\lambda)_{\bm,d_0}(\mu)_{\bn,d_0}}
\cK_{\bm,\bn}\left.\left(y_2;\overline{\frac{\partial}{\partial x_L}},\overline{\frac{\partial}{\partial x_R}}\right)
\right|_{x_L=x_R=y_1} f(x_L,x_R)
\end{gather*}
intertwines the $\Delta(\fg_0,\tilde{K}_0)$-action.
Since this is a finite-order differential operator because \linebreak $\cK_{\bm,\bn}(y_2;z_L,z_R)=0$ unless $|\bm|+|\bn|=\deg \rK$,
this is well-defined as an operator between the space of all holomorphic functions,
and this is meromorphically continued for all $\lambda,\mu\in\BC$.

Therefore in order to compute the intertwining operator, we want to compute the expansion of
$\rK(x_L-x_R,y_2)\in\cP(\fp^+_0)\otimes\cP(\fp^+_0)\otimes\overline{\cP_\bk(\fp^+_0)}$.
Then since this is $\tilde{K}^\BC_0$-invariant, its orthogonal projection
$\cK_{\bm,\bn}(x_L,x_R;y_2)\in \cP_\bm(\fp^+_0)\otimes\cP_\bn(\fp^+_0)\otimes\overline{\cP_\bk(\fp^+_0)}$ is also $\tilde{K}^\BC_0$-invariant,
that is,
\begin{gather*} \cK_{\bm,\bn}(lx_L,lx_R;y_2)=\cK_{\bm,\bn}(x_L,x_R;l^*y_2), \qquad x_L,x_R,y_2\in\fp^+_0, \quad l\in \tilde{K}^\BC_0. \end{gather*}
Such polynomials are uniquely determined by the values on $\fp^+_{\rT,0}\oplus\fp^+_{\rT,0}\oplus\overline{\fp^+_{\rT,0}}$.
\begin{Lemma}
If $\cK_1(x_L,x_R;y_2)$, $\cK_2(x_L,x_R;y_2)\in \cP(\fp^+_0\oplus\fp^+_0\oplus\overline{\fp^+_0})$ satisfy
\begin{alignat*}{3}
& \cK_j(lx_L,lx_R;y_2)=\cK_j(x_L,x_R;l^*y_2), \qquad && x_L,x_R,y_2\in\fp^+_0, \quad l\in \tilde{K}^\BC_0,\quad j=1,2,& \\
& \cK_1(x_L,x_R;y_2)=\cK_2(x_L,x_R;y_2),\qquad && x_L,x_R,y_2\in\fp^+_{\rT,0},&
\end{alignat*}
then $\cK_1(x_L,x_R;y_2)=\cK_2(x_L,x_R;y_2)$ holds for any $x_L,x_R,y_2\in\fp^+_0$.
\end{Lemma}
Therefore it suffices to compute the expansion on $\fp^+_{\rT,0}$.
\begin{proof}
Since $\tilde{K}^\BC_0$ acts transitively on an open dense subset of $\fp^+_0$, by $\tilde{K}^\BC_0$-invariance of $\cK_j$
it suffices to show $\cK_1=\cK_2$ on $\fp^+_0\oplus\fp^+_0\oplus\overline{\fp^+_{\rT,0}}$.
We consider $B(te,te)\in\End(\fp^+_0)$, where $t\in\BC$ and $e\in\fp^+_{\rT,0}\subset\fp^+_0$ is a maximal tripotent.
Then $B(te,te)=B(te,te)^*\in \operatorname{Ad}|_{\fp^+_0}(K^\BC_0)\subset\End(\fp^+_0)$ holds if $|t|\ne 1$.
Moreover, for $x=x_\rT+x_\bot\in\fp^+_{\rT,0}\oplus(\fp^+_{\rT,0})^\bot=\fp^+_0$ we have
\begin{gather*} B(te,te)(x_\rT+x_\bot)=\big(1-|t|^2\big)^2x_\rT+\big(1-|t|^2\big)x_\bot. \end{gather*}
Therefore, for $|t|\ne 1$, $x_L,x_R\in\fp^+_0$, $y_2\in\fp^+_{\rT,0}$ we have
\begin{align*}
\cK_j(x_L,x_R;y_2)&=\cK_j\big(x_L,x_R;\big(1-|t|^2\big)^2B(te,te)^{-1}y_2\big) \\
&=\cK_j\big(\big(1-|t|^2\big)^2B(te,te)^{-1}x_L,\big(1-|t|^2\big)^2B(te,te)^{-1}x_R;y_2\big) \\
&=\cK_j\big(x_{L\rT}+\big(1-|t|^2\big)x_{L\bot},x_{R\rT}+\big(1-|t|^2\big)x_{R\bot};y_2\big),
\end{align*}
where we write $x_L=x_{L\rT}+x_{L\bot}$, $x_R=x_{R\rT}+x_{R\bot}\in\fp^+_{\rT,0}\oplus(\fp^+_{\rT,0})^\bot=\fp^+_0$.
Especially, by taking a limit $|t|\to 1$, we have
\begin{gather*} \cK_j(x_L,x_R;y_2)=\cK_j (x_{L\rT},x_{R\rT};y_2 ). \end{gather*}
Therefore, $\cK_1=\cK_2$ on $\fp^+_{\rT,0}\oplus\fp^+_{\rT,0}\oplus\overline{\fp^+_{\rT,0}}$ implies
$\cK_1=\cK_2$ on $\fp^+_0\oplus\fp^+_0\oplus\overline{\fp^+_{\rT,0}}$, and also on $\fp^+_0\oplus\fp^+_0\oplus\overline{\fp^+_0}$.
\end{proof}

Now we additionally assume that $\bk=(k,\ldots,k)$ with $k\in\BZ_{\ge 0}$.
Then up to constant we have $\rK(x_2,y_2)=\Delta(x_2)^k\overline{\Delta(y_2)^k}$ if $x_2,y_2\in\fp^+_{\rT,0}$.
Then for $x_L,x_R,y_2\in\fp^+_{\rT,0}$, by (\ref{det_norm}), Proposition \ref{Kmxe} and (\ref{ext_of_h}) we have
\begin{align*}
\rK(x_L-x_R,y_2)&=\Delta(x_L-x_R)^k\overline{\Delta(y_2)^k}
=\Delta(x_L)^k\overline{\Delta(y_2)^k}\Delta\big(e-P\big(x^{-1/2}_L\big)x_R\big)^k\\
&=\Delta(x_L)^k\overline{\Delta(y_2)^k}
\sum_{\bm\in\BZ_{++}^{r_0}}(-k)_{\bm,d_0}\frac{d_\bm^{(d_0,r_0,b_0)}}{\big(\frac{n_0}{r_0}\big)_{\bm,d_0}}
\Phi_\bm^{(d_0,r_0)}\big(P\big(x^{-1/2}_L\big)x_R\big).
\end{align*}
By \cite[Lemma XIV.1.2]{FK},
\begin{gather*} \Delta(x_L)^k\Phi_\bm^{(d_0,r_0)}\big(P\big(x^{-1/2}_L\big)x_R\big)=\Delta(x_L)^k\Phi_\bm^{(d_0,r_0)}\big(P\big(x^{1/2}_R\big)x_L^{-1}\big)\end{gather*} holds.
This lies in $\cP_\bm(\fp^+_{\rT,0})$ as a polynomial in~$x_R$, and lies in $\cP_{k-\bm^*}(\fp^+_{\rT,0})$ as a polynomial in $x_L$,
where $k-\bm^*:=(k-m_{r_0},k-m_{r_0-1},\ldots,k-m_1)$.
Now let $\Psi_{k-\bm^*,\bm}^{(d_0,r_0)}(x_L,x_R;y_2)\in \cP(\fp^+_0\times\fp^+_0\times\overline{\fp^+_0})$ be the polynomial satisfying
\begin{alignat*}{3}
&\Psi_{k-\bm^*,\bm}^{(d_0,r_0)}(lx_L,lx_R;y_2)=\Psi_{k-\bm^*,\bm}^{(d_0,r_0)}(x_L,x_R;l^*y_2),\quad && x_L,x_R,y_2\in\fp^+_0,\ l\in K^\BC_0,& \\
& \Psi_{k-\bm^*,\bm}^{(d_0,r_0)}(x_L,x_R;y_2)=\Delta(x_L)^k\overline{\Delta(y_2)^k}\Phi_\bm^{(d_0,r_0)}\big(P\big(x^{-1/2}_L\big)x_R\big),\qquad && x_L,x_R,y_2\in\fp^+_{\rT,0}.&
\end{alignat*}
Such polynomials are unique by the previous lemma. Then we have
\begin{align*}
\rK(x_L-x_R,y_2)=\sum_{\bm\in\BZ_{++}^{r_0}}(-k)_{\bm,d_0}\frac{d_\bm^{(d_0,r_0,b_0)}}{\big(\frac{n_0}{r_0}\big)_{\bm,d_0}}
\Psi_{k-\bm^*,\bm}^{(d_0,r_0)}(x_L,x_R;y_2).
\end{align*}
We write
\begin{gather*} \overline{\Psi_{k-\bm^*,\bm}^{(d_0,r_0)}(x_L,x_R;y_2)}=:\Psi_{k-\bm^*,\bm}^{(d_0,r_0)}(y_2;x_L,x_R). \end{gather*}
Then using this expansion, we get
\begin{gather*} F_{\tau\rho}^*(z_L,z_R;y_2)
=\sum_{\bm\in\BZ_{++}^{r_0}}\frac{(-k)_{\bm,d_0}}{(\lambda)_{k-\bm^*,d_0}(\mu)_{\bm,d_0}}
\frac{d_\bm^{(d_0,r_0,b_0)}}{\big(\frac{n_0}{r_0}\big)_{\bm,d_0}}
\Psi_{k-\bm^*,\bm}^{(d_0,r_0)}(y_2;z_L,z_R). \end{gather*}
We note that the sum is finite because $(-k)_{\bm,d_0}=0$ if $m_1>k$,
and the above formula is symmetric under the exchange of $(z_L,\lambda)$ and $(z_R,\mu)$ up to signature, because
\begin{gather*}
\Psi_{k-\bm^*,\bm}^{(d_0,r_0)}(y_2;z_L,z_R)=\Psi_{\bm,k-\bm^*}^{(d_0,r_0)}(y_2;z_R,z_L),\\
(-k)_{\bm,d_0}\frac{d_\bm^{(d_0,r_0,b_0)}}{\big(\frac{n_0}{r_0}\big)_{\bm,d_0}}
=(-k)_{\bm,d_0}\frac{d_\bm^{(d_0,r_0,0)}}{\big(\frac{n_{0,\rT}}{r_0}\big)_{\bm,d_0}}
=(-1)^{kr}(-k)_{k-\bm^*,d_0}\frac{d_{k-\bm^*}^{(d_0,r_0,0)}}{\big(\frac{n_{0,\rT}}{r_0}\big)_{k-\bm^*,d_0}}\\
\hphantom{(-k)_{\bm,d_0}\frac{d_\bm^{(d_0,r_0,b_0)}}{\big(\frac{n_0}{r_0}\big)_{\bm,d_0}}}{}
=(-1)^{kr}(-k)_{k-\bm^*,d_0}\frac{d_{k-\bm^*}^{(d_0,r_0,b_0)}}{\big(\frac{n_0}{r_0}\big)_{k-\bm^*,d_0}},
\end{gather*}
the latter of which follows from the proof of \cite[Proposition 2.6]{N}.
Therefore we have proved the following.
\begin{Theorem}\label{thm_tensor}
Let $k\in\BZ_{\ge 0}$. Then the linear map
\begin{gather*}
\cF_{\lambda,\mu,k}^*\colon \ \cO_\lambda(D_0)\hboxtimes\cO_\mu(D_0)\to \cO_{\lambda+\mu}(D_0,\cP_{(k,\ldots,k)}(\fp^+_0)), \\
\cF_{\lambda,\mu,k}^*f(y_1,y_2) :=\sum_{\bm\in\BZ_{++}^{r_0}}\frac{(-k)_{\bm,d_0}}{(\lambda)_{k-\bm^*,d_0}(\mu)_{\bm,d_0}}
\frac{d_\bm^{(d_0,r_0,b_0)}}{\big(\frac{n_0}{r_0}\big)_{\bm,d_0}}\\
\hphantom{\cF_{\lambda,\mu,k}^*f(y_1,y_2) :=}{} \times\Psi_{k-\bm^*,\bm}^{(d_0,r_0)}
\left.\left(y_2;\overline{\frac{\partial}{\partial x_L}},\overline{\frac{\partial}{\partial x_R}}\right)
\right|_{x_L=x_R=y_1}f(x_L,x_R)
\end{gather*}
$(y_1\in D_0, y_2\in\fp^+_2)$ intertwines the $\Delta \tilde{G}_0$-action.
\end{Theorem}
This gives essentially the same result as in \cite{PZ}.
If $G_0$ is of tube type, i.e., $G_0=G_{0,\rT}$, then $\cP_{(k,\ldots,k)}(\fp^+_0)$ is 1-dimensional, and we have
$\cO_{\lambda+\mu}(D_0,\cP_{(k,\ldots,k)}(\fp^+_0))\simeq \cO_{\lambda+\mu+2k}(D_0)$ via $f\Delta(y)^k\mapsto f$,
and thus it gives the intertwining operator $\cF_{\lambda,\mu,k}^{\prime*}\colon \cO_\lambda(D_0)\hboxtimes\cO_\mu(D_0)\to\linebreak
\cO_{\lambda+\mu+2k}(D_0)$,
\begin{gather*} \cF_{\lambda,\mu,k}^{\prime *}f(y)
:=\sum_{\bm\in\BZ_{++}^{r_0}}\frac{(-k)_{\bm,d_0}}{(\lambda)_{k-\bm^*,d_0}(\mu)_{\bm,d_0}}
\frac{d_\bm^{(d_0,r_0,b_0)}}{\big(\frac{n_0}{r_0}\big)_{\bm,d_0}}
\Phi_{k-\bm^*,\bm}^{(d_0,r_0)}
\left.\left(\frac{\partial}{\partial x_L},\frac{\partial}{\partial x_R}\right)
\right|_{x_L=x_R=y}\hspace{-10pt}f(x_L,x_R), \end{gather*}
where we write
\begin{gather*} \Phi_{k-\bm^*,\bm}^{(d_0,r_0)}(x_L,x_R):=
\overline{\Delta(y_2)^{-k}}\Psi_{k-\bm^*,\bm}^{(d_0,r_0)}(x_L,x_R;y_2)=\Delta(x_L)^k\Phi_\bm^{(d_0,r_0)}\big(P\big(x^{-1/2}_L\big)x_R\big). \end{gather*}
Also, if $G_0=U(s,1)$, then $\Psi_{k-m,m}^{(2,1)}(y_2;x_L,x_R)
=\big({}^t\hspace{-1pt}y_2\overline{x_L}\big)^{k-m}\big({}^t\hspace{-1pt}y_2\overline{x_R}\big)^m$ holds, and thus
$\cF_{\lambda,\mu,k}^*\colon$ $ \cO_\lambda(D_0)\hboxtimes\cO_\mu(D_0)\to \cO_{\lambda+\mu}(D_0,\cP_{k}(\BC^s))$ becomes
\begin{gather*} \cF_{\lambda,\mu,k}^*f(y_1,y_2)
:=\sum_{m=0}^\infty\frac{(-k)_m}{(\lambda)_{k-m}(\mu)_m}\frac{1}{m!}
\left({}^t\hspace{-1pt}y_2\frac{\partial}{\partial x_L}\right)^{k-m}\left.\left({}^t\hspace{-1pt}y_2\frac{\partial}{\partial x_R}\right)^m
\right|_{x_L=x_R=y_1}f(x_L,x_R). \end{gather*}
This coincides with the Rankin--Cohen bidifferential operator (see \cite[Theorem~7.1]{C}, \cite[Theorem~8.1(2)]{KP2}).

\subsection[$\cF_{\tau\rho}$ for $(G,G_1)=(\operatorname{Sp}(s,\BR), \operatorname{Sp}(s',\BR)\times \operatorname{Sp}(s'',\BR))$, \\
$(U(q,s), U(q',s')\times U(q'',s''))$, $(\operatorname{SO}^*(2s), \operatorname{SO}^*(2s')\times \operatorname{SO}^*(2s''))$, \\
$(E_{6(-14)}, {\rm SL}(2,\BR)\times \operatorname{SU}(1,5))$, $(E_{7(-25)}, {\rm SL}(2,\BR)\times \operatorname{Spin}_0(2,10))$]{$\boldsymbol{\cF_{\tau\rho}}$ for $\boldsymbol{(G,G_1)=(\operatorname{Sp}(s,\BR), \operatorname{Sp}(s',\BR)\times \operatorname{Sp}(s'',\BR))}$, \\
$\boldsymbol{(U(q,s), U(q',s')\times U(q'',s''))}$, $\boldsymbol{(\operatorname{SO}^*(2s), \operatorname{SO}^*(2s')\times \operatorname{SO}^*(2s''))}$, \\
$\boldsymbol{(E_{6(-14)}, {\rm SL}(2,\BR)\times \operatorname{SU}(1,5))}$, $\boldsymbol{(E_{7(-25)}, {\rm SL}(2,\BR)\times \operatorname{Spin}_0(2,10))}$}

In this subsection we set
\begin{align*}
(G,G_1)&=(G,G_{11}\times G_{22})\\
&= \begin{cases}(\operatorname{Sp}(s,\BR), \operatorname{Sp}(s',\BR)\times \operatorname{Sp}(s'',\BR)) \quad (s=s'+s'')& (\text{Case }d=1),\\
(U(q,s), U(q',s')\times U(q'',s'')) \quad \ \ \ (q=q'+q'',s=s'+s'')& (\text{Case }d=2),\\
(\operatorname{SO}^*(2s), \operatorname{SO}^*(2s')\times \operatorname{SO}^*(2s'')) \quad (s=s'+s'')& (\text{Case }d=4),\\
(E_{6(-14)}, {\rm SL}(2,\BR)\times \operatorname{SU}(1,5)) & (\text{Case }d=6),\\
(E_{7(-25)},{\rm SL}(2,\BR)\times \operatorname{Spin}_0(2,10)) & (\text{Case }d=8) \end{cases}
\end{align*}
(up to covering). Then the maximal compact subgroups $(K,K_1)=(K,K_{11}\times K_{22})\subset (G,G_{11}\times G_{22})$ are given by
\begin{align*}
(K,K_1)&=(K,K_{11}\times K_{22})\\
&=\begin{cases} (U(s), U(s')\times U(s''))& (\text{Cases }d=1,4),\\
(U(q)\times U(s), (U(q')\times U(s'))\times (U(q'')\times U(s'')))& (\text{Case }d=2),\\
(U(1)\times \operatorname{Spin}(10), U(1)\times S(U(1)\times U(5)))& (\text{Case }d=6),\\
(U(1)\times E_6,U(1)\times (U(1)\times \operatorname{Spin}(10)))& (\text{Case }d=8)	 \end{cases}
\end{align*}
(up to covering). Also we have
\begin{gather*} \fp^+=\begin{cases}\Sym(s,\BC) & (\text{Case }d=1),\\ M(q,s;\BC) & (\text{Case }d=2),\\ \Skew(s,\BC) & (\text{Case }d=4),\\
M(1,2;\BO)^\BC & (\text{Case }d=6),\\ \Herm(3,\BO)^\BC & (\text{Case }d=8), \end{cases} \end{gather*}
and $\fp^+_1=\fp^+_{11}\oplus\fp^+_{22}:=\fg_1^\BC\cap\fp^+$, $\fp^+_2=\fp^+_{12}:=(\fp^+_1)^\bot$ are realized as
\begin{gather*} (\fp^+_{11},\fp^+_{12},\fp^+_{22})=\begin{cases}
(\Sym(s',\BC), M(s',s'';\BC), \Sym(s'',\BC)) & (\text{Case }d=1),\\
(M(q',s';\BC), (M(q',s'';\BC)\oplus M(q'',s';\BC)), M(q'',s'';\BC))\! & (\text{Case }d=2),\\
(\Skew(s',\BC), M(s',s'';\BC), \Skew(s'',\BC)) & (\text{Case }d=4),\\
(\BC, \Skew(5,\BC), M(1,5;\BC)) & (\text{Case }d=6),\\
\big(\BC, M(1,2;\BO)^\BC, \Herm(2,\BO)^\BC\big) & (\text{Case }d=8). \end{cases} \end{gather*}
For $d=1,4,6,8$, let $\chi$, $\chi_{11}$ and $\chi_{22}$ be the characters of $K^\BC$, $K_{11}^\BC$ and $K_{22}^\BC$ respectively,
normalized as (\ref{char}). Then we have $\chi|_{K_{jj}}=\chi_{jj}$ $(j=1,2)$.
Similarly, for $d=2$, let $\chi^{-\lambda_1-\lambda_2}$, $\chi_{11}^{-\lambda_1-\lambda_2}$ and $\chi_{22}^{-\lambda_1-\lambda_2}$ be
the characters of $K^\BC$, $K_{11}^\BC$ and $K_{22}^\BC$ respectively, as~(\ref{charuqs}).
Then similarly we have $\chi^{-\lambda_1-\lambda_2}|_{K_{jj}}=\chi_{jj}^{-\lambda_1-\lambda_2}$ $(j=1,2)$.

Now let $(\tau,V)=\big(\chi^{-\lambda},\BC\big)=\big(\chi^{-\lambda_1-\lambda_2},\BC\big)$ with $\lambda$ sufficiently large,
$W\subset\cP(\fp^+_{12})\otimes\chi^{-\lambda}$ be an irreducible $\tilde{K}^\BC_1=\tilde{K}^\BC_{11}\times \tilde{K}^\BC_{22}$-submodule,
and $\rK(x_2)\in \cP\big(\fp^+_{12},\Hom\big(W,\chi^{-\lambda}\big)\big)$ be the $\tilde{K}^\BC_1$-invariant polynomial in the sense of~(\ref{K-invariance}).
For $x_2=x_{12}\in\fp^+_2=\fp^+_{12}$, $w_1=w_{11}+w_{22}\in\fp^+_1=\fp^+_{11}\oplus\fp^+_{22}$,
we want to compute
\begin{align*}
F_{\tau\rho}(x_2;w_1)&=F_{\tau\rho}(x_{12};w_{11},w_{22})\\
& =\big\langle {\rm e}^{(y_1|w_1)_{\fp^+_1}}I_W,
\big(h(Q(x_2)y_1,y_1)^{-\lambda/2}\rK\big((x_2)^{Q(y_1)x_2}\big)\big)^*\big\rangle_{\hat{\rho},y_1}\!\\
&=\bigl\langle {\rm e}^{(y_{11}|w_{11})_{\fp^+_{11}}+(y_{22}|w_{22})_{\fp^+_{22}}}I_W,\\
& \qquad \big(h(Q(x_{12})(y_{11}+y_{22}),y_{11}+y_{22})^{-\lambda/2}
\rK\big((x_{12})^{Q(y_{11}+y_{22})x_{12}}\big)\big)^*\bigr\rangle_{\hat{\rho},y_{11},y_{22}}.
\end{align*}
Now since we have $Q(x_{12})y_{11}\in\fp^+_{22}$, $Q(x_{12})y_{22}\in\fp^+_{11}$, it holds that
\begin{align}
h(Q(x_{12})(y_{11}+y_{22}),y_{11}+y_{22}) &=\chi(B(Q(x_{12})y_{11}+Q(x_{12})y_{22},y_{11}+y_{22})) \notag\\
& =\chi(B(Q(x_{12})y_{22},y_{11}))\chi(B(Q(x_{12})y_{11},y_{22})) \notag\\
& =h(Q(x_{12})y_{22},y_{11})h(Q(x_{12})y_{11},y_{22}) \notag\\
& =h(Q(x_{12})y_{11},y_{22})^2=h_{22}(Q(x_{12})y_{11},y_{22})^2,\label{generic_norm_square}
\end{align}
where we have used (\ref{Bergman_decomp}) at the 2nd equality, \cite[Part V, Propositions IV.3.4 and IV.3.5]{FKKLR} at the 4th equality,
and $\chi|_{K_{22}}=\chi_{22}$ at the last equality. Moreover we have the following.
\begin{Lemma}
\begin{gather*} (x_{12})^{Q(y_{11}+y_{22})x_{12}}=B(Q(x_{12})y_{11},y_{22})^{-1}x_{12}. \end{gather*}
\end{Lemma}
\begin{proof}By the definition of the quasi-inverse, we have
\begin{align*}
&(x_{12})^{Q(y_{11}+y_{22})x_{12}}=B(x_{12},Q(y_{11}+y_{22})x_{12})^{-1}(x_{12}-Q(x_{12})Q(y_{11}+y_{22})x_{12})\\
&=B(Q(x_{12})(y_{11}+y_{22}),y_{11}+y_{22})^{-1}(x_{12}-Q(x_{12})(Q(y_{11})+Q(y_{11},y_{22})+Q(y_{22}))x_{12})\\
&=B(Q(x_{12})y_{11}+Q(x_{12})y_{22},y_{11}+y_{22})^{-1}(x_{12}-Q(x_{12})Q(y_{11},y_{22})x_{12})\\
&=B(Q(x_{12})y_{11},y_{22})^{-1}B(Q(x_{12})y_{22},y_{11})^{-1}(x_{12}-Q(x_{12})Q(y_{11},y_{22})x_{12}),
\end{align*}
where we have used Lemma \ref{projlemma}(2), (\ref{Bergman_decomp}),
and $Q(x_{12})y_{jj}\in\fp^+_{3-j,3-j}$, $Q(y_{jj})x_{12}=0$ $(j=1,2)$, which follows from case-by-case analysis.
Thus it suffices to show
\begin{gather*} B(Q(x_{12})y_{22},y_{11})^{-1}(x_{12}-Q(x_{12})Q(y_{11},y_{22})x_{12})=x_{12}. \end{gather*}
This follows from
\begin{align*}
B(Q(x_{12})y_{22},y_{11})x_{12}&=x_{12}-D(Q(x_{12})y_{22},y_{11})x_{12}+Q(Q(x_{12})y_{22})Q(y_{11})x_{12}\\
&=x_{12}-Q(Q(x_{12})y_{22},x_{12})y_{11}+Q(Q(x_{12})y_{22})0\\
&=x_{12}-Q(x_{12})D(y_{22},x_{12})y_{11}\\
& =x_{12}-Q(x_{12})Q(y_{11},y_{22})x_{12},
\end{align*}
where we have used \cite[Part V, Proposition I.2.1 (J3.1$'$)]{FKKLR} at the 3rd equality.
\end{proof}

Therefore we have
\begin{gather*}
F_{\tau\rho}(x_{12};w_{11},w_{22}) =\bigl\langle {\rm e}^{(y_{11}|w_{11})_{\fp^+_{11}}}{\rm e}^{(y_{22}|w_{22})_{\fp^+_{22}}}I_W,\\
\hphantom{F_{\tau\rho}(x_{12};w_{11},w_{22}) =\bigl\langle}{} \big(h_{22}(Q(x_{12})y_{11},y_{22})^{-\lambda}
\rK\big(B(Q(x_{12})y_{11},y_{22})^{-1}x_{12}\big)\big)^*\bigr\rangle_{\hat{\rho},y_{11},y_{22}}.
\end{gather*}
Now we write $(\rho,W)=(\rho_{11}\boxtimes\rho_{22},W_{11}\otimes W_{22})$. Then we have
\begin{gather*}
F_{\tau\rho}(x_{12};w_{11},w_{22})
=\bigl\langle {\rm e}^{(y_{11}|w_{11})_{\fp^+_{11}}}{\rm e}^{(y_{22}|w_{22})_{\fp^+_{22}}}I_{W_{11}\boxtimes W_{22}},\\
\quad \big(h_{22}(Q(x_{12})y_{11},y_{22})^{-\lambda}
\rK\big(B(Q(x_{12})y_{11},y_{22})^{-1}x_{12}\big)\big)^*\bigr\rangle_{\hat{\rho}_{11}\boxtimes\hat{\rho}_{22},y_{11},y_{22}}\\
=\bigl\langle {\rm e}^{(y_{11}|w_{11})_{\fp^+_{11}}}{\rm e}^{(y_{22}|w_{22})_{\fp^+_{22}}}I_{W_{11}\boxtimes W_{22}},
\big(\rK(x_{12})\rho_{22}(B(Q(x_{12})y_{11},y_{22}))\big)^*\bigr\rangle_{\hat{\rho}_{11}\boxtimes\hat{\rho}_{22},y_{11},y_{22}}\\
=\rK(x_{12})\bigl\langle {\rm e}^{(y_{11}|w_{11})_{\fp^+_{11}}}{\rm e}^{(y_{22}|w_{22})_{\fp^+_{22}}}I_{W_{11}\boxtimes W_{22}},
I_{W_{11}}\boxtimes\rho_{22}(B(y_{22},Q(x_{12})y_{11}))\bigr\rangle_{\hat{\rho}_{11}\boxtimes\hat{\rho}_{22},y_{11},y_{22}}\\
=\rK(x_{12})\bigl(\bigl\langle {\rm e}^{(y_{11}|w_{11})_{\fp^+_{11}}}I_{W_{11}},
{\rm e}^{(w_{22}|Q(x_{12})y_{11})_{\fp^+_{22}}}I_{W_{11}}\bigr\rangle_{\hat{\rho}_{11},y_{11}}\boxtimes I_{W_{22}}\bigr)\\
=\rK(x_{12})\bigl(\bigl\langle {\rm e}^{(y_{11}|w_{11})_{\fp^+_{11}}}I_{W_{11}},
{\rm e}^{(y_{11}|Q(x_{12})w_{22})_{\fp^+_{11}}}I_{W_{11}}\bigr\rangle_{\hat{\rho}_{11},y_{11}}\boxtimes I_{W_{22}}\bigr),
\end{gather*}
where we have used (\ref{K-invariance}) at the 2nd equality, and the 4th equality holds since $\rho_{22}(B(y_{22},\allowbreak z_{22}))$
is the reproducing kernel of $\cH_{\rho_{22}}(D_{22},W_{22})$, where $z_{22}=Q(x_{12})y_{11}$.
In the following we omit~$\boxtimes I_{W_{22}}$.
Now we assume $s'\le s''$ when $d=1$, $q'\le s''$ when $d=2$, $2\le s'\le s''$ when $d=4$, and set $W=W_{11}\boxtimes W_{22}$ as
\begin{align*}
W&=\cP_{(\underbrace{\scriptstyle k+1,\ldots,k+1}_l,k,\ldots,k)}(M(s',s'';\BC))\otimes\chi^{-\lambda} \\
&\simeq \big(V_{\langle l\rangle}^{(s')\vee}\otimes\chi_{11}^{-\lambda-k}\big)
\boxtimes \Bigl(V_{(\underbrace{\scriptstyle k+1,\ldots,k+1}_l,\underbrace{\scriptstyle k,\ldots,k}_{s'-l},
\underbrace{\scriptstyle 0,\ldots,0}_{s''-s'})}^{(s'')\vee}
\otimes\chi_{22}^{-\lambda}\Bigr) & (d=1), \\
W&=\cP_{(k,\ldots,k)}(M(q',s'';\BC))\boxtimes\cP_\bl(M(q'',s';\BC))\otimes\chi^{-\lambda_1-\lambda_2} \\
&\simeq \big(\BC^{(q')}\boxtimes V_\bl^{(s')}\otimes\chi_{11}^{-(\lambda_1+k)-\lambda_2}\big)\boxtimes
\Bigl(V_\bl^{(q'')\vee}\boxtimes V_{(\underbrace{\scriptstyle k,\ldots,k}_{q'},0,\ldots,0)}^{(s'')}\otimes\chi_{22}^{-\lambda_1-\lambda_2}\Bigr)
& (d=2), \\
W&=\begin{cases}\cP_{(k+l,k\ldots,k)}(M(s',s'';\BC))\otimes\chi^{-\lambda} & (1) \\
\cP_{(k+l,\ldots,k+l,k)}(M(s',s'';\BC))\otimes\chi^{-\lambda} & (2) \end{cases} \\
&\simeq\begin{cases}\big(V_{(l,0,\ldots,0)}^{(s')\vee}\otimes\chi_{11}^{-\lambda-2k}\big)
\boxtimes \Bigl(V_{(\underbrace{\scriptstyle k+l,k,\ldots,k}_{s'},0,\ldots,0)}^{(s'')\vee}
\otimes\chi_{22}^{-\lambda}\Bigr) & (1) \\
\big(V_{(l,\ldots,l,0)}^{(s')\vee}\otimes\chi_{11}^{-\lambda-2k}\big)
\boxtimes \Bigl(V_{(\underbrace{\scriptstyle k+l,\ldots,k+l,k}_{s'},0,\ldots,0)}^{(s'')\vee}
\otimes\chi_{22}^{-\lambda}\Bigr) & (2) \end{cases} & \begin{pmatrix}d=4,\\ s'\ne 3\end{pmatrix}, \\
W&=\cP_{(k_1,k_2,k_3)}(M(3,s'';\BC))\otimes\chi^{-\lambda} \\
&\simeq\big(V_{(k_1,k_2,k_3)}^{(3)\vee}\otimes\chi_{11}^{-\lambda}\big)
\boxtimes \big(V_{(k_1,k_2,k_3,0,\ldots,0)}^{(s'')\vee}\otimes\chi_{22}^{-\lambda}\big) \hspace{-30pt}\\
&\simeq\big(V_{(0;-k_2-k_3,-k_1-k_3,-k_1-k_2)}^{(1,3)\vee}\otimes\chi_{11}^{-\lambda}\big)
\boxtimes \big(V_{(k_1,k_2,k_3,0,\ldots,0)}^{(s'')\vee}\otimes\chi_{22}^{-\lambda}\big)
& \begin{pmatrix}d=4,\\ s'=3\end{pmatrix}, \\
W&=\cP_{(k_1,k_2)}(\Skew(5,\BC))\otimes\chi^{-\lambda}\\
&\simeq\chi_{11}^{-\lambda-k_1-k_2}\boxtimes\big(V_{(0;-k_2,-k_2,-k_1,-k_1,-k_1-k_2)}^{(1,5)\vee}\otimes\chi_{22}^{-\lambda}\big)
& (d=6), \\
W&=\cP_{(k_1,k_2)}\big(M(1,2;\BO)^\BC\big)\otimes\chi^{-\lambda} \\
&\simeq\chi_{11}^{-\lambda-k_1-k_2}\boxtimes \Bigl(\chi_{22}^{-\lambda-\frac{k_1+k_2}{2}}
\otimes V_{\big(\frac{k_1+k_2}{2},\frac{k_1-k_2}{2},\frac{k_1-k_2}{2},\frac{k_1-k_2}{2},\frac{k_1-k_2}{2}\big)}^{[10]\vee}\Bigr) & (d=8),
\end{align*}
where $k\in\BZ_{\ge 0}$, and $l\in\{0,1,\ldots,s'-1\}$ if $d=1$, $\bl\in\BZ_{++}^{\min\{ q'',s'\}}$ if $d=2$,
$l\in\BZ_{\ge 0}$ if $d=4$ with $s'\ne 3$, $(k_1,k_2,k_3)\in\BZ_{++}^3$ if $d=4$ with $s'=3$,
$(k_1,k_2)\in\BZ_{++}^2$ if $d=6,8$, and we denote $\langle l\rangle:=(\underbrace{1,\ldots,1}_l,0,\ldots,0)$.
We note that when $d=4$, $s'=3$, we identify $\operatorname{SO}^*(6)\simeq \operatorname{SU}(1,3)$ up to covering. We write
\begin{gather*} \rK(x_{12})=\begin{cases}
\rK_{k\langle s'\rangle+\langle l\rangle}^{(2)}(x_{12}) & (d=1),\\
\rK_{(k,\ldots,k)}^{(2)}(x_{12})\rK_\bl^{(2)}(x_{21}) & (d=2),\\
\rK_{(k+l,k,\ldots,k)}^{(2)}(x_{12}) & (d=4,\,(1)),\\
\rK_{(k+l,\ldots,k+l,k)}^{(2)}(x_{12}) & (d=4,\,(2))\\
\rK_{(k_1,k_2,k_3)}^{(2)}(x_{12}) & (d=4,\,s'=3),\\
\rK_{(k_1,k_2)}^{(4)}(x_{12}) & (d=6),\\
\rK_{(k_1,k_2)}^{(6)}(x_{12}) & (d=8),
\end{cases}\quad\in \cP\big(\fp^+_{12},\Hom\big(W,\chi^{-\lambda}\big)\big), \end{gather*}
where the superscripts mean $d_2=d_{12}$. Then $\cP(\fp^+_{11})\otimes W_{11}$ is decomposed under $\tilde{K}_{11}^\BC$ as
\begin{align*}
\cP(\fp^+_{11})\otimes W_{11}
&\simeq\bigoplus_{\bm\in\BZ_{++}^{s'}} V_{2\bm}^{(s')\vee}\otimes V_{\langle l\rangle}^{(s')\vee}\otimes\chi_{11}^{-\lambda-k}\\
& \simeq\bigoplus_{\bm\in\BZ_{++}^{s'}}\bigoplus_{\substack{\bl\in\{0,1\}^{s'},\; |\bl|=l\\ \bm+\bl\in\BZ_{++}^{s'}}}
V_{2\bm+\bl}^{(s')\vee}\otimes \chi_{11}^{-\lambda-k} && (d=1),\\
\cP(\fp^+_{11})\otimes W_{11}
&\simeq\bigoplus_{\bm\in\BZ_{++}^{\min\{q',s'\}}} \big(V_{\bm}^{(q')\vee}\boxtimes V_{\bm}^{(s')}\big)
\otimes \big(\BC^{(q')}\boxtimes V_\bl^{(s')}\otimes\chi_{11}^{-(\lambda_1+k)-\lambda_2}\big) \\
&\simeq\bigoplus_{\bm\in\BZ_{++}^{\min\{q',s'\}}}\bigoplus_{\bn\in\BZ_{++}^{s'}}
c_{\bm,\bl}^\bn V_{\bm}^{(q')\vee}\boxtimes V_\bn^{(s')}\otimes\chi_{11}^{-(\lambda_1+k)-\lambda_2} &&(d=2),\\
\cP(\fp^+_{11})\otimes W_{11}
&\simeq\bigoplus_{\bm\in\BZ_{++}^{\lfloor s'/2\rfloor}} V_{\bm^2}^{(s')\vee}\otimes V_{(l,0,\ldots,0)}^{(s')\vee}\otimes\chi_{11}^{-\lambda-2k}\\
&\simeq\bigoplus_{\bm\in\BZ_{++}^{\lfloor s'/2\rfloor}}
\bigoplus_{\substack{\bl\in(\BZ_{\ge 0})^{\lceil s'/2\rceil},\; |\bl|=l\\ 0\le l_j\le m_{j-1}-m_j}}
V_{\substack{(m_1+l_1,m_1,m_2+l_2,\\
m_2,\ldots,m_{\lfloor s'/2\rfloor}+l_{\lfloor s'/2\rfloor},\\ m_{\lfloor s'/2\rfloor}(,l_{\lceil s'/2\rceil}))}}^{(s')\vee}
\otimes\chi_{11}^{-\lambda-2k} &&(d=4\ (1)), \\
\cP(\fp^+_{11})\otimes W_{11}
&\simeq\bigoplus_{\bm\in\BZ_{++}^{\lfloor s'/2\rfloor}} V_{\bm^2}^{(s')\vee}\otimes V_{(l,\ldots,l,0)}^{(s')\vee}
\otimes \chi_{11}^{-\lambda-2k}\\
&\simeq\bigoplus_{\bm\in\BZ_{++}^{\lfloor s'/2\rfloor}}
\bigoplus_{\substack{\bl\in(\BZ_{\ge 0})^{\lceil s'/2\rceil},\; |\bl|=l\\ 0\le l_j\le m_j-m_{j+1}\\
0\le l_{\lceil s'/2\rceil} \text{ if }s'\colon \text{odd}}}
V_{\substack{(m_1+l,m_1+l-l_1,m_2+l,\\
m_2+l-l_2,\ldots, \\ m_{\lfloor s'/2\rfloor}+l,\\ m_{\lfloor s'/2\rfloor}+l-l_{\lfloor s'/2\rfloor}
(,l-l_{\lceil s'/2\rceil}))}}^{(s')\vee}
\otimes \chi_{11}^{-\lambda-2k}\! &&(d=4\ (2)), \\
\cP(\fp^+_{11})\otimes W_{11}
&\simeq\bigoplus_{m=0}^\infty V_{(m,m,0)}^{(3)\vee}\otimes V_{(k_1,k_2,k_3)}^{(3)\vee}\otimes\chi_{11}^{-\lambda}\\
&\simeq\bigoplus_{m=0}^\infty \;\bigoplus_{\substack{\bm\in(\BZ_{\ge 0})^3,\;|\bm|=m \\ 0\le m_j\le k_j-k_{j+1}}}
V_{(k_1+m-m_1,k_2+m-m_2,k_3+m-m_3)}^{(3)\vee}\otimes\chi_{11}^{-\lambda}\\
&\simeq\bigoplus_{\substack{\bm\in(\BZ_{\ge 0})^3\\ 0\le m_j\le k_j-k_{j+1}}}
V_{(k_1+m_2+m_3,k_2+m_1+m_3,k_3+m_1+m_2)}^{(3)\vee}\otimes\chi_{11}^{-\lambda} &&\begin{pmatrix}d=4\\ s'=3\end{pmatrix},\\
\cP(\fp^+_{11})\otimes W_{11}&\simeq \bigoplus_{m=0}^\infty \chi_{11}^{-\lambda-k_1-k_2-2m} &&(d=6,8),
\end{align*}
where $c_{\bm,\bl}^\bn$ are the Littlewood--Richardson coefficients, and for $d=4$ we write $m_0=+\infty$, $m_{\lfloor s'/2\rfloor+1}=0$.
Also, when $d=4$, $s'=3$, under the identification $\operatorname{SO}^*(6)\simeq \operatorname{SU}(1,3)$ up to covering, we have
\begin{gather*}
 V_{(k_1+m_2+m_3,k_2+m_1+m_3, k_3+m_1+m_2)}^{(3)\vee} \\
 \qquad{} \simeq V_{(0;-k_2-k_3-2m_1-m_2-m_3, -k_1-k_3-m_1-2m_2-m_3, -k_1-k_2-m_1-m_2-2m_3)}^{(1,3)\vee} \\
\qquad{} \simeq V_{(m_1+m_2+m_3;-k_2-k_3-m_1, -k_1-k_3-m_2,-k_1-k_2-m_3)}^{(1,3)\vee}.
\end{gather*}
Then we expand ${\rm e}^{(z_{11}|y_{11})_{\fp^+_{11}}}I_{W_{11}} \in \cO\big(\fp^+_{11}\times\overline{\fp^+_{11}},\End(W_{11})\big)$ according to the above decomposition~as
\begin{align*}
{\rm e}^{(z_{11}|y_{11})_{\fp^+_{11}}}I_{W_{11}}=\begin{cases}
\ds \sum_{\bm,\bl}\bK^{(1)}_{\bm,\bl}(z_{11};y_{11}) & (d=1),\\
\ds \sum_{\bm,\bn}\bK^{(2)}_{\bm,\bn}(z_{11};y_{11}) & (d=2),\\
\ds \sum_{\bm,\bl}\bK^{(4)}_{\bm,\bl}(z_{11};y_{11}) & (d=4\ (1)),\\
\ds \sum_{\bm,\bl}\bK^{(4)}_{\bm,-\bl}(z_{11};y_{11}) & (d=4\ (2)),\\
\ds \sum_{\bm}\bK^{(2)\prime}_{\bm,\bk}(z_{11};y_{11}) & (d=4,\, s'=3),\\
\ds \sum_{m=0}^\infty \frac{1}{m!}(z_{11}\overline{y_{11}})^m & (d=6,8). \end{cases}
\end{align*}
Therefore by \cite{N2} and (\ref{exp_onD}), we have
\begin{align*}
&F_{\tau\rho}(x_{12};w_{11},w_{22})=\rK(x_{12})\bigl\langle {\rm e}^{(y_{11}|w_{11})_{\fp^+_{11}}}I_{W_{11}},
{\rm e}^{(y_{11}|Q(x_{12})w_{22})_{\fp^+_{11}}}I_{W_{11}}\bigr\rangle_{\hat{\rho}_{11},y_{11}}\\
&=\begin{cases}
\ds \sum_{\bm\in\BZ_{++}^{s'}}\sum_{\substack{\bl\in\{0,1\}^{s'},\; |\bl|=l\\ \bm+\bl\in\BZ_{++}^{s'}}}
\frac{1}{(\lambda+k+\langle l\rangle)_{\bm+\bl-\langle l\rangle,1}}\\
\hspace{100pt}\times \rK_{k\langle s'\rangle+\langle l\rangle}^{(2)}(x_{12})
\bK_{\bm,\bl}^{(1)}(x_{12}w_{22}^*{}^t\hspace{-1pt}x_{12};w_{11}) & (d=1), \\
\ds \sum_{\bm\in\BZ_{++}^{\min\{q',s'\}}}\sum_\bn \frac{1}{(\lambda+k+\bl)_{\bn-\bl,2}}\\
\hspace{100pt}\times \rK_{(k,\ldots,k)}^{(2)}(x_{12})\rK_\bl^{(2)}(x_{21})\bK_{\bm,\bn}^{(2)}(x_{12}w_{22}^*x_{21};w_{11}) & (d=2), \\
\ds \sum_{\bm\in\BZ_{++}^{\lfloor s'/2\rfloor}}
\sum_{\substack{\bl\in(\BZ_{\ge 0})^{\lceil s'/2\rceil},\; |\bl|=l\\ 0\le l_j\le m_{j-1}-m_j}}
\frac{1}{(\lambda+2k+(l,0,\ldots,0))_{\bm+\bl-(l,0,\ldots,0),4}},\\
\hspace{100pt}\times \rK_{(k+l,k,\ldots,k)}^{(2)}(x_{12})
\bK_{\bm,\bl}^{(4)}(-x_{12}w_{22}^*{}^t\hspace{-1pt}x_{12};w_{11}) & (d=4,\,(1)),\\
\ds \sum_{\bm\in\BZ_{++}^{\lfloor s'/2\rfloor}}
\sum_{\substack{\bl\in(\BZ_{\ge 0})^{\lceil s'/2\rceil},\; |\bl|=l\\ 0\le l_j\le m_j-m_{j+1}}}
\frac{1}{(\lambda+2k+l+(l,\ldots,l,0))_{\bm-\bl-(0,\ldots,0,-l),4}},\\
\hspace{100pt}\times \rK_{(k+l,\ldots,k+l,k)}^{(2)}(x_{12})
\bK_{\bm,-\bl}^{(4)}(-x_{12}w_{22}^*{}^t\hspace{-1pt}x_{12};w_{11})
& \!\!\begin{pmatrix}d=4,\,(2)\\s\colon \text{even}\end{pmatrix},
\end{cases}\\
& \hphantom{=} \begin{cases}
\ds \sum_{\bm\in\BZ_{++}^{\lfloor s'/2\rfloor}}
\sum_{\substack{\bl\in(\BZ_{\ge 0})^{\lceil s'/2\rceil},\; |\bl|=l\\ 0\le l_j\le m_j-m_{j+1}\\ 0\le l_{\lceil s'/2\rceil}}}
\frac{1}{(\lambda+2k+2l)_{\bm-\bl',4}\left(\lambda+2k+l-2\left\lfloor\frac{s}{2}\right\rfloor+1\right)_{l-l_{\lceil s/2\rceil}}}\hspace{-30pt}\\
\hspace{100pt}\times \rK_{(k+l,\ldots,k+l,k)}^{(2)}(x_{12})
\bK_{\bm,-\bl}^{(4)}(-x_{12}w_{22}^*{}^t\hspace{-1pt}x_{12};w_{11})
& \!\!\begin{pmatrix}d=4,\,(2)\\s\colon \text{odd}\end{pmatrix},\\
\ds \sum_{\substack{\bm\in(\BZ_{\ge 0})^3\\ 0\le m_j\le k_j-k_{j+1}}}
\frac{1}{(\lambda+(k_1+k_2,k_1+k_3,k_2+k_3))_{(m_3,m_2,m_1),2}} \\
\hspace{100pt}\times\rK_{(k_1,k_2,k_3)}^{(2)}(x_{12})\bK_{\bm,\bk}^{(2)\prime}(-x_{12}w_{22}^*{}^t\hspace{-1pt}x_{12};w_{11})
& \!\!\begin{pmatrix}d=4\\ s'=3\end{pmatrix},\\
\ds \sum_{m=0}^\infty \frac{1}{(\lambda+k_1+k_2)_m}\frac{1}{m!}\left(\overline{w_{11}}\overline{w_{22}}
\hspace{1pt}{}^t\mathbf{Pf}(x_{12})\right)^m\rK_{(k_1,k_2)}^{(4)}(x_{12}) & (d=6),\\
\ds \sum_{m=0}^\infty \frac{1}{(\lambda+k_1+k_2)_m}\frac{1}{m!}\left(\overline{w_{11}}
\Re_\BO\left(x_{12}\left(\overline{w_{22}}\hspace{1pt}{}^t\hspace{-1pt}\hat{x}_{12}\right)\right)\right)^m\rK_{(k_1,k_2)}^{(6)}(x_{12}) & (d=8),
\end{cases}
\end{align*}
where for $d=4$ with $s'$ odd, we identify $\bm=(m_1,\ldots,m_{\lfloor s'/2\rfloor})\in\BZ_{++}^{\lfloor s'/2\rfloor}$ with $(m_1,\ldots,\allowbreak m_{\lfloor s'/2\rfloor},0)\in\BZ_{++}^{\lceil s'/2\rceil}$, and for $\bl=(l_1,\ldots,l_{\lfloor s'/2\rfloor},l_{\lceil s'/2\rceil})\in(\BZ_{\ge 0})^{\lceil s'/2\rceil}$, we write $\bl'=(l_1,\ldots,l_{\lfloor s'/2\rfloor})\allowbreak\in(\BZ_{\ge 0})^{\lfloor s'/2\rfloor}$. Also, for $d=6$, $\mathbf{Pf}(x_{12})$ is defined in~(\ref{Pfaff_vector}). By Theorem~\ref{main}, by substitu\-ting $w_{11}$, $w_{22}$ with $\overline{\frac{\partial}{\partial x_{11}}}$, $\overline{\frac{\partial}{\partial x_{22}}}$, we get the intertwining operator from $(\cH_1)_{\tilde{K}_1}$ to $\cH_{\tilde{K}}$, and by Theorem~\ref{extend}, this extends to the intertwining operator between the spaces of all holomorphic functions if $\cH_1$ is holomorphic discrete. Moreover, for the cases when the norm of $\cH_1$ is computed in \cite{N2}, then by Theorem~\ref{continuation}, this continues meromorphically for all $\lambda\in\BC$. Therefore we have the following.
\begin{Theorem}\label{main1}\quad
\begin{enumerate}\itemsep=0pt
\item[$(1)$] Let $(G,G_1)=(\operatorname{Sp}(s,\BR), \operatorname{Sp}(s',\BR)\times \operatorname{Sp}(s'',\BR))$ with $s=s'+s''$, $s'\le s''$.
Let $k\in\BZ_{\ge 0}$, and $l\in\{0,\ldots,s'-1\}$. Then the linear map
\begin{gather*}
\cF_{\lambda,k,l}\colon \ \cO_{\lambda+k}\big(D_{11},V_{\langle l\rangle}^{(s')\vee}\big)_{\tilde{K}_{11}}\boxtimes
\cO_{\lambda}(D_{22},V_{(\underbrace{\scriptstyle k+1,\ldots,k+1}_l,\underbrace{\scriptstyle k,\ldots,k}_{s'-l},
\underbrace{\scriptstyle 0,\ldots,0}_{s''-s'})}^{(s'')\vee})_{\tilde{K}_{22}}
\longrightarrow \cO_\lambda(D)_{\tilde{K}},\\
\cF_{\lambda,k,l}f\begin{pmatrix}x_{11}&x_{12}\\{}^t\hspace{-1pt}x_{12}&x_{22}\end{pmatrix}
=\sum_{\bm\in\BZ_{++}^{s'}}\sum_{\substack{\bl\in\{0,1\}^{s'},\; |\bl|=l\\ \bm+\bl\in\BZ_{++}^{s'}}}
\frac{1}{(\lambda+k+\langle l\rangle)_{\bm+\bl-\langle l\rangle,1}}\\
\hphantom{\cF_{\lambda,k,l}f\begin{pmatrix}x_{11}&x_{12}\\{}^t\hspace{-1pt}x_{12}&x_{22}\end{pmatrix}=}{}
\times \rK_{k\langle s'\rangle+\langle l\rangle}^{(2)}(x_{12})
\bK_{\bm,\bl}^{(1)}\left(x_{12}\frac{\partial}{\partial x_{22}}{}^t\hspace{-1pt}x_{12};
\overline{\frac{\partial}{\partial x_{11}}}\right)f(x_{11},x_{22})
\end{gather*}
intertwines the $\big(\fg_1,\tilde{K}_1\big)$-action. If $s'=s''$ or $k=0$ or $(k,l)=(1,0)$ or $\lambda>s''$,
then this extends to the map between the spaces of all holomorphic functions.
\item[$(2)$] Let $(G,G_1)=(U(q,s), U(q',s')\times U(q'',s''))$ with $q=q'+q''$, $s=s'+s''$, $q'\le s''$.
Let $k\in\BZ_{\ge 0}$, and $\bl\in\BZ_{++}^{\min\{q'',s'\}}$.
Then the linear map
\begin{gather*}
\cF_{\lambda,k,\bl}\colon \ \cO_{(\lambda_1+k)+\lambda_2}\big(D_{11},\BC\boxtimes V_\bl^{(s')}\big)_{\tilde{K}_{11}}\\
\qquad {} \boxtimes
\cO_{\lambda_1+\lambda_2}\big(D_{22},V_\bl^{(q'')\vee} \boxtimes V_{(\underbrace{\scriptstyle k,\ldots,k}_{q'},0,\ldots,0)}^{(s'')}\big)_{\tilde{K}_{22}}
\longrightarrow \cO_{\lambda_1+\lambda_2}(D)_{\tilde{K}},\\
\cF_{\lambda,k,\bl}f\begin{pmatrix}x_{11}&x_{12}\\x_{21}&x_{22}\end{pmatrix}
=\sum_{\bm\in\BZ_{++}^{\min\{q',s'\}}}\sum_\bn \frac{1}{(\lambda_1+\lambda_2+k+\bl)_{\bn-\bl,2}}\\
\qquad{} \times \rK_{(k,\ldots,k)}^{(2)}(x_{12})\rK_\bl^{(2)}(x_{21})\bK_{\bm,\bn}^{(2)}
\left(x_{12}{\vphantom{\biggl(}}^t\!\!\left(\frac{\partial}{\partial x_{22}}\right)x_{21};
\overline{\frac{\partial}{\partial x_{11}}}\right)f(x_{11},x_{22})
\end{gather*}
intertwines the $(\fg_1,\tilde{K}_1)$-action. If $q'=s''$ or $k=0$ or $\bl=(0,\ldots,0)$ or ``$s'\ge q''$ and $\bl=(l,\ldots,l)$''
or ``$\lambda_1+\lambda_2+k+l_{s'}>q'+s'-1$ and $\lambda_1+\lambda_2+l_{q''}>q''+s''-1$'',
then this extends to the map between the spaces of all holomorphic functions.
\item[$(3)$] Let $(G,G_1)=(\operatorname{SO}^*(2s), \operatorname{SO}^*(2s')\times \operatorname{SO}^*(2s''))$ with $s=s'+s''$, $2\le s'\le s''$.
Let $k,l\in\BZ_{\ge 0}$. Then the linear map
\begin{gather*}
\cF_{\lambda,k,l}\colon \ \cO_{\lambda+2k}\big(D_{11},V_{(l,0,\ldots,0)}^{(s')\vee}\big)_{\tilde{K}_{11}}\boxtimes
\cO_\lambda\big(D_{22},V_{(\underbrace{\scriptstyle k+l,k,\ldots,k}_{s'},0,\ldots,0)}^{(s'')\vee}\big)_{\tilde{K}_{22}}
\longrightarrow \cO_\lambda(D)_{\tilde{K}},\\
\cF_{\lambda,k,l}f\begin{pmatrix}x_{11}&x_{12}\\-{}^t\hspace{-1pt}x_{12}&x_{22}\end{pmatrix}
=\sum_{\bm\in\BZ_{++}^{\lfloor s'/2\rfloor}}\sum_{\substack{\bl\in(\BZ_{\ge 0})^{\lceil s'/2\rceil},\; |\bl|=l\\ 0\le l_j\le m_{j-1}-m_j}}
\frac{1}{(\lambda\!+\!2k\!+\!(l,0,\ldots,0))_{\bm+\bl-(l,0,\ldots,0),4}}\\
\hphantom{\cF_{\lambda,k,l}f\begin{pmatrix}x_{11}&x_{12}\\-{}^t\hspace{-1pt}x_{12}&x_{22}\end{pmatrix}=}{}
\times \rK_{(k+l,k,\ldots,k)}^{(2)}(x_{12})
\bK_{\bm,\bl}^{(4)}\left(x_{12}\frac{\partial}{\partial x_{22}}{}^t\hspace{-1pt}x_{12};
\overline{\frac{\partial}{\partial x_{11}}}\right)f(x_{11},x_{22})
\end{gather*}
intertwines the $(\fg_1,\tilde{K}_1)$-action. Here we identify $\bm=(m_1,\ldots,m_{\lfloor s'/2\rfloor})\in\BZ_{++}^{\lfloor s'/2\rfloor}$
with $(m_1,\ldots,m_{\lfloor s'/2\rfloor},0)\in\BZ_{++}^{\lceil s'/2\rceil}$ when $s'$ is odd.
If $s'=s''$ or $k=0$ or ``$s''=s'+1$ and $l=0$''
or $\lambda>2s''-3$ or ``$s''=s'+1$ and $\lambda+k>2s''-3$'',
then this extends to the map between the spaces of all holomorphic functions.
\item[$(4)$] Let $(G,G_1)=(\operatorname{SO}^*(2s), \operatorname{SO}^*(2s')\times \operatorname{SO}^*(2s''))$ with $s=s'+s''$, $3\le s'\le s''$.
Let $k\in\BZ_{\ge 0}$, and $l\in\BZ_{>0}$. Then the linear maps
\begin{gather*}
\cF_{\lambda,k,l}\colon \ \cO_{\lambda+2k}\big(D_{11},V_{(l,\ldots,l,0)}^{(s')\vee}\big)_{\tilde{K}_{11}}\boxtimes
\cO_\lambda\big(D_{22},V_{(\underbrace{\scriptstyle k+l,\ldots,k+l,k}_{s'},0,\ldots,0)}^{(s'')\vee}\big)_{\tilde{K}_{22}}
\longrightarrow \cO_\lambda(D)_{\tilde{K}}, \\
\cF_{\lambda,k,l}f\begin{pmatrix}x_{11}&x_{12}\\-{}^t\hspace{-1pt}x_{12}&x_{22}\end{pmatrix} \\
\qquad{} =\sum_{\bm\in\BZ_{++}^{\lfloor s'/2\rfloor}}\sum_{\substack{\bl\in(\BZ_{\ge 0})^{\lceil s'/2\rceil},\; |\bl|=l\\ 0\le l_j\le m_j-m_{j+1}}}
\frac{1}{(\lambda+2k+l+(l,\ldots,l,0))_{\bm-\bl-(0,\ldots,0,-l),4}}\\
\qquad\quad{} \times \rK_{(k+l,\ldots,k+l,k)}^{(2)}(x_{12})
\bK_{\bm,-\bl}^{(4)}\left(x_{12}\frac{\partial}{\partial x_{22}}{}^t\hspace{-1pt}x_{12};
\overline{\frac{\partial}{\partial x_{11}}}\right)f(x_{11},x_{22}) \quad (s'\colon \text{even}),\\
\cF_{\lambda,k,l}f\begin{pmatrix}x_{11}&x_{12}\\-{}^t\hspace{-1pt}x_{12}&x_{22}\end{pmatrix} \\
\qquad{} =\sum_{\bm\in\BZ_{++}^{\lfloor s'/2\rfloor}}
\sum_{\substack{\bl\in(\BZ_{\ge 0})^{\lceil s'/2\rceil},\; |\bl|=l\\ 0\le l_j\le m_j-m_{j+1}\\ 0\le l_{\lceil s'/2\rceil}}}
\frac{1}{(\lambda+2k+2l)_{\bm-\bl',4}\left(\lambda+2k+l-2\left\lfloor\frac{s}{2}\right\rfloor+1\right)_{l-l_{\lceil s'/2\rceil}}}\\
\qquad\quad{} \times \rK_{(k+l,\ldots,k+l,k)}^{(2)}(x_{12})
\bK_{\bm,-\bl}^{(4)}\left(x_{12}\frac{\partial}{\partial x_{22}}{}^t\hspace{-1pt}x_{12};
\overline{\frac{\partial}{\partial x_{11}}}\right)f(x_{11},x_{22}) \quad (s'\colon \text{odd})
\end{gather*}
intertwine the $\big(\fg_1,\tilde{K}_1\big)$-action. Here if $s'$ is odd, then for $\bl=(l_1,\ldots,l_{\lfloor s'/2\rfloor},l_{\lceil s'/2\rceil})\in(\BZ_{\ge 0})^{\lceil s'/2\rceil}$,
we write $\bl'=(l_1,\ldots,l_{\lfloor s'/2\rfloor})\in(\BZ_{\ge 0})^{\lfloor s'/2\rfloor}$.
If $s'=s''$ or $\lambda>2s''-3$ or ``$s''=s'+1$ and $\lambda+k>2s''-3$'',
then this extends to the map between the spaces of all holomorphic functions.
\item[$(5)$] Let $(G,G_1)=(\operatorname{SO}^*(2(3+s'')), \operatorname{SO}^*(6)\times \operatorname{SO}^*(2s''))$ with $3\le s''$.
Let $(k_1,k_2,k_3)\in\BZ_{++}^3$. Then the linear maps
\begin{gather*}
\cF_{\lambda,k_1,k_2,k_3}\colon \ \cO_\lambda\big(D_{11},V_{(k_1,k_2,k_3)}^{(3)\vee}\big)_{\tilde{K}_{11}}\boxtimes
\cO_\lambda\big(D_{22},V_{(k_1,k_2,k_3,0,\ldots,0)}^{(s'')\vee}\big)_{\tilde{K}_{22}} \longrightarrow \cO_\lambda(D)_{\tilde{K}},\\
\cF_{\lambda,k_1,k_2,k_3}f\begin{pmatrix}x_{11}&x_{12}\\-{}^t\hspace{-1pt}x_{12}&x_{22}\end{pmatrix}
= \sum_{\substack{\bm\in(\BZ_{\ge 0})^3\\ 0\le m_j\le k_j-k_{j+1}}}\!\!\!
\frac{1}{(\lambda+(k_1+k_2,k_1+k_3,k_2+k_3))_{(m_3,m_2,m_1),2}} \\
\hphantom{\cF_{\lambda,k_1,k_2,k_3}f\begin{pmatrix}x_{11}&x_{12}\\-{}^t\hspace{-1pt}x_{12}&x_{22}\end{pmatrix}=}{} \times\rK_{(k_1,k_2,k_3)}^{(2)}(x_{12})\bK_{\bm,\bk}^{(2)\prime}
\left(x_{12}\frac{\partial}{\partial x_{22}}{}^t\hspace{-1pt}x_{12};
\overline{\frac{\partial}{\partial x_{11}}}\right)f(x_{11},x_{22})
\end{gather*}
intertwine the $\big(\fg_1,\tilde{K}_1\big)$-action.
If $s''=3$ or $k_2=k_3=0$ or ``$s''=4$ and $k_1=k_2=k_3$'' or $\lambda>2s''-3$ or ``$s''=4$ and $\lambda+k_3>2s''-3$'',
then this extends to the map between the spaces of all holomorphic functions.
\item[$(6)$] Let $(G,G_1)=(E_{6(-14)}, {\rm SL}(2,\BR)\times \operatorname{SU}(1,5))$ $($up to covering$)$.
Let $(k_1,k_2)\in\BZ_{++}^2$. Then the linear map
\begin{gather*}
\cF_{\lambda,k_1,k_2}\colon \ \cO_{\lambda+k_1+k_2}(D_{11})\hboxtimes
\cO_\lambda\big(D_{22},V_{(0;-k_2,-k_2,-k_1,-k_1,-k_1-k_2)}^{(1,5)\vee}\big) \longrightarrow \cO_\lambda(D),\\
\cF_{\lambda,k_1,k_2}f(x_{11},x_{12},x_{22})\\
\qquad{} =\sum_{m=0}^\infty \frac{1}{(\lambda+k_1+k_2)_m}\frac{1}{m!}\left(\frac{\partial}{\partial x_{11}}
\frac{\partial}{\partial x_{22}}{}^t\mathbf{Pf}(x_{12})\right)^m\rK_{(k_1,k_2)}^{(4)}(x_{12})f(x_{11},x_{22})
\end{gather*}
$(x_{11}\in\BC,\; x_{12}\in\Skew(5,\BC),\; x_{22}\in M(1,5;\BC))$ intertwines the $\tilde{G}_1$-action.
\item[$(7)$] Let $(G,G_1)=(E_{7(-25)}, {\rm SL}(2,\BR)\times \operatorname{Spin}_0(2,10))$ $($up to covering$)$.
Let $(k_1,k_2)\in\BZ_{++}^2$. Then the linear map
\begin{gather*}
\cF_{\lambda,k_1,k_2}\colon \ \cO_{\lambda+k_1+k_2}(D_{11})_{\tilde{K}_{11}}\\
\qquad{} \boxtimes\cO_{\lambda+\frac{k_1+k_2}{2}}\Bigl(D_{22},
V_{\left(\frac{k_1+k_2}{2},\frac{k_1-k_2}{2},\frac{k_1-k_2}{2},\frac{k_1-k_2}{2},\frac{k_1-k_2}{2}\right)}^{[10]\vee}\Bigr)_{\tilde{K}_{22}} \longrightarrow \cO_\lambda(D)_{\tilde{K}},\\
\cF_{\lambda,k_1,k_2}f\begin{pmatrix}x_{11}&x_{12}\\{}^t\hspace{-1pt}\hat{x}_{12}&x_{22}\end{pmatrix}
=\sum_{m=0}^\infty \frac{1}{(\lambda+k_1+k_2)_m} \\
\qquad {} \times\frac{1}{m!}\left(\frac{\partial}{\partial x_{11}}
\Re_\BO\left(x_{12}\left(\frac{\partial}{\partial x_{22}}{}^t\hspace{-1pt}\hat{x}_{12}\right)\right)\right)^m
\rK_{(k_1,k_2)}^{(6)}(x_{12})f(x_{11},x_{22})
\end{gather*}
intertwines the $\big(\fg_1,\tilde{K}_1\big)$-action. If $k_2=0$ or $\lambda>9$,
then this extends to the map between the spaces of all holomorphic functions.
\end{enumerate}
\end{Theorem}
When $d=4$ and $s'=3$, for $(k_1,k_2,k_3)=(k+l,k,k)$ and $m,l_1,l_2\in\BZ_{\ge 0}$, $l_1+l_2=l$ we have
\begin{gather*}
 \frac{1}{(\lambda+2k+(l,0))_{(m+l_1-l,l_2),4}}\bK_{m,(l_1,l_2)}^{(4)}(z_{11};y_{11}) \\
\qquad{} =\frac{1}{(\lambda+(2k+l,2k+l,2k))_{(m+l_1-l,0,l_2),2}}\bK_{(l_2,0,m+l_1-l),(k+l,k,k)}^{(2)\prime}(z_{11};y_{11}) \\
\qquad\quad {} \in V_{(m+l_1,m,l_2)}^{(3)}\otimes \overline{V_{(m+l_1,m,l_2)}^{(3)}}
\simeq V_{(k+m+l_1,k+m,k+l_2)}^{(3)}\otimes \overline{V_{(k+m+l_1,k+m,k+l_2)}^{(3)}},
\end{gather*}
and therefore the results in (3) and (5) coincide. Similarly, for $(k_1,k_2,k_3)=(k+l,k+l,k)$ and $m,l_1,l_2\in\BZ_{\ge 0}$, $l_1+l_2=l$ we have
\begin{gather*}
 \frac{1}{(\lambda+2k+2l)_{m-l_1}(\lambda+2k+l-1)_{l-l_2}}\bK_{m,-(l_1,l_2)}^{(4)}(z_{11};y_{11}) \\
\qquad{} =\frac{1}{(\lambda+(2k+2l,2k+l,2k+l))_{(m-l_1,l-l_2,0),2}}\bK_{(0,l-l_2,m-l_1),(k+l,k+l,k)}^{(2)\prime}(z_{11};y_{11}) \\
\qquad\quad{} \in V_{(m+l,m+l-l_1,l-l_2)}^{(3)}\otimes \overline{V_{(m+l,m+l-l_1,l-l_2)}^{(3)}} \\
\qquad{} \simeq V_{(k+m+l,k+m+l-l_1,k+l-l_2)}^{(3)}\otimes \overline{V_{(k+m+l,k+m+l-l_1,k+l-l_2)}^{(3)}},
\end{gather*}
and therefore the results in (4) and (5) coincide.

When $d=1,4$, if ``$s'=s''$, $l=0$'' or ``$k=l=0$'', then we have $\rK(x_{12})=\det(x_{12})^k$,
and when $d=2$, if $\bl=(l,\ldots,l)$ and ``$q'=s''$ or $k=0$'' and ``$q''=s'$ or $l=0$'',
then we have $\rK(x_{12},x_{21})=\det(x_{12})^k\det(x_{21})^l$, and in these cases ${\rm e}^{(z_{11}|y_{11})_{\fp^+_{11}}}$ is expanded as
\begin{gather*} {\rm e}^{(z_{11}|y_{11})_{\fp^+_{11}}}=\sum_{\bm}\bK_\bm^{(d)}(z_{11},y_{11})=\sum_{\bm}\tilde{\Phi}_\bm^{(d)}(z_{11}y_{11}^*), \end{gather*}
where $\tilde{\Phi}_\bm^{(d)}(z_{11}y_{11}^*)$ is defined in (\ref{Phitilde1}), (\ref{Phitilde2}).
Therefore, when $(G,G_1)=(\operatorname{Sp}(s,\BR), \operatorname{Sp}(s',\BR)\times \operatorname{Sp}(s'',\BR))$, the intertwining operator is reduced to
\begin{gather}
\cF_{\lambda,k}\colon \ \cO_{\lambda+k}(D_{11}) \hboxtimes \cO_{\lambda+k}(D_{22})\to \cO_{\lambda}(D), \label{Sp-SpSp} \\
(\cF_{\lambda,k}f)\begin{pmatrix}x_{11}&x_{12}\\{}^t\hspace{-1pt}x_{12}&x_{22}\end{pmatrix}
=\det(x_{12})^k\sum_{\bm\in\BZ_{++}^{s'}}\frac{1}{(\lambda+k)_{\bm,1}}
\tilde{\Phi}_\bm^{(1)}\left(x_{12}\frac{\partial}{\partial x_{22}}{}^t\hspace{-1pt}x_{12}\frac{\partial}{\partial x_{11}}\right)f(x_{11},x_{22})\notag
\end{gather}
($k=0$ if $s'\ne s''$), when $(G,G_1)=(U(q,s),U(q',s')\times U(q'',s''))$,
\begin{gather}
\cF_{\lambda,k,l}\colon \ \cO_{(\lambda_1+k)+(\lambda_2+l)}(D_{11}) \hboxtimes \cO_{(\lambda_1+l)+(\lambda_2+k)}(D_{22})
\to \cO_{\lambda_1+\lambda_2}(D), \notag \\
(\cF_{\lambda,k,l}f)\begin{pmatrix}x_{11}&x_{12}\\x_{21}&x_{22}\end{pmatrix}
 =\det(x_{12})^k\det(x_{21})^l
\sum_{\bm\in\BZ_{++}^{\min\{q',s'\}}}\frac{1}{(\lambda_1+\lambda_2+k+l)_{\bm,2}}\notag \\
\hphantom{(\cF_{\lambda,k,l}f)\begin{pmatrix}x_{11}&x_{12}\\x_{21}&x_{22}\end{pmatrix}=}{}
\times\tilde{\Phi}_\bm^{(2)}
\left(x_{12}{\vphantom{\biggl(}}^t\!\!\left(\frac{\partial}{\partial x_{22}}\right)x_{21}
{\vphantom{\biggl(}}^t\!\!\left(\frac{\partial}{\partial x_{11}}\right)\right)f(x_{11},x_{22})
\label{U-UU}
\end{gather}
($k=0$ if $q'\ne s''$, $l=0$ if $q''\ne s'$), and when $(G,G_1)=(\operatorname{SO}^*(2s),\operatorname{SO}^*(2s')\times \operatorname{SO}^*(2s''))$,
\begin{gather}
\cF_{\lambda,k}\colon \ \cO_{\lambda+2k}(D_{11}) \hboxtimes \cO_{\lambda+2k}(D_{22})\to \cO_{\lambda}(D), \notag \\
(\cF_{\lambda,k}f)\begin{pmatrix}x_{11}&x_{12}\\-{}^t\hspace{-1pt}x_{12}&x_{22}\end{pmatrix}
=\det(x_{12})^k
\sum_{\bm\in\BZ_{++}^{\lfloor s'/2\rfloor}}\frac{1}{(\lambda+2k)_{\bm,4}}\notag\\
\hphantom{(\cF_{\lambda,k}f)\begin{pmatrix}x_{11}&x_{12}\\-{}^t\hspace{-1pt}x_{12}&x_{22}\end{pmatrix}=}{} \times\tilde{\Phi}_\bm^{(4)}
\left(-x_{12}\frac{\partial}{\partial x_{22}}{}^t\hspace{-1pt}x_{12}\frac{\partial}{\partial x_{11}}\right)f(x_{11},x_{22})
\label{SO*-SO*SO*}
\end{gather}
($k=0$ if $s'\ne s''$).

\subsection[$\cF_{\tau\rho}$ for $(G,G_1)=(\operatorname{Sp}(s,\BR), U(s',s''))$, $(\operatorname{SO}^*(2s), U(s',s''))$, \\
$(E_{6(-14)}, U(1)\times \operatorname{SO}^*(10))$, $(E_{7(-25)}, U(1)\times E_{6(-14)})$]{$\boldsymbol{\cF_{\tau\rho}}$ for $\boldsymbol{(G,G_1)=(\operatorname{Sp}(s,\BR), U(s',s''))}$, $\boldsymbol{(\operatorname{SO}^*(2s), U(s',s''))}$, \\ $\boldsymbol{(E_{6(-14)}, U(1)\times \operatorname{SO}^*(10))}$, $\boldsymbol{(E_{7(-25)}, U(1)\times E_{6(-14)})}$}
In this subsection we set
\begin{gather*} (G,G_1)= \begin{cases}(\operatorname{Sp}(s,\BR), U(s',s'')) \qquad (s=s'+s'')& (\text{Case }d=1),\\
(\operatorname{SO}^*(2s), U(s',s'')) \qquad (s=s'+s'')& (\text{Case }d=4),\\
(E_{6(-14)}, U(1)\times \operatorname{SO}^*(10)) & (\text{Case }d=6),\\
(E_{7(-25)}, U(1)\times E_{6(-14)}) & (\text{Case }d=8) \end{cases} \end{gather*}
(up to covering). Then the maximal compact subgroups $(K,K_1)=(K,K_{11}\times K_{22})\subset (G,G_1)$ are given by
\begin{gather*}
(K,K_1) =(K,K_{11}\times K_{22})
 =\begin{cases} (U(s), U(s')\times U(s''))& (\text{Cases }d=1,4),\\
(U(1)\times \operatorname{Spin}(10), U(1)\times U(5))& (\text{Case }d=6),\\
(U(1)\times E_6,U(1)\times U(1)\times \operatorname{Spin}(10))& (\text{Case }d=8). \end{cases}
\end{gather*}
(up to covering). Also we have
\begin{gather*} \fp^+=\begin{cases}\Sym(s,\BC) & (\text{Case }d=1),\\ \Skew(s,\BC) & (\text{Case }d=4),\\
M(1,2;\BO)^\BC & (\text{Case }d=6),\\ \Herm(3,\BO)^\BC & (\text{Case }d=8), \end{cases} \end{gather*}
and $\fp^+_1=\fp^+_{12}\colon =\fg_1^\BC\cap\fp^+$, $\fp^+_2=\fp^+_{11}\oplus\fp^+_{22}:=(\fp^+_1)^\bot$ are realized as
\begin{gather*} (\fp^+_{11},\fp^+_{12},\fp^+_{22})=\begin{cases}
(\Sym(s',\BC), M(s',s'';\BC), \Sym(s'',\BC)) & (\text{Case }d=1),\\
(\Skew(s',\BC), M(s',s'';\BC), \Skew(s'',\BC)) & (\text{Case }d=4),\\
(\BC, \Skew(5,\BC), M(1,5;\BC)) & (\text{Case }d=6),\\
\big(\BC, M(1,2;\BO)^\BC, \Herm(2,\BO)^\BC\big) & (\text{Case }d=8). \end{cases} \end{gather*}

Now let $(\tau,V)=\big(\chi^{-\lambda},\BC\big)$ with $\lambda$ sufficiently large,
$W=W_{11}\boxtimes W_{22}\subset(\cP(\fp^+_{11})\boxtimes\cP(\fp^+_{22}))\otimes\chi^{-\lambda}$
be an irreducible $\tilde{K}^\BC_1=\tilde{K}^\BC_{11}\times \tilde{K}^\BC_{22}$-submodule,
and $\rK(x_2)=\rK(x_{11}+x_{22})\in\cP\big(\fp^+_{11}\oplus\fp^+_{22},\Hom\big(W,\chi^{-\lambda}\big)\big)$
be the $\tilde{K}^\BC_1$-invariant polynomial in the sense of (\ref{K-invariance}).
For $x_2=x_{11}+x_{22}\in\fp^+_2=\fp^+_{11}\oplus\fp^+_{22}$, $w_1=w_{12}\in\fp^+_1=\fp^+_{12}$,
we want to compute
\begin{gather*}
 F_{\tau\rho}(x_2;w_1)=F_{\tau\rho}(x_{11},x_{22};w_{12})
 =\big\langle {\rm e}^{(y_1|w_1)_{\fp^+_1}}I_W,
\big(h(x_2,Q(y_1)x_2)^{-\lambda/2}\rK\big((x_2)^{Q(y_1)x_2}\big)\big)^*\big\rangle_{\hat{\rho},y_1}\\
 =\bigl\langle {\rm e}^{(y_{12}|w_{12})_{\fp^+_{12}}}I_W, \big(h(x_{11}\!+x_{22},Q(y_{12})(x_{11}+x_{22}))^{-\lambda/2}
\rK\big((x_{11}\!+x_{22})^{Q(y_{12})(x_{11}+x_{22})}\big)\big)^*\bigr\rangle_{\hat{\rho},y_{12}}\!\\
 =\bigl\langle {\rm e}^{(y_{12}|w_{12})_{\fp^+_{12}}}I_W,
\big(h_{22}(x_{22},Q(y_{12})x_{11})^{-\lambda}
\rK\big((x_{11})^{Q(y_{12})x_{22}}+(x_{22})^{Q(y_{12})x_{11}}\big)\big)^*\bigr\rangle_{\hat{\rho},y_{12}},
\end{gather*}
where we have used the similar argument to~(\ref{generic_norm_square}) at the last equality.
Now we assume that $W=W_{11}\boxtimes W_{22}\subset\cP(\fp^+_2)\otimes\chi^{-\lambda}$ is of the form
\begin{gather*} W=W_{11}\boxtimes W_{22}=\big(\cP_{(k,\ldots,k)}(\fp^+_{11})\otimes\chi_{11}^{-\lambda}\big)
\boxtimes\big(\cP_\bl(\fp^+_{22})\otimes\chi_{22}^{-\lambda}\big), \end{gather*}
where $\chi_{11}$, $\chi_{22}$ are as in the previous subsection, and
\begin{gather*} k\begin{cases}\in\BZ_{\ge 0} & (d=1,6,8 \text{ or }d=4, s'\colon \text{even}),\\ =0 & (d=4, s'\colon \text{odd}),\end{cases}\qquad
\bl\in\BZ_{++}^{r''},\quad r''=\begin{cases} s'' & (d=1), \\ \lfloor s''/2\rfloor & (d=4), \\ 1 & (d=6), \\ 2 & (d=8). \end{cases} \end{gather*}
Then the polynomial $\mathrm{K}$ is of the form
\begin{gather*} \rK(x_{11}+x_{22})=\Delta(x_{11})^k\rK_\bl^{(d)}(x_{22}), \end{gather*}
where $\Delta(x_{11})$ is the determinant polynomial of the Jordan algebra $\fp^+_{11}$ when $d=1,6,8$ or $d=4$ with $s'$ even,
and $\rK_\bl^{(d)}(x_{22})\in\cP\big(\fp^+_{22},\Hom\big(\cP_{\bl}(\fp^+_{22})\otimes\chi^{-\lambda}_{22},\chi^{-\lambda}_{22}\big)\big)$ is
$\tilde{K}_{22}^\BC$-invariant in the sense of~(\ref{K-invariance}). Then we have
\begin{gather*}
 F_{\tau\rho}(x_{11},x_{22};w_{12})\\
{}= \big\langle {\rm e}^{(y_{12}|w_{12})_{\fp^+_{12}}}I_{W_{22}},\big(h_{22}(x_{22},Q(y_{12})x_{11})^{-\lambda}
\Delta\big((x_{11})^{Q(y_{12})x_{22}}\big)^k
\rK_\bl^{(d)}\big((x_{22})^{Q(y_{12})x_{11}}\big)\big)^*\big\rangle_{\hat{\rho},y_{12}}.
\end{gather*}
Now since
\begin{align*}
\Delta\big((x_{11})^{Q(y_{12})x_{22}}\big)&=\chi_{11}\big(P\big((x_{11})^{Q(y_{12})x_{22}}\big)\big)
=\chi_{11}\big(B(x_{11},Q(y_{12})x_{22})^{-1}P(x_{11})\big) \\
&=h_{11}(x_{11},Q(y_{12})x_{22})^{-1}\Delta(x_{11})=h(x_{11},Q(y_{12})x_{22})^{-1}\Delta(x_{11}) \\
&=h(x_{22},Q(y_{12})x_{11})^{-1}\Delta(x_{11})=h_{22}(x_{22},Q(y_{12})x_{11})^{-1}\Delta(x_{11}),
\end{align*}
where $P$ is as (\ref{quad_repn}), and we have used \cite[Part~V, Proposition III.3.1, (J6.1)]{FKKLR} at the 2nd equality,
$\chi|_{K_{jj}}=\chi_{jj}$ at the 4th, 6th equalities, and \cite[Part V, Propositions IV.3.4 and~IV.3.5]{FKKLR} at the 5th equality, we have
\begin{gather*}
 F_{\tau\rho}(x_{11},x_{22};w_{12})\\
\qquad{} =\Delta(x_{11})^k\bigl\langle {\rm e}^{(y_{12}|w_{12})_{\fp^+_{12}}}I_{W_{22}},
\big(h_{22}(x_{22},Q(y_{12})x_{11})^{-\lambda-k} \rK_\bl^{(d)}\big((x_{22})^{Q(y_{12})x_{11}}\big)\big)^*\bigr\rangle_{\hat{\rho},y_{12}}.
\end{gather*}
Next we put $\lambda+k=:\mu$, $Q(y_{12})x_{11}=:z_{22}$, and we want to find the expansion formula of $h(x_{22},z_{22})^{-\mu}\rK_\bl^{(d)}\big((x_{22})^{z_{22}}\big)$. For a while we omit the subscript 22. We realize $\overline{W_{22}}=\overline{\cP_\bl(\fp^+_{22})\otimes\chi_{22}^{-\lambda}}$ as a space of polynomials in $y$, and write $\rK_\bl^{(d)}(x)=\bK_\bl^{(d)}(x,y)\in\cP(\fp^+_{22}\times\overline{\fp^+_{22}})$.
We normalize $\bK_\bl^{(d)}(x,y)$ as
\begin{gather*} \bK_\bl^{(d)}(x,y)=\Proj_{\bl,x}\big({\rm e}^{(x|y)_{\fp^+_{22}}}\big)=\Proj_{\bl,\overline{y}}\big({\rm e}^{(x|y)_{\fp^+_{22}}}\big), \end{gather*}
and for $x,y,z\in\fp^+_{22}$ and $\bn,\bl\in\BZ_{++}^{r''}$, we define
$\cK_{\bn,\bl}^{(d)}(x;y,z)\in\cP_\bn(\fp^+_{22})\boxtimes\overline{\cP_\bl(\fp^+_{22})\boxtimes\cP(\fp^+_{22})}$ by
\begin{gather*}
\cK_{\bn,\bl}^{(d)}(x;y,z) :=\Proj_{\bn,x}\Proj_{\bl,\bar{y}}\big({\rm e}^{(x|y+z)_{\fp^+_{22}}}\big)
 =\Proj_{\bl,\bar{y}}\big(\bK_\bn^{(d)}(x,y+z)\big) \\
 \hphantom{\cK_{\bn,\bl}^{(d)}(x;y,z)}{}
 =\Proj_{\bn,x}\big(\bK_\bl^{(d)}(x,y){\rm e}^{(x|z)_{\fp^+_{22}}}\big),
\end{gather*}
where $\Proj_{\bn,x}$ is the orthogonal projection onto $\cP_\bn(\fp^+_{22})$ with respect to the variable $x$.
Clearly $\cK_{\bn,\bl}^{(d)}$ is non-zero only if $\cP_\bn$ appears abstractly in the decomposition of $\cP_\bl\otimes\cP$.
Then the following holds.
\begin{Proposition}
\begin{align*}
h(x,z)^{-\mu}\bK_\bl^{(d)}(x^z,y)
&=\sum_{\bn\in\BZ_{++}^{r''}}\frac{(\mu)_{\bn,d}}{(\mu)_{\bl,d}}\cK_{\bn,\bl}^{(d)}(x;y,z)
=\sum_{\bn\in\BZ_{++}^{r''}}(\mu+\bl)_{\bn-\bl,d}\cK_{\bn,\bl}^{(d)}(x;y,z).
\end{align*}
\end{Proposition}
\begin{proof}
By (\ref{ext_of_h}) we have
\begin{gather*}
 \sum_{\bl\in\BZ_{++}^{r''}}(\mu)_{\bl,d}h(x,z)^{-\mu}\bK_\bl^{(d)}(x^z,y)
=h(x,z)^{-\mu}h(x^z,y)^{-\mu}=h(x,y+z)^{-\mu}\\
\qquad{} =\sum_{\bn\in\BZ_{++}^{r''}}(\mu)_{\bn,d}\bK_\bn^{(d)}(x,y+z)
=\sum_{\bn\in\BZ_{++}^{r''}}\sum_{\bl\in\BZ_{++}^{r''}}(\mu)_{\bn,d}\cK_{\bn,\bl}^{(d)}(x;y,z),
\end{gather*}
where the 2nd equality follows from (\ref{Bergman_left}). Then by projecting both sides to $\cP_\bl$ with respect to the variable $\bar{y}$ and dividing by $(\mu)_{\bl,d}$, we get the desired formula.
\end{proof}

\begin{Corollary}\label{ext_of_hK}
When we define $\cK_{\bn,\bl}^{(d)}(x;z)\in\cP(\fp^+_{22}\times\overline{\fp^+_{22}},\Hom(\cP_{\bl}(\fp^+_{22})\otimes\chi^{-\lambda}_{22},
\chi^{-\lambda}_{22}))$ as
\begin{gather*} \cK_{\bn,\bl}^{(d)}(x;z)=\Proj_{\bn,x}\big({\rm e}^{(x|z)_{\fp^+_{22}}}\rK_\bl^{(d)}(x)\big), \end{gather*}
then it holds that
\begin{gather*} h(x,z)^{-\mu}\rK_\bl^{(d)}(x^z)=\sum_{\bn\in\BZ_{++}^{r''}}(\mu+\bl)_{\bn-\bl,d}\cK_{\bn,\bl}^{(d)}(x;z). \end{gather*}
\end{Corollary}
Therefore we have
\begin{gather*}
 F_{\tau\rho}(x_{11},x_{22};w_{12}) \\
 \qquad{} =\Delta(x_{11})^k\sum_{\bn\in\BZ_{++}^{r''}}(\lambda+k+\bl)_{\bn-\bl,d}
\bigl\langle {\rm e}^{(y_{12}|w_{12})_{\fp^+_{12}}}I_{W_{22}},
\cK_{\bn,\bl}^{(d)}(x_{22};Q(y_{12})x_{11})^*\bigr\rangle_{\hat{\rho},y_{12}}.
\end{gather*}
From now on we consider the cases $d=1,4$. In these cases we have
\begin{gather*} W=\big(\cP_{(k,\ldots,k)}(\fp^+_{11})\otimes\chi_{11}^{-\lambda}\big)\boxtimes\big(\cP_\bl(\fp^+_{22})\otimes\chi_{22}^{-\lambda}\big)
\simeq\begin{cases} \big(\BC\boxtimes V_{2\bl}^{(s'')}\big)\otimes \chi_1^{-(\lambda+2k)-\lambda} & (d=1), \\
\big(\BC\boxtimes V_{\bl^2}^{(s'')}\big)\otimes \chi_1^{-\left(\frac{\lambda}{2}+k\right)-\frac{\lambda}{2}} & (d=4). \end{cases} \end{gather*}
By the ${\rm GL}(s'',\BC)$-invariance, the representation containing
$\overline{\cK_{\bn,\bl}^{(d)}(x_{22};y_{21}x_{11}^*y_{12})}$ as functions on $(y_{12},y_{21})$ is isomorphic to that
as functions on $\overline{x_{22}}$, namely $\cP_\bn(\fp^+_{22})$, which has the lowest weight~$-2\bn$ when $d=1$ case, $-\bn^2$ when $d=4$ case. Moreover, $\cK_{\bn,\bl}(x_{22};y_{21}x_{11}^*y_{12})$ is non-zero only if~$\cP_\bn(\fp^+_{22})$ appears abstractly in the decomposition of $\cP(\fp^+_{12})\otimes\cP_\bl(\fp^+_{22})$.
Therefore by the result of~\cite{N2} and~(\ref{exp_onD}), we get
\begin{gather*}
 F_{\tau\rho}(x_{11},x_{22};w_{12})\\
 =\begin{cases}\displaystyle \det(x_{11})^k\sum_{\bn\in\BZ_{++}^{s''}}
\frac{(\lambda+k+\bl)_{\bn-\bl,1}}{(2(\lambda+k+\bl))_{2(\bn-\bl),2}}\cK_{\bn,\bl}^{(1)}\big(x_{22};{}^t\hspace{-1pt}w_{12}x_{11}^*w_{12}\big)&
 (d=1),\\
\displaystyle \Pf(x_{11})^k\sum_{\bn\in\BZ_{++}^{\lfloor s''/2\rfloor}}
\frac{(\lambda+k+\bl)_{\bn-\bl,4}}{(\lambda+k+\bl)_{(\bn-\bl)^2,2}}\cK_{\bn,\bl}^{(4)}\big(x_{22};-{}^t\hspace{-1pt}w_{12}x_{11}^*w_{12}\big)
&(d=4)\end{cases}\\
 =\begin{cases}\displaystyle \det(x_{11})^k\sum_{\bn\in\BZ_{++}^{s''}}
\frac{1}{2^{2|\bn-\bl|}\left(\lambda+k+\bl+\frac{1}{2}\right)_{\bn-\bl,1}}
\cK_{\bn,\bl}^{(1)}\big(x_{22};{}^t\hspace{-1pt}w_{12}\overline{x_{11}}w_{12}\big)
&(d=1),\\
\displaystyle \Pf(x_{11})^k\sum_{\bn\in\BZ_{++}^{\lfloor s''/2\rfloor}}
\frac{1}{(\lambda+k+\bl-1)_{\bn-\bl,4}}\cK_{\bn,\bl}^{(4)}\big(x_{22};{}^t\hspace{-1pt}w_{12}\overline{x_{11}}w_{12}\big)
&(d=4)\end{cases}\\
 =\begin{cases}\displaystyle \det(x_{11})^k\hspace{-15pt}\sum_{\substack{\bn\in\BZ_{++}^{s''}\\ l_j\le n_j\; (1\le j\le s') \\
l_j\le n_j\le l_{j-s'}\; (s'+1\le j\le s'')}}\hspace{-15pt}
\frac{1}{\left(\lambda+k+\bl+\frac{1}{2}\right)_{\bn-\bl,1}}
\cK_{\bn,\bl}^{(1)}\big(x_{22};\tfrac{1}{4}{}^t\hspace{-1pt}w_{12}\overline{x_{11}}w_{12}\big)&(d=1),\\
\displaystyle \Pf(x_{11})^k\hspace{-30pt}\sum_{\substack{\bn\in\BZ_{++}^{\lfloor s''/2\rfloor}\\ l_j\le n_j\; (1\le j\le \lfloor s'/2\rfloor) \\
l_j\le n_j\le l_{j-\lfloor s'/2\rfloor}\; (\lfloor s'/2\rfloor+1\le j\le \lfloor s''/2\rfloor)}}\hspace{-30pt}
\frac{1}{(\lambda+k+\bl-1)_{\bn-\bl,4}}\cK_{\bn,\bl}^{(4)}\big(x_{22};{}^t\hspace{-1pt}w_{12}\overline{x_{11}}w_{12}\big)
&(d=4).\end{cases}
\end{gather*}
Here the condition $l_j\le n_j\le l_{j-s'}$ $(s'+1\le j\le s'')$ or $l_j\le n_j\le l_{j-\lfloor s'/2\rfloor}$
$\bigl(\big\lfloor \frac{s'}{2}\big\rfloor+1\le j\le \big\lfloor \frac{s''}{2}\big\rfloor\bigr)$ appears only when
$s'< s''$ or $\big\lfloor \frac{s'}{2}\big\rfloor< \big\lfloor \frac{s''}{2}\big\rfloor$ respectively.
Next we consider $d=6$ case. In this case we have
\begin{align*}
W=(\cP_k(\BC)\boxtimes\cP_l(M(1,5;\BC)))\otimes\chi^{-\lambda}
&\simeq \Big(V_{\left(\frac{l}{2},\frac{l}{2},\frac{l}{2},\frac{l}{2},-\frac{l}{2}\right)}^{(5)\vee}
\otimes\chi_{\operatorname{SO}^*(10)}^{-\lambda-k}\Big)\boxtimes\chi_{U(1)}^{-\lambda+3k-3l} \\
&\simeq \Big(V_{(l,l,l,l,0)}^{(5)\vee}\otimes\chi_{\operatorname{SO}^*(10)}^{-\lambda-k+l}\Big)\boxtimes\chi_{U(1)}^{-\lambda+3k-3l}.
\end{align*}
Then $x_{11}^k{\rm e}^{(x_{22}|Q(y_{12})x_{11})_{\fp^+_{22}}}\rK_l(x_{22})$ is decomposed as
\begin{align*}
x_{11}^k{\rm e}^{(x_{22}|Q(y_{12})x_{11})_{\fp^+_{22}}}\rK_l(x_{22})
&=x_{11}^k\sum_{m=0}^\infty \frac{1}{m!}(x_{22}|Q(y_{12})x_{11})_{\fp^+_{22}}^m\rK_l(x_{22}) \\
&=x_{11}^k\sum_{n=l}^\infty \frac{1}{(n-l)!}(x_{22}|Q(y_{12})x_{11})_{\fp^+_{22}}^{n-l}\rK_l(x_{22}),
\end{align*}
and as a function of $(x_{11},x_{22})$, we have
\begin{gather*}
 x_{11}^k\cK_{n,l}^{(6)}(x_{22};Q(y_{12})x_{11})
=x_{11}^k\frac{1}{(n-l)!}(x_{22}|Q(y_{12})x_{11})_{\fp^+_{22}}^{n-l}\rK_l(x_{22}) \\
\qquad{} \in(\cP_{k+n-l}(\fp^+_{11})\boxtimes\cP_n(\fp^+_{22}))\otimes\chi^{-\lambda}
\simeq V_{(n,n,n,n,0)}^{(5)\vee}\otimes\chi^{-\lambda-k+l}_{\operatorname{SO}^*(10)}\boxtimes\chi_{U(1)}^{-\lambda+3k-3l},
\end{gather*}
and by the $K_1^\BC$-invariance, as a function of $y_{12}$, $x_{11}^k\frac{1}{(n-l)!}(x_{22}|Q(y_{12})x_{11})_{\fp^+_{22}}^{n-l}\rK_l(x_{22})$ sits in the space with the same lowest weight.
Therefore we have
\begin{gather*}
 F_{\tau\rho}(x_{11},x_{22};w_{12})\\
 =x_{11}^k\sum_{n=l}^\infty(\lambda+k+l)_{n-l}\left\langle {\rm e}^{(y_{12}|w_{12})_{\fp^+_{12}}}I_{W_{22}},
\left(\frac{1}{(n-l)!}(x_{22}|Q(y_{12})x_{11})_{\fp_{22}}^{n-l}\rK_l(x_{22})\right)^*\right\rangle_{\hat{\rho},y_{12}} \\
 =x_{11}^k\sum_{m=0}^\infty \frac{(\lambda+k+l)_m}{(\lambda+k+l)_{(m,m),4}}\frac{1}{m!}
\big(x_{11}x_{22}{}^t\mathbf{Pf}(\overline{w_{12}})\big)^m\rK_l(x_{22}) \\
 =x_{11}^k\sum_{m=0}^\infty\frac{1}{(\lambda+k+l-2)_m}\frac{1}{m!}\big(x_{11}x_{22}{}^t\mathbf{Pf}(\overline{w_{12}})\big)^m\rK_l(x_{22}),
\end{gather*}
where $\mathbf{Pf}(w_{12})$ is as (\ref{Pfaff_vector}).
Next we consider $d=8$ case. In this case we have
\begin{gather*} W=\big(\cP_k(\BC)\boxtimes\cP_\bl\big(\Herm(2,\BO)^\BC\big)\big)\otimes\chi^{-\lambda}
\simeq V_{(l_1-l_2,0,0,0,0)}^{[10]\vee}
\boxtimes\chi_{E_{6(-14)}}^{-\lambda-k-\frac{|\bl|}{2}}\boxtimes\chi_{U(1)}^{-\lambda+2k-2|\bl|}. \end{gather*}
Then $x_{11}^k{\rm e}^{(x_{22}|Q(y_{12})x_{11})_{\fp^+_{22}}}\rK_{(l_1,l_2)}^{(8)}(x_{22})
=x_{11}^k\Delta(x_{22})^{l_2}{\rm e}^{(x_{22}|Q(y_{12})x_{11})_{\fp^+_{22}}}\rK_{(l_1-l_2,0)}^{(8)}(x_{22})$
(under suitable normalization) is decomposed as
\begin{align*}
x_{11}^k{\rm e}^{(x_{22}|Q(y_{12})x_{11})_{\fp^+_{22}}}\rK_{(l_1,l_2)}^{(8)}(x_{22})
&=x_{11}^k\sum_{\bn\in\BZ_{++}^2}\cK_{\bn,\bl}^{(8)}(x_{22};Q(y_{12})x_{11}) \\
&=x_{11}^k\Delta(x_{22})^{l_2}\sum_{\bn\in\BZ_{++}^2}\cK_{(n_1-l_2,n_2-l_2),(l_1-l_2,0)}^{(8)}(x_{22};Q(y_{12})x_{11}),
\end{align*}
and as a function of $(x_{11},x_{22})$, we have
\begin{gather*}
 x_{11}^k\cK_{\bn,\bl}^{(8)}(x_{22};Q(y_{12})x_{11})
=x_{11}^k\Delta(x_{22})^{l_2}\cK_{(n_1-l_2,n_2-l_2),(l_1-l_2,0)}^{(8)}(x_{22};Q(y_{12})x_{11}) \\
\qquad {} \in (\cP_{k+|\bn|-|\bl|}(\fp^+_{11})\boxtimes\cP_\bn(\fp^+_{22}))\otimes\chi^{-\lambda}
\simeq V_{(n_1-n_2,0,0,0,0)}^{[10]\vee}\boxtimes\chi_{E_{6(-14)}}^{-\lambda-k+|\bl|-\frac{3}{2}|\bn|}
\boxtimes\chi_{U(1)}^{-\lambda+2k-2|\bl|}.
\end{gather*}
On the other hand, by using the multi-minuscule rule \cite[Corollary~2.16]{S}, we can show that
$\cH_{\lambda+k+\frac{|\bl|}{2}}(D_1,V_{(l_1-l_2,0,0,0,0)}^{[10]\vee})_{\tilde{K}_1}$ is decomposed as
\begin{gather*}
 \cP(\fp_{12}^+)\otimes V_{(l_1-l_2,0,0,0,0)}^{[10]\vee}\boxtimes \chi_{E_{6(-14)}}^{-\lambda-k-\frac{|\bl|}{2}} \\
\qquad{} \simeq \bigoplus_{\bm\in\BZ_{++}^2}
V_{\left(\frac{m_1+m_2}{2},\frac{m_1-m_2}{2},\frac{m_1-m_2}{2},\frac{m_1-m_2}{2},\frac{m_1-m_2}{2}\right)}^{[10]\vee}
\otimes V_{(l_1-l_2,0,0,0,0)}^{[10]\vee}\boxtimes \chi_{E_{6(-14)}}^{-\lambda-k-\frac{|\bl|}{2}-\frac{3}{4}|\bm|} \\
\qquad{} \simeq \bigoplus_{\bm\in\BZ_{++}^2}
\bigoplus_{\substack{\bk\in(\BZ_{\ge 0})^4,\; |\bk|=l_1-l_2\\ k_2+k_4\le m_2\\ k_3\le m_1-m_2}}
V_{\substack{\bigl(\frac{m_1+m_2}{2}+k_1-k_4,\frac{m_1-m_2}{2}+k_2, \hspace{10pt} \\
\hspace{15pt} \frac{m_1-m_2}{2},\frac{m_1-m_2}{2},\frac{m_1-m_2}{2}-k_3\bigr)}}^{[10]\vee}
\boxtimes \chi_{E_{6(-14)}}^{-\lambda-k-\frac{|\bl|}{2}-\frac{3}{4}|\bm|}.
\end{gather*}
Therefore $\cK_{\bn,\bl}^{(8)}(x_{22};Q(y_{12})x_{11})$ is non-zero only if $(n_1,n_2)$ is of the form
\begin{gather*} n_1-l_2=m+k_1,\; n_2-l_2=k_4 \qquad (m,k_1,k_4\in\BZ_{\ge 0},\; k_1+k_4=l_1-l_2,\; k_4\le m) \end{gather*}
(in this case $m_1=m_2=m$, $k_2=k_3=0$).
Therefore, if we assume \cite[Conjecture 5.11]{N2} on the norms of holomorphic discrete series representations of $E_{6(-14)}$ is true,
then we have
\begin{gather*}
 F_{\tau\rho}(x_{11},x_{22};w_{12}) \\
 =x_{11}^k\sum_{\bn\in\BZ_{++}^{2}}(\lambda+k+\bl)_{\bn-\bl,8}\Bigl\langle {\rm e}^{(y_{12}|w_{12})_{\fp^+}}I_{W_{22}},
\cK_{\bn,\bl}^{(8)}(x_{22};y_{22},Q(y_{12})x_{11})^*\Bigr\rangle_{\hat{\rho},y_{12}} \\
=x_{11}^k\sum_{m=0}^\infty \sum_{\substack{k_1,k_4\in\BZ_{\ge 0}\\ k_1+k_4=l_1-l_2\\ k_4\le m}} \\
\quad{}\times\frac{(\lambda+k+l_1)_{m+k_1-l_1+l_2}(\lambda+k+l_2-4)_{k_4}}
{(\lambda+k+l_1)_{m+k_1-l_1+l_2}(\lambda+k+l_1-3)_{m+k_1-l_1+l_2}(\lambda+k+l_2-4)_{k_4}(\lambda+k+l_2-7)_{k_4}}\\
\quad{} \times \Delta(x_{22})^{l_2}\cK_{(m+k_1,k_4),(l_1-l_2,0)}^{(8)} (x_{22};Q(w_{12})x_{11} )\\
=x_{11}^k\Delta(x_{22})^{l_2}\sum_{m=0}^\infty \sum_{\substack{k_1,k_4\in\BZ_{\ge 0}\\ k_1+k_4=l_1-l_2\\ k_4\le m}}
\frac{1}{(\lambda+k+l_1-3)_{m-k_4}(\lambda+k+l_2-7)_{k_4}} \\
\quad{} \times \cK_{(m+k_1,k_4),(l_1-l_2,0)}^{(8)}
\big(x_{22};\overline{x_{11}}\hspace{1pt}{}^t\hspace{-1pt}\hat{w}_{12}w_{12}\big).
\end{gather*}
By Theorem \ref{main}, by substituting $w_{12}$ with $\overline{\frac{\partial}{\partial x_{12}}}$,
we get the intertwining operator from $(\cH_1)_{\tilde{K}_1}$ to $\cH_{\tilde{K}}$, and by Theorem~\ref{extend},
this extends to the intertwining operator between the spaces of all holomorphic functions if $\cH_1$ is holomorphic discrete.
Also, by Theorem \ref{continuation}, this continues meromorphically for all $\lambda\in\BC$.
Therefore we get the following.
\begin{Theorem}\label{main2}\quad
\begin{enumerate}\itemsep=0pt
\item[$(1)$] Let $(G,G_1)=(\operatorname{Sp}(s,\BR), U(s',s''))$ with $s=s'+s''$. Let $k\in\BZ_{\ge 0}$, $\bl\in\BZ_{++}^{s''}$.
Then the linear map
\begin{gather*}
\cF_{\lambda,k,\bl}\colon \ \cO_{(\lambda+2k)+\lambda}(D_1,\BC\boxtimes V_{2\bl}^{(s'')})\to \cO_\lambda(D), \\
(\cF_{\lambda,k,\bl}f)\begin{pmatrix}x_{11}&x_{12}\\{}^t\hspace{-1pt}x_{12}&x_{22}\end{pmatrix}
=\det(x_{11})^k\hspace{-15pt}\sum_{\substack{\bn\in\BZ_{++}^{s''}\\ l_j\le n_j\; (1\le j\le s') \\
l_j\le n_j\le l_{j-s'}\; (s'+1\le j\le s'')}}\hspace{-15pt}
\frac{1}{\left(\lambda+k+\bl+\frac{1}{2}\right)_{\bn-\bl,1}}\\
\hphantom{(\cF_{\lambda,k,\bl}f)\begin{pmatrix}x_{11}&x_{12}\\{}^t\hspace{-1pt}x_{12}&x_{22}\end{pmatrix}=}{} \times \cK_{\bn,\bl}^{(1)}\left(x_{22};\frac{1}{4}\overline{{\vphantom{\biggl(}}^t\!\!\left(
\frac{\partial}{\partial x_{12}}\right)x_{11}\frac{\partial}{\partial x_{12}}}\right)f(x_{12})
\end{gather*}
intertwines the $\tilde{G}_1$-action.
\item[$(2)$] Let $(G,G_1)=(\operatorname{SO}^*(2s), U(s',s''))$ with $s=s'+s''$.
Let $k\in\BZ_{\ge 0}$ if $s'$ is even, $k=0$ if $s'$ is odd, and $\bl\in\BZ_{++}^{\lfloor s''/2\rfloor}$. Then the linear map
\begin{gather*}
\cF_{\lambda,k,\bl}\colon \ \cO_{\left(\frac{\lambda}{2}+k\right)+\frac{\lambda}{2}}\big(D_1,\BC\boxtimes V_{\bl^2}^{(s'')}\big)
\to \cO_\lambda(D), \\
(\cF_{\lambda,k,\bl}f)\begin{pmatrix}x_{11}&x_{12}\\-{}^t\hspace{-1pt}x_{12}&x_{22}\end{pmatrix}
=\Pf(x_{11})^k\hspace{-30pt}\sum_{\substack{\bn\in\BZ_{++}^{\lfloor s''/2\rfloor}\\ l_j\le n_j\; (1\le j\le \lfloor s'/2\rfloor) \\
l_j\le n_j\le l_{j-\lfloor s'/2\rfloor}\; (\lfloor s'/2\rfloor+1\le j\le \lfloor s''/2\rfloor)}}\hspace{-30pt}
\frac{1}{(\lambda+k+\bl-1)_{\bn-\bl,4}}\\
\hphantom{(\cF_{\lambda,k,\bl}f)\begin{pmatrix}x_{11}&x_{12}\\-{}^t\hspace{-1pt}x_{12}&x_{22}\end{pmatrix}=}{} \times \cK_{\bn,\bl}^{(4)}\left(x_{22};
\overline{{\vphantom{\biggl(}}^t\!\!\left(\frac{\partial}{\partial x_{12}}\right)x_{11}\frac{\partial}{\partial x_{12}}}\right)
f(y_{12})
\end{gather*}
intertwines the $\tilde{G}_1$-action.
\item[$(3)$] Let $(G,G_1)=(E_{6(-14)}, U(1)\times \operatorname{SO}^*(10))$ $($up to covering$)$. Let $k,l\in\BZ_{\ge 0}$.
Then the linear map
\begin{gather*}
\cF_{\lambda,k,l}\colon \ \cO_{\lambda+k}\big(D_1,\BC\boxtimes V_{\left(\frac{l}{2},\frac{l}{2},\frac{l}{2},\frac{l}{2},-\frac{l}{2}\right)}^{(5)\vee}\big)
\boxtimes\chi_{U(1)}^{-\lambda+3k-3l}\to \cO_\lambda(D), \\
(\cF_{\lambda,k,l}f)(x_{11},x_{12},x_{22})=x_{11}^k\sum_{m=0}^\infty\frac{1}{(\lambda+k+l-2)_m}\\
\hphantom{(\cF_{\lambda,k,l}f)(x_{11},x_{12},x_{22})=}{} \times \frac{1}{m!}\left(x_{11}x_{22}{}^t\mathbf{Pf}\left(\frac{\partial}{\partial x_{12}}\right)\right)^m\rK_l(x_{22})f(x_{12})
\end{gather*}
$(x_{11}\in\BC,\; x_{12}\in\Skew(5,\BC),\; x_{22}\in M(1,5;\BC))$ intertwines the $\tilde{G}_1$-action.
\item[$(4)$] Let $(G,G_1)=(E_{7(-25)}, U(1)\times E_{6(-14)})$ $($up to covering$)$. Let $k\in\BZ_{\ge 0}$, $\bl\in\BZ_{++}^2$.
We assume that the conjecture on the norms of holomorphic discrete series representations of~$E_{6(-14)}$ {\rm \cite[Conjecture 5.11]{N2}} is true.
Then the linear map
\begin{gather*}
\cF_{\lambda,k,\bl}\colon \ \cO_{\lambda+k+\frac{|\bl|}{2}}\big(D_1,V_{(l_1-l_2,0,0,0,0)}^{[10]\vee}\big)
\boxtimes\chi_{U(1)}^{-\lambda+2k-2|\bl|}\to \cO_\lambda(D), \\
(\cF_{\lambda,k,\bl}f)\begin{pmatrix}x_{11}&x_{12}\\{}^t\hspace{-1pt}\hat{x}_{12}&x_{22}\end{pmatrix} \\
\qquad{} =x_{11}^k\det(x_{22})^{l_2}\sum_{m=0}^\infty \sum_{\substack{k_1,k_4\in\BZ_{\ge 0}\\ k_1+k_4=l_1-l_2\\ k_4\le m}}
\frac{1}{(\lambda+k+l_1-3)_{m-k_4}(\lambda+k+l_2-7)_{k_4}}\\
\qquad\quad {}\times \cK_{(m+k_1,k_4),(l_1-l_2,0)}^{(8)}
\left(x_{22};\overline{x_{11}{\vphantom{\biggl(}}^t\!\!\biggl(\widehat{\frac{\partial}{\partial x_{12}}}\biggr)
\frac{\partial}{\partial x_{12}}}\right)f(x_{12})
\end{gather*}
intertwines the $\tilde{G}_1$-action.
\end{enumerate}
\end{Theorem}
For $d=1,4,8$, if $\bl=(l,\ldots,l)$ (when $d=4$ and $s''$ is odd, additionally assume $l=0$), we have $\rK_\bl^{(d)}(x_{22})=\Delta(x_{22})^l$,
and $\cK_{\bn,\bl}^{(d)}(x_{22};z_{22})=\Delta(x_{22})^l\bK_{\bn-\bl}^{(d)}(x_{22},z_{22})
=\Delta(x_{22})^l\tilde{\Phi}_{\bn-\bl}^{(d)}(x_{22}z_{22}^*)$ holds,
where $\tilde{\Phi}_\bm^{(d)}(x_{22}z_{22}^*)$ is defined in (\ref{Phitilde1}), (\ref{Phitilde2}).
Therefore by replacing $\bn-\bl=\bm$, the intertwining operators are rewritten as, when $(G,G_1)=(\operatorname{Sp}(s,\BR), U(s',s''))$,
\begin{gather}
\cF_{\lambda,k,l}\colon \ \cO_{(\lambda+2k)+(\lambda+2l)}(D_1)\to \cO_\lambda(D),\notag \\
(\cF_{\lambda,k,l}f)\begin{pmatrix}x_{11}&x_{12}\\{}^t\hspace{-1pt}x_{12}&x_{22}\end{pmatrix}=\det(x_{11})^k\det(x_{22})^l\nonumber\\
\qquad{} \times\sum_{\bm\in\BZ_{++}^{\min\{s',s''\}}}
\frac{1}{\left(\lambda+k+l+\frac{1}{2}\right)_{\bm,1}}\tilde{\Phi}_\bm^{(1)}
\left(\frac{1}{4}x_{11}\frac{\partial}{\partial x_{12}}x_{22}{\vphantom{\biggl(}}^t\!\!\left(\frac{\partial}{\partial x_{12}}\right)\right)
f(x_{12}),\label{Sp-U}
\end{gather}
when $(G,G_1)=(\operatorname{SO}^*(2s), U(s',s''))$,
\begin{gather}
\cF_{\lambda,k,l}\colon \ \cO_{\left(\frac{\lambda}{2}+k\right)+\left(\frac{\lambda}{2}+l\right)}(D_1)\to \cO_\lambda(D), \notag \\
(\cF_{\lambda,k,l}f)\begin{pmatrix}x_{11}&x_{12}\\-{}^t\hspace{-1pt}x_{21}&x_{22}\end{pmatrix}
=\Pf(x_{11})^k\Pf(x_{22})^l\notag\\
\qquad{} \times\sum_{\bm\in\BZ_{++}^{\min\{\lfloor s'/2\rfloor,\lfloor s''/2\rfloor\}}}
\frac{1}{(\lambda+k+l-1)_{\bm,4}}\tilde{\Phi}_\bm^{(4)}
\left(-x_{11}\frac{\partial}{\partial x_{12}}x_{22}{\vphantom{\biggl(}}^t\!\!\left(\frac{\partial}{\partial x_{12}}\right)\right)f(x_{12}),\label{SO*-U}
\end{gather}
and when $(G,G_1)=(E_{7(-25)},U(1)\times E_{6(-14)})$,
\begin{gather*}
\cF_{\lambda,k,l}\colon \ \cO_{\lambda+k+l}(D_1)\boxtimes\chi_{U(1)}^{-\lambda+2k-4l}\to \cO_\lambda(D), \\
(\cF_{\lambda,k,l}f)\begin{pmatrix}x_{11}&x_{12}\\{}^t\hspace{-1pt}\hat{x}_{12}&x_{22}\end{pmatrix} \\
\qquad {} =x_{11}^k\det(x_{22})^l\sum_{m=0}^\infty \frac{1}{(\lambda+k+l-3)_m}
\frac{1}{m!}\left(x_{11}\Re_\BO\left(\frac{\partial}{\partial x_{12}}
\left(x_{22}{\vphantom{\biggl(}}^t\!\!\biggl(\widehat{\frac{\partial}{\partial x_{12}}}\biggr)\right)\right)\right)^m\!\!f(x_{12}).
\end{gather*}
(This holds without the assumption \cite[Conjecture 5.11]{N2} since the norm of $\cH_{\lambda+k+l}(D_1)$ is computed in~\cite{FK0}.)

\subsection[$\cF_{\tau\rho}$ for $(G,G_1)=(\operatorname{SU}(3,3), \operatorname{SO}^*(6))$, $(E_{7(-25)}, \operatorname{SU}(2,6))$]{$\boldsymbol{\cF_{\tau\rho}}$ for $\boldsymbol{(G,G_1)=(\operatorname{SU}(3,3), \operatorname{SO}^*(6))}$, $\boldsymbol{(E_{7(-25)}, \operatorname{SU}(2,6))}$}
In this subsection we set
\begin{gather*} (G,G_1)=\begin{cases} (\operatorname{SU}(3,3), \operatorname{SO}^*(6))\simeq (\operatorname{SU}(3,3), \operatorname{SU}(1,3)) & (\text{Case }d_2=1), \\
(E_{7(-25)}, \operatorname{SU}(2,6)) & (\text{Case }d_2=4) \end{cases} \end{gather*}
(up to covering). We can also compute for $(G,G_1)=(\operatorname{SO}^*(12), \operatorname{SO}^*(6)\times \operatorname{SO}^*(6))$ in a similar way,
but we omit this case since this is contained in Theorem \ref{main1} (5).
Then the maximal compact subgroups are
\begin{gather*} (K,K_1)=\begin{cases} (S(U(3)\times U(3)), U(3))\simeq (S(U(3)\times U(3)), S(U(1)\times U(3)))\! & (\text{Case }d_2=1), \\
(U(1)\times E_6, S(U(2)\times U(6))) & (\text{Case }d_2=4) \end{cases} \end{gather*}
(up to covering), and $\fp^+$, $\fp^+_1:=\fg_1^\BC\cap\fp^+$, $\fp^+_2:=(\fp^+_1)^\bot$ are realized as
\begin{gather*}
\fp^+ =\begin{cases} \Herm(3,\BC)^\BC\simeq M(3,\BC) & (\text{Case }d_2=1), \\ \Herm(3,\BO)^\BC & (\text{Case }d_2=4), \end{cases} \\
\fp^+_1 =\begin{cases} \Skew(3,\BR)^\BC\simeq \big(\BR^3\big)^\BC\simeq M(1,3;\BC) & (\text{Case }d_2=1), \\
\Skew(3,\BH)^\BC\simeq \big(\BH^3\big)^\BC\simeq M(2,6;\BC) & (\text{Case }d_2=4), \end{cases} \\
\fp^+_2 =\begin{cases} \Sym(3,\BR)^\BC\simeq \Sym(3,\BC) & (\text{Case }d_2=1), \\
\Herm(3,\BH)^\BC\simeq \Skew(6,\BC) & (\text{Case }d_2=4). \end{cases}
\end{gather*}
$K_1\simeq S(U(1)\times U(3))$ resp. $S(U(2)\times U(6))$ acts on $\fp^+_1\oplus\fp^+_2\simeq M(1,3;\BC)\oplus \Sym(3,\BC)$ resp.
$M(2,6;\BC)\oplus\Skew(6,\BC)$ by
\begin{gather*} (k_1,k_2).(x_1,x_2)=\big(k_1x_1k_2^{-1},\det(k_2)^{-2/\varepsilon}k_2x_2{}^t\hspace{-1pt}k_2\big), \end{gather*}
where $\varepsilon=1$ if $d_2=1$, $\varepsilon=2$ if $d_2=4$.
Let $\chi$, $\chi_1$ be the characters of $K^\BC$, $K_1^\BC$ respectively, normalized as~(\ref{char}),
and also let $\chi_2$ be the character of $K_1^\BC$ normalized as~(\ref{char}) with respect to the
Lie algebra $\fp^+_2\oplus\fk_1^\BC\oplus\fp^-_2$. Then $\chi|_{K_1}=\chi_1^{2/\varepsilon}=\chi_2$ holds.

Now let $(\tau,V)=\big(\chi^{-\lambda},\BC\big)$ with $\lambda$ sufficiently large,
$W\subset\cP(\fp^+_2)\otimes\chi^{-\lambda}$ be an irreducible $\tilde{K}^\BC_1$-submodule,
and $\rK(x_2)\in \cP(\fp^+_2,\Hom(W,\chi^{-\lambda}))$ be the $\tilde{K}^\BC_1$-invariant polynomial in the sense of (\ref{K-invariance}).
For $x_2\in\fp^+_2$, $w_1\in\fp^+_1$, we want to compute
\begin{gather*}
F_{\tau\rho}(x_2;w_1)
=\big\langle {\rm e}^{(y_1|w_1)_{\fp^+_1}}I_W,
\big(h(Q(x_2)y_1,y_1)^{-\lambda/2}\rK\big((x_2)^{Q(y_1)x_2}\big)\big)^*\big\rangle_{\hat{\rho},y_1}.
\end{gather*}
First we compute $h(Q(x_2)y_1,y_1)^{-\lambda/2}$ and $(x_2)^{Q(y_1)x_2}$.
\begin{Lemma}
For $x_2\in\Sym(3,\BC)$, $y_1\in M(1,3;\BC)$,
\begin{enumerate}\itemsep=0pt
\item[$(1)$] $h(Q(x_2)y_1,y_1)=\det\big(I-x_2^\sharp y_1^*\overline{y_1}\big)^2$.
\item[$(2)$] $(x_2)^{Q(y_1)x_2}=\det(x_2)^{-1}\big((I-x_2^\sharp y_1^*\overline{y_1})^{-1}x_2^\sharp \big)^\sharp$.
\end{enumerate}
For $x_2\in\Skew(6,\BC)$, $y_1\in M(2,6;\BC)$,
\begin{enumerate}\itemsep=0pt
\item[$(3)$] $h(Q(x_2)y_1,y_1)=\det\big(I-x_2^\#y_1^*J_2\overline{y_1}\big)$.
\item[$(4)$] $(x_2)^{Q(y_1)x_2}=\Pf(x_2)^{-1}\big((I-x_2^\#y_1^*J_2\overline{y_1})^{-1}x_2^\#\big)^\#$.
\end{enumerate}
\end{Lemma}
Here $x^\sharp$ on $\Sym(3,\BC)$ is defined in (\ref{adjoint3}), $x^\#$ on $\Skew(6,\BC)$ is defined in (\ref{adjoint6}),
and $J_2=\left(\begin{smallmatrix}0&1\\-1&0\end{smallmatrix}\right)$.
\begin{proof}
(1), (3) Let $x_2\in \Sym(3,\BC)$ resp. $\Skew(6,\BC)$, $y_1\in M(1,3;\BC)$ resp. $M(2,6;\BC)$.
Then we have $Q(x_2)y_1=\overline{y_1}x_2^\sharp\in M(1,3;\BC)$ resp.~$J_2\overline{y_1}x_2^\#\in M(2,6;\BC)$,
and since $\chi|_{K_1}=\chi_1^{2/\varepsilon}$, we get
\begin{align*}
h(Q(x_2)y_1,y_1)&=h_1(Q(x_2)y_1,y_1)^{2/\varepsilon} \\
&=\begin{cases} \big(1-\overline{y_1}x_2^\sharp y_1^*\big)^2=\det\big(I-x_2^\sharp y_1^*\overline{y_1}\big)^2 & (d_2=1), \\
\det\big(I-J_2\overline{y_1}x_2^\#y_1^*\big)=\det\big(I-x_2^\#y_1^*J_2\overline{y_1}\big) & (d_2=4). \end{cases}
\end{align*}
(2), (4) Let $x_2\in \Herm(3,\BK')$, $y_1\in M(1,3;\BK')$ with $\BK'=\BR,\BH$.
Then we have $Q(y_1)x_2=\big({}^t\hspace{-1pt}\hat{y}_1y_1\big)\times x_2\in\Herm(3,\BK')$,
where $x\times y$ is as (\ref{Freudenthal_def}), and hence
\begin{gather*} (x_2)^{Q(y_1)x_2}=x_2\big(I-\big(\big({}^t\hspace{-1pt}\hat{y}_1y_1\big)\times x_2\big)x_2\big)^{-1}. \end{gather*}
By (\ref{Freudenthal}), it holds that
\begin{align*}
\big({}^t\hspace{-1pt}\hat{y}_1y_1\big)\times x_2&=\det(x_2)^{-1}\big({}^t\hspace{-1pt}\hat{y}_1y_1\big)\times(x_2^\sharp)^\sharp
 =\det(x_2)^{-1}\big({-}x_2^\sharp{}^t\hspace{-1pt}\hat{y}_1y_1x_2^\sharp +\Re_{\BK'}\Tr(x_2^\sharp{}^t\hspace{-1pt}\hat{y}_1y_1)x_2^\sharp\big) \\
& =-x_2^\sharp{}^t\hspace{-1pt}\hat{y}_1y_1x_2^{-1}+\big(y_1x_2^\sharp{}^t\hspace{-1pt}\hat{y}_1\big)x_2^{-1},
\end{align*}
where $\det(x_2)$ means the determinant polynomial in the sense of Jordan algebras when $\BK'=\BH$, and thus we have
\begin{gather*} I-\big(\big({}^t\hspace{-1pt}\hat{y}_1y_1\big)\times x_2\big)x_2=I+x_2^\sharp{}^t\hspace{-1pt}\hat{y}_1y_1-\big(y_1x_2^\sharp{}^t\hspace{-1pt}\hat{y}_1\big)I. \end{gather*}
Now it holds that
\begin{gather*} \big(I-\big(\big({}^t\hspace{-1pt}\hat{y}_1y_1\big)\times x_2\big)x_2\big)^{-1}
=\big(1-y_1x_2^\sharp{}^t\hspace{-1pt}\hat{y}_1\big)^{-1}\big(I-x_2^\sharp{}^t\hspace{-1pt}\hat{y}_1y_1\big), \end{gather*}
since
\begin{gather*}
 \big(I+x_2^\sharp{}^t\hspace{-1pt}\hat{y}_1y_1-\big(y_1x_2^\sharp{}^t\hspace{-1pt}\hat{y}_1\big)I\big)
\big(I-x_2^\sharp{}^t\hspace{-1pt}\hat{y}_1y_1\big) \\
\qquad{} =I-x_2^\sharp{}^t\hspace{-1pt}\hat{y}_1y_1x_2^\sharp{}^t\hspace{-1pt}\hat{y}_1y_1
-\big(y_1x_2^\sharp{}^t\hspace{-1pt}\hat{y}_1\big)\big(I-x_2^\sharp{}^t\hspace{-1pt}\hat{y}_1y_1\big) \\
\qquad{} =I-\big(y_1x_2^\sharp{}^t\hspace{-1pt}\hat{y}_1\big)x_2^\sharp{}^t\hspace{-1pt}\hat{y}_1y_1
-\big(y_1x_2^\sharp{}^t\hspace{-1pt}\hat{y}_1\big)\big(I-x_2^\sharp{}^t\hspace{-1pt}\hat{y}_1y_1\big)
=\big(1-y_1x_2^\sharp{}^t\hspace{-1pt}\hat{y}_1\big)I.
\end{gather*}
Therefore we get
\begin{align*}
(x_2)^{Q(y_1)x_2}&=x_2\big(I-\big(\big({}^t\hspace{-1pt}\hat{y}_1y_1\big)\times x_2\big)x_2\big)^{-1}
=\big(1-y_1x_2^\sharp{}^t\hspace{-1pt}\hat{y}_1\big)^{-1}x_2\big(I-x_2^\sharp{}^t\hspace{-1pt}\hat{y}_1y_1\big) \\
&=\det(x_2)^{-1}\det\big(\big(I-x_2^\sharp{}^t\hspace{-1pt}\hat{y}_1y_1\big)^{-1}x_2^\sharp\big)
\big(\big(I-x_2^\sharp{}^t\hspace{-1pt}\hat{y}_1y_1\big)^{-1}x_2^\sharp\big)^{-1} \\
&=\det(x_2)^{-1}\big(\big(I-x_2^\sharp{}^t\hspace{-1pt}\hat{y}_1y_1\big)^{-1}x_2^\sharp\big)^\sharp.
\end{align*}
Then by complexifying holomorphically in $x_2$, anti-holomorphically in $y_1$
(as in Section \ref{EJTS} for $\BK'=\BH$), we get the desired formulas.
\end{proof}

Now we set
\begin{align*}
W&=\cP_\bk(\fp^+_2)\otimes\chi^{-\lambda} \\
&\simeq \begin{cases} V_{2(k_1,k_2,k_3)}^{(3)\vee}\otimes\chi_1^{-2\lambda}
\simeq V_{2(0;-k_2-k_3,-k_1-k_3,-k_1-k_2)}^{(1,3)\vee}\otimes\chi_1^{-2\lambda} & (d_2=1), \\
V_{(0,0;-k_2-k_3,-k_2-k_3,-k_1-k_3,-k_1-k_3,-k_1-k_2,-k_1-k_2)}^{(2,6)\vee}\otimes\chi_1^{-\lambda} & (d_2=4), \end{cases}
\end{align*}
and denote $\rK(x_2)=\rK_{(k_1,k_2,k_3)}^{(d_2)}(x_2)$. Then by the previous lemma,
\begin{gather*}
 h(Q(x_2)y_1,y_1)^{-\lambda/2}\rK_{(k_1,k_2,k_3)}^{(d_2)}\big((x_2)^{Q(y_1)x_2}\big) \\
 =\begin{cases} \det\big(I-y_1^*\overline{y_1}x_2^\sharp\big)^{-\lambda}\rK_{(k_1,k_2,k_3)}^{(1)}
\big(\det(x_2)^{-1}\big(\big(I-x_2^\sharp y_1^*\overline{y_1}\big)^{-1}x_2^\sharp\big)^\sharp\big) & (d_2=1),\\
\det\big(I-y_1^*J_2\overline{y_1}x_2^\#\big)^{-\lambda/2}\rK_{(k_1,k_2,k_3)}^{(4)}
\big(\Pf(x_2)^{-1}\big(\big(I-x_2^\#y_1^*J_2\overline{y_1}\big)^{-1}x_2^\#\big)^\#\big) & (d_2=4) \end{cases} \\
 =\begin{cases} \det(x_2)^{-|\bk|}\det\big(I-y_1^*\overline{y_1}x_2^\sharp\big)^{-\lambda}\rK_{(k_1,k_2,k_3)}^{(1)}
\big(\big(\big(I-x_2^\sharp y_1^*\overline{y_1}\big)^{-1}x_2^\sharp\big)^\sharp\big) & (d_2=1),\\
\Pf(x_2)^{-|\bk|}\det\big(I-y_1^*J_2\overline{y_1}x_2^\#\big)^{-\lambda/2}\rK_{(k_1,k_2,k_3)}^{(4)}
\big(\big(\big(I-x_2^\#y_1^*J_2\overline{y_1}\big)^{-1}x_2^\#\big)^\#\big) & (d_2=4). \end{cases}
\end{gather*}
Then since the map $f(x_2)\mapsto f(x_2^{\sharp(\#)})$ yields
$\cP_{(k_1,k_2,k_3)}(\fp^+_2)\to \cP_{(k_1+k_2,k_1+k_3,k_2+k_3)}(\fp^+_2)$,
we write $\rK_{(k_1,k_2,k_3)}^{(d_2)}(x_2^{\sharp(\#)})=\rK_{(k_1+k_2,k_1+k_3,k_2+k_3)}^{(d_2)}(x_2)$. Then we have
\begin{gather*}
 h(Q(x_2)y_1,y_1)^{-\lambda/2}\rK_{(k_1,k_2,k_3)}^{(d_2)}\big((x_2)^{Q(y_1)x_2}\big) \\
 =\begin{cases} \det(x_2)^{-|\bk|}\det\!\big(I-y_1^*\overline{y_1}x_2^\sharp\big)^{-\lambda}\rK_{(k_1+k_2,k_1+k_3,k_2+k_3)}^{(1)}
\big(\big(I-x_2^\sharp y_1^*\overline{y_1}\big)^{-1}x_2^\sharp\big) & (d_2=1), \\
\Pf(x_2)^{-|\bk|}\det\!\big(I-y_1^*J_2\overline{y_1}x_2^\#\big)^{-\lambda/2}\rK_{(k_1+k_2,k_1+k_3,k_2+k_3)}^{(4)}
\big(\big(I-x_2^\#y_1^*J_2\overline{y_1}\big)^{-1}x_2^\#\big)\!\!\!\!\! & (d_2=4). \end{cases}
\end{gather*}
Now for $x_2,z_2\in\fp^+_2=\Sym(3,\BC)$ resp. $\Skew(6,\BC)$ let
\begin{gather*}
\cK_{\bn,(k_1+k_2,k_1+k_3,k_2+k_3)}^{(d_2)}(x_2;z_2)
:=\Proj_{\bn,x}\big({\rm e}^{\frac{1}{\varepsilon}\tr(x_2z_2^*)}\rK_{(k_1+k_2,k_1+k_3,k_2+k_3)}^{(d_2)}(x_2)\big) \\
\hphantom{\cK_{\bn,(k_1+k_2,k_1+k_3,k_2+k_3)}^{(d_2)}(x_2;z_2):=}{} \in\cP\big(\fp^+_2\times\overline{\fp^+_2},\Hom\big(W,\chi^{-\lambda}\big)\big).
\end{gather*}
This is non-zero only if $n_1\ge k_1+k_2$, $n_2\ge k_1+k_3$, $n_3\ge k_2+k_3$ holds.
Then by Corollary \ref{ext_of_hK} we have
\begin{gather*}
 h(Q(x_2)y_1,y_1)^{-\lambda/2}\rK_{(k_1,k_2,k_3)}^{(d_2)}\big((x_2)^{Q(y_1)x_2}\big) \\
\qquad{}=\sum_{\bn\in\BZ_{++}^3}(\lambda+(k_1+k_2,k_1+k_3,k_2+k_3))_{(n_1-k_1-k_2,n_2-k_1-k_3,n_3-k_2-k_3),d_2} \\
\qquad\quad{} \times \begin{cases} \det(x_2)^{-|\bk|}\cK_{\bn,(k_1+k_2,k_1+k_3,k_2+k_3)}^{(1)}\big(x_2^\sharp;{}^t\hspace{-1pt}y_1y_1\big) & (d_2=1), \\
\Pf(x_2)^{-|\bk|}\cK_{\bn,(k_1+k_2,k_1+k_3,k_2+k_3)}^{(4)}\big(x_2^\#;-{}^t\hspace{-1pt}y_1J_2y_1\big) & (d_2=4). \end{cases}
\end{gather*}
We define $\cK_{\bm,\bk}^{(d_2)\prime}(x_2;y_1)\in\cP\big(\fp^+_2\times\overline{\fp^+_1},\Hom\big(W,\chi^{-\lambda}\big)\big)$ by
\begin{gather*}
\cK_{\bm,\bk}^{(1)\prime}(x_2;y_1) :=\det(x_2)^{-|\bk|}
\cK_{\substack{(k_1+k_2+m_3,k_1+k_3+m_2,k_2+k_3+m_1),\quad \\ \hspace{70pt}(k_1+k_2,k_1+k_3,k_2+k_3)}}^{(1)}
\big(x_2^\sharp;{}^t\hspace{-1pt}y_1y_1\big), \\
\cK_{\bm,\bk}^{(4)\prime}(x_2;y_1) :=\Pf(x_2)^{-|\bk|}
\cK_{\substack{(k_1+k_2+m_3,k_1+k_3+m_2,k_2+k_3+m_1),\quad \\ \hspace{70pt}(k_1+k_2,k_1+k_3,k_2+k_3)}}^{(4)}
\big(x_2^\#;-{}^t\hspace{-1pt}y_1J_2y_1\big),
\end{gather*}
so that
\begin{align*}
\sum_{\bm\in(\BZ_{\ge 0})^3}\cK_{\bm,\bk}^{(d_2)\prime}(x_2;y_1)
&=\begin{cases}\det(x_2)^{-|\bk|}{\rm e}^{\tr(x_2^\sharp y_1^*\overline{y_1})}\rK_{(k_1+k_2,k_1+k_3,k_2+k_3)}^{(1)}\big(x_2^\sharp\big) & (d_1=1), \\
\Pf(x_2)^{-|\bk|}{\rm e}^{\frac{1}{2}\tr(x_2^\#y_1^*J_2\overline{y_1})}\rK_{(k_1+k_2,k_1+k_3,k_2+k_3)}^{(4)}\big(x_2^\#\big) & (d_2=4) \end{cases} \\
&=\begin{cases}{\rm e}^{\overline{y_1}x_2^\sharp y_1^*}\rK_{(k_1,k_2,k_3)}^{(1)}(x_2) & (d_1=1), \\
{\rm e}^{\frac{1}{2}\tr(J_2\overline{y_1}x_2^\#y_1^*)}\rK_{(k_1,k_2,k_3)}^{(4)}(x_2) & (d_2=4). \end{cases}
\end{align*}
Then we get
\begin{gather*}
 h(Q(x_2)y_1,y_1)^{-\lambda/2}\rK_{(k_1,k_2,k_3)}^{(d_2)}\big((x_2)^{Q(y_1)x_2}\big) \\
\qquad{} =\sum_{\bm\in(\BZ_{\ge 0})^3}(\lambda+(k_1+k_2,k_1+k_3,k_2+k_3))_{(m_3,m_2,m_1),d_2}\cK_{\bm,\bk}^{(d_2)\prime}(x_2;y_1)
\end{gather*}
Now as a function of $x_2\in\fp^+_2$, $\cK_{\bm,\bk}^{(d_2)\prime}(x_2;y_1)$ sits in
\begin{gather*}
 \cP_{(k_1+m_2+m_3,k_2+m_1+m_3,k_3+m_1+m_2)}(\fp^+_2)
\simeq V_{2(k_1+m_2+m_3,k_2+m_1+m_3,k_3+m_1+m_2)}^{(3)\vee} \\
\qquad{} \simeq V_{2(0;-k_2-k_3-2m_1-m_2-m_3,-k_1-k_3-m_1-2m_2-m_3,-k_1-k_2-m_1-m_2-2m_3)}^{(1,3)\vee} \\
\qquad{} \simeq V_{2(|\bm|;-k_2-k_3-m_1,-k_1-k_3-m_2,-k_1-k_2-m_3)}^{(1,3)\vee}
\end{gather*}
for $d_2=1$ case, and
\begin{gather*}
 \cP_{(k_1+m_2+m_3,k_2+m_1+m_3,k_3+m_1+m_2)}(\fp^+_2)\\
\qquad{} \simeq V_{(0;-k_2-k_3-2m_1-m_2-m_3,-k_1-k_3-m_1-2m_2-m_3,-k_1-k_2-m_1-m_2-2m_3)^2}^{(2,6)\vee}\\
\qquad{} \simeq V_{(|\bm|,|\bm|;-k_2-k_3-m_1,-k_2-k_3-m_1,-k_1-k_3-m_2,-k_1-k_3-m_2,-k_1-k_2-m_3,-k_1-k_2-m_3)}^{(2,6)\vee}
\end{gather*}
for $d_2=4$ case, and by the $K_1$-invariance, as a function of $y_1$, $\cK_{\bm,\bk}^{(d_2)\prime}(x_2;y_1)$ sits in the space with the same lowest weight,
\begin{gather*}
 V_{2(|\bm|;-k_2-k_3-m_1,-k_1-k_3-m_2,-k_1-k_2-m_3)}^{(1,3)\vee} \\
\qquad{} \subset V_{2(|\bm|;0,0,-|\bm|)}^{(1,3)\vee}\otimes V_{2(0;-k_2-k_3,-k_1-k_3,-k_1-k_2)}^{(1,3)\vee}
\simeq \cP_{2|\bm|}(\fp^+_1)\otimes\cP_{(k_1,k_2,k_3)}(\fp^+_2)
\end{gather*}
for $d_2=1$ case, and
\begin{gather*}
 V_{(|\bm|,|\bm|;-k_2-k_3-m_1,-k_2-k_3-m_1,-k_1-k_3-m_2,-k_1-k_3-m_2,-k_1-k_2-m_3,-k_1-k_2-m_3)}^{(2,6)\vee}\\
\qquad{} \subset V_{(|\bm|,|\bm|;0,0,0,0,-|\bm|,-|\bm|)}^{(2,6)\vee}\otimes V_{(0;-k_2-k_3,-k_1-k_3,-k_1-k_2)^2}^{(2,6)\vee}\\
\qquad{} \simeq \cP_{(|\bm|,|\bm|)}(\fp^+_1)\otimes\cP_{(k_1,k_2,k_3)}(\fp^+_2)
\end{gather*}
for $d_2=4$ case. $\cK_{\bm,\bk}^{(d_2)\prime}(x_2;y_1)$ is non-zero only if these inclusions hold, that is,
$0\le m_1\le k_1-k_2$, $0\le m_2\le k_2-k_3$, $0\le m_3$ hold.
Therefore by the result of~\cite{N2} and~(\ref{exp_onD}) we get
\begin{gather*}
 F_{\tau\rho}(x_2;w_1)\\
 = \sum_{\substack{\bm\in(\BZ_{\ge 0})^3\\ 0\le m_j\le k_j-k_{j+1}}} \!\!\!\!
(\lambda+(k_1+k_2,k_1+k_3,k_2+k_3))_{(m_3,m_2,m_1),d_2}
\big\langle {\rm e}^{(y_1|w_1)_{\fp^+_1}}I_W,\cK_{\bm,\bk}^{(d_2)\prime}(x_2;y_1)^*\big\rangle_{\hat{\rho},y_1}\\
 =\begin{cases}\ds \sum_{\substack{\bm\in(\BZ_{\ge 0})^3\\ 0\le m_j\le k_j-k_{j+1}}}\frac{(\lambda+(k_1+k_2,k_1+k_3,k_2+k_3))_{(m_3,m_2,m_1),1}}
{(2\lambda+2(k_1+k_2,k_1+k_3,k_2+k_3))_{2(m_3,m_2,m_1),2}}\cK_{\bm,\bk}^{(1)\prime}(x_2;w_1) & (d_2=1), \\
\ds \sum_{\substack{\bm\in(\BZ_{\ge 0})^3\\ 0\le m_j\le k_j-k_{j+1}}}\frac{(\lambda+(k_1+k_2,k_1+k_3,k_2+k_3))_{(m_3,m_2,m_1),4}}
{(\lambda+(k_1+k_2,k_1+k_3,k_2+k_3)^2)_{(m_3,m_2,m_1)^2,2}}\cK_{\bm,\bk}^{(4)\prime}(x_2;w_1) & (d_2=4), \end{cases}\\
 =\begin{cases}\ds \sum_{\substack{\bm\in(\BZ_{\ge 0})^3\\ 0\le m_j\le k_j-k_{j+1}}}\!\!\!
\frac{1}{\left(\lambda+(k_1+k_2,k_1+k_3,k_2+k_3)+\frac{1}{2}\right)_{(m_3,m_2,m_1),2}}
\cK_{\bm,\bk}^{(1)\prime}\left(\!x_2;\frac{1}{2}w_1\!\right) \hspace{-18pt}& (d_2=1), \\
\ds \sum_{\substack{\bm\in(\BZ_{\ge 0})^3\\ 0\le m_j\le k_j-k_{j+1}}}
\frac{1}{(\lambda+(k_1+k_2,k_1+k_3,k_2+k_3)-1)_{(m_3,m_2,m_1),2}}\cK_{\bm,\bk}^{(4)\prime}(x_2;w_1) & (d_2=4). \end{cases}
\end{gather*}
By Theorem \ref{main}, by substituting $w_1$ with $\overline{\frac{\partial}{\partial x_1}}$, we get the intertwining operator from $(\cH_1)_{\tilde{K}_1}$ to $\cH_{\tilde{K}}$, and by Theorem~\ref{extend},
this extends to the intertwining operator between the spaces of all holomorphic functions if $\cH_1$ is holomorphic discrete. Also, by Theorem~\ref{continuation}, this continues meromorphically for all~$\lambda\in\BC$.
Therefore we have the following.
\begin{Theorem}\label{main3}\quad
\begin{enumerate}\itemsep=0pt
\item[$(1)$] Let $(G,G_1)=(\operatorname{SU}(3,3), \operatorname{SO}^*(6))$, and $\bk\in\BZ_{++}^3$.
Then the linear map
\begin{gather*}
\cF_{\lambda,\bk}\colon \ \cO_\lambda(D_1,V_{2(k_1,k_2,k_3)}^{(3)\vee})\to \cO_\lambda(D), \\
(\cF_{\lambda,\bk}f)(x_1,x_2)=\sum_{\substack{\bm\in(\BZ_{\ge 0})^3\\ 0\le m_j\le k_j-k_{j+1}}}
\frac{1}{\left(\lambda+(k_1+k_2,k_1+k_3,k_2+k_3)+\frac{1}{2}\right)_{(m_3,m_2,m_1),2}}\\
\hphantom{(\cF_{\lambda,\bk}f)(x_1,x_2)=}{} \times \cK_{\bm,\bk}^{(1)\prime}\left(x_2;\frac{1}{2}\overline{\frac{\partial}{\partial x_1}}\right)f(x_1)
\end{gather*}
$(x_1\in \Skew(3,\BC)\simeq M(1,3;\BC),\; x_2\in\Sym(3,\BC))$ intertwines the $\tilde{G}_1$-action.
\item[$(2)$] Let $(G,G_1)=(E_{7(-25)}, \operatorname{SU}(2,6))$ (up to covering), and $\bk\in\BZ_{++}^3$.
Then the linear map
\begin{gather*}
\cF_{\lambda,\bk}\colon \ \cO_\lambda(D_1,V_{(0,0;-k_2-k_3,-k_2-k_3,-k_1-k_3,-k_1-k_3,-k_1-k_2,-k_1-k_2)}^{(2,6)\vee})
\to \cO_\lambda(D), \\
(\cF_{\lambda,\bk}f)(x_1,x_2)=\sum_{\substack{\bm\in(\BZ_{\ge 0})^3\\ 0\le m_j\le k_j-k_{j+1}}}
\frac{1}{(\lambda+(k_1+k_2,k_1+k_3,k_2+k_3)-1)_{(m_3,m_2,m_1),2}}\\
\hphantom{(\cF_{\lambda,\bk}f)(x_1,x_2)=}{} \times \cK_{\bm,\bk}^{(4)\prime}\left(x_2;\overline{\frac{\partial}{\partial x_1}}\right)f(x_1)
\end{gather*}
$(x_1\in M(2,6;\BC),\; x_2\in\Skew(6,\BC))$ intertwines the $\tilde{G}_1$-action.
\end{enumerate}
\end{Theorem}
Especially, if $\bk=(k,k,k)$, $\bm=(0,0,m)$, then we have
\begin{gather*}
 \cK_{\bm,\bk}^{(d_2)\prime}(x_2;y_1) \\
 =\begin{cases} \ds \det(x_2)^{-3k}\det\big(x_2^\sharp\big)^{2k}\tilde{\Phi}_{(m,0,0)}^{(1)}\big(x_2^\sharp y_1^*\overline{y_1}\big)
=\frac{1}{m!}\det(x_2)^k\big(\overline{y_1}x_2^\sharp y_1^*\big)^m & (d_2=1), \\
\ds \Pf(x_2)^{-3k}\Pf\big(x_2^\#\big)^{2k}\tilde{\Phi}_{(m,0,0)}^{(4)}\big(x_2^\#y_1^*J_2\overline{y_1}\big)
=\frac{(-1)^m}{m!}\Pf(x_2)^k\Pf\big(\overline{y_1}x_2^\#y_1^*\big)^m & (d_2=4). \end{cases}
\end{gather*}
Therefore, for $(G,G_1)=(\operatorname{SU}(3,3),\operatorname{SO}^*(6))$ we get
\begin{gather*}
\cF_{\lambda,k}\colon \ \cO_{\lambda+2k}(D_1)\to \cO_\lambda(D), \\
(\cF_{\lambda,k}f)(x_1,x_2)=\det(x_2)^k\sum_{m=0}^\infty \frac{1}{\left(\lambda+2k+\frac{1}{2}\right)_m}\frac{1}{m!}
\left(\frac{1}{4}\frac{\partial}{\partial x_1}x_2^\sharp
{\vphantom{\biggl(}}^t\!\!\left(\frac{\partial}{\partial x_1}\right)\right)^mf(x_1)
\end{gather*}
$(x_1\in M(1,3;\BC), \;x_2 \in \Sym(3,\BC))$, and for $(G,G_1)=(E_{7(-25)},\operatorname{SU}(2,6))$ we get
\begin{gather*}
\cF_{\lambda,k}\colon \ \cO_{\lambda+2k}(D_1)\to \cO_\lambda(D), \\
(\cF_{\lambda,k}f)(x_1,x_2)
=\Pf(x_2)^k\sum_{m=0}^\infty \frac{1}{(\lambda+2k-1)_m}\frac{(-1)^m}{m!}
\Pf\left(\frac{\partial}{\partial x_1}x_2^\#
{\vphantom{\biggl(}}^t\!\!\left(\frac{\partial}{\partial x_1}\right)\right)^mf(x_1)
\end{gather*}
$(x_1\in M(2,6;\BC), \;x_2 \in \Skew(6,\BC))$.

\subsection[$\cF_{\tau\rho}$ for $(G,G_1)=(\operatorname{SU}(s,s), \operatorname{Sp}(s,\BR))$, $(\operatorname{SU}(s,s), \operatorname{SO}^*(2s))$]{$\boldsymbol{\cF_{\tau\rho}}$ for $\boldsymbol{(G,G_1)=(\operatorname{SU}(s,s), \operatorname{Sp}(s,\BR))}$, $\boldsymbol{(\operatorname{SU}(s,s), \operatorname{SO}^*(2s))}$}
In this subsection we set
\begin{gather*} (G,G_1)=\begin{cases}(\operatorname{SU}(s,s), \operatorname{Sp}(s,\BR))& (\text{Case }1),\\
(\operatorname{SU}(s,s), \operatorname{SO}^*(2s))& (\text{Cases }2,3) \end{cases} \end{gather*}
with $s\ge 2$. Then the maximal compact subgroups are $(K,K_1)=(S(U(s)\times U(s)), U(s))$, and~$\fp^+$, $\fp^+_1:=\fg_1^\BC\cap\fp^+$, $\fp^+_2:=(\fp^+_1)^\bot$ are realized as
\begin{gather*} \fp^+=M(s,\BC),\qquad (\fp^+_1,\fp^+_2)=\begin{cases}
(\Sym(s,\BC),\Skew(s,\BC))& (\text{Case }1),\\ (\Skew(s,\BC),\Sym(s,\BC))& (\text{Cases }2,3). \end{cases} \end{gather*}
Let $\chi$, $\chi_1$ be the characters of $K^\BC$, $K_1^\BC$ respectively, normalized as (\ref{char}).
Then for $k\in K_1=U(s)$, $\chi(k)=\det(k)$, $\chi_1(k)=\det(k)^{1/\varepsilon}$ holds, where $\varepsilon=1$ for Case~1,
$\varepsilon=2$ for Cases 2,~3.

Now let $(\tau,V)=\big(\chi^{-\lambda},\BC\big)$ with $\lambda$ sufficiently large,
$W\subset\cP(\fp^+_2)\otimes\chi^{-\lambda}$ be an irreducible $\tilde{K}^\BC_1$-submodule,
and $\rK(x_2)\in \cP(\fp^+_2,\Hom(W,\chi^{-\lambda}))$ be the $\tilde{K}^\BC_1$-invariant polynomial in the sense of (\ref{K-invariance}).
For $x_2\in\fp^+_2$, $w_1\in\fp^+_1$, we want to compute
\begin{align*}
F_{\tau\rho}(x_2;w_1)
&=\big\langle {\rm e}^{(y_1|w_1)_{\fp^+_1}}I_W,
\big(h(Q(x_2)y_1,y_1)^{-\lambda/2}\rK\big((x_2)^{Q(y_1)x_2}\big)\big)^*\big\rangle_{\hat{\rho},y_1} \\
&=\big\langle {\rm e}^{(y_1|w_1)_{\fp^+_1}}I_W,
\big(\det(I-x_2y_1^*x_2y_1^*)^{-\lambda/2}\rK\big(x_2(I-y_1^*x_2y_1^*x_2)^{-1}\big)\big)^*\big\rangle_{\hat{\rho},y_1}.
\end{align*}
Now we consider $W$ of the form
\begin{gather*} W=\begin{cases}\cP_{(\underbrace{\scriptstyle k+1,\ldots,k+1}_l,k,\ldots,k)}(\Skew(s,\BC))\otimes\chi^{-\lambda}
\simeq V_{\langle 2l\rangle}^{(s)\vee}\otimes\chi_1^{-\lambda-k} &\\
\hspace{48pt} \left(k\in\BZ_{\ge 0}\text{ if }s\colon \text{even},\; k=0\text{ if }s\colon \text{odd},\;
l=0,1,\ldots,\left\lceil\frac{s}{2}\right\rceil-1\right)&(\text{Case }1),\\
\cP_{(k+l,k,\ldots,k)}(\Sym(s,\BC))\otimes\chi^{-\lambda}
\simeq V_{(2l,0,\ldots,0)}^{(s)\vee}\otimes\chi_1^{-2\lambda-4k} \qquad (k,l\in\BZ_{\ge 0})&(\text{Case }2),\\
\cP_{(k+l,\ldots,k+l,k)}(\Sym(s,\BC))\otimes\chi^{-\lambda}
\simeq V_{(2l,\ldots,2l,0)}^{(s)\vee}\otimes\chi_1^{-2\lambda-4k} \quad (k,l\in\BZ_{\ge 0})&(\text{Case }3),\end{cases} \end{gather*}
where we denote $(\underbrace{1,\ldots,1}_l,0,\ldots,0)=:\langle l\rangle$.
Then $\cH_{\varepsilon\lambda}(D_1,W)_{\tilde{K}_1}$ becomes multiplicity-free under~$\tilde{K}_1$.
However, when $(G,G_1)=(\operatorname{SU}(3,3),\operatorname{SO}^*(6))$ this list does not exhaust all $\tilde{K}_1$-multipli\-city-free submodules of $\cH_\lambda(D)$.
For this pair see Theorem \ref{main3}(1).
We write the polynomial $\rK(x_2)$ as
\begin{alignat*}{3}
&\rK(x_2)=\Pf(x_2)^k \rK_{\langle l\rangle}^{(4)}(x_2)\qquad & &(\text{Case }1),&\\
&\rK(x_2)=\det(x_2)^k \rK_{(l,0,\ldots,0)}^{(1)}(x_2)\qquad & &(\text{Case }2),&\\
&\rK(x_2)=\det(x_2)^k \rK_{(l,\ldots,l,0)}^{(1)}(x_2)\qquad & &(\text{Case }3).&
\end{alignat*}
Let $W_{(l)}'$ stand for $V_{\langle 2l\rangle}^{(s)\vee}$, $V_{(2l,0,\ldots,0)}^{(s)\vee}$ and $V_{(2l,\ldots,2l,0)}^{(s)\vee}$,
and let $\rK_{(l)}'(x_2)$ stand for $\rK_{\langle l\rangle}^{(4)}(x_2)$, \linebreak $\rK_{(l,0,\ldots,0)}^{(1)}(x_2)$ and
$\rK_{(l,\ldots,l,0)}^{(1)}(x_2)$ respectively, so that
\begin{gather*}
 F_{\tau\rho}(x_2;w_1)=\det(x_2)^{k\varepsilon/2}\\
 \qquad{}\times\big\langle {\rm e}^{(y_1|w_1)_{\fp^+_1}}I_W,
\big(\det(I-x_2y_1^*x_2y_1^*)^{-(\lambda+k\varepsilon)/2}\rK_{(l)}'\big(x_2(I-y_1^*x_2y_1^*x_2)^{-1}\big)\big)^*\big\rangle_{\hat{\rho},y_1}.
\end{gather*}
Now we put $\lambda+k\varepsilon=:\mu$, and find the expansion of $\det(I-x_2y_1^*x_2y_1^*)^{-\mu/2}\rK_{(l)}'\big(x_2(I-y_1^*x_2y_1^*x_2)^{-1}\big)$
by using the expansion of ${\rm e}^{\frac{1}{2}\tr(x_2y_1^*x_2y_1^*)}\rK_{(l)}'(x_2)$.
Since $\cP(\fp^+_2)$ and $\cP(\fp^+_1)\otimes W_{(l)}'$ are decomposed under $K_1$ as
\begin{alignat*}{3}
& \cP(\fp^+_2) \simeq\bigoplus_{\bm\in\BZ_{++}^{\lfloor s/2\rfloor}}\cP_\bm(\Skew(s,\BC))
\simeq\bigoplus_{\bm\in\BZ_{++}^{\lfloor s/2\rfloor}}V_{\bm^2}^{(s)\vee} \quad & &(\text{Case }1),& \\
& \cP(\fp^+_2) \simeq\bigoplus_{\bm\in\BZ_{++}^s}\cP_\bm(\Sym(s,\BC))
\simeq\bigoplus_{\bm\in\BZ_{++}^s}V_{2\bm}^{(s)\vee} \quad & &(\text{Cases }2,3), \\
& \cP(\fp^+_1)\otimes W_{(l)}'
 \simeq \bigoplus_{\bm\in\BZ_{++}^s}V_{2\bm}^{(s)\vee}\otimes V_{\langle 2l\rangle}^{(s)\vee}
\simeq \bigoplus_{\bm\in\BZ_{++}^s}\bigoplus_{\substack{\bl\in\{0,1\}^s,\; |\bl|=2l\\ \bm+\bl\in\BZ_{++}^s}}V_{2\bm+\bl}^{(s)\vee}\quad &&(\text{Case }1),& \\
& \cP(\fp^+_1)\otimes W_{(l)}' \simeq \bigoplus_{\bm\in\BZ_{++}^{\lfloor s/2\rfloor}}V_{\bm^2}^{(s)\vee}\otimes V_{(2l,0,\ldots,0)}^{(s)\vee}&&& \\
&\hphantom{\cP(\fp^+_1)\otimes W_{(l)}'}{} \simeq \bigoplus_{\bm\in\BZ_{++}^{\lfloor s/2\rfloor}}
\bigoplus_{\substack{\bl\in(\BZ_{\ge 0})^{\lceil s/2\rceil},\; |\bl|=2l\\ 0\le l_j\le m_{j-1}-m_j}}
V_{\substack{(m_1+l_1,m_1,m_2+l_2,m_2,\ldots,\hspace{30pt}\\ m_{\lfloor s/2\rfloor}+l_{\lfloor s/2\rfloor},m_{\lfloor s/2\rfloor}
(,l_{\lceil s/2\rceil}))}}^{(s)\vee} \quad & &(\text{Case }2), & \\
&\cP(\fp^+_1)\otimes W_{(l)}' \simeq \bigoplus_{\bm\in\BZ_{++}^{\lfloor s/2\rfloor}}V_{\bm^2}^{(s)\vee}\otimes V_{(2l,\ldots,2l,0)}^{(s)\vee}&&&\\
&\hphantom{\cP(\fp^+_1)\otimes W_{(l)}'}{} \simeq \bigoplus_{\bm\in\BZ_{++}^{\lfloor s/2\rfloor}}
\bigoplus_{\substack{\bl\in(\BZ_{\ge 0})^{\lceil s/2\rceil},\; |\bl|=2l\\ 0\le l_j\le m_j-m_{j+1}\\
0\le l_{\lceil s/2\rceil} \text{ if }s\colon \text{odd}}}
V_{\substack{(m_1+2l,m_1+2l-l_1,m_2+2l,m_2+2l-l_2,\ldots,\hspace{10pt}\\
m_{\lfloor s/2\rfloor}+2l,m_{\lfloor s/2\rfloor}+2l-l_{\lfloor s/2\rfloor}(,2l-l_{\lceil s/2\rceil}))}}^{(s)\vee} \quad\! && (\text{Case }3),&
\end{alignat*}
and since ${\rm e}^{\frac{1}{2}\tr(x_2y_1^*x_2y_1^*)}\rK_{(l)}'(x_2)$ is $K_1^\BC={\rm GL}(s,\BC)$-invariant, this is decomposed as
\begin{gather*} {\rm e}^{\frac{1}{2}\tr(x_2y_1^*x_2y_1^*)}\rK_{(l)}'(x_2)=\sum_{W''} \cK_{W''}(x_2;y_1), \end{gather*}
where $\cK_{W''}(x_2;y_1)\in (W''\boxtimes \overline{W''})^{K_1^\BC}$, and $W''$ runs over all irreducible $K_1^\BC$-modules
which appear commonly in $\cP(\fp^+_2)$ and $\cP(\fp^+_1)\otimes W_{(l)}'$, namely,
\begin{align}
{\rm e}^{\frac{1}{2}\tr(x_2y_1^*x_2y_1^*)}\rK_{(l)}'(x_2)&=\sum_{\bm\in\BZ_{++}^{\lfloor s/2\rfloor}}\sum_\bl \cK_{\bm,\bl}(x_2;y_1) \notag \\
&=\begin{cases}
\ds \sum_{\bm\in\BZ_{++}^{\lfloor s/2\rfloor}}
\sum_{\substack{\bl\in\{0,1\}^{\lfloor s/2\rfloor},\; |\bl|=l\\ \bm+\bl\in\BZ_{++}^{\lfloor s/2\rfloor}}}
\cK_{\bm,\bl}^{(4,1)}(x_2;y_1) & (\text{Case }1), \\
\ds \sum_{\bm\in\BZ_{++}^{\lfloor s/2\rfloor}}
\sum_{\substack{\bl\in(\BZ_{\ge 0})^{\lceil s/2\rceil},\; |\bl|=l\\ 0\le l_j\le m_{j-1}-m_j}}
\cK_{\bm,\bl}^{(1,4)}(x_2;y_1) & (\text{Case }2), \\
\ds \sum_{\bm\in\BZ_{++}^{\lfloor s/2\rfloor}}
\sum_{\substack{\bl\in(\BZ_{\ge 0})^{\lceil s/2\rceil},\; |\bl|=l\\
0\le l_j\le m_j-m_{j+1}\\ 0\le l_{\lceil s/2\rceil} \text{ if }s\colon \text{odd}}}
\cK_{\bm,-\bl}^{(1,4)}(x_2;y_1) & (\text{Case }3),
\end{cases} \label{indexl}
\end{align}
where for Cases 2, 3, we put $m_0=+\infty$, $m_{\lfloor s/2\rfloor+1}=0$,
and $\cK_{\bm,\bl}(x_2;y_1)$ sits in
\begin{alignat*}{3}
&\left(V_{(2\bm+\bl)^2}^{(s)\vee}\otimes \overline{V_{(2\bm+\bl)^2}^{(s)\vee}}\right)^{K_1^\BC}&&&\\
&\quad {} \subset \cP_{2\bm+\bl}(\Skew(s,\BC))_{x_2}\otimes\overline{\cP_{\bm^2}(\Sym(s,\BC))_{y_1}\otimes V_{\langle 2l\rangle}^{(s)\vee}}
& &(\text{Case }1),& \\
&\left(V_{\substack{2(m_1+l_1,m_1,m_2+l_2,m_2,\ldots,\hspace{15pt}\\ \quad m_{\lfloor s/2\rfloor}+l_{\lfloor s/2\rfloor},m_{\lfloor s/2\rfloor}
(,l_{\lceil s/2\rceil}))}}^{(s)\vee}\otimes \overline{
V_{\substack{2(m_1+l_1,m_1,m_2+l_2,m_2,\ldots,\hspace{15pt}\\ \quad m_{\lfloor s/2\rfloor}+l_{\lfloor s/2\rfloor},m_{\lfloor s/2\rfloor}
(,l_{\lceil s/2\rceil}))}}^{(s)\vee}}\right)^{K_1^\BC}&&&\\
&\quad{} \subset\cP_{\substack{(m_1+l_1,m_1,m_2+l_2,m_2,\ldots,\hspace{20pt}\\ \; m_{\lfloor s/2\rfloor}+l_{\lfloor s/2\rfloor},m_{\lfloor s/2\rfloor}
(,l_{\lceil s/2\rceil}))}}(\Sym(s,\BC))_{x_2}
\otimes\overline{\cP_{2\bm}(\Skew(s,\BC))_{y_1}\otimes V_{(2l,0,\ldots,0)}^{(s)\vee}}\ && (\text{Case }2),&\\
&\left(V_{\substack{2(m_1+l,m_1+l-l_1,m_2+l,m_2+l-l_2,\ldots,\hspace{10pt}\\ \quad
m_{\lfloor s/2\rfloor}+l,m_{\lfloor s/2\rfloor}+l-l_{\lfloor s/2\rfloor}(,l-l_{\lceil s/2\rceil}))}}^{(s)\vee}\otimes\overline{
V_{\substack{2(m_1+l,m_1+l-l_1,m_2+l,m_2+l-l_2,\ldots,\hspace{10pt}\\ \quad
m_{\lfloor s/2\rfloor}+l,m_{\lfloor s/2\rfloor}+l-l_{\lfloor s/2\rfloor}(,l-l_{\lceil s/2\rceil}))}}^{(s)\vee}}\right)^{K_1^\BC}\hspace{-30pt} &&&\\
&\quad {} \subset \cP_{\substack{(m_1+l,m_1+l-l_1,m_2+l,m_2+l-l_2,\ldots,\hspace{15pt}\\ \;
m_{\lfloor s/2\rfloor}+l,m_{\lfloor s/2\rfloor}+l-l_{\lfloor s/2\rfloor}(,l-l_{\lceil s/2\rceil}))}}(\Sym(s,\BC))_{x_2}&&&\\
& \quad\quad {} \otimes\overline{\cP_{2\bm}(\Skew(s,\BC))_{y_1}\otimes V_{(2l,\ldots,2l,0)}^{(s)\vee}} && (\text{Case }3).&
\end{alignat*}
Then $\det(I-x_2y_1^*x_2y_1^*)^{-\mu/2}\rK_{(l)}'\big(x_2(I-y_1^*x_2y_1^*x_2)^{-1}\big)$ is expanded as follows.
\begin{Proposition}\label{exp_susp}\quad
\begin{enumerate}\itemsep=0pt
\item[$(1)$] For Case~{\rm 1},
\begin{gather*}
 \det(I-x_2y_1^*x_2y_1^*)^{-\mu/2}\rK_{\langle l\rangle}^{(4)}\big(x_2(I-y_1^*x_2y_1^*x_2)^{-1}\big)\\
 =\sum_{\bm\in\BZ_{++}^{\lfloor s/2\rfloor}}\sum_\bl \frac{(\mu)_{\bm+\bl,2}}{(\mu)_{\langle l\rangle,2}}\cK_{\bm,\bl}^{(4,1)}(x_2;y_1)
=\sum_{\bm\in\BZ_{++}^{\lfloor s/2\rfloor}}\sum_\bl (\mu+\langle l\rangle)_{\bm+\bl-\langle l\rangle,2}\cK_{\bm,\bl}^{(4,1)}(x_2;y_1).
\end{gather*}
\item[$(2)$] For Case~{\rm 2},
\begin{gather*}
 \det(I-x_2y_1^*x_2y_1^*)^{-\mu/2}\rK_{(l,0,\ldots,0)}^{(1)}\big(x_2(I-y_1^*x_2y_1^*x_2)^{-1}\big)\\
 \qquad{}=\sum_{\bm\in\BZ_{++}^{\lfloor s/2\rfloor}}\sum_\bl \frac{(\mu)_{\bm+\bl,2}}{(\mu)_{(l,0,\ldots,0),2}}\cK_{\bm,\bl}^{(1,4)}(x_2;y_1)\\
 \qquad{} =\sum_{\bm\in\BZ_{++}^{\lfloor s/2\rfloor}}\sum_\bl (\mu+(l,0,\ldots,0))_{\bm+\bl-(l,0,\ldots,0),2}\cK_{\bm,\bl}^{(1,4)}(x_2;y_1),
\end{gather*}
where we identify $\bm=(m_1,\ldots,m_{\lfloor s/2\rfloor})\in\BZ_{++}^{\lfloor s/2\rfloor}$ with
$(m_1,\ldots,m_{\lfloor s/2\rfloor},0)\in\BZ_{++}^{\lceil s/2\rceil}$ if $s$ is odd.
\item[$(3)$] For Case~{\rm 3} with $s$ even,
\begin{gather*}
 \det(I-x_2y_1^*x_2y_1^*)^{-\mu/2}\rK_{(\underbrace{\scriptstyle l,\ldots,l}_{s-1},0)}^{(1)}\big(x_2(I-y_1^*x_2y_1^*x_2)^{-1}\big)\\
\qquad{} =\sum_{\bm\in\BZ_{++}^{s/2}}\sum_\bl \frac{(\mu+l)_{\bm-\bl+(\smash{\overbrace{\scriptstyle l,\ldots,l}^{s/2}}),2}}
{(\mu+l)_{(\underbrace{\scriptstyle l,\ldots,l}_{s/2-1},0),2}}\cK_{\bm,-\bl}^{(1,4)}(x_2;y_1)\\
\qquad{} =\sum_{\bm\in\BZ_{++}^{s/2}}\sum_\bl (\mu+l+(\underbrace{l,\ldots,l}_{s/2-1},0))_{\bm-\bl+(\underbrace{\scriptstyle 0,\ldots,0}_{s/2-1},l),2}
\cK_{\bm,-\bl}^{(1,4)}(x_2;y_1).
\end{gather*}
\item[$(4)$] For Case~{\rm 3} with $s$ odd, there exist monic polynomials $\varphi_{\bm,-\bl}(\mu)\in\BC[\mu]$ of degree $l-l_{\lceil s/2\rceil}$
such that
\begin{gather*}
 \det(I-x_2y_1^*x_2y_1^*)^{-\mu/2}\rK_{(\underbrace{\scriptstyle l,\ldots,l}_{s-1},0)}^{(1)}\big(x_2(I-y_1^*x_2y_1^*x_2)^{-1}\big)\\
\qquad{} =\sum_{\bm\in\BZ_{++}^{\lfloor s/2\rfloor}}\sum_\bl
\frac{(\mu+l)_{\bm-\bl'+\smash{(\overbrace{\scriptstyle l,\dots,l}^{\lfloor s/2\rfloor}}),2}\varphi_{\bm,-\bl}(\mu)}
{(\mu+l)_{(\underbrace{\scriptstyle l,\ldots,l}_{\lfloor s/2\rfloor}),2}}
\cK_{\bm,-\bl}^{(1,4)}(x_2;y_1)\\
\qquad{} =\sum_{\bm\in\BZ_{++}^{\lfloor s/2\rfloor}}\sum_\bl (\mu+2l)_{\bm-\bl',2}\varphi_{\bm,-\bl}(\mu)\cK_{\bm,-\bl}^{(1,4)}(x_2;y_1),
\end{gather*}
where for $\bl=(l_1,\ldots,l_{\lfloor s/2\rfloor},l_{\lceil s/2\rceil})\in(\BZ_{\ge 0})^{\lceil s/2\rceil}$, we put
$\bl'=(l_1,\ldots,l_{\lfloor s/2\rfloor})\in(\BZ_{\ge 0})^{\lfloor s/2\rfloor}$.
\end{enumerate}
\end{Proposition}
First we prove Proposition \ref{exp_susp} when $s$ is even. Let $s=2r$. We realize
\begin{gather*} \overline{W_{(l)}'}=\begin{cases}
\overline{V_{\langle 2l\rangle}^{(2r)\vee}}\simeq \overline{\cP_{\langle l\rangle}(\Skew(2r,\BC))}&(\text{Case }1),\\
\overline{V_{(2l,0,\ldots,0)}^{(2r)\vee}}\simeq \overline{\cP_{(l,0,\ldots,0)}(\Sym(2r,\BC))}&(\text{Case }2),\\
\overline{V_{(2l,\ldots,2l,0)}^{(2r)\vee}}\simeq \overline{\cP_{(l,\ldots,l,0)}(\Sym(2r,\BC))}&(\text{Case }3)\end{cases} \end{gather*}
as a space of polynomials in $w_2$, and write $\rK_{(l)}'(x_2)=\rK_{(l)}'(x_2,w_2)\in\cP(\fp^+_2\times\overline{\fp^+_2})$,
$\cK_{\bm,\bl}(x_2;y_1)=\cK_{\bm,\bl}(x_2;y_1,w_2)\in\cP(\fp^+_2\times\overline{\fp^+_1\times\fp^+_2})$.
Now we define $\Rest\colon \cP(\fp^+_2)\to \cP(M(r,\BC))$ by
\begin{gather*} (\Rest f)(x):=f\begin{pmatrix} 0&x\\ \mp {}^t\hspace{-1pt}x&0 \end{pmatrix}. \end{gather*}
We find $\Rest(W_{(l)}')$.
\begin{Lemma}\quad
\begin{enumerate}\itemsep=0pt
\item[$(1)$] For Case~{\rm 1}, $\Rest(\cP_{\langle l\rangle}(\Skew(2r,\BC)))=\cP_{\langle l\rangle}(M(r,\BC))$.
\item[$(2)$] For Case~{\rm 2}, $\Rest(\cP_{(l,0,\ldots,0)}(\Sym(2r,\BC)))=\cP_{(l,0,\ldots,0)}(M(r,\BC))$. \vspace{2mm}

\item[$(3)$] For Case~{\rm 3}, $\Rest(\cP_{(\smash{\overbrace{\scriptstyle l,\ldots,l}^{2r-1}},0)}(\Sym(2r,\BC)))
=\cP_{(\smash{\overbrace{\scriptstyle 2l,\ldots,2l}^{r-1}},l)}(M(r,\BC))$.
\end{enumerate}
\end{Lemma}
\begin{proof}
Since $W_{(l)}'$ and $\cP(M(r,\BC))$ are decomposed under $U(r)\times U(r)$ as
\begin{gather*}
\cP_{\langle l\rangle}(\Skew(2r,\BC))\simeq V_{\langle 2l\rangle}^{(2r)\vee}
\simeq \bigoplus_{k=0}^{2l} V_{\langle k\rangle}^{(r)\vee}\boxtimes V_{\langle 2l-k\rangle}^{(r)},\\
\cP_{(l,0,\ldots,0)}(\Sym(2r,\BC))\simeq V_{(2l,0,\ldots,0)}^{(2r)\vee}
\simeq \bigoplus_{k=0}^{2l} V_{(k,0,\ldots,0)}^{(r)\vee}\boxtimes V_{(2l-k,0,\ldots,0)}^{(r)},\\
\cP_{(\underbrace{\scriptstyle l,\ldots,l}_{2r-1},0)}(\Sym(2r,\BC))\simeq V_{(\underbrace{\scriptstyle 2l,\ldots,2l}_{2r-1},0)}^{(2r)\vee}
\simeq \bigoplus_{k=0}^{2l} V_{(\underbrace{\scriptstyle 2l,\ldots,2l}_{r-1},k)}^{(r)\vee}
\boxtimes V_{(\underbrace{\scriptstyle 2l,\ldots,2l}_{r-1},2l-k)}^{(r)},\\
\cP(M(r,\BC))\simeq \bigoplus_{\bm\in\BZ_{++}^r}\cP_\bm(M(r,\BC))
\simeq \bigoplus_{\bm\in\BZ_{++}^r} V_\bm^{(r)\vee}\boxtimes V_\bm^{(r)},
\end{gather*}
and $\Rest(W_{(l)}')$ appears commonly in both decomposition, only $k=l$ component remains.
\end{proof}

Now we write $\Rest(W_{(l)}')=:W_{(l)}''$, $\rK_{(l)}'\!\left(\!\Big(\begin{smallmatrix} 0&x\\\mp{}^t\hspace{-1pt}x&0 \end{smallmatrix}\Big),
\Big(\begin{smallmatrix} 0&y\\\mp{}^t\hspace{-1pt}y&0 \end{smallmatrix}\Big)\!\right)=:\rK_{(l)}''(x,y)
\!\in\! (W_{(l)}''\boxtimes\overline{W_{(l)}''})^{U(r)\times U(r)}$.
Next we consider $\cP(M(r,\BC)\oplus M(r,\BC))$, on which $U(r)\times U(r)\times U(r)=:K_{xL}\times K_{zL}\times K_R$ acts by
\begin{gather*} f(x,z)\mapsto f\big(k_{xL}^{-1}xk_R,k_{zL}^{-1}zk_R\big) \qquad ((k_{xL},k_{zL},k_R)\in K_{xL}\times K_{zL}\times K_R). \end{gather*}
Under this action we expand ${\rm e}^{\tr(xy^*)}\rK_{(l)}''(z,w)$ as
\begin{gather*} {\rm e}^{\tr(xy^*)}\rK_{(l)}''(z,w)=\sum_{\bm\in\BZ_{++}^r}\sum_\bl \bK_{\bm,\bl}^{(2)}(x,z;y,w), \end{gather*}
where $\bl$ runs over the same sets as (\ref{indexl}), and as functions on $(x,z)$, $\bK_{\bm,\bl}^{(2)}(x,z;y,w)$ sits in
\begin{alignat*}{3}
&V_{\bm,xL}^{(r)\vee}\boxtimes V_{\langle l\rangle,zL}^{(r)\vee}\boxtimes V_{\bm+\bl,R}^{(r)}
\subset \cP_\bm(M(r,\BC))_x\otimes\cP_{\langle l\rangle}(M(r,\BC))_z\quad & &(\text{Case }1),& \\
&V_{\bm,xL}^{(r)\vee}\boxtimes V_{(l,0,\ldots,0),zL}^{(r)\vee}\boxtimes V_{\bm+\bl,R}^{(r)}
\subset \cP_\bm(M(r,\BC))_x\otimes\cP_{(l,0,\ldots,0)}(M(r,\BC))_z \quad & &(\text{Case }2),& \\
&V_{\bm,xL}^{(r)\vee}\boxtimes V_{(2l,\ldots,2l,l),zL}^{(r)\vee}\boxtimes V_{\bm-\bl+2l,R}^{(r)}
\subset \cP_\bm(M(r,\BC))_x\otimes\cP_{(2l,\ldots,2l,l)}(M(r,\BC))_z \quad & &(\text{Case }3).&
\end{alignat*}
Now we claim the following.
\begin{Lemma}\label{lemma_suss}
$\cK_{\bm,\bl}\left(\Big(\begin{smallmatrix}0&x\\ \mp{}^t\hspace{-1pt}x&0\end{smallmatrix}\Big);
\Big(\begin{smallmatrix}0&y\\ \pm{}^t\hspace{-1pt}y&0\end{smallmatrix}\Big),\Big(\begin{smallmatrix}0&w\\ \mp{}^t\hspace{-1pt}w&0\end{smallmatrix}\Big)\right)
=\bK_{\bm,\bl}^{(2)}(x,x;yx^*y,w)$ holds, and these are linearly independent.
\end{Lemma}
To prove this, we prepare some notations. For $\bm\in\BZ_{++}^r$, as (\ref{Schur}) let
\begin{gather}\label{Schur2}
\tilde{\Phi}_\bm^{(2)}(t_1,\ldots,t_r)=\frac{\prod\limits_{i<j}(m_i-m_j-i+j)}{\prod\limits_{i=1}^r (m_i+r-i)!}
\frac{\det\big(\big(t_i^{m_j+r-j}\big)_{i,j}\big)}{\det\big(\big(t_i^{r-j}\big)_{i,j}\big)}
\end{gather}
be the renormalized Schur polynomial, so that for $x\in M(r,\BC)$, $\tilde{\Phi}_\bm^{(2)}(x)=\tilde{\Phi}_\bm^{(2)}(t_1,\ldots,t_r)
\in\cP_\bm(M(r,\BC))^{\Delta U(r)}$ holds, where $t_1,\ldots,t_r$ are the eigenvalues of $x$.
Next, for $(y,z)\in \Sym(2r,\BC)\times\Skew(2r,\BC)$ (resp.~$\Sym(2r+1,\BC)\times\Skew(2r+1,\BC)$), when the eigenvalues of $yz$ are
$t_1,-t_1,t_2,\allowbreak-t_2,\ldots,t_r,-t_r$ (resp.~$t_1,-t_1,t_2,-t_2,\ldots,t_r,-t_r,0$), we define
\begin{gather}\label{Schur3}
\tilde{\Phi}_\bm^{(2)\prime}\big((yz)^2\big)=\tilde{\Phi}_\bm^{(2)\prime}\big((zy)^2\big):=\tilde{\Phi}_\bm^{(2)}\big(t_1^2,\ldots,t_r^2\big).
\end{gather}
For a while we realize $\operatorname{Sp}(r,\BC)$, $O(2r,\BC)$ as all linear automorphisms on $\BC^{2r}$
preserving the bilinear forms $\left(\begin{smallmatrix}0&I\\-I&0\end{smallmatrix}\right)$ and
$\left(\begin{smallmatrix}0&I\\I&0\end{smallmatrix}\right)$ respectively. Then the following holds.
\begin{Lemma}\quad
\begin{enumerate}\itemsep=0pt
\item[$(1)$] For $x\in\Sym(2r,\BC)$,
\begin{gather*} \tilde{\Phi}_\bm^{(2)\prime}\left(\left(x\begin{pmatrix}0&I\\-I&0\end{pmatrix}\right)^2\right)
\in\cP_{\bm^2}(\Sym(2r,\BC))^{\operatorname{Sp}(r,\BC)}\simeq \big(V_{(2\bm)^2}^{(2r)\vee}\big)^{\operatorname{Sp}(r,\BC)}. \end{gather*}
\item[$(2)$] For $x\in\Skew(2r,\BC)$,
\begin{gather*} \tilde{\Phi}_\bm^{(2)\prime}\left(\left(x\begin{pmatrix}0&I\\I&0\end{pmatrix}\right)^2\right)
\in\cP_{2\bm}(\Skew(2r,\BC))^{O(2r,\BC)}\simeq \big(V_{(2\bm)^2}^{(2r)\vee}\big)^{O(2r,\BC)}. \end{gather*}
\end{enumerate}
\end{Lemma}
\begin{proof}
(1) Follows from \cite[Proposition 7.6]{Z} with $(G,K,H,L)=(\operatorname{Sp}(2r,\BR),U(2r),\operatorname{Sp}(r,\BC),\allowbreak \operatorname{Sp}(r))$.
(2) Follows from \cite[Proposition 8.3]{Z} with $(G,K,H,L)=(\operatorname{SO}^*(4r),U(2r),\operatorname{SO}(2r,\BC),\allowbreak \operatorname{SO}(2r))$.
\end{proof}

\begin{proof}[Proof of Lemma \ref{lemma_suss}]
We define two linear maps $\alpha$, $\beta$ by
\begin{align*}
&\alpha\colon \ \cP(M(r,\BC))\boxtimes\overline{\cP(M(r,\BC))\otimes\cP(M(r,\BC))}\longrightarrow \overline{\cP(M(r,\BC))},&
&f(x;y,w)\mapsto f(I;y,I),\\
&\beta\colon \ \cP(M(r,\BC))\boxtimes\overline{\cP(M(r,\BC))\otimes\cP(M(r,\BC))}\longrightarrow \cP(M(r,\BC)),&
&f(x;y,w)\mapsto f(x;I,x^*).
\end{align*}
Then by the $U(r)\times U(r)\times U(r)$-invariance of $\bK_{\bm,\bl}^{(2)}$, we have
\begin{gather*} \bK_{\bm,\bl}^{(2)}(I,I;y,I)\in \overline{\cP_\bm(M(r,\BC))^{\Delta U(r)}}=\overline{\BC\tilde{\Phi}_\bm^{(2)}(y)}, \end{gather*}
where $\tilde{\Phi}_\bm^{(2)}(y)$ is defined in (\ref{Phitilde1}), (\ref{Schur}). Therefore it holds that
\begin{gather*} \alpha\big(\bK_{\bm,\bl}^{(2)}(x,x;yx^*y,w)\big)=\bK_{\bm,\bl}^{(2)}\big(I,I;y^2,I\big)\in\overline{\BC\tilde{\Phi}_\bm^{(2)}\big(y^2\big)}. \end{gather*}
Next, as a function of $x$, under the action of $K_L=\Delta U(r)\subset K_{xL}\times K_{zL}$ and $K_R$,
$\bK_{\bm,\bl}^{(2)}(x,x;y,y)$ sits in
\begin{gather*} \bK_{\bm,\bl}^{(2)}(x,x;y,y)\in\begin{cases}
\big(V_{\bm}^{(r)\vee}\otimes V_{\langle l\rangle}^{(r)\vee}\big)_L\boxtimes V_{\bm+\bl,R}^{(r)} & (\text{Case }1),\\
\big(V_{\bm}^{(r)\vee}\otimes V_{(l,0,\ldots,0)}^{(r)\vee}\big)_L\boxtimes V_{\bm+\bl,R}^{(r)} & (\text{Case }2),\\
\big(V_{\bm}^{(r)\vee}\otimes V_{(2l,\ldots,2l,l)}^{(r)\vee}\big)_L\boxtimes V_{\bm-\bl+2l,R}^{(r)} & (\text{Case }3), \end{cases} \end{gather*}
and also in $\cP(M(r,\BC))_x$. Hence it holds that
\begin{gather*} \bK_{\bm,\bl}^{(2)}(x,x;y,y)\in \begin{cases}
\big(\cP_{\bm+\bl}(M(r,\BC))_x\boxtimes\overline{\cP_{\bm+\bl}(M(r,\BC))_y}\big)^{\Delta U(r)}\\
\hspace{105pt} {}=\BC \bK_{\bm+\bl}^{(2)}(x,y)\quad (\text{Cases 1, 2}),\\
\big(\cP_{\bm-\bl+2l}(M(r,\BC))_x\boxtimes\overline{\cP_{\bm-\bl+2l}(M(r,\BC))_y}\big)^{\Delta U(r)}\\
\hspace{105pt} {}=\BC \bK_{\bm-\bl+2l}^{(2)}(x,y)\quad (\text{Case }3). \end{cases} \end{gather*}
Therefore we get
\begin{gather*} \beta\big(\bK_{\bm,\bl}^{(2)}(x,x;yx^*y,w)\big)\\
\qquad{} =\bK_{\bm,\bl}^{(2)}(x,x;x^*,x^*)\in \begin{cases}
\BC \bK_{\bm+\bl}^{(2)}(x,x^*)=\BC\tilde{\Phi}_{\bm+\bl}^{(2)}\big(x^2\big) \quad & (\text{Cases 1, 2}),\\
\BC \bK_{\bm-\bl+2l}^{(2)}(x,x^*)=\BC\tilde{\Phi}_{\bm-\bl+2l}^{(2)}\big(x^2\big) \quad & (\text{Case }3). \end{cases} \end{gather*}
Next we consider $\cK_{\bm,\bl}(x_2;y_1,w_2)$. We have
\begin{gather*}
\cK_{\bm,\bl}\left(\begin{pmatrix}0&I\\-I&0\end{pmatrix};y_1,\begin{pmatrix}0&I\\-I&0\end{pmatrix}\right)
 \in\overline{\cP_{\bm^2}(\Sym(2r,\BC))^{\operatorname{Sp}(r,\BC)}}\\
 \qquad{}=\overline{\BC\tilde{\Phi}_\bm^{(2)\prime}
\left(\left(y_1\begin{pmatrix}0&I\\-I&0\end{pmatrix}\right)^2\right)} \quad (\text{Case 1}),\\
\cK_{\bm,\bl}\left(\begin{pmatrix}0&I\\I&0\end{pmatrix};y_1,\begin{pmatrix}0&I\\I&0\end{pmatrix}\right)
 \in\overline{\cP_{2\bm}(\Skew(2r,\BC))^{O(2r,\BC)}}\\
 \qquad{} =\overline{\BC\tilde{\Phi}_\bm^{(2)\prime}
\left(\left(y_1\begin{pmatrix}0&I\\I&0\end{pmatrix}\right)^2\right)} \quad (\text{Cases 2, 3}).
\end{gather*}
Especially we get
\begin{gather*}
 \alpha\left(\cK_{\bm,\bl}\left(\begin{pmatrix}0&x\\ \mp{}^t\hspace{-1pt}x&0\end{pmatrix};
\begin{pmatrix}0&y\\ \pm{}^t\hspace{-1pt}y&0\end{pmatrix},\begin{pmatrix}0&w\\ \mp{}^t\hspace{-1pt}w&0\end{pmatrix}\right)\right)\\
\qquad{} =\cK_{\bm,\bl}\left(\begin{pmatrix}0&I\\ \mp I&0\end{pmatrix};
\begin{pmatrix}0&y\\ \pm{}^t\hspace{-1pt}y&0\end{pmatrix},\begin{pmatrix}0&I\\ \mp I&0\end{pmatrix}\right)
\in \overline{\BC\Phi_\bm^{(2)}\big(y^2\big)}.
\end{gather*}
Next, we regard $\overline{\cP_\bm(\fp^+_2)}$ as the complex conjugate representation of $K_1^\BC={\rm GL}(2r,\BC)$ with respect to the real form
\begin{alignat*}{3}
& \left\{k\in {\rm GL}(2r,\BC)\colon k\begin{pmatrix}0&I\\I&0\end{pmatrix}=\begin{pmatrix}0&I\\I&0\end{pmatrix}\overline{k}\right\}
 \simeq {\rm GL}(2r,\BR) \quad & &(\text{Case 1}),& \\
&\left\{k\in {\rm GL}(2r,\BC)\colon k\begin{pmatrix}0&I\\-I&0\end{pmatrix}=\begin{pmatrix}0&I\\-I&0\end{pmatrix}\overline{k}\right\}
\simeq {\rm GL}(r,\BH) \quad & &(\text{Cases 2, 3}).&
\end{alignat*}
Then the linear map $\overline{f(y)}\mapsto \overline{f\left(\left(\begin{smallmatrix}0&I\\\pm I&0\end{smallmatrix}\right)y^*
\left(\begin{smallmatrix}0&I\\\pm I&0\end{smallmatrix}\right)\right)}$ induces
the linear isomorphism $\overline{\cP_\bm(\fp^+_2)}\allowbreak\simeq \cP_\bm(\fp^+_2)$. Therefore,
\begin{alignat*}{3}
&\cK_{\bm,\bl}^{(4,1)}\left(x_2;\begin{pmatrix}0&I\\I&0\end{pmatrix},
\begin{pmatrix}0&I\\I&0\end{pmatrix}x_2^*\begin{pmatrix}0&I\\I&0\end{pmatrix}\right)&&&\\
&\qquad\quad{} \in\left(V_{(2\bm+\bl)^2}^{(2r)\vee}\otimes V_{\langle 2l\rangle}^{(2r)\vee}\right)^{O(2r,\BC)}=\left(V_{(2(\bm+\bl))^2}^{(2r)\vee}\right)^{O(2r,\BC)}&&& \\
&\qquad{} =\cP_{2(\bm+\bl)}(\Skew(2r,\BC))^{O(2r,\BC)}
=\BC\tilde{\Phi}_{\bm+\bl}^{(2)\prime}\left(\left(x_2\begin{pmatrix}0&I\\I&0\end{pmatrix}\right)^2\right) \quad & &(\text{Case }1),& \\
&\cK_{\bm,\bl}^{(1,4)}\left(x_2;\begin{pmatrix}0&I\\-I&0\end{pmatrix},
\begin{pmatrix}0&I\\-I&0\end{pmatrix}x_2^*\begin{pmatrix}0&I\\-I&0\end{pmatrix}\right)&&&\\
&\qquad \quad {} \in\big(V_{\substack{2(m_1+l_1,m_1,m_2+l_2,m_2,\ldots,m_r+l_r,m_r)}}^{(2r)\vee}
\otimes V_{(2l,0,\ldots,0)}^{(2r)\vee}\big)^{\operatorname{Sp}(r,\BC)}
=\big(V_{(2(\bm+\bl))^2}^{(2r)\vee}\big)^{\operatorname{Sp}(r,\BC)}\!\!\! &&& \\
&\qquad {} =\cP_{(\bm+\bl)^2}(\Sym(2r,\BC))^{\operatorname{Sp}(r,\BC)}
=\BC\tilde{\Phi}_{\bm+\bl}^{(2)\prime}\left(\left(x_2\begin{pmatrix}0&I\\-I&0\end{pmatrix}\right)^2\right)\quad & &(\text{Case 2}),& \\
&\cK_{\bm,-\bl}^{(1,4)}\left(x_2;\begin{pmatrix}0&I\\-I&0\end{pmatrix},
\begin{pmatrix}0&I\\-I&0\end{pmatrix}x_2^*\begin{pmatrix}0&I\\-I&0\end{pmatrix}\right)&&&\\
&\qquad\quad {} \in\big( V_{\substack{2(m_1+l,m_1+l-l_1,m_2+l,m_2+l-l_2,\ldots,m_r+l,m_r+l-l_r)}}^{(2r)\vee}
\otimes V_{(2l,\ldots,2l,l)}^{(2r)\vee}\big)^{\operatorname{Sp}(r,\BC)} &&&\\
& \qquad{} =\big(V_{(2(\bm-\bl+2l))^2}^{(2r)\vee}\big)^{\operatorname{Sp}(r,\BC)} &&& \\
&\qquad {}=\cP_{(\bm-\bl+2l)^2}(\Sym(2r,\BC))^{\operatorname{Sp}(r,\BC)}
=\BC\tilde{\Phi}_{\bm-\bl+2l}^{(2)\prime}\left(\left(x_2\begin{pmatrix}0&I\\-I&0\end{pmatrix}\right)^2\right) \quad & &(\text{Case 3}).&
\end{alignat*}
Especially, we have
\begin{gather*}
 \beta\left(\cK_{\bm,\bl}\left(\begin{pmatrix}0&x\\ \mp{}^t\hspace{-1pt}x&0\end{pmatrix};
\begin{pmatrix}0&y\\ \pm{}^t\hspace{-1pt}y&0\end{pmatrix},\begin{pmatrix}0&w\\ \mp{}^t\hspace{-1pt}w&0\end{pmatrix}\right)\right)\\
\qquad{} =\cK_{\bm,\bl}\left(\begin{pmatrix}0&x\\ \mp{}^t\hspace{-1pt}x&0\end{pmatrix};
\begin{pmatrix}0&I\\ \pm I&0\end{pmatrix},\begin{pmatrix}0&x^*\\ \mp \overline{x}&0\end{pmatrix}\right)\in\begin{cases}
\BC\tilde{\Phi}_{\bm+\bl}^{(2)}\big(x^2\big) & (\text{Cases 1, 2}),\\
\BC\tilde{\Phi}_{\bm-\bl+2l}^{(2)}\big(x^2\big) & (\text{Case }3). \end{cases}
\end{gather*}
Therefore both $\cK_{\bm,\bl}\left(\Big(\begin{smallmatrix}0&x\\ \mp{}^t\hspace{-1pt}x&0\end{smallmatrix}\Big);
\Big(\begin{smallmatrix}0&y\\ \pm{}^t\hspace{-1pt}y&0\end{smallmatrix}\Big),
\Big(\begin{smallmatrix}0&w\\ \mp{}^t\hspace{-1pt}w&0\end{smallmatrix}\Big)\right)$
and $\bK_{\bm,\bl}^{(2)}(x,x;yx^*y,w)$ sits in
\begin{gather*} \alpha^{-1}\big(\overline{\BC\Phi_\bm^{(2)}\big(y^2\big)}\big)\cap \beta^{-1}\big(\BC\Phi_{\bm+(2l-)\bl}^{(2)}\big(x^2\big)\big), \end{gather*}
and since $\{\Phi_\bm(x^2)\}_\bm$ are linearly independent,
\begin{gather*} \bigoplus_{\bm\in\BZ_{++}^r}\bigoplus_\bl
\alpha^{-1}\big(\overline{\BC\Phi_\bm^{(2)}\big(y^2\big)}\big)\cap \beta^{-1}\big(\BC\Phi_{\bm+(2l-)\bl}^{(2)}\big(x^2\big)\big) \end{gather*}
is a direct sum. Finally, we have
\begin{gather*}
 {\rm e}^{\frac{1}{2}\tr\left(\Big(\begin{smallmatrix}0&x\\ \mp{}^t\hspace{-1pt}x&0\end{smallmatrix}\Big)
\Big(\begin{smallmatrix}0&y\\ \pm{}^t\hspace{-1pt}y&0\end{smallmatrix}\Big)^*
\Big(\begin{smallmatrix}0&x\\ \mp{}^t\hspace{-1pt}x&0\end{smallmatrix}\Big)
\Big(\begin{smallmatrix}0&y\\ \pm{}^t\hspace{-1pt}y&0\end{smallmatrix}\right)^*\Big)}
\rK_{(l)}'\left(\begin{pmatrix}0&x\\ \mp{}^t\hspace{-1pt}x&0\end{pmatrix},
\begin{pmatrix}0&w\\ \mp{}^t\hspace{-1pt}w&0\end{pmatrix}\right)\\
\qquad{} ={\rm e}^{\tr(xy^*xy^*)}\rK_{(l)}''(x,w)={\rm e}^{\tr(x(yx^*y)^*)}\rK_{(l)}''(x,w),
\end{gather*}
and projecting both sides to each component, we get the claim.
\end{proof}

\begin{proof}[Proof of Proposition \ref{exp_susp} ($\boldsymbol{s=2r}$: even)]
We write
\begin{gather*} \det(I-x_2y_1^*x_2y_1^*)^{-\mu/2}\rK_{(l)}'\big(x_2(I-y_1^*x_2y_1^*x_2)^{-1},y_2\big)
=\sum_{\bm\in\BZ_{++}^r}\sum_\bl C_{\bm,\bl}(\mu)\cK_{\bm,\bl}(x_2;y_1,y_2), \end{gather*}
and find the functions $C_{\bm,\bl}(\mu)$. From this expression, we have
\begin{gather*}
 \det(I-xy^*xy^*)^{-\mu}\rK_{(l)}''\big(x(I-y^*xy^*x)^{-1},w\big)\\
 =\det\left(I-\begin{pmatrix}0&x\\ \mp{}^t\hspace{-1pt}x&0\end{pmatrix}
\begin{pmatrix}0&y\\ \pm{}^t\hspace{-1pt}y&0\end{pmatrix}\!\!{\vphantom{\biggr)}}^*\!
\begin{pmatrix}0&x\\ \mp{}^t\hspace{-1pt}x&0\end{pmatrix}
\begin{pmatrix}0&y\\ \pm{}^t\hspace{-1pt}y&0\end{pmatrix}\!\!{\vphantom{\biggr)}}^*\right)^{-\mu/2}\\
\quad{} \times \rK_{(l)}'\left(\begin{pmatrix}0&x\\ \mp{}^t\hspace{-1pt}x&0\end{pmatrix}\left(I
-\begin{pmatrix}0&y\\ \pm{}^t\hspace{-1pt}y&0\end{pmatrix}\!\!{\vphantom{\biggr)}}^*\!
\begin{pmatrix}0&x\\ \mp{}^t\hspace{-1pt}x&0\end{pmatrix}
\begin{pmatrix}0&y\\ \pm{}^t\hspace{-1pt}y&0\end{pmatrix}\!\!{\vphantom{\biggr)}}^*\!
\begin{pmatrix}0&x\\ \mp{}^t\hspace{-1pt}x&0\end{pmatrix}\right)^{-1},
\begin{pmatrix}0&w\\ \mp{}^t\hspace{-1pt}w&0\end{pmatrix}\right)\\
 =\sum_{\bm\in\BZ_{++}^r}\sum_\bl C_{\bm,\bl}(\mu)\cK_{\bm,\bl}\left(\begin{pmatrix}0&x\\ \mp{}^t\hspace{-1pt}x&0\end{pmatrix};
\begin{pmatrix}0&y\\ \pm{}^t\hspace{-1pt}y&0\end{pmatrix},\begin{pmatrix}0&w\\ \mp{}^t\hspace{-1pt}w&0\end{pmatrix}\right).
\end{gather*}
On the other hand, we can regard $\det(I-xy^*)^{-\mu}\rK_{(l)}''\big(z(I-y^*x)^{-1},w\big)$ as the reproducing kernel of
the holomorphic discrete series representation of $\widetilde{U}(r,r)$. Thus by \cite{N2} it holds that
\begin{alignat*}{3}
& \det(I-xy^*)^{-\mu}\bK_{\langle l\rangle}^{(2)}\big(z(I-y^*x)^{-1},w\big)&&&\\
& \qquad {} =\sum_{\bm\in\BZ_{++}^r}\sum_\bl \frac{(\mu)_{\bm+\bl,2}}{(\mu)_{\langle l\rangle,2}}\bK_{\bm,\bl}^{(2)}(x,z;y,w)\quad && (\text{Case }1),& \\
& \det(I-xy^*)^{-\mu}\bK_{(l,0,\ldots,0)}^{(2)}\big(z(I-y^*x)^{-1},w\big)&&&\\
&\qquad {} =\sum_{\bm\in\BZ_{++}^r}\sum_\bl \frac{(\mu)_{\bm+\bl,2}}{(\mu)_{(l,0,\ldots,0),2}}\bK_{\bm,\bl}^{(2)}(x,z;y,w) \quad&& (\text{Case }2)& \\
& \det(I-xy^*)^{-\mu}\bK_{(\smash{\underbrace{\scriptstyle 2l,\ldots,2l}_{r-1}},l)}^{(2)}\big(z(I-y^*x)^{-1},w\big)&&&\\[5mm]
& \qquad {} = \sum_{\bm\in\BZ_{++}^r}\!\sum_\bl \frac{(\mu+l)_{\bm-\bl+(\smash{\overbrace{\scriptstyle l,\ldots,l}^r}),2}}
{(\mu+l)_{(\smash{\underbrace{\scriptstyle l,\ldots,l}_{r-1}},0),2}}\bK_{\bm,\bl}^{(2)}(x,z;y,w) \quad&& (\text{Case }3).&
\end{alignat*}
Therefore substituting $(x,z;y,w)$ with $(x,x;yx^*y,w)$ and comparing with the above formula, we can determine $C_{\bm,\bl}(\mu)$.
\end{proof}

When $s=2r+1$ is odd, we compute the expansion by restricting the function for $s=2r+2$ to $s=2r+1$.
We redefine the restriction map
\begin{gather*}
 \Rest\colon \ \cP(\Sym(2r+2,\BC))\longrightarrow\cP(\Sym(2r+1,\BC)),\\
 \Rest\colon \ \cP(\Skew(2r+2,\BC))\longrightarrow\cP(\Skew(2r+1,\BC))
\end{gather*}
(we use the same symbol) by $(\Rest f)(x):=f\left(\begin{smallmatrix}x&0\\0&0\end{smallmatrix}\right)$, and let
\begin{gather*} \Proj\colon \ V_{\langle 2l\rangle}^{(2r+2)\vee}\longrightarrow V_{\langle 2l\rangle}^{(2r+1)\vee},\qquad
V_{(2l,0,\ldots,0)}^{(2r+2)\vee}\longrightarrow V_{(2l,0,\ldots,0)}^{(2r+1)\vee},\qquad
V_{(2l,\ldots,2l,0)}^{(2r+2)\vee}\longrightarrow V_{(2l,\ldots,2l,0)}^{(2r+1)\vee} \end{gather*}
(we use the same symbol) be the $U(2r+1)$-equivariant orthogonal projection.
\begin{proof}[Proof of Proposition \ref{exp_susp} ($\boldsymbol{s=2r+1}$: odd)]
First we deal with Case 1. Since
\begin{gather*} \Rest(\cP_{2\bm+\bl}(\Skew(2r+2,\BC)))=\begin{cases}
\cP_{2\bm+\bl}(\Skew(2r+1,\BC)) & (2m_{r+1}+l_{r+1}=0), \\ \{0\} & (2m_{r+1}+l_{r+1}>0), \end{cases} \end{gather*}
by projecting ${\rm e}^{\frac{1}{2}\tr(x_2y_1^*x_2y_1^*)}\rK_{\langle l\rangle}^{(4)}(x_2)$ to each component
with respect to $x_2\in\Skew(2r+2,\BC)$, if $2m_{r+1}+l_{r+1}=0$ we get
\begin{gather*} \Rest\otimes\overline{\Rest\otimes\Proj}\big(\cK_{\bm,\bl}^{(4,1)(2r+2)}(x_2;y_1)\big)=\cK_{\bm,\bl}^{(4,1)(2r+1)}(x_2;y_1). \end{gather*}
Therefore by restricting $\det(I-x_2y_1^*x_2y_1^*)^{-\mu/2}\rK_{\langle l\rangle}^{(4)}\big(x_2(I-y_1^*x_2y_1^*x_2)^{-1}\big)$
for $s=2r+2$ to $s=2r+1$, we get the claim. Next we deal with Case 2. In this case we have
\begin{gather*}
 \Rest\big(\cP_{\substack{(m_1+l_1,m_1,m_2+l_2,m_2,\ldots,\\\hspace{25pt} m_{r+1}+l_{r+1},m_{r+1})}}(\Sym(2r+2,\BC))\big) \\
\qquad{} \subset\bigoplus_{\substack{\bn\in\BZ_{++}^{2r+1}\\ m_j+l_j\ge n_{2j-1}\ge m_j\\ m_j\ge n_{2j}\ge m_{j+1}+l_{j+1}}}\cP_\bn(\Sym(2r+1,\BC))
=\bigoplus_{\substack{\bn\in\BZ_{++}^{2r+1}\\ m_j+l_j\ge n_{2j-1}\ge m_j\\ m_j\ge n_{2j}\ge m_{j+1}+l_{j+1}}}V_{2\bn}^{(2r+1)\vee},
\end{gather*}
and if $m_{r+1}=0$,
\begin{gather*}
 \Rest\otimes\Proj\big(\cP_{2\bm}(\Skew(2r+2,\BC))\otimes V_{(2l,0,\ldots,0)}^{(2r+2)\vee}\big)\\
\qquad{}=\cP_{2\bm}(\Skew(2r+1,\BC))\otimes V_{(2l,0,\ldots,0)}^{(2r+1)\vee}\\
\qquad{} =\bigoplus_{\substack{\bk\in(\BZ_{\ge 0})^{r+1},\; |\bk|=2l\\ 0\le k_j\le m_{j-1}-m_j }}
V_{(2m_1+k_1,2m_1,2m_2+k_2,2m_2,\ldots,2m_r+k_r,2m_r,k_{r+1})}^{(2r+1)\vee},
\end{gather*}
and only $V_{(2m_1+2l_1,2m_1,2m_2+2l_2,2m_2,\ldots,2m_r+2l_r,2m_r,2l_{r+1})}^{(2r+1)\vee}$ appears commonly in both decomposition. Therefore when $m_{r+1}=0$ we have
\begin{gather*}
 \Rest\otimes\overline{\Rest\otimes\Proj}\Big(\Big(
\cP_{\substack{(m_1+l_1,m_1,m_2+l_2,m_2,\ldots,\;\\\hspace{30pt} m_r+l_r,m_r,l_{r+1},0)}}(\Sym(2r+2,\BC))_{x_2}\\
\qquad\quad{} \otimes\overline{\cP_{2\bm}(\Skew(2r+2,\BC))_{y_1}
\otimes V_{(2l,0,\ldots,0)}^{(2r+2)\vee}}\Big)^{{\rm GL}(2r+2,\BC)}\Big)\\
\qquad{} =\Big(\cP_{\substack{(m_1+l_1,m_1,m_2+l_2,m_2,\ldots,\\\hspace{35pt} m_r+l_r,m_r,l_{r+1})}}(\Sym(2r+1,\BC))_{x_2}\\
\qquad\quad {} \otimes\overline{\cP_{2\bm}(\Skew(2r+1,\BC))_{y_1}
\otimes V_{(2l,0,\ldots,0)}^{(2r+1)\vee}}\Big)^{{\rm GL}(2r+1,\BC)},
\end{gather*}
and hence
\begin{gather*} \Rest\otimes\overline{\Rest\otimes\Proj} \big(\cK_{\bm,\bl}^{(1,4)(2r+2)}(x_2;y_1)\big)=\cK_{\bm,\bl}^{(1,4)(2r+1)}(x_2;y_1) \end{gather*}
holds. Therefore the claim follows. Finally we deal with Case~3. In this case we have
\begin{gather*}
 \Rest\big(\cP_{\substack{(m_1+l,m_1+l-k_1,m_2+l,m_2+l-k_2,\ldots,\;\\\hspace{45pt} m_{r+1}+l,m_{r+1}+l-k_{r+1})}}(\Sym(2r+2,\BC))\big) \\
\qquad{} \subset\bigoplus_{\substack{\bn\in\BZ_{++}^{2r+1}\\ m_j+l\ge n_{2j-1}\ge m_j+l-k_j\\ m_j+l-k_j\ge n_{2j}\ge m_{j+1}+l}}\cP_\bn(\Sym(2r+1,\BC))
=\bigoplus_{\substack{\bn\in\BZ_{++}^{2r+1}\\ m_j+l\ge n_{2j-1}\ge m_j+l-k_j\\ m_j+l-k_j\ge n_{2j}\ge m_{j+1}+l}}V_{2\bn}^{(2r+1)\vee},
\end{gather*}
and if $m_{r+1}=0$,
\begin{gather*}
 \Rest\otimes\Proj\big(\cP_{2\bm}(\Skew(2r+2,\BC))\otimes V_{(2l,\ldots,2l,0)}^{(2r+2)\vee}\big) =\cP_{2\bm}(\Skew(2r+1,\BC))\otimes V_{(2l,\ldots,2l,0)}^{(2r+1)\vee}\\
\qquad{} =\bigoplus_{\substack{\bl\in(\BZ_{\ge 0})^{r+1},\; |\bl|=2l\\ 0\le l_j\le m_j-m_{j+1}\\ 0\le l_{r+1}}}
V_{(2m_1+2l,2m_1+2l-l_1,2m_2+2l,2m_2+2l-l_2,\ldots,2m_r+2l,2m_r+2l-l_r,2l-l_{r+1})}^{(2r+1)\vee}.
\end{gather*}
Therefore when $m_{r+1}=0$ we have
\begin{gather*}
 \Rest\otimes\overline{\Rest\otimes\Proj}\Big(\Big(
\cP_{\substack{(m_1+l,m_1+l-k_1,m_2+l,m_2+l-k_2,\ldots,\;\\\hspace{45pt} m_r+l,m_r+l-k_r,l,l-k_{r+1})}}(\Sym(2r+2,\BC))_{x_2}\\
\qquad\quad {} \otimes\overline{\cP_{2\bm}(\Skew(2r+2,\BC))_{y_1}
\otimes V_{(2l,\ldots,2l,0)}^{(2r+2)\vee}}\Big)^{{\rm GL}(2r+2,\BC)}\Big)\\
\qquad{} \subset\bigoplus_{\substack{\bl\in(\BZ_{\ge 0})^{r+1},\; |\bl|=l\\ k_j\le l_j\le m_j-m_{j+1}\\ 0\le l_{r+1}\le k_{r+1}}}
\Big(\cP_{\substack{(m_1+l,m_1+l-l_1,m_2+l,m_2+l-l_2,\ldots,\;\\\hspace{45pt} m_r+l,m_r+l-l_r,l-l_{r+1})}}(\Sym(2r+1,\BC))_{x_2}\\
\qquad\quad{} \otimes\overline{\cP_{2\bm}(\Skew(2r+1,\BC))_{y_1}
\otimes V_{(2l,\ldots,2l,0)}^{(2r+1)\vee}}\Big)^{{\rm GL}(2r+1,\BC)},
\end{gather*}
and hence there exist constants $c_{\bm,\bk,\bl}$ such that
\begin{gather*} \Rest\otimes\overline{\Rest\otimes\Proj} \big(\cK_{\bm,-\bk}^{(1,4)(2r+2)}(x_2;y_1)\big)
=\sum_{\substack{\bl\in(\BZ_{\ge 0})^{r+1},\; |\bl|=l\\ k_j\le l_j\le m_j-m_{j+1}\\ 0\le l_{r+1}\le k_{r+1}}}
c_{\bm,\bk,\bl}\cK_{\bm,-\bl}^{(1,4)(2r+1)}(x_2;y_1) \end{gather*}
holds. Comparing two expansion of ${\rm e}^{\frac{1}{2}\tr(x_2y_1^*x_2y_1^*)}\rK_{(l,\ldots,l,0)}^{(1)}(x_2)$,
\begin{gather*}
 {\rm e}^{\frac{1}{2}\tr(x_2y_1^*x_2y_1^*)}\rK_{(l,\ldots,l,0)}^{(1)(2r+1)}(x_2)
=\Rest\otimes\overline{\Rest\otimes\Proj}\big({\rm e}^{\frac{1}{2}\tr(x_2y_1^*x_2y_1^*)}\rK_{(l,\ldots,l,0)}^{(1)(2r+2)}(x_2)\big)\\
\qquad{}=\sum_{\bm\in\BZ_{++}^r}\sum_{\substack{\bk\in(\BZ_{\ge 0})^{r+1},\; |\bk|=l\\ 0\le k_j\le m_j-m_{j+1}\\ 0\le k_{r+1}}}
\Rest\otimes\overline{\Rest\otimes\Proj}\big(\cK_{\bm,-\bk}^{(1,4)(2r+2)}(x_2;y_1)\big) \\
\qquad{} =\sum_{\bm\in\BZ_{++}^r}\sum_{\substack{\bk\in(\BZ_{\ge 0})^{r+1},\; |\bk|=l\\ 0\le k_j\le m_j-m_{j+1}\\ 0\le k_{r+1}}}
\;\sum_{\substack{\bl\in(\BZ_{\ge 0})^{r+1},\; |\bl|=l\\ k_j\le l_j\le m_j-m_{j+1}\\ 0\le l_{r+1}\le k_{r+1}}}
c_{\bm,\bk,\bl}\cK_{\bm,-\bl}^{(1,4)(2r+1)}(x_2;y_1) \\
\qquad{} =\sum_{\bm\in\BZ_{++}^r}\sum_{\substack{\bl\in(\BZ_{\ge 0})^{r+1},\; |\bl|=l\\ 0\le l_j\le m_j-m_{j+1}\\ 0\le l_{r+1}}}
\;\sum_{\substack{\bk\in(\BZ_{\ge 0})^{r+1},\; |\bk|=l\\ 0\le k_j\le l_j\\ l_{r+1}\le k_{r+1}}}
c_{\bm,\bk,\bl}\cK_{\bm,-\bl}^{(1,4)(2r+1)}(x_2;y_1)
\end{gather*}
and
\begin{gather*} {\rm e}^{\frac{1}{2}\tr(x_2y_1^*x_2y_1^*)}\rK_{(l,\ldots,l,0)}^{(1)(2r+1)}(x_2)
=\sum_{\bm\in\BZ_{++}^r}\sum_{\substack{\bl\in(\BZ_{\ge 0})^{r+1},\; |\bl|=l\\ 0\le l_j\le m_j-m_{j+1}\\ 0\le l_{r+1}}}
\cK_{\bm,-\bl}^{(1,4)(2r+1)}(x_2;y_1), \end{gather*}
we get
\begin{gather*} \sum_{\substack{\bk\in(\BZ_{\ge 0})^{r+1},\; |\bk|=l\\ 0\le k_j\le l_j\\ l_{r+1}\le k_{r+1}}}c_{\bm,\bk,\bl}=1. \end{gather*}
Therefore we have
\begin{gather*}
 \det(I-x_2y_1^*x_2y_1^*)^{-\mu/2}\rK_{(l,\ldots,l,0)}^{(1)(2r+1)}\big(x_2(I-y_1^*x_2y_1^*x_2)^{-1}\big) \\
 \qquad{} =\Rest\otimes\overline{\Rest\otimes\Proj}\big(
\det(I-x_2y_1^*x_2y_1^*)^{-\mu/2}\rK_{(l,\ldots,l,0)}^{(1)(2r+2)}\big(x_2(I-y_1^*x_2y_1^*x_2)^{-1}\big)\big) \\
\qquad{} =\sum_{\bm\in\BZ_{++}^r}\sum_{\substack{\bk\in(\BZ_{\ge 0})^{r+1},\; |\bk|=l\\ 0\le k_j\le m_j-m_{j+1}\\ 0\le k_{r+1}}}
(\mu+2l)_{\bm-\bk',2}(\mu+l-r)_{l-k_{r+1}}\\
\qquad\quad{} \times\Rest\otimes\overline{\Rest\otimes\Proj}\big(\cK_{\bm,-\bk}^{(1,4)(2r+2)}(x_2;y_1)\big) \\
\qquad{} =\sum_{\bm\in\BZ_{++}^r}\sum_{\substack{\bk\in(\BZ_{\ge 0})^{r+1},\; |\bk|=l\\ 0\le k_j\le m_j-m_{j+1}\\ 0\le k_{r+1}}}
\;\sum_{\substack{\bl\in(\BZ_{\ge 0})^{r+1},\; |\bl|=l\\ k_j\le l_j\le m_j-m_{j+1}\\ 0\le l_{r+1}\le k_{r+1}}}
c_{\bm,\bk,\bl}(\mu+2l)_{\bm-\bk',2}(\mu+l-r)_{l-k_{r+1}}\\
\qquad\quad{} \times\cK_{\bm,-\bl}^{(1,4)(2r+1)}(x_2;y_1) \\
\qquad{} =\sum_{\bm\in\BZ_{++}^r}\sum_{\substack{\bl\in(\BZ_{\ge 0})^{r+1},\; |\bl|=l\\ 0\le l_j\le m_j-m_{j+1}\\ 0\le l_{r+1}}}
\;\sum_{\substack{\bk\in(\BZ_{\ge 0})^{r+1},\; |\bk|=l\\ 0\le k_j\le l_j\\ l_{r+1}\le k_{r+1}}}
c_{\bm,\bk,\bl}(\mu+2l)_{\bm-\bk',2}(\mu+l-r)_{l-k_{r+1}}\\
\qquad\quad{} \times\cK_{\bm,-\bl}^{(1,4)(2r+1)}(x_2;y_1) \\
\qquad{}=\sum_{\bm\in\BZ_{++}^r}\sum_{\substack{\bl\in(\BZ_{\ge 0})^{r+1},\; |\bl|=l\\ 0\le l_j\le m_j-m_{j+1}\\ 0\le l_{r+1}}}(\mu+2l)_{\bm-\bl',2}\\
\qquad\quad{} \times\sum_{\substack{\bk\in(\BZ_{\ge 0})^{r+1},\; |\bk|=l\\ 0\le k_j\le l_j\\ l_{r+1}\le k_{r+1}}}
c_{\bm,\bk,\bl}(\mu+2l+\bm-\bl')_{\bl'-\bk',2}(\mu+l-r)_{l-k_{r+1}}\cK_{\bm,-\bl}^{(1,4)(2r+1)}(x_2;y_1),
\end{gather*}
where for $\bk=(k_1,\ldots,k_r,k_{r+1})\in(\BZ_{\ge 0})^{r+1}$, we put $\bk'=(k_1,\ldots,k_r)\in(\BZ_{\ge 0})^r$, and similar for~$\bl$,~$\bl'$. Now we put
\begin{gather*} \varphi_{\bm,-\bl}(\mu):=\sum_{\substack{\bk\in(\BZ_{\ge 0})^{r+1},\; |\bk|=l\\ 0\le k_j\le l_j\\ l_{r+1}\le k_{r+1}}}
c_{\bm,\bk,\bl}(\mu+2l+\bm-\bl')_{\bl'-\bk',2}(\mu+l-r)_{l-k_{r+1}}. \end{gather*}
Then this is a monic polynomial of degree $l-l_{r+1}$, and we have
\begin{gather*}
 \det(I-x_2y_1^*x_2y_1^*)^{-\mu/2}\rK_{(l,\ldots,l,0)}^{(1)(2r+1)}\big(x_2(I-y_1^*x_2y_1^*x_2)^{-1}\big) \\
\qquad{} =\sum_{\bm\in\BZ_{++}^r}\sum_{\substack{\bl\in(\BZ_{\ge 0})^{r+1},\; |\bl|=l\\ 0\le l_j\le m_j-m_{j+1}\\ 0\le l_{r+1}}}
(\mu+2l)_{\bm-\bl',2}\varphi_{\bm,-\bl}(\mu)\cK_{\bm,-\bl}^{(1,4)(2r+1)}(x_2;y_1).
\end{gather*}
Therefore the claim follows.
\end{proof}

Therefore, by using the results of \cite{N2} and (\ref{exp_onD}), for Case 1,
\begin{gather*}
 F_{\tau\rho}(x_2;w_1) =\Pf(x_2)^k\! \sum_\bm\!\sum_\bl (\lambda\!+\!k\!+\!\langle l\rangle)_{\bm{+}\bl{-}\langle l\rangle,2}
\big\langle {\rm e}^{\tr(y_1w_1^*)}I_W,\overline{\cK_{\bm,\bl}^{(4,1)}(x_2;y_1)} \big\rangle_{\cH_{\lambda{+}k}(D_1,W_{(l)}'),y_1} \\
\qquad{} =\Pf(x_2)^k\sum_\bm\sum_\bl \frac{(\lambda+k+\langle l\rangle)_{\bm+\bl-\langle l\rangle,2}}{(\lambda+k+\langle 2l\rangle)_{(\bm+\bl-\langle l\rangle)^2,1}}
\cK_{\bm,\bl}^{(4,1)}(x_2;w_1) \\
\qquad{} =\Pf(x_2)^k\sum_\bm\sum_\bl \frac{1}{\left(\lambda+k+\langle l\rangle-\frac{1}{2}\right)_{\bm+\bl-\langle l\rangle,2}}\cK_{\bm,\bl}^{(4,1)}(x_2;w_1),
\end{gather*}
for Case 2,
\begin{gather*}
 F_{\tau\rho}(x_2;w_1) =\det(x_2)^{k}\sum_\bm\sum_\bl (\lambda+2k+(l,0,\ldots,0))_{\bm+\bl-(l,0,\ldots,0),2}\\
\qquad\quad{} \times\bigl\langle {\rm e}^{\frac{1}{2}\tr(y_1w_1^*)}I_W,\overline{\cK_{\bm,\bl}^{(1,4)}(x_2;y_1)}
\bigr\rangle_{\cH_{2\lambda+4k}(D_1,W_{(l)}'),y_1} \\
\qquad{} =\det(x_2)^{k}\sum_\bm\sum_\bl
\frac{(\lambda+2k+(l,0,\ldots,0))_{\bm+\bl-(l,0,\ldots,0),2}}{(2\lambda+4k+(2l,0,\ldots,0))_{2(\bm+\bl)-(2l,0,\ldots,0),4}}
\cK_{\bm,\bl}^{(1,4)}(x_2;w_1)\\
\qquad{} =\det(x_2)^{k}\sum_\bm\sum_\bl
\frac{1}{2^{2|\bm|}\left(\lambda+2k+(l,0,\ldots,0)+\frac{1}{2}\right)_{\bm+\bl-(l,0,\ldots,0),2}}\cK_{\bm,\bl}^{(1,4)}(x_2;w_1)\\
\qquad{} =\det(x_2)^{k}\sum_\bm\sum_\bl
\frac{1}{\left(\lambda+2k+(l,0,\ldots,0)+\frac{1}{2}\right)_{\bm+\bl-(l,0,\ldots,0),2}}\cK_{\bm,\bl}^{(1,4)}\left(x_2;\frac{1}{2}w_1\right),
\end{gather*}
for Case 3 with $s$ even,
\begin{gather*}
 F_{\tau\rho}(x_2;w_1,w_2) =\det(x_2)^{k}\sum_\bm\sum_\bl (\lambda+2k+l+(\smash{\overbrace{l,\ldots,l}^{s/2-1}},0))_{\bm-\bl
+(\smash{\overbrace{\scriptstyle 0,\ldots,0}^{s/2-1}},l),2}\\
\qquad\quad{} \times\bigl\langle {\rm e}^{\frac{1}{2}\tr(y_1w_1^*)}I_W,\overline{\cK_{\bm,-\bl}^{(1,4)}(x_2;y_1)}
\bigr\rangle_{\cH_{2\lambda+4k}(D_1,W_{(l)}'),y} \\
\qquad{} =\det(x_2)^{k}\sum_\bm\sum_\bl
\frac{(\lambda+2k+l+(l,\ldots,l,0))_{\bm-\bl-(0,\ldots,0,l),2}}{(2\lambda+4k+2l+(2l,\ldots,2l,0))_{2(\bm-\bl)+(0,\ldots,0,2l),4}}
\cK_{\bm,-\bl}^{(1,4)}(x_2;w_1)\\
\qquad{} =\det(x_2)^{k}\sum_\bm\sum_\bl
\frac{1}{2^{2|\bm|}\left(\lambda+2k+l+(l,\ldots,l,0)+\frac{1}{2}\right)_{\bm-\bl+(0,\ldots,0,l),2}}\cK_{\bm,-\bl}^{(1,4)}(x_2;w_1)\\
\qquad{} =\det(x_2)^{k}\sum_\bm\sum_\bl
\frac{1}{\left(\lambda+2k+l+(l,\ldots,l,0)+\frac{1}{2}\right)_{\bm-\bl+(0,\ldots,0,l),2}}\cK_{\bm,-\bl}^{(1,4)}\big(x_2;\tfrac{1}{2}w_1\big).
\end{gather*}
and for Case 3 with $s$ odd,
\begin{gather*}
 F_{\tau\rho}(x_2;w_1,w_2) =\det(x_2)^{k}\sum_\bm\sum_\bl (\lambda+2k+2l)_{\bm-\bl',2}\varphi_{\bm,-\bl}(\lambda+2k)\\
\qquad\quad{} \times\bigl\langle {\rm e}^{\frac{1}{2}\tr(y_1w_1^*)}I_W,\overline{\cK_{\bm,-\bl}^{(1,4)}(x_2;y_1)}
\bigr\rangle_{\cH_{2\lambda+4k}(D_1,W_{(l)}'),y} \\
\qquad{} =\det(x_2)^{k}\sum_\bm\sum_\bl \frac{(\lambda+2k+2l)_{\bm-\bl',2}\varphi_{\bm,-\bl}(\lambda+2k)}
{(2\lambda+4k+4l)_{2(\bm-\bl'),4}\left(2\lambda+4k+2l-2\left\lfloor\frac{s}{2}\right\rfloor+1\right)_{2l-2l_{\lceil s/2\rceil}}}\\
\qquad\quad{} \times\cK_{\bm,-\bl}^{(1,4)}(x_2;w_1)\\
\qquad{} =\det(x_2)^{k}\sum_\bm\sum_\bl \frac{1}{2^{2|\bm|}\left(\lambda+2k+2l+\frac{1}{2}\right)_{\bm-\bl',2}
\left(\lambda+2k+l-\left\lfloor\frac{s}{2}\right\rfloor+1\right)_{l-l_{\lceil s/2\rceil}}}\\
\qquad\quad{} \times\frac{\varphi_{\bm,-\bl}(\lambda+2k)}
{\left(\lambda+2k+l-\left\lfloor\frac{s}{2}\right\rfloor+\frac{1}{2}\right)_{l-l_{\lceil s/2\rceil}}}
\cK_{\bm,-\bl}^{(1,4)}(x_2;w_1)\\
\qquad{} =\det(x_2)^{k}\sum_\bm\sum_\bl \frac{1}{\left(\lambda+2k+2l+\frac{1}{2}\right)_{\bm-\bl',2}
\left(\lambda+2k+l-\left\lfloor\frac{s}{2}\right\rfloor+1\right)_{l-l_{\lceil s/2\rceil}}}\\
\qquad\quad{} \times\frac{\varphi_{\bm,-\bl}(\lambda+2k)}
{\left(\lambda+2k+l-\left\lfloor\frac{s}{2}\right\rfloor+\frac{1}{2}\right)_{l-l_{\lceil s/2\rceil}}}
\cK_{\bm,-\bl}^{(1,4)}\big(x_2;\tfrac{1}{2}w_1\big).
\end{gather*}
By Theorem~\ref{main}, by substituting $w_1$ with $\overline{\frac{\partial}{\partial x_1}}$, we get the intertwining operator from $(\cH_1)_{\tilde{K}_1}$ to $\cH_{\tilde{K}}$, and by Theorem~\ref{extend}, this extends to the intertwining operator between the spaces of all holomorphic functions if $\cH_1$ is holomorphic discrete. Also, by Theorem~\ref{continuation}, this continues meromorphically for all $\lambda\in\BC$. Therefore we have the following.
\begin{Theorem}\label{main4}\quad
\begin{enumerate}\itemsep=0pt
\item[$(1)$] Let $(G,G_1)=(\operatorname{SU}(s,s), \operatorname{Sp}(s,\BR))$ with $s\ge 2$. Let $k\in\BZ_{\ge 0}$ if $s$ is even, $k=0$ if $s$ is odd,
and $l\in\left\{0,\ldots,\left\lceil\frac{s}{2}\right\rceil-1\right\}$. Then the linear map
\begin{gather*}
\cF_{\lambda,k,l}\colon \ \cO_{\lambda+k}\big(D_1,V_{\langle 2l\rangle}^{(s)\vee}\big)\to \cO_\lambda(D), \\
(\cF_{\lambda,k,l}f)(x_1+x_2) =\Pf(x_2)^k\sum_{\bm\in\BZ_{++}^{\lfloor s/2\rfloor}}
\sum_{\substack{\bl\in\{0,1\}^{\lfloor s/2\rfloor},\; |\bl|=l\\ \bm+\bl\in\BZ_{++}^{\lfloor s/2\rfloor}}}
\frac{1}{\left(\lambda+k+\langle l\rangle-\frac{1}{2}\right)_{\bm+\bl-\langle l\rangle,2}}\\
\hphantom{(\cF_{\lambda,k,l}f)(x_1+x_2) =}{} \times
\cK_{\bm,\bl}^{(4,1)}\left(x_2;\overline{\frac{\partial}{\partial x_1}}\right)f(x_1)
\end{gather*}
$(x_1\in\Sym(s,\BC), x_2\in\Skew(s,\BC))$ intertwines the $\tilde{G}_1$-action.
\item[$(2)$] Let $(G,G_1)=(\operatorname{SU}(s,s), \operatorname{SO}^*(2s))$ with $s\ge 2$, and $k,l\in\BZ_{\ge 0}$. Then the linear map
\begin{gather*}
\cF_{\lambda,k,l}\colon \ \cO_{2\lambda+4k}\big(D_1,V_{(2l,0,\ldots,0)}^{(s)\vee}\big)\to \cO_\lambda(D), \\
(\cF_{\lambda,k,l}f)(x_1+x_2)=\det(x_2)^k\sum_{\bm\in\BZ_{++}^{\lfloor s/2\rfloor}}
\sum_{\substack{\bl\in(\BZ_{\ge 0})^{\lceil s/2\rceil},\; |\bl|=l\\ 0\le l_j\le m_{j-1}-m_j}}\\
\qquad{} \times\frac{1}{\left(\lambda+2k+(l,0,\ldots,0)+\frac{1}{2}\right)_{\bm+\bl-(l,0,\ldots,0),2}}
\cK_{\bm,\bl}^{(1,4)}\left(x_2;\frac{1}{2}\overline{\frac{\partial}{\partial x_1}}\right)f(y_1)
\end{gather*}
$(x_1\in\Skew(s,\BC), x_2\in\Sym(s,\BC))$ intertwines the $\tilde{G}_1$-action. Here we identify $\bm=(m_1,\ldots,m_{\lfloor s/2\rfloor})\in\BZ_{++}^{\lfloor s/2\rfloor}$ with
$(m_1,\ldots,m_{\lfloor s/2\rfloor},0)\in\BZ_{++}^{\lceil s/2\rceil}$ if $s$ is odd.
\item[$(3)$] Let $(G,G_1)=(\operatorname{SU}(s,s), \operatorname{SO}^*(2s))$ with $s\ge 2$ even, and $k,l\in\BZ_{\ge 0}$. Then the linear map
\begin{gather*}
\cF_{\lambda,k,l}\colon \ \cO_{2\lambda+4k}\big(D_1,V_{(2l,\dots,2l,0)}^{(s)\vee}\big)\to \cO_\lambda(D), \\
(\cF_{\lambda,k,l}f)(x_1+x_2)=\det(x_2)^k\sum_{\bm\in\BZ_{++}^{s/2}}
\sum_{\substack{\bl\in(\BZ_{\ge 0})^{s/2},\; |\bl|=l\\ 0\le l_j\le m_j-m_{j+1}}}\\
\qquad{} \times\frac{1}{\left(\lambda+2k+l+(l,\ldots,l,0)+\frac{1}{2}\right)_{\bm-\bl+(0,\ldots,0,l),2}}
\cK_{\bm,-\bl}^{(1,4)}\left(x_2;\frac{1}{2}\overline{\frac{\partial}{\partial x_1}}\right)f(x_1)
\end{gather*}
$(x_1\in\Skew(s,\BC), x_2\in\Sym(s,\BC))$ intertwines the $\tilde{G}_1$-action.
\end{enumerate}
\end{Theorem}
\begin{Proposition}\label{main_prop}Let $(G,G_1)=(\operatorname{SU}(s,s), \operatorname{SO}^*(2s))$ with $s\ge 2$ odd, and $k,l\in\BZ_{\ge 0}$. Then there exist monic polynomials $\varphi_{\bm,-\bl}(\mu)\in\BC[\mu]$ of degree $l-l_{\lceil s/2\rceil}$ such that the linear map
\begin{gather*}
\cF_{\lambda,k,l}\colon \ \cO_{2\lambda+4k}\big(D_1,V_{(2l,\dots,2l,0)}^{(s)\vee}\big)\to \cO_\lambda(D), \\
(\cF_{\lambda,k,l}f)(x_1+x_2)=\det(x_2)^k\sum_{\bm\in\BZ_{++}^{\lfloor s/2\rfloor}}
\sum_{\substack{\bl\in(\BZ_{\ge 0})^{\lceil s/2\rceil},\; |\bl|=l\\ 0\le l_j\le m_j-m_{j+1}\\ 0\le l_{\lceil s/2\rceil}}}\\
\quad {} \times\frac{\varphi_{\bm,-\bl}(\lambda+2k)}{\left(\lambda+2k+2l+\frac{1}{2}\right)_{\bm-\bl',2}
\left(\lambda+2k+l-\left\lfloor\frac{s}{2}\right\rfloor+1\right)_{l-l_{\lceil s/2\rceil}}
\left(\lambda+2k+l-\left\lfloor\frac{s}{2}\right\rfloor+\frac{1}{2}\right)_{l-l_{\lceil s/2\rceil}}} \\
\quad{}\times\cK_{\bm,-\bl}^{(1,4)}\left(x_2;\frac{1}{2}\overline{\frac{\partial}{\partial x_1}}\right)f(x_1)
\end{gather*}
$(x_1\in\Skew(s,\BC), x_2\in\Sym(s,\BC))$ intertwines the $\tilde{G}_1$-action.
Here for $\bl=(l_1,\ldots,l_{\lfloor s/2\rfloor},\allowbreak l_{\lceil s/2\rceil})\in(\BZ_{\ge 0})^{\lceil s/2\rceil}$, we put
$\bl'=(l_1,\ldots,l_{\lfloor s/2\rfloor})\in(\BZ_{\ge 0})^{\lfloor s/2\rfloor}$.
\end{Proposition}
Later we prove $\varphi_{\bm,-\bl}(\mu)=\left(\mu+l-\left\lfloor\frac{s}{2}\right\rfloor+\frac{1}{2}\right)_{l-l_{\lceil s/2\rceil}}$
in Section \ref{sect_remaining}.

In Theorem \ref{main4}, if $l=0$, then these maps are reduced to
\begin{gather}
\cF_{\lambda,k}\colon \ \cO_{\lambda+k}(D_1)\to \cO_\lambda(D), \notag \\
(\cF_{\lambda,k}f)(x_1+x_2)=\Pf(x_2)^k
\sum_{\bm\in\BZ_{++}^{\lfloor s/2\rfloor}}\frac{1}{\left(\lambda+k-\frac{1}{2}\right)_{\bm,2}}
\tilde{\Phi}_\bm^{(2)\prime}\left(\left(x_2\frac{\partial}{\partial x_1}\right)^2\right)f(x_1) \label{SU-Sp}
\end{gather}
when $G_1=\operatorname{Sp}(s,\BR)$, and
\begin{gather}
\cF_{\lambda,k}\colon \ \cO_{2\lambda+4k}(D_1)\to \cO_\lambda(D), \notag \\
(\cF_{\lambda,k}f)(x_1+x_2)=\det(x_2)^k
\sum_{\bm\in\BZ_{++}^{\lfloor s/2\rfloor}}\frac{1}{\left(\lambda+2k+\frac{1}{2}\right)_{\bm,2}}
\tilde{\Phi}_\bm^{(2)\prime}\left(\left(\frac{1}{2}x_2\frac{\partial}{\partial x_1}\right)^2\right)f(x_1) \label{SU-SO*}
\end{gather}
when $G_1=\operatorname{SO}^*(2s)$, where $\tilde{\Phi}^{(2)\prime}_\bm\left((x_2w_1)^2\right)$ is defined in~(\ref{Schur2}),~(\ref{Schur3}). We note that the difference between $\frac{\partial}{\partial x_1}$ in $G_1=\operatorname{Sp}(s,\BR)$ case and $\frac{1}{2}\frac{\partial}{\partial x_1}$ in $G_1=\operatorname{SO}^*(2s)$ case is caused by the difference of the normalization of the inner product on $\Sym(s,\BC)$ and $\Skew(s,\BC)$. In fact, on $\Sym(s,\BC)$ we have $\frac{\partial}{\partial x}=\big(\frac{1+\delta_{ij}}{2}\frac{\partial}{\partial x_{ij}}\big)_{ij}$, and on $\Skew(s,\BC)$ we have $ \frac{1}{2}\frac{\partial}{\partial x}=\big(\frac{\sgn(j-i)}{2}\frac{\partial}{\partial x_{ij}}\big)_{ij}$, so both are similar.

\subsection[$\cF_{\tau\rho}$ for $(G,G_1)=(\operatorname{SO}_0(2,n), \operatorname{SO}_0(2,n')\times \operatorname{SO}(n''))$, \\
$(E_{6(-14)}, \operatorname{SU}(2,4)\times \operatorname{SU}(2))$, $(E_{7(-25)}, \operatorname{SU}(2)\times \operatorname{SO}^*(12))$]{$\boldsymbol{\cF_{\tau\rho}}$ for $\boldsymbol{(G,G_1)=(\operatorname{SO}_0(2,n), \operatorname{SO}_0(2,n')\times \operatorname{SO}(n''))}$, \\
$\boldsymbol{(E_{6(-14)}, \operatorname{SU}(2,4)\times \operatorname{SU}(2))}$, $\boldsymbol{(E_{7(-25)}, \operatorname{SU}(2)\times \operatorname{SO}^*(12))}$}
In this subsection we set
\begin{gather*} (G,G_1)= \begin{cases}(\operatorname{SO}_0(2,n), \operatorname{SO}_0(2,n')\times \operatorname{SO}(n'')) \quad (n=n'+n'')& (\text{Case }1), \\
(E_{6(-14)}, \operatorname{SU}(2,4)\times \operatorname{SU}(2)) & (\text{Case }2), \\
(E_{7(-25)}, \operatorname{SU}(2)\times \operatorname{SO}^*(12)) & (\text{Case }3) \end{cases} \end{gather*}
(up to covering). Then the maximal compact subgroups $(K,K_1)\subset (G,G_1)$ are given by
\begin{gather*}
(K,K_1)=\begin{cases} (\operatorname{SO}(2)\times \operatorname{SO}(n), \operatorname{SO}(2)\times \operatorname{SO}(n')\times \operatorname{SO}(n''))& (\text{Case }1),\\
(U(1)\times \operatorname{Spin}(10), S(U(2)\times U(4))\times \operatorname{SU}(2))& (\text{Case }2),\\
(U(1)\times E_6,\operatorname{SU}(2)\times U(6))& (\text{Case }3) \end{cases}
\end{gather*}
(up to covering). Also we have
\begin{gather*} \fp^+=\begin{cases}\BC^n & (\text{Case }1),\\ M(1,2;\BO)^\BC & (\text{Case }2),\\ \Herm(3,\BO)^\BC & (\text{Case }3), \end{cases} \end{gather*}
and $\fp^+_1:=\fg_1^\BC\cap\fp^+$, $\fp^+_2:=(\fp^+_1)^\bot$ are realized as
\begin{gather*} (\fp^+_1,\fp^+_2)=\begin{cases}
\big(\BC^{n'}, \BC^{n''}\big) & (\text{Case }1),\\
(M(2,4;\BC), M(4,2;\BC)) & (\text{Case }2),\\
(\Skew(6,\BC), M(2,6;\BC)) & (\text{Case }3). \end{cases} \end{gather*}
Let $\chi$, $\chi_1$ be the characters of $K^\BC$, $K_1^\BC$ respectively, normalized as (\ref{char}), and also let $\chi_2$ be the character of $K_1^\BC$ normalized as (\ref{char}) with respect to the Lie algebra $\fp^+_2\oplus\fk_1^\BC\oplus\fp^-_2$. Then $\chi|_{K_1}=\chi_1=\chi_2$ holds.

Now let $(\tau,V)=\big(\chi^{-\lambda},\BC\big)$ with $\lambda$ sufficiently large, $W\subset\cP(\fp^+_2)\otimes\chi^{-\lambda}$ be an irreducible $\tilde{K}^\BC_1$-submodule, and $\rK(x_2)\in \cP\big(\fp^+_2,\Hom\big(W,\chi^{-\lambda}\big)\big)$ be the $\tilde{K}^\BC_1$-invariant polynomial in the sense of~(\ref{K-invariance}). For $x_2\in\fp^+_2$, $w_1\in\fp^+_1$, we want to compute
\begin{align*}
F_{\tau\rho}(x_2;w_1)
=\big\langle {\rm e}^{(y_1|w_1)_{\fp^+_1}}I_W,
\big(h(x_2,Q(y_1)x_2)^{-\lambda/2}\rK\big((x_2)^{Q(y_1)x_2}\big)\big)^*\big\rangle_{\hat{\rho},y_1}.
\end{align*}
Now for $y_1\in\fp^+_1$ and $x_2\in\fp^+_2$, $Q(y_1)x_2\in\fp^+_2$ is given by
\begin{gather*} Q(y_1)x_2=\begin{cases} 2q(y_1,\overline{x_2})y_1-q(y_1)\overline{x_2}=-q(y_1)\overline{x_2} & (\text{Case }1),\\
-\big({}^t\hspace{-1pt}y_1J_2y_1\big)^\#\overline{x_2}J_2 & (\text{Case }2),\\
J_2\overline{x_2}(y_1)^\# & (\text{Case }3), \end{cases} \end{gather*}
where $q(\cdot)$ in Case 1 is as (\ref{quadra}), $x^\#$ in Cases 2 and 3 are defined in (\ref{adjoint4}) and (\ref{adjoint6}) respectively,
and $J_2=\left(\begin{smallmatrix}0&1\\-1&0\end{smallmatrix}\right)\in\Skew(2,\BC)$.
Similarly, $h(x_2,Q(y_1)x_2)^{-\lambda/2}=h_2(x_2,Q(y_1)x_2)^{-\lambda/2}$ is given by
\begin{align*}
h_2(x_2,Q(y_1)x_2)^{-\lambda/2}&=\begin{cases}
\big(1-2q(x_2,-q(\overline{y_1})x_2)+q(x_2)q(-q(\overline{y_1})x_2)\big)^{-\lambda/2} & (\text{Case }1),\\
\det\big(I_2+J_2{}^t\hspace{-1pt}x_2(y_1^*J_2\overline{y_1})^\#x_2\big)^{-\lambda/2} & (\text{Case }2),\\
\det\big(I_2+x_2(y_1^*)^\#{}^t\hspace{-1pt}x_2J_2\big)^{-\lambda/2} & (\text{Case }3) \end{cases}\\
&=\begin{cases}
\big(1+2q(x_2)q(\overline{y_1})+q(x_2)^2q(\overline{y_1})^2\big)^{-\lambda/2} & (\text{Case }1),\\
\det\big(J_2-{}^t\hspace{-1pt}x_2(y_1^*J_2\overline{y_1})^\#x_2\big)^{-\lambda/2} & (\text{Case }2),\\
\det\big(J_2-x_2(y_1^*)^\#{}^t\hspace{-1pt}x_2\big)^{-\lambda/2} & (\text{Case }3) \end{cases}\\
&=\begin{cases}
 (1+q(x_2)q(\overline{y_1}) )^{-\lambda} & (\text{Case }1),\\
\big(1-\Pf\big({}^t\hspace{-1pt}x_2(y_1^*J_2\overline{y_1})^\#x_2\big)\big)^{-\lambda} & (\text{Case }2),\\
\left(1-\Pf(x_2(y_1^*)^\#{}^t\hspace{-1pt}x_2)\right)^{-\lambda} & (\text{Case }3) \end{cases}\\
&=\big(1-\tfrac{1}{2}(x_2|Q(y_1)x_2)_{\fp^+}\big)^{-\lambda},
\end{align*}
and $(x_2)^{Q(y_1)x_2}\in\fp^+_2$ is given by
\begin{align*}
(x_2)^{Q(y_1)x_2}&=\begin{cases}
 (1+q(x_2)q(\overline{y_1}) )^{-2} (x_2+q(x_2)q(\overline{y_1})x_2 ) & (\text{Case }1),\\
x_2\big(I_2+J_2{}^t\hspace{-1pt}x_2(y_1^*J_2\overline{y_1})^\#x_2\big)^{-1} & (\text{Case }2),\\
\big(I_2+x_2(y_1^*)^\#{}^t\hspace{-1pt}x_2J_2\big)^{-1}x_2 & (\text{Case }3) \end{cases}\\
&=\begin{cases}
 (1+q(x_2)q(\overline{y_1}) )^{-1}x_2 & (\text{Case }1),\\
\big(1-\Pf({}^t\hspace{-1pt}x_2(y_1^*J_2\overline{y_1})^\#x_2)\big)^{-1}x_2 & (\text{Case }2),\\
\big(1-\Pf(x_2(y_1^*)^\#{}^t\hspace{-1pt}x_2)\big)^{-1}x_2 & (\text{Case }3) \end{cases}\\
&=\big(1-\tfrac{1}{2}(x_2|Q(y_1)x_2)_{\fp^+}\big)^{-1}x_2.
\end{align*}
Therefore if $\rK(x_2)$ is homogeneous of degree $k$, then we have
\begin{gather*}
 h(x_2,Q(y_1)x_2)^{-\lambda/2}\rK\big((x_2)^{Q(y_1)x_2}\big)\\
 \qquad{} =\big(1-\tfrac{1}{2}(x_2|Q(y_1)x_2)_{\fp^+}\big)^{-\lambda}\rK\big(\big(1-\tfrac{1}{2}(x_2|Q(y_1)x_2)_{\fp^+}\big)^{-1}x_2\big)\\
\qquad{}=\big(1-\tfrac{1}{2}(x_2|Q(y_1)x_2)_{\fp^+}\big)^{-\lambda-k}\rK(x_2) =\sum_{m=0}^\infty \frac{(\lambda+k)_m}{m!}\big(\tfrac{1}{2}(x_2|Q(y_1)x_2)_{\fp^+}\big)^m\rK(x_2),
\end{gather*}
and hence we have
\begin{gather*}
F_{\tau\rho}(x_2;w_1)
=\sum_{m=0}^\infty \frac{(\lambda+k)_m}{m!}\big\langle {\rm e}^{(y_1|w_1)_{\fp^+}}I_W,
\big(\big(\tfrac{1}{2}(x_2|Q(y_1)x_2)_{\fp^+}\big)^m\rK(x_2)\big)^*\big\rangle_{\hat{\rho},y_1}.
\end{gather*}
Now we set
\begin{alignat*}{3}
& W=\cP_{(k_1,k_2)}(\fp^+_2)\otimes\chi^{-\lambda}
\simeq\chi_1^{-\lambda-k_1-k_2}\boxtimes \BC^{[n']}\boxtimes V_{(k_1-k_2,0,\ldots,0)}^{[n'']\vee} \quad && (\text{Case }1),& \\
& W=\cP_{(k_1,k_2)}(\fp^+_2)\otimes\chi^{-\lambda}
\simeq\big(\chi_1^{-\lambda}\otimes V_{(0,0;-k_2,-k_1,-k_1-k_2,-k_1-k_2)}^{(2,4)\vee}\big)\boxtimes V_{(k_1-k_2,0)}^{(2)\vee} \quad && (\text{Case }2),& \\
& W=\cP_{(k,0)}(\fp^+_2)\otimes\chi^{-\lambda}
\simeq V_{(k,0)}^{(2)\vee}\boxtimes\big(\chi_1^{-\lambda}\otimes
V_{\left(\frac{k}{2},\frac{k}{2},\frac{k}{2},\frac{k}{2},\frac{k}{2},-\frac{k}{2}\right)}^{(6)\vee}\big) \quad && (\text{Case }3),&
\end{alignat*}
where for Case 1, if $n''=1$ we assume $k_1=k_2$ or $k_1=k_2+1$ so that $(k_1,k_2)=\big(\big\lceil\frac{l}{2}\big\rceil,\big\lfloor\frac{l}{2}\big\rfloor\big)$ for some $l\in\BZ_{\ge 0}$, and regard $V_{(k_1-k_2,0,\ldots,0)}^{[n'']\vee}$ as trivial. Also, if $n''=2$ we do not assume $k_1\ge k_2$ (see Section~\ref{realize}). Then $\rK(x_2)=\rK_\bk^{(d_2)}(x_2)$ is homogeneous of degree $|\bk|$, where $\bk=(k_1,k_2)$ for Cases~1,~2, and $\bk=(k,0)$ for Case~3. Also, we have
\begin{gather*} \big(\tfrac{1}{2}(x_2|Q(y_1)x_2)_{\fp^+}\big)^m\rK(x_2)
\in\Big(\bigoplus_{\substack{\bm\in\BZ_{++}^2\\ |\bm|=2m+|\bk|}}\cP_\bm(\fp^+_2)\otimes
\overline{\bigoplus_{\substack{\bn\in\BZ_{++}^{r_1}\\ |\bn|=2m}}\cP_\bn(\fp^+_1)\otimes\cP_{\bk}(\fp^+_2)}\Big)^{K_1}, \end{gather*}
where $r_1=2$ for Cases 1, 2, and $r_1=3$ for Case 3. This space is computed as
\begin{gather*}
 \Big(\bigoplus_{\substack{\bm\in\BZ_{++}^2\\ |\bm|=2m+|\bk|}}\chi_1^{-|\bm|}\boxtimes\BC^{[n']}\boxtimes V_{(m_1-m_2,0,\ldots,0)}^{[n'']\vee} \\
\qquad\quad{} \otimes\overline{\bigoplus_{\substack{\bn\in\BZ_{++}^2\\ |\bn|=2m}}
\chi_1^{-|\bn|-|\bk|}\boxtimes V_{(n_1-n_2,0,\ldots,0)}^{[n']\vee}\boxtimes V_{(k_1-k_2,0,\ldots,0)}^{[n'']\vee}}
\Big)^{K_1}\\
 \qquad{} =\Big(\chi_1^{-2m-|\bk|}\boxtimes \BC^{[n']}\boxtimes V_{(k_1-k_2,0,\ldots,0)}^{[n'']\vee}
\otimes\overline{\chi_1^{-2m-|\bk|}\boxtimes \BC^{[n']}\boxtimes V_{(k_1-k_2,0,\ldots,0)}^{[n'']\vee}}\Big)^{K_1}
\end{gather*}
for Case 1 (if $n'$ or $n''$ equals 1 or 2, then the range of $\bm$, $\bn$ changes, but the result does not change),
\begin{gather*}
 \Big(\bigoplus_{\substack{\bm\in\BZ_{++}^2\\ |\bm|=2m+|\bk|}}
V_{(0,0;-m_2,-m_1,-m_1-m_2,-m_1-m_2)}^{(2,4)\vee}\boxtimes V_{(m_1-m_2,0)}^{(2)\vee}\\
\qquad\quad{} \otimes\overline{\bigoplus_{\substack{\bn\in\BZ_{++}^2\\ |\bn|=2m}}
V_{(n_1,n_2;0,0,-n_2,-n_1)}^{(2,4)\vee}\boxtimes\BC^{(2)}
\otimes V_{(0,0;-k_2,-k_1,-k_1-k_2,-k_1-k_2)}^{(2,4)\vee}\boxtimes V_{(k_1-k_2,0)}^{(2)\vee}}
\Big)^{K_1}\\
\qquad{} =\Big(V_{(m,m;-k_2,-k_1,-m-k_1-k_2,-m-k_1-k_2)}^{(2,4)\vee}\boxtimes V_{(k_1-k_2,0)}^{(2)\vee} \\
\qquad\quad{} \otimes\overline{V_{(m,m;-k_2,-k_1,-m-k_1-k_2,-m-k_1-k_2)}^{(2,4)\vee}\boxtimes V_{(k_1-k_2,0)}^{(2)\vee}}\Big)^{K_1}
\end{gather*}
for Case 2 (we note that $V_{\substack{(m,m;-k_2,-k_1,\hspace{40pt}\\ \; -m-k_1-k_2,-m-k_1-k_2)}}^{(2,4)\vee}
\simeq V_{\substack{(0,0;-m-k_2,-m-k_1,\hspace{30pt}\\ \; -2m-k_1-k_2,-2m-k_1-k_2)}}^{(2,4)\vee}$ as $S(U(2)\times U(4))$-modules), and
\begin{gather*}
\Big(\bigoplus_{\substack{\bm\in\BZ_{++}^2\\ |\bm|=2m+k}}V_{(m_1-m_2,0)}^{(2)\vee}\boxtimes
V_{\left(\frac{m_1+m_2}{2},\frac{m_1+m_2}{2},\frac{m_1+m_2}{2},\frac{m_1+m_2}{2},\frac{m_1-m_2}{2},\frac{m_2-m_1}{2}\right)}^{(6)\vee} \\
\qquad\quad{} \otimes\overline{\bigoplus_{\substack{\bn\in\BZ_{++}^3\\ |\bn|=2m}}
\BC^{(2)}\boxtimes V_{(n_1,n_1,n_2,n_2,n_3,n_3)}^{(6)\vee}
\otimes V_{(k,0)}^{(2)\vee}\boxtimes V_{\left(\frac{k}{2},\frac{k}{2},\frac{k}{2},\frac{k}{2},\frac{k}{2},-\frac{k}{2}\right)}^{(6)\vee}}
\Big)^{K_1}\\
\qquad{}=\Big(V_{(k,0)}^{(2)\vee}\boxtimes V_{\left(m+\frac{k}{2},m+\frac{k}{2},m+\frac{k}{2},m+\frac{k}{2},\frac{k}{2},-\frac{k}{2}\right)}^{(6)\vee}
\otimes\overline{V_{(k,0)}^{(2)\vee}\boxtimes
V_{\left(m+\frac{k}{2},m+\frac{k}{2},m+\frac{k}{2},m+\frac{k}{2},\frac{k}{2},-\frac{k}{2}\right)}^{(6)\vee}}\Big)^{K_1}
\end{gather*}
for Case 3. Therefore by using the results of \cite{N2} and (\ref{exp_onD}) we have
\begin{gather*}
 F_{\tau\rho}(x_2;w_1) =\begin{cases}
\ds \sum_{m=0}^{\infty}\frac{(\lambda+k_1+k_2)_m}{(\lambda+k_1+k_2)_{(m,m),n'-2}}\frac{1}{m!}& \\
\quad{}\times \big(\tfrac{1}{2}(x_2|Q(w_1)x_2)_{\fp^+}\big)^m
\rK_{(k_1,k_2)}^{(n''-2)}(x_2) & (\text{Case }1),\\
\ds \sum_{m=0}^\infty \frac{(\lambda+k_1+k_2)_m}{(\lambda+(k_1+k_2,k_1+k_2,k_1,k_2))_{(m,m,0,0),2}}& \\
\quad{}\times
\dfrac{1}{m!}\big(\tfrac{1}{2}(x_2|Q(w_1)x_2)_{\fp^+}\big)^m\rK_{(k_1,k_2)}^{(2)}(x_2),& (\text{Case }2),\\
\ds \sum_{m=0}^\infty \frac{(\lambda+k)_m}{(\lambda+(k,k,0))_{(m,m,0),4}}& \\
\quad{}\times \dfrac{1}{m!}\big(\tfrac{1}{2}(x_2|Q(w_1)x_2)_{\fp^+}\big)^m
\rK_{(k,0)}^{(2)}(x_2) & (\text{Case }3)\end{cases}\\
 \hphantom{F_{\tau\rho}(x_2;w_1)}{} =\begin{cases}
\ds \sum_{m=0}^{\infty}\frac{1}{\left(\lambda+k_1+k_2-\frac{n'-2}{2}\right)_m}\frac{(-1)^m}{m!}q(x_2)^m\overline{q(w_1)}^m
\rK_{(k_1,k_2)}^{(n''-2)}(x_2) & (\text{Case }1),\\
\ds \sum_{m=0}^\infty \frac{1}{(\lambda+k_1+k_2-1)_m}\frac{1}{m!}\Pf\left({}^t\hspace{-1pt}x_2(w_1^*J_2\overline{w_1})^\#x_2\right)^m
\rK_{(k_1,k_2)}^{(2)}(x_2)\!\! & (\text{Case }2),\\
\ds \sum_{m=0}^\infty \frac{1}{(\lambda+k-2)_m}\frac{1}{m!}\Pf\left(x_2(w_1^*)^\#{}^t\hspace{-1pt}x_2\right)^m
\rK_{(k,0)}^{(2)}(x_2) & (\text{Case }3).
\end{cases}
\end{gather*}
By Theorem \ref{main}, by substituting $w_1$ with $\overline{\frac{\partial}{\partial x_1}}$, we get the intertwining operator from $(\cH_1)_{\tilde{K}_1}$ to $\cH_{\tilde{K}}$, and by Theorem~\ref{extend}, this extends to the intertwining operator between the spaces of all holomorphic functions if $\cH_1$ is holomorphic discrete. Also, by Theorem \ref{continuation}, this continues meromorphically for all~$\lambda\in\BC$. Therefore we have the following.
\begin{Theorem}\label{main5}\quad
\begin{enumerate}\itemsep=0pt
\item[$(1)$] Let $(G,G_1)=(\operatorname{SO}_0(2,n), \operatorname{SO}_0(2,n')\times \operatorname{SO}(n''))$ with $n=n'+n''$, and $(k_1,k_2)\in\BZ_{++}^2$ if $n''\ge 3$,
$(k_1,k_2)\in(\BZ_{\ge 0})^2$ if $n''=2$, $(k_1,k_2)=\left(\left\lceil\frac{l}{2}\right\rceil,\left\lfloor\frac{l}{2}\right\rfloor\right)$
for some $l\in\BZ_{\ge 0}$ if $n''=1$. Then the linear map
\begin{gather}
\cF_{\lambda,k_1,k_2}\colon \cO_{\lambda+k_1+k_2}(D_1)\boxtimes V_{(k_1-k_2,0,\ldots,0)}^{[n'']\vee}
\to \cO_\lambda(D), \notag \\
 (\cF_{\lambda,k_1,k_2}f)(x_1,x_2)\notag \\ =\sum_{m=0}^{\infty}
\frac{1}{\left(\lambda+k_1+k_2-\frac{n'-2}{2}\right)_m}\frac{(-1)^m}{m!}q(x_2)^m
q\left(\frac{\partial}{\partial x_1}\right)^m\rK_{(k_1,k_2)}^{(n''-2)}(x_2)f(x_1)\label{SO-SOSO}
\end{gather}
$(x_1\in\BC^{n'}, x_2\in\BC^{n''})$ intertwines the $\tilde{G}_1$-action.
\item[$(2)$] Let $(G,G_1)=(E_{6(-14)}, \operatorname{SU}(2,4)\times \operatorname{SU}(2))$ $($up to covering$)$, and $(k_1,k_2)\in\BZ_{++}^2$.
Then the linear map
\begin{gather*}
\cF_{\lambda,k_1,k_2}\colon \ \cO_{\lambda}\big(D_1,V_{(0,0;-k_2,-k_1,-k_1-k_2,-k_1-k_2)}^{(2,4)\vee}\big)
\boxtimes V_{(k_1-k_2,0)}^{(2)\vee}\to \cO_\lambda(D), \\
 (\cF_{\lambda,k_1,k_2}f)(x_1,x_2)\\
=\sum_{m=0}^\infty \frac{1}{(\lambda+k_1+k_2-1)_m}\frac{1}{m!}
\Pf\left({}^t\hspace{-1pt}x_2\left({\vphantom{\Bigl(}}^t\!\!\left(\frac{\partial}{\partial x_1}\right)J_2\frac{\partial}{\partial x_1}
\right)^\#x_2\right)^m\rK_{(k_1,k_2)}^{(2)}(x_2)f(x_1)
\end{gather*}
$(x_1\in M(2,4;\BC),x_2\in M(4,2;\BC))$ intertwines the $\tilde{G}_1$-action.
\item[$(3)$] Let $(G,G_1)=(E_{7(-25)}, \operatorname{SU}(2)\times \operatorname{SO}^*(12))$ $($up to covering$)$, and $k\in\BZ_{\ge 0}$.
Then the linear map
\begin{gather*}
\cF_{\lambda,k}\colon \ \cO_{\lambda}\big(D_1,V_{\left(\frac{k}{2},\frac{k}{2},\frac{k}{2},\frac{k}{2},\frac{k}{2},-\frac{k}{2}\right)}^{(6)\vee}\big)
\boxtimes V_{(k,0)}^{(2)\vee}\to \cO_\lambda(D), \\
(\cF_{\lambda,k}f)(x_1,x_2)=\sum_{m=0}^\infty \frac{1}{(\lambda+k-2)_m}\frac{1}{m!}
\Pf\left(x_2{\vphantom{\Bigl(}}^t\!\!\left(\frac{\partial}{\partial x_1}\right)^\#{}^t\hspace{-1pt}x_2\right)^m\rK_{(k,0)}^{(2)}(x_2)f(x_1)
\end{gather*}
$(x_1\in\Skew(6,\BC),x_2\in M(2,6;\BC))$ intertwines the $\tilde{G}_1$-action.
\end{enumerate}
\end{Theorem}

\section[Behavior of $\cF_{\tau\rho}$ when $\lambda$ is a pole]{Behavior of $\boldsymbol{\cF_{\tau\rho}}$ when $\boldsymbol{\lambda}$ is a pole}\label{section6}

In this section, we look at the behavior of $\cF_\lambda$ when $\lambda$ is a pole.
For simplicity we only treat the case that both $G$ and $G_1$ are classical and both $\cH$ and $\cH_1$ are of scalar type.
In this case, the underlying $(\fg,\tilde{K})$-module of the holomorphic discrete series representation $\cH_\lambda(D)$ of scalar type
is decomposed as
\begin{gather*} \cH_\lambda(D)_{\tilde{K}}=\cO_\lambda(D)_{\tilde{K}}=\cP(\fp^+)=\bigoplus_{\bm\in\BZ_{++}^r}\cP_\bm(\fp^+), \end{gather*}
and by (\ref{norm_compare}), (\ref{scalar_norm}), its analytic continuation $\cO_\lambda(D)_{\tilde{K}}$ is reducible if and only if
\begin{gather*} \begin{cases} \lambda\in\frac{1}{2}\BZ,\; \lambda\le\frac{1}{2}(r-1) & (G=\operatorname{Sp}(r,\BR),\;r\ge 2,\;(r,d)=(r,1)),\\
\lambda\in\BZ,\; \lambda\le \min\{q,s\}-1 & (G=\operatorname{SU}(q,s),\;(r,d)=(\min\{q,s\},2)),\\
\lambda\in\BZ,\; \lambda\le 2\left(\left\lfloor\frac{s}{2}\right\rfloor-1\right)
& (G=\operatorname{SO}^*(2s),\;(r,d)=(\left\lfloor\frac{s}{2}\right\rfloor,4)),\\
\lambda\in\BZ,\; \lambda\le \frac{n-2}{2} & (G=\operatorname{SO}_0(2,n),\; n\colon \text{even},\; (r,d)=(2,n-2)),\\
\lambda\in\BZ,\; \lambda\le 0 \text{ or }\lambda\in\BZ+\frac{1}{2},\; \lambda\le \frac{n-2}{2} & (G=\operatorname{SO}_0(2,n),\; n\colon \text{odd},\; (r,d)=(2,n-2)).
\end{cases} \end{gather*}
For these $\lambda$ and for $i=1,2,\ldots,r$, let
\begin{gather*} M_i(\lambda)=M_i^\fg(\lambda):=\bigoplus_{\substack{\bm\in\BZ_{++}^{r}\\ m_i\le\frac{d}{2}(i-1)-\lambda}}\cP_\bm(\fp^+). \end{gather*}
Also, since $\mathfrak{so}(2,1)\simeq\mathfrak{sl}(2,\BR)$ and $\cO_{\lambda}(D_{\operatorname{SO}_0(2,1)})\simeq\cO_{2\lambda}(D_{{\rm SL}(2,\BR)})$
(see Section~\ref{realize}), for $\lambda\in\BZ_{\le 0}$ we write
\begin{gather*} M_1^{\mathfrak{sl}(2,\BR)}(2\lambda)=\bigoplus_{m=0}^{-2\lambda}\cP_m(\fp^+_{\mathfrak{sl}(2,\BR)})
=\bigoplus_{m=0}^{-2\lambda}\cP_{\left(\left\lceil\frac{m}{2}\right\rceil,\left\lfloor\frac{m}{2}\right\rfloor\right)}(\fp^+_{\mathfrak{so}(2,1)})
=:\begin{cases} M_1^{\mathfrak{so}(2,1)}(\lambda) & (2\lambda\colon \text{even}),\\ M_2^{\mathfrak{so}(2,1)}(\lambda) & (2\lambda\colon \text{odd}).
\end{cases} \end{gather*}
Then the composition series are given by, when $G=\operatorname{Sp}(r,\BR)$ with $r\ge 2$,
\begin{align*}
&\cO_\lambda(D)_{\tilde{K}}\supset M_{2\left\lceil\frac{r}{2}\right\rceil-1}(\lambda)\supset M_{2\left\lceil\frac{r}{2}\right\rceil-3}(\lambda)
\supset\cdots\supset M_{\max\{2\lambda,0\}+1}(\lambda)\supset\{0\} && (\lambda\in\BZ), \\
&\cO_\lambda(D)_{\tilde{K}}\supset M_{2\left\lfloor\frac{r}{2}\right\rfloor}(\lambda)\supset M_{2\left\lfloor\frac{r}{2}\right\rfloor-2}(\lambda)
\supset\cdots\supset M_{\max\{2\lambda,1\}+1}(\lambda)\supset\{0\} && \big(\lambda\in\BZ+ \tfrac{1}{2}\big),
\end{align*}
when $G=\operatorname{SU}(q,s)$,
\begin{gather*} \cO_\lambda(D)_{\tilde{K}}\supset M_{\min\{q,s\}}(\lambda)\supset M_{\min\{q,s\}-1}(\lambda)
\supset\cdots\supset M_{\max\{\lambda,0\}+1}(\lambda)\supset\{0\}, \end{gather*}
when $G=\operatorname{SO}^*(2s)$,
\begin{gather*} \cO_\lambda(D)_{\tilde{K}}\supset M_{\left\lfloor\frac{s}{2}\right\rfloor}(\lambda)\supset M_{\left\lfloor\frac{s}{2}\right\rfloor-1}(\lambda)
\supset\cdots\supset M_{\max\left\{\left\lceil\frac{\lambda}{2}\right\rceil,0\right\}+1}(\lambda)\supset\{0\}, \end{gather*}
and when $G=\operatorname{SO}_0(2,n)$ with $n\ne 2$,
\begin{alignat*}{3}
&\cO_\lambda(D)_{\tilde{K}}\supset M_2(\lambda)\supset M_1(\lambda)\supset\{0\} \qquad & &(n\colon \text{even},\;\lambda\in\BZ,\;\lambda\le 0),&\\
&\cO_\lambda(D)_{\tilde{K}}\supset M_2(\lambda)\supset\{0\}\qquad & &\big(n\colon \text{even},\;\lambda\in\BZ,\;1\le\lambda\le \tfrac{n-2}{2}\big),&\\
&\cO_\lambda(D)_{\tilde{K}}\supset M_1(\lambda)\supset\{0\} \qquad& &(n\colon \text{odd},\;\lambda\in\BZ,\;\lambda\le 0),&\\
&\cO_\lambda(D)_{\tilde{K}}\supset M_2(\lambda)\supset\{0\} \qquad& &\big(n\colon \text{odd},\;\lambda\in\BZ+ \tfrac{1}{2},\;\lambda\le \tfrac{n-2}{2}\big)&
\end{alignat*}
(see \cite{FK0}). For $G=U(q,s)$ case, we use the same symbol $M_i(\lambda_1+\lambda_2)\subset\cO_{\lambda_1+\lambda_2}(D)_{\tilde{K}}$ as in the $G=\operatorname{SU}(q,s)$ case. We also write $M_0(\lambda)=M_{-1}(\lambda)=\{0\}$, $M_{r+1}(\lambda)=M_{r+2}(\lambda)=\cO_\lambda(D)_{\tilde{K}}$. We note that when $G=\operatorname{SO}_0(2,2)$, we have $\cO_\lambda(D_{\operatorname{SO}_0(2,2)})_{\tilde{K}}\simeq \cO_\lambda(D_{{\rm SL}(2,\BR)})_{\tilde{K}}\boxtimes\cO_\lambda(D_{{\rm SL}(2,\BR)})_{\tilde{K}}$, and its submodules are given by tensor products of $M_1^{\mathfrak{sl}(2,\BR)}(\lambda)$ or $\cO_\lambda(D_{{\rm SL}(2,\BR)})_{\tilde{K}}$.

First we consider $(G,G_1)=(G,G_{11}\times G_{22})=(U(q,s),U(q',s')\times U(q'',s''))$, and let $\cH=\cO_{\lambda_1+\lambda_2}(D)$, $\cH_1=\cO_{(\lambda_1+k)+(\lambda_2+l)}(D_{11})\hboxtimes\cO_{(\lambda_1+l)+(\lambda_2+k)}(D_{22})$, where $k=0$ if $q'\ne s''$, $l=0$ if $q''\ne s'$. Without loss of generality we may assume $q'\le q'',s',s''$. Then the intertwining operator
\begin{gather*} \cF_{\lambda,k,l}\colon \ \cO_{(\lambda_1+k)+(\lambda_2+l)}(D_{11}) \hboxtimes \cO_{(\lambda_1+l)+(\lambda_2+k)}(D_{22})
\to \cO_{\lambda_1+\lambda_2}(D) \end{gather*}
is given by (\ref{U-UU}),
\begin{gather*}
(\cF_{\lambda,k,l}f)\begin{pmatrix}x_{11}&x_{12}\\x_{21}&x_{22}\end{pmatrix}
 =\det(x_{12})^k\det(x_{21})^l
\sum_{\bm\in\BZ_{++}^{q'}}\frac{1}{(\lambda_1+\lambda_2+k+l)_{\bm,2}}\\
\hphantom{(\cF_{\lambda,k,l}f)\begin{pmatrix}x_{11}&x_{12}\\x_{21}&x_{22}\end{pmatrix}=}{} \times\tilde{\Phi}_\bm^{(2)}
\left(x_{12}{\vphantom{\biggl(}}^t\!\!\left(\frac{\partial}{\partial x_{22}}\right)x_{21}
{\vphantom{\biggl(}}^t\!\!\left(\frac{\partial}{\partial x_{11}}\right)\right)f(x_{11},x_{22}).
\end{gather*}
The poles are at $\lambda_1+\lambda_2+k+l\in \BZ_{\le q'-1}$. We write $\lambda_1+\lambda_2+k+l=:\mu$. Then the order of poles are $q'-\max\{\mu,0\}$.

Now we consider the residue of $\cF_{\lambda,k,l}$. Since as a function of $w_{11}$ and $w_{22}$, we have
\begin{gather*} \tilde{\Phi}_\bm^{(2)} (x_{12}w_{22}^*x_{21}w_{11}^* )\in\overline{\cP_{\bm}(\fp^+_{11})_{w_{11}}\boxtimes\cP_{\bm}(\fp^+_{22})_{w_{22}}},\end{gather*}
if $f(x_{11},x_{22})\in\cP_\bk(\fp^+_{11})\boxtimes\cP_\bl(\fp^+_{22})$ and $\bm$ satisfies $m_j>k_j$ or $m_j>l_j$ for some $j$, then we have
\begin{gather*} \tilde{\Phi}_\bm^{(2)}
\left(x_{12}{\vphantom{\biggl(}}^t\!\!\left(\frac{\partial}{\partial x_{22}}\right)
x_{21}{\vphantom{\biggl(}}^t\!\!\left(\frac{\partial}{\partial x_{11}}\right)\right)f(x_{11},x_{22})=0. \end{gather*}
Therefore, if $f\in M_{i+1}^{\fg_{11}}(\mu)\boxtimes \cO_{\mu}(D_{22})_{\tilde{K}_{22}}
+\cO_{\mu}(D_{11})_{\tilde{K}_{11}}\boxtimes M_{i+1}^{\fg_{22}}(\mu)$, then it holds that
\begin{gather*}
(\cF_{\lambda,k,l}f)\begin{pmatrix}x_{11}&x_{12}\\x_{21}&x_{22}\end{pmatrix}
 =\det(x_{12})^k\det(x_{21})^l
\sum_{\substack{\bm\in\BZ_{++}^{q'}\\ m_{i+1}\le i-\mu}}\frac{1}{(\lambda_1+\lambda_2+k+l)_{\bm,2}}\\
\hphantom{(\cF_{\lambda,k,l}f)\begin{pmatrix}x_{11}&x_{12}\\x_{21}&x_{22}\end{pmatrix}=}{} \times\tilde{\Phi}_\bm^{(2)}
\left(x_{12}{\vphantom{\biggl(}}^t\!\!\left(\frac{\partial}{\partial x_{22}}\right)x_{21}
{\vphantom{\biggl(}}^t\!\!\left(\frac{\partial}{\partial x_{11}}\right)\right)f(x_{11},x_{22}).
\end{gather*}
This has a pole of order $i-\max\{\mu,0\}$ at $\lambda_1+\lambda_2+k+l=\mu$. Therefore
\begin{gather*}
 \big(\tilde{\cF}^i_{\lambda,k,l}f\big)\begin{pmatrix}x_{11}&x_{12}\\x_{21}&x_{22}\end{pmatrix}
:=\lim_{(\nu_1,\nu_2)\to(\lambda_1,\lambda_2)}(\nu_1+\nu_2-\lambda_1-\lambda_2)^{i-\max\{\mu,0\}}
(\cF_{\nu,k,l}f)\begin{pmatrix}x_{11}&x_{12}\\x_{21}&x_{22}\end{pmatrix} \\
{}=\lim_{(\nu_1,\nu_2)\to(\lambda_1,\lambda_2)}\det(x_{12})^k\det(x_{21})^l \\
\qquad{}\times
\sum_{\substack{\bm\in\BZ_{++}^{q'}\\ m_{i+1}\le i-\mu}}\frac{(\nu_1+\nu_2-\lambda_1-\lambda_2)^{i-\max\{\mu,0\}}}{(\nu_1+\nu_2+k+l)_{\bm,2}}
\tilde{\Phi}_\bm^{(2)}\left(x_{12}{\vphantom{\biggl(}}^t\!\!\left(\frac{\partial}{\partial x_{22}}\right)
x_{21}{\vphantom{\biggl(}}^t\!\!\left(\frac{\partial}{\partial x_{11}}\right)\right)f(x_{11},x_{22})
\end{gather*}
is well-defined for $f\in M_{i+1}^{\fg_{11}}(\mu)\boxtimes \cO_{\mu}(D_{22})_{\tilde{K}_{22}} +\cO_{\mu}(D_{11})_{\tilde{K}_{11}}\boxtimes M_{i+1}^{\fg_{22}}(\mu)$. Similarly, $\tilde{\cF}^{i-1}_{\lambda,k,l}$ is well-defined on $M_{i}^{\fg_{11}}(\mu)\boxtimes \cO_{\mu}(D_{22})_{\tilde{K}_{22}} +\cO_{\mu}(D_{11})_{\tilde{K}_{11}}\boxtimes M_{i}^{\fg_{22}}(\mu)$, and therefore $\tilde{\cF}^i_{\lambda,k,l}$ is trivial on $M_i^{\fg_{11}}(\mu)\boxtimes \cO_{\mu}(D_{22})_{\tilde{K}_{22}}+\cO_{\mu}(D_{11})_{\tilde{K}_{11}}\boxtimes M_i^{\fg_{22}}(\mu)$. That is, the linear map
\begin{gather*}
\tilde{\cF}^i_{\lambda,k,l}\colon \
\big(M_{i+1}^{\fg_{11}}(\mu)\boxtimes \cO_{\mu}(D_{22})_{\tilde{K}_{22}}+\cO_{\mu}(D_{11})_{\tilde{K}_{11}}\boxtimes M_{i+1}^{\fg_{22}}(\mu)\big)\\
\hphantom{\tilde{\cF}^i_{\lambda,k,l}\colon}{} \ /\big(M_i^{\fg_{11}}(\mu)\boxtimes \cO_{\mu}(D_{22})_{\tilde{K}_{22}}+\cO_{\mu}(D_{11})_{\tilde{K}_{11}}\boxtimes M_i^{\fg_{22}}(\mu)\big)
\longrightarrow \cO_{\lambda_1+\lambda_2}(D)_{\tilde{K}}
\end{gather*}
is well-defined. Moreover, if $f\in M_{i+1}^{\fg_{11}}(\mu)\boxtimes M_{i+1}^{\fg_{22}}(\mu)$, then for $(g_{11},g_{22})\in\fg_{11}^\BC\oplus\fg_{22}^\BC$ we have
\begin{gather*}
{\rm d}\big(\rho_{(\nu_1+k)+(\nu_2+l)}^{\fg_{11}}\boxtimes \rho_{(\nu_1+l)+(\nu_2+k)}^{\fg_{22}}\big)(g_{11},g_{22})f \\
 \qquad{} \in M_{i+1}^{\fg_{11}}(\mu)\boxtimes \cO(D_{22})_{\tilde{K}_{22}}+\cO(D_{11})_{\tilde{K}_{11}}\boxtimes M_{i+1}^{\fg_{22}}(\mu)
\end{gather*}
for generic $(\nu_1,\nu_2)$. Therefore taking the limit $(\nu_1,\nu_2)\to(\lambda_1,\lambda_2)$ in the both sides of
\begin{gather*}
 {\rm d}\tau_{\nu_1+\nu_2}(g_{11},g_{22})(\nu_1+\nu_2-\lambda_1-\lambda_2)^{i-\max\{\mu,0\}}\cF_{\nu,k,l}f \\
\qquad{} =(\nu_1+\nu_2-\lambda_1-\lambda_2)^{i-\max\{\mu,0\}}\cF_{\nu,k,l}
{\rm d}\big(\rho_{(\nu_1+k)+(\nu_2+l)}^{\fg_{11}}\boxtimes \rho_{(\nu_1+l)+(\nu_2+k)}^{\fg_{22}}\big)(g_{11},g_{22})f,
\end{gather*}
we get
\begin{gather*} {\rm d}\tau_{\lambda_1+\lambda_2}(g_{11},g_{22})\tilde{\cF}_{\lambda,k,l}^if
=\tilde{\cF}_{\lambda,k,l}^i
{\rm d}\big(\rho_{(\lambda_1+k)+(\lambda_2+l)}^{\fg_{11}}\boxtimes \rho_{(\lambda_1+l)+(\lambda_2+k)}^{\fg_{22}}\big)(g_{11},g_{22})f. \end{gather*}
We note that if $f\in M_{i+1}^{\fg_{11}}(\mu)\boxtimes \cO(D_{22})_{\tilde{K}_{22}}+\cO(D_{11})_{\tilde{K}_{11}}\boxtimes M_{i+1}^{\fg_{22}}(\mu)$ is not in $M_{i+1}^{\fg_{11}}(\mu)\boxtimes M_{i+1}^{\fg_{22}}(\mu)$, then this does not hold since we cannot take the limits of $(\nu_1+\nu_2-\lambda_1-\lambda_2)^{i-\max\{\mu,0\}}\cF_{\nu,k,l}$ and ${\rm d}\big(\rho_{(\nu_1+k)+(\nu_2+l)}^{\fg_{11}}\boxtimes \rho_{(\nu_1+l)+(\nu_2+k)}^{\fg_{22}}\big)(g_{11},g_{22})f$ separately. Therefore the restriction of $\tilde{\cF}_{\lambda,k,l}^i$,
\begin{gather*}
\tilde{\cF}^i_{\lambda,k,l}\colon \ \big(M_{i+1}^{\fg_{11}}(\mu)\boxtimes M_{i+1}^{\fg_{22}}(\mu)\big)
/\big(M_i^{\fg_{11}}(\mu)\boxtimes M_{i+1}^{\fg_{22}}(\mu)+M_{i+1}^{\fg_{11}}(\mu)\boxtimes M_i^{\fg_{22}}(\mu)\big) \\
\hphantom{\tilde{\cF}^i_{\lambda,k,l}\colon}{} \ \simeq \big(M_{i+1}^{\fg_{11}}(\mu)/M_i^{\fg_{11}}(\mu)\big)\boxtimes\big(M_{i+1}^{\fg_{22}}(\mu)/M_i^{\fg_{22}}(\mu)\big)
\longrightarrow \cO_{\lambda_1+\lambda_2}(D)_{\tilde{K}}
\end{gather*}
intertwines the $(\fg_1,\tilde{K}_1)$-action. Similar phenomena also occur for $(G,G_1)=(G,G_{11}\times G_{22}) =(\operatorname{Sp}(s,\BR),\operatorname{Sp}(s',\BR)\times \operatorname{Sp}(s'',\BR))$, $(\operatorname{SO}^*(2s),\operatorname{SO}^*(2s')\times \operatorname{SO}^*(2s''))$, $(\operatorname{SO}_0(2,2+n''),\operatorname{SO}_0(2,2)\times \operatorname{SO}(n''))$.
\begin{Theorem}\quad
\begin{enumerate}\itemsep=0pt
\item[$(1)$] Let $(G,G_1)=(\operatorname{Sp}(s,\BR),\operatorname{Sp}(s',\BR)\times \operatorname{Sp}(s'',\BR))$ with $s=s'+s''$, $s'\le s''$.
Let $k\in\BZ_{\ge 0}$ $(k=0$ if $s'\ne s'')$. We assume $\mu:=\lambda+k\in\frac{1}{2}\BZ$, $\mu\le \frac{1}{2}(s'-1)$.
Then for $s'\ge 1$, for $\mu\in\BZ$, $\max\{0,\lfloor\mu\rfloor\}\le i\le \big\lceil \frac{s'}{2}\big\rceil$,
\begin{gather*}
\tilde{\cF}_{\lambda,k}^i:=\lim_{\nu\to\lambda} (\nu-\lambda)^{i-\max\{0,\lfloor\mu\rfloor\}}\cF_{\nu,k}\colon \\
\qquad{} \big(M_{2i+1}^{\fg_{11}}(\mu)\boxtimes \cO_\mu(D_{22})_{\tilde{K}_{22}}
+\cO_\mu(D_{11})_{\tilde{K}_{11}}\boxtimes M_{2i+1}^{\fg_{22}}(\mu)\big) \\
\qquad\quad {} /\big(M_{2i-1}^{\fg_{11}}(\mu)\boxtimes \cO_\mu(D_{22})_{\tilde{K}_{22}}
+\cO_\mu(D_{11})_{\tilde{K}_{11}}\boxtimes M_{2i-1}^{\fg_{22}}(\mu)\big)
\longrightarrow \cO_\lambda(D)_{\tilde{K}},
\end{gather*}
and for $s'\ge 2$, for $\mu\in\BZ+\frac{1}{2}$, $\max\{0,\lfloor\mu\rfloor\}\le i\le \big\lfloor \frac{s'}{2}\big\rfloor$,
\begin{gather*}
 \tilde{\cF}_{\lambda,k}^i:=\lim_{\nu\to\lambda} (\nu-\lambda)^{i-\max\{0,\lfloor\mu\rfloor\}}\cF_{\nu,k}\colon \\
\qquad{} \big(M_{2i+2}^{\fg_{11}}(\mu)\boxtimes \cO_\mu(D_{22})_{\tilde{K}_{22}}+\cO_\mu(D_{11})_{\tilde{K}_{11}}\boxtimes M_{2i+2}^{\fg_{22}}(\mu)\big) \\
\qquad\quad {} /\big(M_{2i}^{\fg_{11}}(\mu)\boxtimes \cO_\mu(D_{22})_{\tilde{K}_{22}}+\cO_\mu(D_{11})_{\tilde{K}_{11}}\boxtimes M_{2i}^{\fg_{22}}(\mu)\big)
\longrightarrow \cO_\lambda(D)_{\tilde{K}}
\end{gather*}
are well-defined $(\cF_{\lambda,k}$ is as \eqref{Sp-SpSp}$)$, and their restriction
\begin{alignat*}{3}
 & \tilde{\cF}_{\lambda,k}^i\colon \ \big(M_{2i+1}^{\fg_{11}}(\mu)/M_{2i-1}^{\fg_{11}}(\mu)\big)\boxtimes \big(M_{2i+1}^{\fg_{22}}(\mu)/M_{2i-1}^{\fg_{22}}(\mu)\big)
\longrightarrow \cO_\lambda(D)_{\tilde{K}} \quad && (\mu \in\BZ),&\\
&\tilde{\cF}_{\lambda,k}^i\colon \ \big(M_{2i+2}^{\fg_{11}}(\mu)/M_{2i}^{\fg_{11}}(\mu)\big)\boxtimes \big(M_{2i+2}^{\fg_{22}}(\mu)/M_{2i}^{\fg_{22}}(\mu)\big)
\longrightarrow \cO_\lambda(D)_{\tilde{K}} \quad && \big(\mu \in\BZ+\tfrac{1}{2}\big),&
\end{alignat*}
intertwine the $(\fg_1,\tilde{K}_1)$-action.
\item[$(2)$] Let $(G,G_1)=(U(q,s),U(q',s')\times U(q'',s''))$ with $q=q'+q''$, $s'=s'+s''$, $q'\le q'',s',s''$.
Let $k,l\in\BZ_{\ge 0}$ $(k=0$ if $q'\ne s''$, $l=0$ if $q''\ne s')$.
We assume $\mu:=\lambda_1+\lambda_2+k+l\in\BZ$, $\mu\le q'-1$. Let $\max\{0,\mu\}\le i\le q'$. Then
\begin{gather*}
 \tilde{\cF}_{\lambda,k,l}^i:=\lim_{(\nu_1,\nu_2)\to(\lambda_1,\lambda_2)} (\nu_1+\nu_2-\lambda_1-\lambda_2)^{i-\max\{0,\mu\}}\cF_{\nu,k,l} \colon\\
 \qquad{} \big(M_{i+1}^{\fg_{11}}(\mu)\boxtimes \cO_\mu(D_{22})_{\tilde{K}_{22}}+\cO_\mu(D_{11})_{\tilde{K}_{11}}\boxtimes M_{i+1}^{\fg_{22}}(\mu)\big) \\
\qquad\quad {} /\big(M_i^{\fg_{11}}(\mu)\boxtimes \cO_\mu(D_{22})_{\tilde{K}_{22}}+\cO_\mu(D_{11})_{\tilde{K}_{11}}\boxtimes M_i^{\fg_{22}}(\mu)\big)
\longrightarrow \cO_{\lambda_1+\lambda_2}(D)_{\tilde{K}}
\end{gather*}
is well-defined $(\cF_{\lambda,k,l}$ is as \eqref{U-UU}$)$, and its restriction
\begin{gather*} \tilde{\cF}_{\lambda,k,l}^i\colon \ \big(M_{i+1}^{\fg_{11}}(\mu)/M_i^{\fg_{11}}(\mu)\big)\boxtimes \big(M_{i+1}^{\fg_{22}}(\mu)/M_i^{\fg_{22}}(\mu)\big)
\longrightarrow \cO_{\lambda_1+\lambda_2}(D)_{\tilde{K}} \end{gather*}
intertwines the $(\fg_1,\tilde{K}_1)$-action.
\item[$(3)$] Let $(G,G_1)=(\operatorname{SO}^*(2s),\operatorname{SO}^*(2s')\times \operatorname{SO}^*(2s''))$ with $s=s'+s''$, $2\le s'\le s''$.
Let $k\in\BZ_{\ge 0}$ $(k=0$ if $s'\ne s'')$.
We assume $\mu:=\lambda+2k\in\BZ$, $\mu\le 2\big(\big\lfloor\frac{s'}{2}\big\rfloor-1\big)$. Let $\max\big\{0,\big\lceil\frac{\mu}{2}\big\rceil\big\}\le i\le \big\lfloor\frac{s'}{2}\big\rfloor$. Then
\begin{gather*}
 \tilde{\cF}_{\lambda,k}^i:=\lim_{\nu\to\lambda} (\nu-\lambda)^{i-\max\{0,\lceil\mu/2\rceil\}}\cF_{\nu,k}\colon \\
 \qquad {} \big(M_{i+1}^{\fg_{11}}(\mu)\boxtimes \cO_\mu(D_{22})_{\tilde{K}_{22}}+\cO_\mu(D_{11})_{\tilde{K}_{11}}\boxtimes M_{i+1}^{\fg_{22}}(\mu)\big) \\
\qquad \quad {} /\big(M_i^{\fg_{11}}(\mu)\boxtimes \cO_\mu(D_{22})_{\tilde{K}_{22}}+\cO_\mu(D_{11})_{\tilde{K}_{11}}\boxtimes M_i^{\fg_{22}}(\mu)\big)
\longrightarrow \cO_\lambda(D)_{\tilde{K}}
\end{gather*}
is well-defined $(\cF_{\lambda,k}$ is as \eqref{SO*-SO*SO*}$)$, and its restriction
\begin{gather*} \tilde{\cF}_{\lambda,k}^i\colon \ \big(M_{i+1}^{\fg_{11}}(\mu)/M_i^{\fg_{11}}(\mu)\big)\boxtimes \big(M_{i+1}^{\fg_{22}}(\mu)/M_i^{\fg_{22}}(\mu)\big)
\longrightarrow \cO_\lambda(D)_{\tilde{K}} \end{gather*}
intertwines the $(\fg_1,\tilde{K}_1)$-action.
\item[$(4)$] Let $(G,G_1)=(\operatorname{SO}_0(2,2+n''),\operatorname{SO}_0(2,2)\times \operatorname{SO}(n''))\simeq (\operatorname{SO}_0(2,2+n''),{\rm SL}(2,\BR)\times {\rm SL}(2,\BR)\times \operatorname{SO}(n''))$ $($up to covering$)$. Let $(k_1,k_2)\in\BZ_{++}^2$. We assume $\mu:=\lambda+k_1+k_2\in\BZ$, $\mu\le 0$. Then
\begin{gather*}
\tilde{\cF}_{\lambda,k_1,k_2}^0:=\cF_{\lambda,k_1,k_2}\colon \\
\quad \big(M_1^{\fg_{11}}(\mu)\boxtimes \cO_\mu(D_{22})_{\tilde{K}_{22}}+\cO_\mu(D_{11})_{\tilde{K}_{11}}\boxtimes M_1^{\fg_{22}}(\mu)\big)
\boxtimes V_{(k_1-k_2,0,\ldots,0)}^{[n'']\vee} \longrightarrow \cO_\lambda(D)_{\tilde{K}},
\\
\tilde{\cF}_{\lambda,k_1,k_2}^1:=\lim_{\nu\to\lambda} (\nu-\lambda)\cF_{\nu,k_1,k_2} \colon \ \bigl(\big(\cO_\mu(D_{11})_{\tilde{K}_{11}}\boxtimes \cO_\mu(D_{22})_{\tilde{K}_{22}}\big) \\
/\big(M_1^{\fg_{11}}(\mu)\boxtimes \cO_\mu(D_{22})_{\tilde{K}_{22}}+\cO_\mu(D_{11})_{\tilde{K}_{11}}\boxtimes M_1^{\fg_{22}}(\mu)\big)\bigr)
\boxtimes V_{(k_1-k_2,0,\ldots,0)}^{[n'']\vee} \longrightarrow \cO_\lambda(D)_{\tilde{K}}
\end{gather*}
are well-defined $(\cF_{\lambda,k_1,k_2}$ is as \eqref{SO-SOSO}$)$. Moreover,
\begin{gather*}
 \tilde{\cF}_{\lambda,k_1,k_2}^1\colon \ \big(\cO_\mu(D_{11})_{\tilde{K}_{11}}/M_1^{\fg_{11}}(\mu)\big)
\boxtimes \big(\cO_\mu(D_{22})_{\tilde{K}_{22}}/M_1^{\fg_{22}}(\mu)\big)\\
\hphantom{\tilde{\cF}_{\lambda,k_1,k_2}^1\colon}{} \ \boxtimes V_{(k_1-k_2,0,\ldots,0)}^{[n'']\vee} \longrightarrow \cO_\lambda(D)_{\tilde{K}},
\end{gather*}
and the restriction
\begin{gather*} \tilde{\cF}_{\lambda,k_1,k_2}^0\colon \ M_1^{\fg_{11}}(\mu)\boxtimes M_1^{\fg_{22}}(\mu)\boxtimes V_{(k_1-k_2,0,\ldots,0)}^{[n'']\vee}
\longrightarrow \cO_\lambda(D)_{\tilde{K}} \end{gather*}
intertwine the $(\fg_1,\tilde{K}_1)$-action.
\end{enumerate}
\end{Theorem}

Next we consider the case $(G,G_1)=(\operatorname{Sp}(s,\BR),U(s',s''))$, and let $\cH=\cO_\lambda(D)$, $\cH_1=\cO_{(\lambda+2k)+(\lambda+2l)}(D_1)$. Without loss of generality we may assume $s'\le s''$. Then the intertwining operator
\begin{gather*} \cF_{\lambda,k,l}\colon \ \cO_{(\lambda+2k)+(\lambda+2l)}(D_1)\to \cO_\lambda(D) \end{gather*}
is given by (\ref{Sp-U}),
\begin{gather*}
 (\cF_{\lambda,k,l}f)\begin{pmatrix}x_{11}&x_{12}\\{}^t\hspace{-1pt}x_{12}&x_{22}\end{pmatrix}=\det(x_{11})^k\det(x_{22})^l\\
\qquad{} \times\sum_{\bm\in\BZ_{++}^{s'}}\frac{1}{\big(\lambda+k+l+\frac{1}{2}\big)_{\bm,1}}\tilde{\Phi}_\bm^{(1)}
\left(\frac{1}{4}x_{11}\frac{\partial}{\partial x_{12}}x_{22}{\vphantom{\biggl(}}^t\!\!\left(\frac{\partial}{\partial x_{12}}\right)\right)f(x_{12}).
\end{gather*}
The poles are at $2(\lambda+k+l)\in \BZ_{\le s'-2}$ if $s'\ge 2$, at $2(\lambda+k+l)\in 2\BZ_{\le 0}-1$ if $s'=1$. We write $\lambda+k+l=:\mu$. Then the order of poles are
\begin{gather*} \begin{cases} \left\lfloor \dfrac{s'}{2}\right\rfloor-\max\{0,\lceil\mu\rceil\} & (\mu\in \BZ), \vspace{1mm}\\
\left\lceil \dfrac{s'}{2}\right\rceil-\max\{0,\lceil\mu\rceil\} & \big(\mu\in \BZ+\frac{1}{2}\big). \end{cases} \end{gather*}

Now we consider the residue of $\cF_{\lambda,k,l}$. Since as a function of $w_{12}$, $\tilde{\Phi}_\bm^{(1)} (x_{11}\overline{w_{12}}x_{22}w_{12}^* )\in\overline{\cP_{2\bm}(\fp^+_1)_{w_{12}}}$ holds, if $f(x_{12})\in\cP_\bk(\fp^+_1)$ and $\bm$ satisfies $2m_j>k_j$ for some $j$, then we have
\begin{gather*} \tilde{\Phi}_\bm^{(1)}
\left(\frac{1}{4}x_{11}\frac{\partial}{\partial x_{12}}x_{22}{\vphantom{\biggl(}}^t\!\!\left(\frac{\partial}{\partial x_{12}}\right)\right)f(x_{12})=0. \end{gather*}
Therefore, if $s'\ge 2$, $2\mu\le s'-2$ is even and $f\in M_{2i+2}^{\fg_1}(2\mu)$, then it holds that
\begin{gather*}
 (\cF_{\lambda,k,l}f)\begin{pmatrix}x_{11}&x_{12}\\{}^t\hspace{-1pt}x_{12}&x_{22}\end{pmatrix}=\det(x_{11})^k\det(x_{22})^l\\
\qquad{} \times\sum_{\substack{\bm\in\BZ_{++}^{s'}\\ m_{2i+2}\le i-\mu}}\frac{1}{\big(\lambda+k+l+\frac{1}{2}\big)_{\bm,1}}\tilde{\Phi}_\bm^{(1)}
\left(\frac{1}{4}x_{11}\frac{\partial}{\partial x_{12}}x_{22}{\vphantom{\biggl(}}^t\!\!\left(\frac{\partial}{\partial x_{12}}\right)\right)
f(x_{12}),
\end{gather*}
and if $s'\ge 1$, $2\mu\le s'-2$ is odd and $f\in M_{2i+1}^{\fg_1}(2\mu)$, then it holds that
\begin{gather*}
 (\cF_{\lambda,k,l}f)\begin{pmatrix}x_{11}&x_{12}\\{}^t\hspace{-1pt}x_{12}&x_{22}\end{pmatrix}=\det(x_{11})^k\det(x_{22})^l\\
\qquad{}\times\sum_{\substack{\bm\in\BZ_{++}^{s'}\\ m_{2i+1}\le i-\lceil\mu\rceil}}\frac{1}{\big(\lambda+k+l+\frac{1}{2}\big)_{\bm,1}}\tilde{\Phi}_\bm^{(1)}
\left(\frac{1}{4}x_{11}\frac{\partial}{\partial x_{12}}x_{22}{\vphantom{\biggl(}}^t\!\!\left(\frac{\partial}{\partial x_{12}}\right)\right) f(x_{12}),
\end{gather*}
These have a pole of order $i-\max\{0,\lceil \mu\rceil\}$ at $\lambda+k+l=\mu$. Therefore if $\mu\in\BZ$, then
\begin{gather*}
\big(\tilde{\cF}^i_{\lambda,k,l}f\big)\begin{pmatrix}x_{11}&x_{12}\\x_{21}&x_{22}\end{pmatrix}
:=\lim_{\nu\to\lambda}(\nu-\lambda)^{i-\max\{0,\lceil \mu\rceil\}}(\cF_{\nu,k,l}f)\begin{pmatrix}x_{11}&x_{12}\\x_{21}&x_{22}\end{pmatrix}\\
\qquad{}=\lim_{\nu\to\lambda}\det(x_{12})^k\det(x_{21})^l \\
\qquad\quad{}\times \sum_{\substack{\bm\in\BZ_{++}^{s'}\\ m_{2i+2}\le i-\mu}}\frac{(\nu-\lambda)^{i-\max\{0,\lceil \mu\rceil\}}}{\big(\nu+k+l+\frac{1}{2}\big)_{\bm,1}}
\tilde{\Phi}_\bm^{(1)}\left(\frac{1}{4}x_{11}\frac{\partial}{\partial x_{12}}x_{22}{\vphantom{\biggl(}}^t\!\!\left(\frac{\partial}{\partial x_{12}}\right)\right)f(x_{12})
\end{gather*}
is well-defined on $M_{2i+2}^{\fg_1}(2\mu)$, and is trivial on $M_{2i}^{\fg_1}(2\mu)$. That is, the linear map
\begin{gather*} \tilde{\cF}^i_{\lambda,k,l}\colon \ M_{2i+2}^{\fg_1}(2\mu)/ M_{2i}^{\fg_1}(2\mu)\longrightarrow \cO_\lambda(D)_{\tilde{K}} \end{gather*}
is well-defined. Moreover, if $i=\big\lfloor \frac{s'}{2}\big\rfloor$, then
\begin{gather*} \tilde{\cF}^{\lfloor s'/2\rfloor}_{\lambda,k,l}\colon \ \cO_{2\mu}(D_1)_{\tilde{K}_1}
/M_{2\lfloor \frac{s'}{2}\rfloor}^{\fg_1}(2\mu)\longrightarrow \cO_\lambda(D)_{\tilde{K}} \end{gather*}
is clearly intertwining since $(\nu-\lambda)^{\lfloor s'/2\rfloor-\max\{0,\lceil \mu\rceil\}}\cF_{\nu,k,l}$ does not have a pole at $\nu=\lambda$, and even if $i<\big\lfloor \frac{s'}{2}\big\rfloor$,
since ${\rm d}\rho_{(\nu+2k)+(\nu+2l)}(g_1)M_{2i+1}^{\fg_1}(2\mu)\subset M_{2i+2}^{\fg_1}(2\mu)$ holds for any $g_1\in\fg_1^\BC$ and for generic $\nu$, the restriction
\begin{gather*} \tilde{\cF}^i_{\lambda,k,l}\colon \ M_{2i+1}^{\fg_1}(2\mu)/ M_{2i}^{\fg_1}(2\mu)\longrightarrow \cO_\lambda(D)_{\tilde{K}} \end{gather*}
is also intertwining. Similarly, if $\mu\in\BZ+\frac{1}{2}$, then
\begin{gather*}
\big(\tilde{\cF}^i_{\lambda,k,l}f\big)\begin{pmatrix}x_{11}&x_{12}\\x_{21}&x_{22}\end{pmatrix}
:=\lim_{\nu\to\lambda}(\nu-\lambda)^{i-\max\{0,\lceil \mu\rceil\}}(\cF_{\nu,k,l}f)\begin{pmatrix}x_{11}&x_{12}\\x_{21}&x_{22}\end{pmatrix}\\
\qquad{} =\lim_{\nu\to\lambda}\det(x_{12})^k\det(x_{21})^l \\
\qquad\quad{} \times \sum_{\substack{\bm\in\BZ_{++}^{s'}\\ m_{2i+1}\le i-\lceil\mu\rceil}}
\frac{(\nu-\lambda)^{i-\max\{0,\lceil \mu\rceil\}}}{\big(\nu+k+l+\frac{1}{2}\big)_{\bm,1}}
\tilde{\Phi}_\bm^{(1)} \left(\frac{1}{4}x_{11}\frac{\partial}{\partial x_{12}}x_{22}{\vphantom{\biggl(}}^t\!\!\left(\frac{\partial}{\partial x_{12}}\right)\right)f(x_{12})
\end{gather*}
is well-defined on $M_{2i+1}^{\fg_1}(2\mu)/M_{2i-1}^{\fg_1}(2\mu)$, and is intertwining if $i=\big\lceil \frac{s'}{2}\big\rceil$ or if it is restricted on $M_{2i}^{\fg_1}(2\mu)/M_{2i-1}^{\fg_1}(2\mu)$.
Similar phenomena also occur for $(G,G_1)=(\operatorname{SO}^*(2s),U(s',s''))$, $(\operatorname{SO}_0(2,n),\operatorname{SO}_0(2,n')\times \operatorname{SO}(n''))$ ($n'$: even), that is, a residue of $\cF_{\tau\rho}$ gives a well-defined map from some submodule $M\subset\cO(D_1)_{\tilde{K}_1}$, and is trivial on the 2nd smaller submodule, and moreover, the restricted map to the 1st smaller submodule than $M$ is intertwining. On the other hand, for $(G,G_1)=(\operatorname{SU}(s,s),\operatorname{Sp}(s,\BR))$, $(\operatorname{SU}(s,s),\operatorname{SO}^*(2s))$, $(\operatorname{SO}_0(2,n),\operatorname{SO}_0(2,n')\times \operatorname{SO}(n''))$ ($n'$: odd), a residue of $\cF_{\tau\rho}$ gives a well-defined map from some submodule $M\subset\cO(D_1)_{\tilde{K}_1}$, and is trivial on the 1st smaller submodule. In this case this map is not intertwining except for the one from (the quotient module of) whole $\cO(D_1)_{\tilde{K}_1}$.

\begin{Theorem}\quad
\begin{enumerate}\itemsep=0pt
\item[$(1)$] Let $(G,G_1)=(\operatorname{Sp}(s,\BR),U(s',s''))$ with $s=s'+s''$, $s'\le s''$. Let $k,l\in\BZ_{\ge 0}$. We assume $\mu:=\lambda+k+l\in\frac{1}{2}\BZ$, $\mu\le \frac{1}{2}(s'-2)$. Then for $s'\ge 2$, for $\mu\in\BZ$, $\max\{0,\lceil\mu\rceil\}\le i\le \big\lfloor \frac{s'}{2}\big\rfloor$,
\begin{gather*} \tilde{\cF}^i_{\lambda,k,l}=\lim_{\nu\to\lambda}(\nu-\lambda)^{i-\max\{0,\lceil \mu\rceil\}}\cF_{\nu,k,l}
\colon \ M_{2i+2}^{\fg_1}(2\mu)/M_{2i}^{\fg_1}(2\mu)\to \cO_\lambda(D)_{\tilde{K}}, \end{gather*}
and for $s'\ge 1$, for $\mu\in\BZ+\frac{1}{2}$, $\max\{0,\lceil\mu\rceil\}\le i\le \big\lceil \frac{s'}{2}\big\rceil$,
\begin{gather*} \tilde{\cF}^i_{\lambda,k,l}=\lim_{\nu\to\lambda}(\nu-\lambda)^{i-\max\{0,\lceil \mu\rceil\}}\cF_{\nu,k,l}
\colon \ M_{2i+1}^{\fg_1}(2\mu)/M_{2i-1}^{\fg_1}(2\mu)\to \cO_\lambda(D)_{\tilde{K}} \end{gather*}
are well-defined $(\cF_{\lambda,k,l}$ is as~\eqref{Sp-U}$)$. Moreover,
\begin{alignat*}{3}
&\tilde{\cF}^{\lfloor s'/2\rfloor}_{\lambda,k,l}\colon \ \cO_{2\mu}(D_1)_{\tilde{K}_1}
/M_{2\lfloor \frac{s'}{2}\rfloor}^{\fg_1}(2\mu)\longrightarrow \cO_\lambda(D)_{\tilde{K}}\quad && (\mu\in\BZ), &\\
&\tilde{\cF}^{\lceil s'/2\rceil}_{\lambda,k,l}\colon \ \cO_{2\mu}(D_1)_{\tilde{K}_1}
/M_{2\lceil \frac{s'}{2}\rceil-1}^{\fg_1}(2\mu)\longrightarrow \cO_\lambda(D)_{\tilde{K}} \quad & &\big(\mu\in\BZ+\tfrac{1}{2}\big),&
\end{alignat*}
and the restriction
\begin{alignat*}{3}
&\tilde{\cF}^i_{\lambda,k,l}\colon \ M_{2i+1}^{\fg_1}(2\mu)/M_{2i}^{\fg_1}(2\mu)\to \cO_\lambda(D)_{\tilde{K}} \quad && (\mu\in\BZ),& \\
&\tilde{\cF}^i_{\lambda,k,l}\colon \ M_{2i}^{\fg_1}(2\mu)/M_{2i-1}^{\fg_1}(2\mu)\to \cO_\lambda(D)_{\tilde{K}} \quad && \big(\mu\in\BZ+\tfrac{1}{2}\big)&
\end{alignat*}
intertwine the $\big(\fg_1,\tilde{K}_1\big)$-action.
\item[$(2)$] Let $(G,G_1)=(\operatorname{SO}^*(2s),U(s',s''))$ with $s=s'+s''$, $2\le s'\le s''$.
Let $k,l\in\BZ_{\ge 0}$ ($k=0$ if $s'$ is odd, $l=0$ if $s''$ is odd).
We assume $\mu:=\lambda+k+l\in\BZ$, $\mu\le 2\big\lfloor \frac{s'}{2}\big\rfloor-1$.
Let $\max\big\{0,\big\lfloor \frac{\mu}{2}\big\rfloor\big\}\le i\le \big\lfloor\frac{s'}{2}\big\rfloor$. Then
\begin{gather*} \tilde{\cF}_{\lambda,k,l}^i:=\lim_{\nu\to\lambda} (\nu-\lambda)^{i-\max\left\{0,\left\lfloor\frac{\mu}{2}\right\rfloor\right\}}\cF_{\nu,k,l}
\colon \ M_{2i+2}^{\fg_1}(\mu)/M_{2i}^{\fg_1}(\mu)\longrightarrow \cO_\lambda(D)_{\tilde{K}} \end{gather*}
is well-defined $(\cF_{\lambda,k,l}$ is as \eqref{SO*-U}$)$. Moreover,
\begin{gather*} \tilde{\cF}_{\lambda,k,l}^{\lfloor s'/2\rfloor}\colon \
\cO_\mu(D_1)_{\tilde{K}_1}/M_{2\left\lfloor\frac{s'}{2}\right\rfloor}^{\fg_1}(\mu)\longrightarrow \cO_\lambda(D)_{\tilde{K}}, \end{gather*}
and the restriction
\begin{gather*} \tilde{\cF}_{\lambda,k,l}^i\colon \ M_{2i+1}^{\fg_1}(\mu)/M_{2i}^{\fg_1}(\mu)\longrightarrow \cO_\lambda(D)_{\tilde{K}}
\qquad \left(\max\left\{0,\left\lceil\frac{\mu}{2}\right\rceil\right\}\le i\le \left\lfloor\frac{s'}{2}\right\rfloor\right) \end{gather*}
intertwine the $(\fg_1,\tilde{K}_1)$-action. $($We note that $M_{2\left\lfloor\frac{\mu}{2}\right\rfloor+1}^{\fg_1}(\mu)=\{0\}$
for $\mu=1,3,\ldots,\allowbreak 2\big\lfloor \frac{s'}{2}\big\rfloor-1$,
and hence the restriction of $\tilde{\cF}_{\lambda,k,l}^{\lfloor \mu/2\rfloor}$ is trivial for these $\mu.)$
\item[$(3)$] Let $(G,G_1)=(\operatorname{SU}(s,s),\operatorname{Sp}(s,\BR))$ with $s\ge 2$. Let $k\in\BZ_{\ge 0}$ ($k=0$ if $s$ is odd).
We assume $\mu:=\lambda+k\in\BZ+\frac{1}{2}$, $\mu\le \big\lfloor \frac{s}{2}\big\rfloor-\frac{1}{2}$.
Let $\max\big\{0,\mu-\frac{1}{2}\big\}\le i\le \big\lfloor\frac{s}{2}\big\rfloor$. Then
\begin{gather*} \tilde{\cF}_{\lambda,k}^i:=\lim_{\nu\to\lambda} (\nu-\lambda)^{i-\max\left\{0,\mu-\frac{1}{2}\right\}}\cF_{\nu,k}
\colon \ M_{2i+2}^{\fg_1}(\mu)/M_{2i}^{\fg_1}(\mu)\longrightarrow \cO_\lambda(D)_{\tilde{K}} \end{gather*}
is well-defined $(\cF_{\lambda,k}$ is as \eqref{SU-Sp}$)$. Moreover,
\begin{gather*} \tilde{\cF}_{\lambda,k,l}^{\lfloor s/2\rfloor}\colon \
\cO_\mu(D_1)_{\tilde{K}_1}/M_{2\left\lfloor\frac{s}{2}\right\rfloor}^{\fg_1}(\mu)\longrightarrow \cO_\lambda(D)_{\tilde{K}} \end{gather*}
intertwines the $\big(\fg_1,\tilde{K}_1\big)$-action.
\item[$(4)$] Let $(G,G_1)=(\operatorname{SU}(s,s),\operatorname{SO}^*(2s))$ with $s\ge 2$. $k\in\BZ_{\ge 0}$.
We assume $\mu:=\lambda+2k\in\BZ+\frac{1}{2}$, $\mu\le \big\lfloor \frac{s}{2}\big\rfloor-\frac{3}{2}$.
Let $\max\big\{0,\mu+\frac{1}{2}\big\}\le i\le \big\lfloor\frac{s}{2}\big\rfloor$. Then
\begin{gather*} \tilde{\cF}_{\lambda,k}^i:=\lim_{\nu\to\lambda} (\nu-\lambda)^{i-\max\left\{0,\mu+\frac{1}{2}\right\}}\cF_{\nu,k}
\colon \ M_{i+1}^{\fg_1}(2\mu)/M_{i}^{\fg_1}(2\mu)\longrightarrow \cO_\lambda(D)_{\tilde{K}} \end{gather*}
is well-defined $(\cF_{\lambda,k}$ is as \eqref{SU-SO*}$)$. Moreover,
\begin{gather*} \tilde{\cF}_{\lambda,k,l}^{\lfloor s/2\rfloor}\colon \
\cO_{2\mu}(D_1)_{\tilde{K}_1}/M_{\left\lfloor\frac{s}{2}\right\rfloor}^{\fg_1}(2\mu)\longrightarrow \cO_\lambda(D)_{\tilde{K}} \end{gather*}
intertwines the $\big(\fg_1,\tilde{K}_1\big)$-action.
\item[$(5)$] Let $(G,G_1)=(\operatorname{SO}_0(2,n),\operatorname{SO}_0(2,n')\times \operatorname{SO}(n''))$ with $n=n'+n''$, $n'\ne 2$. Let $(k_1,k_2)\in\BZ_{++}^2$.
We assume $\mu:=\lambda+k_1+k_2\le \frac{n'-2}{2}$, $\mu-\frac{n'-2}{2}\in\BZ$. Then
\begin{gather*} \tilde{\cF}_{\lambda,k_1,k_2}^0:=\cF_{\lambda,k_1,k_2}
\colon \ M_2^{\mathfrak{so}(2,n')}(\mu)\boxtimes V_{(k_1-k_2,0,\ldots,0)}^{[n'']\vee}\longrightarrow \cO_\lambda(D)_{\tilde{K}},\\
 \tilde{\cF}_{\lambda,k_1,k_2}^1:=\lim_{\nu\to\lambda} (\nu-\lambda)\cF_{\nu,k_1,k_2}
\colon \ \!\big(\cO_\mu(D_1)_{\tilde{K}_1}\!/M_2^{\mathfrak{so}(2,n')}(\mu)\big)\boxtimes V_{(k_1{-}k_2,0,\ldots,0)}^{[n'']\vee}\!\longrightarrow\! \cO_\lambda(D)_{\tilde{K}}\! \end{gather*}
are well-defined $(\cF_{\lambda,k_1,k_2}$ is as \eqref{SO-SOSO}$)$. Moreover,
\begin{gather*} \tilde{\cF}_{\lambda,k_1,k_2}^1\colon \ \big(\cO_\mu(D_1)_{\tilde{K}_1}/M_2^{\mathfrak{so}(2,n')}(\mu)\big)\boxtimes V_{(k_1-k_2,0,\ldots,0)}^{[n'']\vee}
\longrightarrow \cO_\lambda(D)_{\tilde{K}} \end{gather*}
and the restriction
\begin{gather*} \tilde{\cF}_{\lambda,k_1,k_2}^0\colon \ M_1^{\mathfrak{so}(2,n')}(\mu)\boxtimes V_{(k_1-k_2,0,\ldots,0)}^{[n'']\vee}\longrightarrow \cO_\lambda(D)_{\tilde{K}}
\qquad \left(\mu\le 0, \;n'\colon \text{even}\right) \end{gather*}
intertwine the $\big(\fg_1,\tilde{K}_1\big)$-action.
\end{enumerate}
\end{Theorem}

We can also do similar computation when $G$ is exceptional, or when the representation $\cO_\mu(D_1,W)$ of $\tilde{G}_1$ is not of scalar type,
but we omit the detail.
On the other hand, for $(G,G_1)=(U(q,s),U(q,s')\times U(s''))$, $(\operatorname{SO}^*(2s),\operatorname{SO}^*(2(s-1))\times \operatorname{SO}(2))$, $(\operatorname{SO}^*(2s),U(s-1,1))$,
$(\operatorname{SO}_0(2,2s),U(1,s))$ and $(E_{6(-14)},U(1)\times \operatorname{Spin}_0(2,8))$,
as in Section \ref{sect_normal_derivative}, the intertwining operators $\cO_\mu(D_1)\to\cO_\lambda(D)$ are given by
multiplication operators, and do not depend on the parameter $\lambda$.
Therefore there are no poles, and hence such phenomena do not occur.

\section{Explicit calculation of intertwining operators: remaining case}\label{sect_remaining}
In this section we again set $(G,G_1)=(\operatorname{SU}(s,s),\operatorname{SO}^*(2s))$ with $s=2r+1\ge 2$ odd,
and for $k,l\in\BZ_{\ge 0}$, we determine the $\tilde{G}_1$-intertwining operator
\begin{gather*} \cF_{\lambda,k,l}\colon \ \cO_{2\lambda+4k}\big(D_1,V_{(2l,\dots,2l,0)}^{(s)\vee}\big)\to \cO_\lambda(D). \end{gather*}
Let $\rK_{(l,\ldots,l,0)}^{(1)}(x_2)\in \cP\big(\Sym(s,\BC),\Hom\big(V_{(2l,\dots,2l,0)}^{(s)\vee},\BC\big)\big)$ be the $\tilde{K}^\BC_1$-invariant polynomial in the sense of~(\ref{K-invariance}), and let $\cK_{\bm,-\bl}^{(1,4)}(x_2;y_1) \in\cP\big(\Sym(s,\BC)\times\overline{\Skew(s,\BC)},\Hom\big(V_{(2l,\dots,2l,0)}^{(s)\vee},\BC\big)\big)$ be the orthogonal projection of
${\rm e}^{\frac{1}{2}\tr(x_2y_1^*x_2y_1^*)}\rK_{(l,\ldots,l,0)}^{(1)}(x_2)$ onto
\begin{gather*}
 \left(V_{\substack{2(m_1+l,m_1+l-l_1,m_2+l,m_2+l-l_2,\ldots,\;\\ \hspace{50pt}
m_r+l,m_r+l-l_r,l-l_{r+1})}}^{(s)\vee}\otimes\overline{
V_{\substack{2(m_1+l,m_1+l-l_1,m_2+l,m_2+l-l_2,\ldots,\;\\ \hspace{50pt}
m_r+l,m_r+l-l_r,l-l_{r+1})}}^{(s)\vee}}\right)^{K_1^\BC}\\
\qquad{} \subset \cP_{\substack{(m_1+l,m_1+l-l_1,m_2+l,m_2+l-l_2,\ldots,\;\\ \hspace{45pt}
m_r+l,m_r+l-l_r,l-l_{r+1})}}(\Sym(s,\BC))_{x_2}
\otimes\overline{\cP_{2\bm}(\Skew(s,\BC))_{y_1}\otimes V_{(2l,\ldots,2l,0)}^{(s)\vee}}.
\end{gather*}
Then we proved in Proposition \ref{main_prop} that there exist monic polynomials $\varphi_{\bm,-\bl}(\mu)\in\BC[\mu]$ of degree
$l-l_{r+1}$ such that the intertwining operator is given by
\begin{gather*}
 (\cF_{\lambda,k,l}f)(x_1+x_2)=\det(x_2)^k\sum_{\bm\in\BZ_{++}^r}
\sum_{\substack{\bl\in(\BZ_{\ge 0})^{r+1},\; |\bl|=l\\ 0\le l_j\le m_j-m_{j+1}\\ 0\le l_{r+1}}}\\
 \qquad{}\times\frac{\varphi_{\bm,-\bl}(\lambda+2k)}{\left(\lambda+2k+2l+\frac{1}{2}\right)_{\bm-\bl',2}
\left(\lambda+2k+l-r+1\right)_{l-l_{r+1}}\left(\lambda+2k+l-r+\frac{1}{2}\right)_{l-l_{r+1}}} \\
\qquad{} \times\cK_{\bm,-\bl}^{(1,4)}\left(x_2;\frac{1}{2}\overline{\frac{\partial}{\partial x_1}}\right)f(x_1)
\end{gather*}
$(x_1\in\Skew(s,\BC), x_2\in\Sym(s,\BC))$. Here for $\bl=(l_1,\ldots,l_r,l_{r+1})\in(\BZ_{\ge 0})^{r+1}$, we put $\bl'=(l_1,\ldots,l_r)\in(\BZ_{\ge 0})^r$. In this section we want to prove
\begin{Lemma}\label{remaining_eq}
$\varphi_{\bm,-\bl}(\mu)=\big(\mu+l-r+\tfrac{1}{2}\big)_{l-l_{r+1}}$.
\end{Lemma}
To do this, first we consider the reducibility of
\begin{gather*} \cO_\mu\big(D_1,V_{(l,\ldots,l,0)}^{(s)\vee}\big)_{\tilde{K}_1}\simeq
\cO_{\mu+l}\big(D_1,V_{\left(\frac{l}{2},\ldots,\frac{l}{2},-\frac{l}{2}\right)}^{(s)\vee}\big)_{\tilde{K}_1}. \end{gather*}
By \cite[Theorem 6.2(6)]{N2}, this is reducible if and only if $\mu\in\BZ$, $\mu\le 2r-l-1$ and $\mu\ne 2r-2l-1$.
Moreover, if $\mu=2r-l-i-1$ with $i=0,1,\ldots,l-1$, then
\begin{gather*} M_i^l:=\bigoplus_{\bm\in\BZ_{++}^r}
\bigoplus_{\substack{\bl\in(\BZ_{\ge 0})^{r+1},\; |\bl|=l\\ 0\le l_j\le m_j-m_{j+1}\\ l-l_{r+1}\le i}}
V_{\substack{(m_1+l,m_1+l-l_1,m_2+l,m_2+l-l_2,\ldots,\;\\
\hspace{45pt}m_r+l,m_r+l-l_r,l-l_{r+1})}}^{(s)\vee} \subset \cO_{2r-l-i-1}\big(D_1,V_{(l,\dots,l,0)}^{(s)\vee}\big)_{\tilde{K}_1} \end{gather*}
is an irreducible $(\fg_1,\tilde{K}_1)$-submodule, and the quotient $\cO_{2r-l-i-1}\big(D_1,V_{(l,\dots,l,0)}^{(s)\vee}\big)_{\tilde{K}_1}/M_i^l$ is infinitesimally unitary.

\begin{proof}[Proof of Lemma \ref{remaining_eq}] For $i=0,1,\ldots,l-1$, let $\lambda=-2k-l+r-i-\frac{1}{2}$, and consider the map
\begin{gather*}
\tilde{\cF}_{k,l}^i:=\lim_{\nu\to\lambda}\big(\nu+2k+l-r+\tfrac{1}{2}\big)_{i+1}
\cF_{\nu,k,l}\colon \cO_{2r-2l-2i-1}\big(D_1,V_{(2l,\dots,2l,0)}^{(s)}\big)_{\tilde{K}_1}\to \cO_\lambda(D)_{\tilde{K}}, \\
(\tilde{\cF}_{k,l}^if)(x_1+x_2)=\det(x_2)^k\sum_{\bm\in\BZ_{++}^{\lfloor s/2\rfloor}}
\sum_{\substack{\bl\in(\BZ_{\ge 0})^{r+1},\; |\bl|=l\\ 0\le l_j\le m_j-m_{j+1}\\ l-l_{r+1}\ge i+1}}\\
\qquad{}\times\frac{\varphi_{\bm,-\bl}\left(-l+r-i-\frac{1}{2}\right)}{\left(r-i+l\right)_{\bm-\bl',2}
\left(\frac{1}{2}-i\right)_{l-l_{r+1}}(1)_{l-l_{r+1}-i-1}}
\cK_{\bm,-\bl}^{(1,4)}\left(x_2;\frac{1}{2}\overline{\frac{\partial}{\partial x_1}}\right)f(x_1).
\end{gather*}
This is well-defined, and intertwines the $\tilde{G}_1$-action. Then since
$\cK_{\bm,-\bl}^{(1,4)}\left(x_2;\frac{1}{2}\overline{\frac{\partial}{\partial x_1}}\right)f(x_1)=0$ holds if
$f\in V_{(n_1+2l,n_1+2l-k_1,n_2+2l,n_2+2l-k_2,\ldots,n_r+2l,n_r+2l-k_r,2l-k_{r+1})}^{(s)\vee}$
with $n_j<2m_j$, $n_j-k_j<2m_j-2l_j$, or $2l-k_{r+1}<2l-2l_{r+1}$, we have
\begin{gather*}
\operatorname{Ker}\tilde{\cF}_{k,l}^i \supset \bigoplus_{\bn\in\BZ_{++}^r}
\bigoplus_{\substack{\bk\in(\BZ_{\ge 0})^{r+1},\; |\bk|=2l\\ 0\le k_j\le n_j-n_{j+1}\\ 2l-k_{r+1}\le 2i+1}}
V_{\substack{(n_1+2l,n_1+2l-k_1,n_2+2l,n_2+2l-k_2,\ldots,\;\\
\hspace{45pt}n_r+2l,n_r+2l-k_r,2l-k_{r+1})}}^{(s)\vee} \supsetneq M_{2i}^{2l}.
\end{gather*}
Now we prove that $\operatorname{Ker}\tilde{\cF}_{k,l}^i=\cO_{2r-2l-2i-1}\big(D_1,V_{(2l,\dots,2l,0)}^{(s)\vee}\big)_{\tilde{K}_1}$ holds. Let $M\subset\cO_{2r-2l-2i-1}\big(D_1,\allowbreak V_{(2l,\dots,2l,0)}^{(s)\vee}\big)_{\tilde{K}_1}/M_{2i}^{2l}$ be any irreducible $\big(\fg_1,\tilde{K}_1\big)$-module, and $\hat{M}\subset \cO_{2r-2l-2i-1}\big(D_1,V_{(2l,\dots,2l,0)}^{(s)\vee}\big)_{\tilde{K}_1}$ be the preimage of~$M$. Then $M$ contains a $\tilde{K}_1$-type which is annihilated by the quotient map of ${\rm d}\rho_{2r-2l-2i-1}(\fp^+_1)$. That is, the preimage $\hat{M}$ contains a $\tilde{K}_1$-type $V_M$ such that
\begin{gather*} V_M\not\subset M_{2i}^{2l}, \qquad {\rm d}\rho_{2r-2l-2i-1}(\fp^+_1)V_M\subset M_{2i}^{2l}. \end{gather*}
Then since the action ${\rm d}\rho_{2r-2l-2i-1}(\fp^+_1)$ is given by 1st order differential operators of constant coefficients, in general we have
\begin{gather*}
 {\rm d}\rho_{2r-2l-2i-1}(\fp^+_1)V_{(n_1+2l,n_1+2l-k_1,n_2+2l,n_2+2l-k_2,\ldots,n_r+2l,n_r+2l-k_r,2l-k_{r+1})}^{(s)\vee} \\
 \qquad{} \subset V_{(n_1+2l,n_1+2l-k_1,n_2+2l,n_2+2l-k_2,\ldots,n_r+2l,n_r+2l-k_r,2l-k_{r+1})}^{(s)\vee}\otimes V_{(0,\ldots,0,-1,-1)}^{(s)\vee} \\
 \qquad{} =\bigoplus_{1\le i<j\le s}V_{(n_1+2l,n_1+2l-k_1,n_2+2l,n_2+2l-k_2,\ldots,n_r+2l,n_r+2l-k_r,2l-k_{r+1})-\be_{ij}}^{(s)\vee}
\end{gather*}
holds, where $\be_{ij}=\big(0,\ldots,\overset{i\text{-th}}{\check{1}},\ldots,\overset{j\text{-th}}{\check{1}},\ldots,0\big)$,
and this is non-zero unless $\bn=(0,\ldots,0)$, $\bk=(0,\ldots,0,-2l)$. Thus $V_M$ must be of the form
\begin{gather*} V_M=V_{(n_1+2l,n_1+2l-k_1,n_2+2l,n_2+2l-k_2,\ldots,n_r+2l,n_r+2l-k_r,2l-k_{r+1})}^{(s)\vee} \end{gather*}
for some $\bn$, $\bk$ with $2l-k_{r+1}=2i+1$, and hence $V_M\subset\operatorname{Ker}\tilde{\cF}_{k,l}^i$ holds. Therefore we have $\big(\operatorname{Ker}\tilde{\cF}_{k,l}^i\big)/M_{2i}^{2l}\cap M\ne \{0\}$,
and by the irreducibility of $M$, we get $\operatorname{Ker}\tilde{\cF}_{k,l}^i\supset \hat{M}$. Since $\cO_{2r-2l-2i-1}\big(D_1,V_{(2l,\dots,2l,0)}^{(s)\vee}\big)_{\tilde{K}_1}/M_{2i}^{2l}$ is infinitesimally unitary,
this is completely reducible, and any $\tilde{K}_1$-type of this module is contained in some irreducible submodule. Therefore we have $\operatorname{Ker}\tilde{\cF}_{k,l}^i=\cO_{2r-2l-2i-1}\big(D_1,V_{(2l,\dots,2l,0)}^{(s)\vee}\big)_{\tilde{K}_1}$, that is, $\tilde{\cF}_{k,l}^i=0$ holds. Thus $\varphi_{\bm,-\bl}\left(-l+r-i-\frac{1}{2}\right)\allowbreak=0$ holds if $l-l_{r+1}\ge i+1$. Since $\varphi_{\bm,-\bl}(\mu)$ is monic of degree $l-l_{r+1}$, we get Lemma~\ref{remaining_eq}.
\end{proof}

Hence we have proved the following.
\begin{Theorem}\label{main6}
Let $(G,G_1)=(\operatorname{SU}(s,s), \operatorname{SO}^*(2s))$ with $s\ge 2$ odd, and $k,l\in\BZ_{\ge 0}$. Then the linear map
\begin{gather*}
\cF_{\lambda,k,l}\colon \ \cO_{2\lambda+4k}\big(D_1,V_{(2l,\dots,2l,0)}^{(s)\vee}\big)\to \cO_\lambda(D), \\
(\cF_{\lambda,k,l}f)(x_1+x_2)=\det(x_2)^k\sum_{\bm\in\BZ_{++}^{\lfloor s/2\rfloor}}
\sum_{\substack{\bl\in(\BZ_{\ge 0})^{\lceil s/2\rceil},\; |\bl|=l\\ 0\le l_j\le m_j-m_{j+1}\\ 0\le l_{\lceil s/2\rceil}}}\\
\qquad{} \times\frac{1}{\left(\lambda+2k+2l+\frac{1}{2}\right)_{\bm-\bl',2}
\left(\lambda+2k+l-\left\lfloor\frac{s}{2}\right\rfloor+1\right)_{l-l_{\lceil s/2\rceil}}}
\cK_{\bm,-\bl}^{(1,4)}\left(x_2;\frac{1}{2}\overline{\frac{\partial}{\partial x_1}}\right)f(x_1)
\end{gather*}
$(x_1\in\Skew(s,\BC), x_2\in\Sym(s,\BC))$ intertwines the $\tilde{G}_1$-action. Here for $\bl=(l_1,\ldots,l_{\lfloor s/2\rfloor},\allowbreak l_{\lceil s/2\rceil})\in(\BZ_{\ge 0})^{\lceil s/2\rceil}$, we put $\bl'=(l_1,\ldots,l_{\lfloor s/2\rfloor})\in(\BZ_{\ge 0})^{\lfloor s/2\rfloor}$.
\end{Theorem}

\subsection*{Acknowledgments}
The author would like to thank his supervisor Professor T.~Kobayashi for a lot of helpful advice on this paper. He also thanks Professor H.~Ochiai for helpful advice on the structure of this paper, and his colleagues, especially Dr.~M.~Kitagawa for a lot of helpful discussion. In addition he would like to thank anonymous referees for a lot of helpful suggestion to improve this paper.

\LastPageEnding

\end{document}